\documentclass[a4]{scrartcl}     % onecolumn (ditto)
\usepackage{color}
\usepackage{graphicx}
\usepackage{algorithm}
\usepackage{mathtools}
\mathtoolsset{showonlyrefs}
\usepackage{amsmath}
\usepackage{amssymb}
\usepackage{amsthm}
\usepackage{array}
\usepackage{subfigure}
\usepackage{mathrsfs}
\usepackage[noend]{algpseudocode}
\usepackage{hyperref}

\newtheorem{theorem}{Theorem}[section]
\newtheorem{lemma}[theorem]{Lemma}
\newtheorem{remark}[theorem]{Remark}

\newtheorem{corollary}[theorem]{Corollary}
\newtheorem{proposition}[theorem]{Proposition}

\newcommand{\weakly}{\rightharpoonup}
\newcommand{\N}{\ensuremath{\mathbb{N}}}
\newcommand{\T}{\ensuremath{\mathbb{T}}}
\renewcommand{\S}{\ensuremath{\mathbb{S}}}

\newcommand{\R}{\ensuremath{\mathbb{R}}}

\newcommand{\ii}{\textnormal{i}}
\newcommand{\X}{\mathbb{X}}
\newcommand{\Y}{\mathbb{Y}}
\newcommand{\e}{\textnormal{e}}
\newcommand{\supp}{\textnormal{supp}}

\newcommand{\zb}[1]{\ensuremath{\boldsymbol{#1}}}

\newcommand{\dx}{\,\mathrm{d}}
\newcommand{\tT}{\mathrm{T}}

\newcommand{\uL}{\scriptscriptstyle{L}}

\DeclareMathOperator{\emp}{emp}
\DeclareMathOperator{\atom}{atom}
\DeclareMathOperator{\curve}{curv}
\DeclareMathOperator{\Lcurve}{{\lambda}-curv}
\DeclareMathOperator{\Acurve}{ a-curv}

\DeclareMathOperator{\Lip}{Lip}

\DeclareMathOperator*{\diam}{diam}
\DeclareMathOperator*{\spann}{span}
\DeclareMathOperator{\dist}{dist}

\DeclareMathOperator{\G}{\mathcal{G}}

\DeclareMathOperator{\OOO}{O}
\DeclareMathOperator*{\SO}{SO}

\DeclareMathOperator{\TSP}{TSP}
\DeclareMathOperator{\MST}{MST}

\begin{document}

\title{Curve Based Approximation of Measures on Manifolds by Discrepancy Minimization}

\author{
	Martin Ehler\footnotemark[1]
	\and
	Manuel Gr\"af\footnotemark[2]
	\and
	Sebastian Neumayer\footnotemark[3] 
	\and
	and Gabriele Steidl\footnotemark[3] 
}
\maketitle 

\date{\today}

\footnotetext[3]{Department of Mathematics,
	TU Kaiserslautern,
	Paul-Ehrlich-Str.~31, D-67663 Kaiserslautern, Germany,
	\{neumayer,steidl\}@mathematik.uni-kl.de
} 
\footnotetext[1]{University of Vienna, Department of Mathematics, Vienna, Austria,
	\{martin.ehler\}@univie.ac.at	
}
\footnotetext[2]{Austrian Academy of Sciences, Acoustics Research Institute, Vienna, Austria,
	\{mgraef\}@kfs.oeaw.ac.at
}
\begin{abstract}
	The approximation of probability measures on compact metric spaces and in particular on Riemannian manifolds
	by atomic or empirical ones is a classical task in approximation and complexity theory 
	with a wide range of applications.
	Instead of point measures we are concerned  with the approximation by  measures supported on Lipschitz curves.
	Special attention is paid to push-forward measures of Lebesgue measures on the unit interval by such curves.
	Using the discrepancy as distance between measures, we prove optimal approximation rates in terms of the curve's length and Lipschitz constant.
	Having established the theoretical convergence rates, we are interested in the numerical minimization of the
	discrepancy between a given probability measure and the set of push-forward measures of Lebesgue measures on the unit interval by Lipschitz curves.
	We present numerical examples for measures on the 2- and 3-dimensional torus, the 2-sphere, the rotation group on $\mathbb R^3$ and the Grassmannian of all 2-dimensional linear subspaces of $\R^4$.
	Our algorithm of choice is a conjugate gradient method on these manifolds which incorporates second-order information.
	For efficiently computing the gradients and the Hessians within the algorithm, 
	we approximate the given measures by truncated Fourier series and use fast Fourier transform techniques on these manifolds.
\end{abstract}	

The approximation of probability measures by atomic or empirical ones based on their discrepancies 
is a well examined problem in approximation and complexity theory \cite{Kuipers:1974la,Matousek:2010kb,Nowak:2010rr}
with a wide range of applications, e.g., in the derivation of quadrature rules and in the construction of designs.
Recently, discrepancies  were also used in image processing for dithering \cite{Graf:2013fk,SGBW2010,TSGSW2011}, 
i.e., for representing a gray-value image by a finite number of black dots,
and in generative adversarial networks \cite{DRG2015}. 

Besides discrepancies, Optimal Transport (OT) and in particular Wasserstein distances have emerged as powerful tools to compare probability measures in recent years,
see \cite{CP2019,Villani2003} and the references therein.
In fact, so-called Sinkhorn divergences, which are computationally much easier to handle than OT, 
are known to interpolate between OT and  discrepancies \cite{FSVATP2018}. 
For the sample complexity of Sinkhorn divergences we refer to \cite{GCBCP2018}.
The rates for approximating probability measures by atomic or empirical ones with respect to Wasserstein distances depend on the dimension of the underlying spaces, see \cite{Che18,Kloeckner2012}.
In contrast, approximation rates based on discrepancies can be given independently of the dimension~\cite{Nowak:2010rr}, i.e., they do not suffer from the curse of dimensionality.
Additionally, we should keep in mind that the computation of discrepancies does not involve a minimization problem, which is a major drawback of OT and Sinkhorn divergences.
Moreover, discrepancies admit a simple description in Fourier domain and hence the use of fast Fourier transforms is possible, leading to better scalability than the aforementioned methods. 

Instead of point measures, we are interested in approximations
with respect to measures supported on curves.
More precisely, we consider push-forward measures of probability measures $\omega \in {\mathcal P} ([0,1])$ by Lipschitz curves of bounded speed, with special focus on absolutely continuous measures $\omega = \rho \lambda$ and the
Lebesgue measure $\omega = \lambda$.
In this chapter, we focus on approximation with respect to discrepancies.
For related results on quadrature and approximation on manifolds, 
we refer to \cite{Filbir:2011fk,Grochenig:2015ya,Mhaskar:2010kx,Mhaskar:2017nr} and the references therein.
An approximation model based on the 2-Wasserstein distance was proposed in \cite{LGKW2018}.
That work exploits completely different techniques than ours both in the theoretical and numerical part.
Finally, we want to point out a relation to principal curves which are
used in computer science and graphics for approximating  distributions approximately supported on curves
\cite{HS89,Hauberg15,KKLZ00,Hauberg15,KLO20}.
For the interested reader, we further comment on this direction of research
in Remark~\ref{rem:principalcurves} and in the conclusions.
Next, we want to motivate our framework by numerous potential applications:
\begin{itemize}
	\item In MRI sampling \cite{BCCKW2016,CCKW2014}, it is desirable to construct sampling curves with short sampling times (short curve) and high reconstruction quality.
	Unfortunately, these requirements usually contradict each other and finding a good trade-off is necessary.
	Experiments demonstrating the power of this novel approach on a real-world scanner are presented in \cite{LWCM19}.
	\item For laser engraving \cite{LGKW2018} and 3D printing \cite{CSGC17}, we require nozzle trajectories based on our (continuous) input densities.
	Compared to the approach in \cite{CSGC17}, where points given by Llyod's method are connected as a solution of the TSP (traveling salesman problem), our method jointly selects the points and the corresponding curve.
	This avoids the necessity of solving a TSP, which can be quite costly, although efficient approximations exist. 
	Further, it is not obvious that the fixed initial point approximation is a good starting point for constructing a curve.
	\item The model can be used for wire sculpture creation \cite{AXGT13}. In view of this, our numerical experiment presented in Fig.~\ref{fig:spock} can be interpreted as a building plan for a wire sculpture of the Spock head, namely of a 2D surface.
	Clearly, the approach can be also used to create images similar to TSP Art \cite{KP05}, where images are created from points by solving the corresponding TSP.
	\item In a more manifold related setting, the approach can be used for grand tour computation on $\mathcal{G}_{2,4}$ \cite{Asimov:1985aa}, see also our numerical experiment in Fig.~\ref{fig:uniform_g24}. More technical details are provided in the corresponding section.
\end{itemize}

Our contribution is two-fold.
On the theoretical side, we provide estimates of the approximation rates in terms of the maximal speed of the curve.
First, we prove approximation rates for general probability measures on compact Ahlfors $d$-regular length spaces $\X$.
These spaces include many compact sets in the Euclidean space $\R^d$, e.g., the unit ball or the unit cube as well as $d$-dimensional compact Riemannian manifolds  without boundary.
The basic idea consists in combining the known convergence rates for approximation by atomic measures with cost estimates for the traveling salesman problem.
As for point measures, the approximation rate  $L^{d/(2d-2)} \le L^{-1/2}$ for general $\omega \in {\mathcal P} ([0,1])$
and $L^{d/(3d-2)} \le L^{-1/3}$ for $\omega = \lambda$ in terms of the maximal Lipschitz constant (speed) $L$ of the curves
does not crucially depend on the dimension of $\X$. In particular, the second estimate improves a result given in~\cite{CCKW2017} for the torus.

If the measures fulfill additional smoothness properties, these estimates can be improved on compact, connected, $d$-dimensional Riemannian manifolds without boundary.
Our results are formulated for absolutely continuous measures (with respect to the Riemannian measure) having densities in the Sobolev space $H^s(\X)$, $s> d/2$.
In this setting, the optimal approximation rate becomes roughly speaking $L^{-s/(d-1)}$.
Our proofs rely on a general result of Brandolini et al.~\cite{Brandolini:2014oz} on the quadrature error achievable by integration with respect to a measure that exactly integrates all eigenfunctions of the Laplace--Beltrami with eigenvalues smaller than a fixed number.
Hence, we need to construct measures supported on curves that fulfill the above exactness criterion.
More precisely, we construct such curves for the $d$ dimensional torus $\T^d$, the spheres $\S^d$, the rotation group $\SO(3)$ and the Grassmannian $\mathcal{G}_{2,4}$.

On the numerical side, we are interested in finding (local) minimizers of discrepancies between a given continuous measure and those from the set of push-forward measures of the Lebesgue measure by bounded Lipschitz curves.
This problem is tackled numerically on $\T^2$, $\T^3$, $\S^2$ as well as $\SO(3)$ and $\mathcal{G}_{2,4}$
by switching to the Fourier domain. %We suppose that the given continuous measure is approximated  by a truncated Fourier series.
The minimizers are computed using the method of conjugate gradients (CG) on manifolds, which incorporates second order information in form of a multiplication by the Hessian.
Thanks to the approach in the Fourier domain, the required gradients and the calculations involving the Hessian 
can be performed efficiently by fast Fourier transform techniques at arbitrary nodes on the respective manifolds.
Note that in contrast to our approach, semi-continuous OT minimization relies on Laguerre tessellations \cite{GKL2019}, which
are not available in the required form on the 2-sphere, $\SO(3)$ or $\mathcal{G}_{2,4}$.

This chapter is organized as follows: 
In Section \ref{sec:basics} we give the necessary preliminaries on probability measures.
In particular, we introduce the different sets of measures supported on Lipschitz curves that are used for the approximation.
Note that measures supported on continuous curves of finite length 
can be equivalently characterized by push-forward measures of probability measures by Lipschitz curves.
Section~\ref{sec:discrepancies}
provides the notation on reproducing kernel Hilbert spaces and discrepancies including their representation in the Fourier domain.
Section~\ref{sec:2} contains our estimates of the approximation rates for general given measures and different approximation spaces of measures supported on curves.
Following the usual lines in approximation theory, we are then concerned with the
approximation of absolutely continuous measures with density functions lying in Sobolev spaces.
Our main results on the approximation rates of smoother measures are contained in 
Section \ref{sec:3}, where we distinguish between the approximation 
with respect to the  push-forward of general measures $\omega \in {\mathcal P}[0,1]$,
absolute continuous measures and the Lebesgue measure on $[0,1]$. 
In Section \ref{sec:discretization_dith} we formulate our numerical minimization problem.
Our numerical algorithms of choice are briefly described in Section~\ref{sec:algs}.
For a comprehensive description of the algorithms on the different manifolds, we refer to respective papers.
Section \ref{sec:numerics_dith} contains  numerical results demonstrating the practical feasibility of our findings.
Conclusions are drawn in Section \ref{sec:conclusions_dith}.
Finally, Appendix  \ref{sec:examples} briefly introduces the different manifolds $\X$ used in our numerical examples together with the  Fourier
representation of probability  measures on $\X$.
%------------------------------------------------------------------------
\section{Probability measures and curves} \label{sec:basics}
%------------------------------------------------------------------------
In this section, the basic notation on measure spaces is provided, see \cite{AFP2000,FL2007},
with focus on probability measures supported on curves. At this point, let us assume that 
%\begin{quotation}
%$\X$ is a compact Hausdorff space,
%\end{quotation}
\begin{quotation}
	$\X$ is a compact metric space endowed with a bounded non-negative Borel measure $\sigma_\X \in \mathcal{M} (\X)$
	such that $\supp(\sigma_\X)=\X$.
	Further, we denote the metric by $\dist_\X$. 
\end{quotation}
Additional requirements on $\X$ are added along the way and notations are explained below.
By $\mathcal{B}(\X)$ we denote the Borel $\sigma$-algebra on $\X$ 
and by $\mathcal M(\X)$ the linear space of all finite signed Borel measures on $\X$, 
i.e., the space of all $\mu\colon \mathcal{B}(\X) \rightarrow \mathbb R$ satisfying $\mu(\X) < \infty$ and for any sequence 
$(B_k)_{k \in \N} \subset \mathcal{B}(\X)$ of pairwise disjoint sets the relation
$\mu(\bigcup_{k=1}^\infty B_k) = \sum_{k=1}^\infty \mu(B_k)$.
The \emph{support of a measure} $\mu$ is the closed set 
\[\supp(\mu) \coloneqq \bigl\{ x \in \X: B \subset \X \text{ open, }x \in B  \implies \mu(B) >0\bigr\}.\]
For $\mu \in \mathcal M(\X)$ the total variation measure is defined by
\[
|\mu|(B) \coloneqq \sup \biggl\{ \sum_{k=1}^\infty |\mu(B_k)|:
\bigcup_{k=1}^\infty B_k = B, \, B_k \; \mbox{pairwise disjoint}\biggr\}.
\]
With the norm $\| \mu\|_{\mathcal M} = |\mu|(\X)$ the space $\mathcal M(\X)$ becomes a Banach space.
By ${\mathcal C}(\X)$ we denote the Banach space of continuous real-valued functions on 
$\X$ equipped with the norm $\| \varphi\|_{{\mathcal C}(\X)} \coloneqq \max_{x \in \X} |\varphi(x)|$.
The space $\mathcal M(\X)$ can be identified via Riesz' theorem with the dual space of ${\mathcal C}(\X)$ 
and the weak-$^\ast$ topology on $\mathcal M(\X)$
gives rise to the \emph{weak convergence of measures},
i.e., a sequence $(\mu_k )_k \subset \mathcal M(\X)$
converges \emph{weakly} to $\mu$ and we write $\mu_k \weakly \mu$, if
\begin{equation}
	\lim_{k \to \infty} \int_{\X} \varphi \dx \mu_k = \int_{\X} \varphi \dx \mu \quad \text{for all } \varphi \in {\mathcal  C}(\X).
\end{equation} 
For a non-negative, finite measure $\mu$, let $L^p(\X,\mu)$ be the Banach space (of equivalence classes)
of complex-valued functions with norm
\[\|f\|_{L^p(\X,\mu)} = \left( \int_\X |f|^p \dx \mu \right)^\frac1p < \infty.\]

By $\mathcal P (\X)$ we denote the space of Borel probability measures on $\X$,
i.e., non-negative Borel measures with $\mu(\X) = 1$. This space is \emph{weakly compact}, i.e., compact with respect to the topology of weak convergence.
We are interested in the approximation of measures in $\mathcal P (\X)$
by probability measures supported on points and curves in $\X$. 
To this end, we associate with $x \in \X$ a probability measure $\delta_x$ with values 
$\delta_x(B) = 1$ if $x \in B$ and $\delta_x(B) = 0$ otherwise.

The \emph{atomic probability measures} at $N$ points are defined by
\begin{align}
	\mathcal P_N^{\atom} (\X) &\coloneqq
	\biggl\{\sum_{k=1}^N w_k \delta_{x_k} :  
	(x_k)_{k=1}^N \in \X^N, \,  (w_k)_{k=1}^N \in [0,1]^N,\, \sum_{k=1}^N w_k = 1
	\biggr\}.
\end{align}
In other words, $\mathcal P_N^{\atom} (\X)$ 
is the collection of probability measures, whose support consists of at most $N$ points. Further restriction 
to equal mass distribution leads to the \emph{empirical probability measures} at $N$ points denoted by 
\begin{equation}\label{eq:def empirical}
	\mathcal P_N^{\emp} (\X)\coloneqq\biggl\{ \frac{1}{N}\sum_{k=1}^N \delta_{x_k} :  (x_k)_{k=1}^N \in \X^N\biggr\}.
\end{equation}

In this chapter, we are interested in the approximation by measures having their support on curves.
Let $\mathcal{C}([a,b],\X)$ denote the set of  closed,
continuous curves $\gamma\colon[a,b]\rightarrow \X$.
Although our presented experiments involve solely closed curves, some applications might require open curves. Hence, we want to point out that all of our approximation results still hold without this requirement. Upper bounds would not get worse and we have not used the closedness for the lower bounds on the approximation rates.
%Although our presented experiments involve solely closed curves, some applications might require open curves. Hence, we want to point out that all of our approximation results still hold without this requirement as any open curve can be extended to a closed one with increased length bounded by $\diam(\X)$. Consequently, this additional part of the curve is negligible in asymptotic considerations involving the length.
The \emph{length of} a curve  $\gamma\in\mathcal{C}([a,b],\X) $ is given by
\begin{equation}\label{eq:Laenge}
	\ell(\gamma)\coloneqq \sup_{\substack{a\leq t_0\leq \ldots\leq t_n\leq b\\ n\in\N }} 
	\sum_{k=1}^n \dist_\X \bigl(\gamma(t_k),\gamma(t_{k-1})\bigr).
\end{equation}
If $\ell(\gamma)<\infty$, then $\gamma$ is called \emph{rectifiable}. By reparametrization, see \cite[Thm.~3.2]{Hhjlasz:2003sy}, 
the image of any rectifiable curve in $\mathcal{C}([a,b],\X)$ can be derived from the set of closed \emph{Lipschitz continuous curves}
\begin{align*} \label{eq:lip}
	\Lip(\X) \coloneqq 
	\bigl\{\gamma \in \mathcal{C}([0,1],\X) : \exists L \in \R \; \text{with} \;
	\dist_\X\bigl(\gamma(s),\gamma(t)\bigr)\leq L|s-t|\;\forall s,t\in[0,1]\bigr\}.
\end{align*}
The \emph{speed of a curve} $\gamma \in \Lip(\X)$ is defined a.e.~by the metric derivative
\begin{equation*}
	|\dot \gamma|(t) \coloneqq\lim_{s\rightarrow t} 
	\frac{\dist_\X\bigl(\gamma(s),\gamma(t)\bigr)}{|s-t|},\qquad t\in[0,1],
\end{equation*}
cf.~\cite[Sec.~1.1]{Ambrosio}.
The optimal Lipschitz constant $L=L(\gamma)$ of a curve $\gamma$ is given by
$L(\gamma) = \| \, |\dot \gamma| \, \|_{^\infty([0,1])}$.
For a constant speed curve it holds $L(\gamma) = \ell(\gamma)$.

We aim to approximate measures in $\mathcal{P}(\X)$ from those of the subset 
\begin{equation}\label{eq:def P curve 010}
	\mathcal{P}^{\curve}_L(\X)\coloneqq\bigl\{\nu\in\mathcal{P}(\X)  : 
	\exists \gamma\in \mathcal{C}([a,b],\X),\; \supp(\nu)\subset \gamma([a,b]),\; \ell(\gamma)\leq L\bigr\}.
\end{equation}
This space is quite large and in order to define further meaningful subsets, we derive an equivalent formulation in terms of push-forward measures. 
For $\gamma\in \mathcal{C}([0,1],\X)$, the \emph{push-forward} $\gamma{_*} \omega \in {\mathcal P}(\X)$ 
of a probability measure $\omega \in {\mathcal P}([0,1])$ is defined by $\gamma{_*} \omega(B)\coloneqq\omega (\gamma^{-1} (B))$ 
for $B \in \mathcal B(\X)$.
We directly observe $\supp(\gamma{_*} \omega)=\gamma(\supp(\omega))$.
By the following lemma, $\mathcal{P}^{\curve}_L(\X)$ consists of the push-forward of measures in $\mathcal{P}([0,1])$ by constant speed curves.

\begin{lemma}\label{lem:repa}
	The space $\mathcal{P}^{\curve}_L(\X)$ in \eqref{eq:def P curve 010} is equivalently given by
	\begin{equation}\label{eq:def P curve simpler}
		\mathcal{P}^{\curve}_L(\X)= \bigl\{ \gamma{_*} \omega : \gamma\in\Lip(\X) \text{ has constant speed } L(\gamma)\leq L  ,\; \omega\in\mathcal{P}([0,1]) \bigr\}.
	\end{equation}
\end{lemma}

\begin{proof}
	Let $\nu\in \mathcal{P}^{\curve}_L(\X)$ as in \eqref{eq:def P curve 010}. If $\supp(\nu)$ consists of a single point $x \in \X$ only, 
	then the constant curve $\gamma \equiv x$ pushes forward an arbitrary $\delta_t$ for $t\in[a,b]$, which shows that $\nu$ is contained 
	in \eqref{eq:def P curve simpler}. 
	
	Suppose that $\supp(\nu)$ contains at least two distinct points and let $\gamma\in\mathcal{C}([a,b],\X)$ 
	with $\supp(\nu)\subset \gamma([a,b])$ and $\ell(\gamma)<\infty$.
	According to \cite[Prop.~2.5.9]{Burago:2001hs}, there exists a continuous curve 
	$\tilde \gamma \in \Lip(\X)$ with constant speed $\ell(\gamma)$ 
	and a continuous non-decreasing function $\varphi \colon [a,b] \to [0,1]$ with $\gamma = \tilde \gamma \circ \varphi$.
	Now, define $f\colon \X \to [0,1]$ by $f(x) \coloneqq \min \{\tilde \gamma^{-1}(x)\}$. This function is measurable, 
	since for every $t \in [0,1]$ it holds that 
	\[\bigl\{x \in \X: f(x) \leq t\bigr\} = \bigl\{ x \in \X: \min \{\tilde \gamma^{-1} (x) \} \le t \bigr\} = \tilde \gamma([0,t])\]
	is compact.
	Due to $\supp(\nu)\subset \tilde \gamma([0,1])$, we can define $\omega \coloneqq f{_*}\nu \in \mathcal P([0,1])$.
	By construction, $\omega$ satisfies $\tilde{\gamma}{_*} \omega (B)=\omega(\tilde{\gamma}^{-1}(B)) =\nu(f^{-1} \circ \tilde{\gamma}^{-1}(B))= \nu(B)$ for all $B\in \mathcal{B}(\X)$.
	This concludes the proof.
\end{proof}

The set $\mathcal{P}^{\curve}_L(\X)$  contains $\mathcal P_N^{\atom} (\X)$ if $L$ is sufficiently large compared to $N$ and $\X$ is sufficiently nice, cf.~Section \ref{sec:2}. 
%Indeed, if $\omega \in \mathcal P([0,1])$ is just a point measure, 
%we will see in Section \ref{subsec:measures_curve_omega} by considering the traveling salesman problem that $\mathcal{P}^{\curve}_L(\X)$ contains the point measures
%$\mathcal P_N^{\atom} (\X)$ once $L \sim N^{\frac{d-1}{d} }$.
It is reasonable to ask for more restrictive sets of approximation measures, e.g., when $\omega \in \mathcal P([0,1])$ is assumed to be absolutely continuous.
For the Lebesgue measure $\lambda$ on $[0,1]$, we consider
\begin{equation}\label{eq:LcurveNew}
	\mathcal{P}^{\Acurve}_{L}(\X) 
	\coloneqq
	\bigl\{ \gamma{_*} \omega : \gamma \in \Lip(\X), \; L(\gamma) \leq L, 
	\; \omega = \rho \lambda \in \mathcal P([0,1]), \,  L(\rho) \leq L\bigr\}.
\end{equation}
%Here, the mass distribution along the curve depends on the speed of the curve.
%If the support of the approximated measure $\mu$ consists of (at least) two disconnected components, 
%then the curve may need to rush from one component to the other, which crucially 
%affects the approximation behavior.

In the literature \cite{CCKW2017,LGKW2018}, the special case 
of push-forward of the Lebesgue measure $\omega = \lambda$ on $[0,1]$  by Lipschitz curves
in $\mathbb T^d$ was discussed and successfully used in certain applications \cite{BCCKW2016,CCKW2014}.
Therefore, we also consider approximations from 
\begin{equation}\label{eq:Lcurve}
	\mathcal{P}^{\Lcurve}_L(\X) 
	\coloneqq
	\bigl\{ \gamma{_*}\lambda : \gamma \in \Lip(\X), \; L(\gamma) \leq L \bigr\}.
\end{equation}
It is obvious that our probability spaces related to curves are nested,
\begin{equation*}
	\mathcal{P}^{\Lcurve}_L(\X) \subset \mathcal{P}^{\Acurve}_L(\X) \subset \mathcal{P}^{\curve}_L(\X).
\end{equation*}
Hence, one may expect that establishing good approximation rates is most difficult for $\mathcal{P}^{\Lcurve}_L(\X)$ and easier for $\mathcal{P}^{\curve}_L(\X)$.

%------------------------------------------------------------------------
\section{Discrepancies and RKHS} \label{sec:discrepancies}
%------------------------------------------------------------------------
The aim of this section is to introduce the way we quantify the distance (``discrepancy'') between two probability measures.
To this end, choose a continuous, symmetric function $K\colon \X \times \X \rightarrow \mathbb R$ that is positive definite, i.e., for any finite number $n \in \mathbb N$ 
of points $x_j\in \X$, $j=1,\ldots,n$, the relation
\[
\sum_{i,j=1}^n a_i a_j K(x_i,x_j) \ge 0
\]
is satisfied for all $a_j\in \mathbb R$, $j=1,\ldots,n$. We know by Mercer's theorem \cite{CS2002,Mer1909,Steinwart:2011it} 
that there exists an orthonormal basis $\{\phi_k: k \in \N\}$ of $L^2(\X,\sigma_\X)$ and non-negative coefficients $(\alpha_k)_{k \in \N} \in \ell_1$ such that $K$ has the Fourier expansion
\begin{equation} \label{mercer}
	K(x,y) = \sum_{k=0}^\infty \alpha_k \phi_k(x)\overline{\phi_k(y)}
\end{equation}
with absolute and uniform convergence of the right-hand side.
If $\alpha_k > 0$ for some $k \in \mathbb N_0$, the corresponding function $\phi_k$ is continuous.
Every function $f\in L^2(\X,\sigma_\X)$ has a Fourier expansion
\[
f = \sum_{k=0}^\infty \hat f_k \phi_k, \quad \hat f_k \coloneqq \int_{\X} f \overline{\phi_k} \dx \sigma_\X.
\]
The kernel $K$ gives rise to a \emph{reproducing kernel Hilbert space} (RKHS).
More precisely, the function space
\[
H_{K} (\X) \coloneqq \Bigl\{f \in L^2(\X,\sigma_\X) : \sum_{k=0}^\infty \alpha_k^{-1} |\hat f_k|^2 < \infty \Bigr\}
\]
equipped with the inner product and the corresponding norm
\begin{equation} \label{norm_rkhs}
	\langle f,g \rangle_{H_{K} (\X)} 
	= \sum_{k=0}^\infty \alpha_k^{-1} \hat f_k \overline{\hat g_k}, \qquad \|f\|_{H_{K} (\X)} = \sqrt{\langle f,f \rangle}_{H_{K} (\X)}
\end{equation}
forms a Hilbert space with reproducing kernel, i.e.,
\begin{align}
	K (x,\cdot) \in H_{K} (\X) \qquad\qquad &\mbox{for all} \; x \in \X,\\
	f(x) = \bigl\langle f, K (x,\cdot) \bigr\rangle_{H_{K} (\X)} \qquad &\mbox{for all} \; f \in H_{K} (\X), \; x \in \X.
\end{align}
Note that $f\in H_K(\X)$ implies $\hat{f}_k=0$ if $\alpha_k=0$, in which case we make the convention \smash{$\alpha_k^{-1} \hat f_k=0$} in \eqref{norm_rkhs}. 
The space $H_{K} (\X)$ is the closure of the linear span of $\{ K (x_j,\cdot): x_j \in \X \}$ with respect to the norm \eqref{norm_rkhs}, and $H_{K} (\X)$ 
is continuously embedded in $C(\X)$. In particular, the point evaluations in $H_{K} (\X)$ are continuous.

The \emph{discrepancy} $\mathscr{D}_K(\mu,\nu)$ is defined as the dual norm on $H_{K}(\X)$ of the linear operator $T\colon H_K(\X) \rightarrow \mathbb C$ 
with $\varphi \mapsto \int_{\X} \varphi  \dx (\mu -  \nu)$:
\begin{equation}\label{equiv_1}
	\mathscr{D}_K(\mu,\nu)
	=
	\max_{\| \varphi\|_{ H_{K} (\X) } \le 1}
	\Bigl|\int_{\X} \varphi  \dx (\mu -  \nu) \Bigr| ,
\end{equation}
see \cite{Gnewuch:2012jy,Nowak:2010rr}.
Note that this looks similar to the 1-Wasserstein distance, 
where the space of test functions consists of Lipschitz continuous functions and is larger.
Since
\begin{align}
	\int_\X \varphi \dx \mu = \int_\X \bigl\langle \varphi, K(x,\cdot) \bigr\rangle_{H_K(\X)} \dx \mu(x)
	= \Bigl\langle \varphi, \int_\X K(x,\cdot) \dx \mu(x) \Bigr\rangle_{H_K(\X)},
\end{align}
we obtain by Riesz's representation theorem
$$
\max_{\| \varphi\|_{ H_{K} (\X) } \le 1} \int_{\X} \varphi  \dx \mu = \Bigl\| \int_\X K(x,\cdot) 
\dx \mu(x)\Bigr\|_{H_{K}(\X)},
$$
which yields by Fubini's theorem, \eqref{mercer}, \eqref{norm_rkhs} and symmetry of $K$ that
\begin{align} \label{mercer_1}
	\mathscr{D}^2_K(\mu,\nu)
	=& \iint\limits_{\X\times\X} K \dx\mu \dx\mu - 2\iint\limits_{\X\times\X} K\dx\mu \dx\nu
	+\iint\limits_{\X\times\X} K\dx\nu \dx\nu  
	\\
	=&\sum_{k=0}^\infty \alpha_k  | \hat{\mu}_{k}-\hat{\nu}_{k}  |^2,\label{mercer_2}
\end{align}
where the \emph{Fourier coefficients} of $\mu, \nu\in \mathcal P(\X)$ are well-defined for $k$ with $\alpha_k\neq 0$ by
\begin{equation} \label{fourier_measure}
	\hat{\mu}_k\coloneqq \int_{\X}\overline{\phi_k} \dx\mu,\qquad \hat{\nu}_k\coloneqq \int_{\X}\overline{\phi_k} \dx\nu.
\end{equation}
\begin{remark}
	The Fourier coefficients $\hat{\mu}_{k}$ and $\hat{\nu}_{k}$ depend on both $K$ and $\sigma_\X$, but the identity \eqref{mercer_1} shows that $\mathscr{D}_K(\mu,\nu)$ only depends on $K$. 
	Thus, our approximation rates do not depend on the choice of $\sigma_\X$.
	On the other hand, our numerical algorithms in Section \ref{sec:algs} depend on $\phi_k$ and hence on the choice of $\sigma_\X$. 
\end{remark}

If $\mu_n \weakly \mu$ and $\nu_n \weakly \nu$ as $n\rightarrow \infty$, then also $\mu_n \otimes \nu_n \weakly \mu \otimes \nu$.
Therefore, the continuity of $K$ implies that $\lim_{n \rightarrow \infty} \mathscr{D}_K(\mu_n,\nu_n) = \mathscr{D}_K(\mu,\nu)$, so that $\mathscr{D}_K$ is  
continuous with respect to weak convergence in both arguments. Thus, for any weakly compact subset $P\subset\mathcal{P}(\X)$, the infimum 
\begin{equation*}
	\inf_{\nu\in P} \mathscr{D}_K(\mu,\nu)
\end{equation*}
is actually a minimum. All of the subsets introduced in the previous section are weakly compact. 

\begin{lemma}
	The sets $\mathcal{P}_N^{\atom}(\X)$, $\mathcal{P}_N^{\emp}(\X)$, $\mathcal{P}^{\curve}_L(\X)$, $\mathcal{P}^{\Acurve}_L(\X)$, and $\mathcal{P}^{\Lcurve}_L(\X)$ are weakly compact. 
\end{lemma}

\begin{proof}
	It is well-known that $\mathcal{P}_N^{\atom}(\X)$ and $\mathcal{P}_N^{\emp}(\X) $ are weakly compact. 
	
	We show that $\mathcal{P}^{\curve}_L(\X)$ is weakly compact.
	In view of \eqref{eq:def P curve simpler}, let $(\gamma_k)_{k\in\N}$ be Lipschitz curves with constant speed $L(\gamma_k)\leq L$ 
	and $(\omega_k)_{k\in\N} \subset {\mathcal P}([0,1])$. Since ${\mathcal P}([0,1])$ is weakly compact, 
	we can extract a subsequence $(\omega_{k_j})_{j\in\N}$ with weak limit $\hat \omega \in {\mathcal P}([0,1])$.
	Now, we observe that 
	$
	\dist_\X( \gamma_{k_j} (s), \gamma_{k_j} (t)) \le L |s-t|
	$
	for all $j\in \mathbb N$.
	Since $\X$ is compact, the Arzelà--Ascoli theorem implies that there exists a subsequence of 
	$(\gamma_{k_j})_{j\in\N}$ which converges uniformly towards $\hat \gamma\in \Lip(\X)$ with  $L(\hat{\gamma})\leq L$.
	Then, $\hat \nu\coloneqq \hat{\gamma}{_*}\hat{\omega}$ fulfills $\supp(\hat{\nu})\subset \hat{\gamma}([0,1])$, 
	so that $\hat{\nu}\in {\mathcal P}_{L}^{\curve}(\X)$ by \eqref{eq:def P curve 010}. Thus, $\mathcal{P}^{\curve}_L(\X)$ is weakly compact. 
	
	The proof for $\mathcal{P}^{\Acurve}_L(\X)$ and $\mathcal{P}^{\Lcurve}_L(\X)$ is analogous and hence omitted.
\end{proof}

\begin{remark} \textrm{(Discrepancies and Convolution Kernels)} \label{rem:conv}
	Let $\X = \mathbb T^d \coloneqq \R^d / \mathbb{Z}^d$ be the torus and $h \in \mathcal{C}(\mathbb T^d)$ be a function with Fourier series
	\[
	h(x) = \sum_{k \in \mathbb Z^d} \hat h_k \e^{2 \pi\ii \langle k,x\rangle}, \quad
	\hat h_k \coloneqq \int_{\mathbb T^d} h(x)  \e^{- 2 \pi\ii \langle k,x\rangle} \dx \sigma_{\mathbb{T}^d}(x),
	\]
	which converges in $L^2(\mathbb T^d)$ so that $\sum_k |\hat h_k|^2 < \infty$.
	%	Note that for simplicity, we prefer complex eigenfunctions instead of their real-valued trigonometric counterparts.
	Assume that $\hat h_k \not = 0$ for all $k \in \mathbb Z^d$.
	We consider the special Mercer kernel
	\begin{equation} \label{hi_2}
		K(x,y) \coloneqq \sum_{k \in \mathbb Z^d} |\hat h_k|^2 \e^{2 \pi \ii \langle k,x-y\rangle} = \sum_{k \in \mathbb Z^d} |\hat h_k|^2 \cos \bigl(2 \pi \langle k,x-y\rangle\bigr)
	\end{equation}
	with associated discrepancy $\mathscr{D}_h$ via \eqref{mercer_1},
	i.e., $\phi_k(x) = \e^{2 \pi \ii \langle k,x\rangle}$, $\alpha_k = |\hat h_k|^2$, $k \in \mathbb Z^d$ in~\eqref{mercer}.
	The convolution of $h$ with $\mu \in {\mathcal M}(\mathbb T^d)$ is the function
	$h * \mu \in C(\mathbb T^d)$
	defined by
	\[
	(h * \mu) (x) \coloneqq \int_{\mathbb T^d} h(x-y) \dx \mu(y).
	\]
	By the convolution theorem for Fourier transforms it holds
	$\widehat{(h * \mu)}_k = \hat h_k \hat \mu_k$, $k \in \mathbb Z^d$, and
	we obtain by Parseval's identity for $\mu,\nu \in \mathcal M(\mathbb T^d)$ and \eqref{mercer_2} that
	\begin{align*}
		\|h*(\mu - \nu)\|_{L^2(\mathbb T^d)}^2
		&=
		\bigl\| \bigl( \hat h_k \, (\hat \mu_k - \hat \nu_k) \bigr)_{k \in \mathbb Z^d}\bigr\|_{\ell_2}^2
		= \sum_{k \in \mathbb Z^d} |\hat h_k|^2 |\hat \mu_k - \hat \nu_k|^2
		=
		\mathscr{D}_h^2(\mu,\nu).
	\end{align*}
	In image processing, metrics of this kind were considered in \cite{CCKW2017,FHS2013,TSGSW2011}. 
\end{remark}
\begin{remark}
	\textrm{(Relations to Principal Curves)} \label{rem:principalcurves} 
	A similar concept, sharing the common theme of ``a curve which passes through the \emph{middle} of a distribution'' with the intention of our chapter, is that of principle curves. The notion of principal curves has been developed in a statistical framework and was successfully applied in statistics and machine learning, see \cite{GW13,KKLZ00,KLO20}. 
	The idea is to generalize the concept of principal components with just one direction to so-called self-consistent (principal) curves. In the seminal paper \cite{HS89}, the authors showed that these principal curves $\gamma$ are critical points of the energy functional
	\begin{equation}
		\label{eq:PCEnergy}
		E(\gamma,\mu) = \int_{\mathbb X} \Vert x - \mathrm{proj}_\gamma(x) \Vert^2_2 \mathrm d \mu(x),  		
	\end{equation}
	where $\mu$ is a given probability measure on $\mathbb X$ and $\mathrm{proj}_\gamma(x) = \mathrm{argmin}_{y \in \gamma} \Vert x - y\Vert_2$ is a projection of a point $x \in \mathbb X$ on $\gamma$. This notion has also been generalized to Riemannian manifolds in \cite{Hauberg15}, see also \cite{KLO20} for an application on the sphere. Further investigation of principal curves in the plane, cf.~\cite{DS96}, showed that self-consistent curves are not (local) minimizers, but saddle points of \eqref{eq:PCEnergy}. Moreover, the existence of such curves is established only for certain classes of measures, such as elliptical ones. By additionally constraining the length of curves minimizing~\eqref{eq:PCEnergy}, these unfavorable effects were eliminated, cf.~\cite{KKLZ00}.	In comparison to the objective~\eqref{eq:PCEnergy}, the discrepancy \eqref{mercer_1} averages for fixed $x \in \mathbb X$ the distance encoded by $K$ to any point on $\gamma$, instead of averaging over the squared minimal distances to $\gamma$. 
\end{remark}
%----------------------------------------------------------------------------------------------------
\section{Approximation of general probability measures}\label{sec:2}
%----------------------------------------------------------------------------------------------------
Given $\mu\in\mathcal{P}(\X)$, the estimates\footnote{ We use the symbols 
	$\lesssim$ and $\gtrsim$ to indicate that the corresponding inequalities hold up to a positive constant factor on the respective right-hand side. 
	The notation $\sim$ means that both relations $\lesssim$ and $\gtrsim$ hold. The dependence of the constants on other parameters shall either be explicitly stated or clear from
	the context.}
\begin{equation}\label{eq:DiscConv}
	\min_{\nu \in \mathcal P_N^{\atom}(\X)} \mathscr{D}_K(\mu,\nu) \leq \min_{\nu \in \mathcal P_N^{\emp}(\X)} 
	\mathscr{D}_K(\mu,\nu) \lesssim N^{-\frac12},
\end{equation}
are well-known, cf.~\cite[Cor.~2.8]{Graf:2013zl}. Here, the constant hidden in $\lesssim$ depends on $\X$ and $K$ but is independent of $\mu$ and $N\in\N$. 
In this section, we are interested in approximation rates with respect to measures supported on curves.

Our approximation rates for $\mathcal{P}^{\curve}_L(\X)$ are based on those for $\mathcal{P}_N^{\atom}(\X)$
combined with estimates for the traveling salesman problem (TSP).
Let $\TSP_{\X}(N)$ denote the worst case minimal cost tour in a fully connected graph $G$ of $N$ arbitrary nodes represented by $x_1,\ldots,x_N\in\X$ and edges with cost $\dist_\X(x_i,x_j)$, $i,j=1,\ldots,N$.
Similarly, let $\MST_{\X}(N)$ denote the worst case cost of the minimal spanning tree of $G$.
To derive suitable estimates, we require that $\X$ is \emph{Ahlfors $d$-regular} (sometimes also called Ahlfors-David $d$-regular), i.e., there exists $0<d<\infty$ such that
\begin{equation}\label{eq:ahlfors stuff}
	%A_\X r^d\leq \sigma_\X\bigl(B_r(x)\bigr) \leq B_\X r^d, \quad\text{for all } x\in\X,\quad 0<r\leq \diam(\X),
	\sigma_\X\bigl(B_r(x)\bigr) \sim r^d, \quad\text{for all } x\in\X,\quad 0<r\leq \diam(\X),
\end{equation}
where $B_r(x)=\{y\in\X : \dist_{\X}(x,y)\leq r\}$ and the constants in $\sim$ do not depend on $x$ or $r$.
Note that $d$ is not required to be an integer and turns out to be the Hausdorff dimension.
For $\X$ being the unit cube the following lemma was proved in \cite{SS89}.

\begin{lemma}\label{lemma TSP}
	If $\X$ is a compact Ahlfors $d$-regular metric space, then there is a constant $0<C_{\TSP}<\infty$ depending on $\X$ such that
	\begin{equation}\label{eq:Patom subset Pcurve}
		\TSP_{\X}(N) \leq C_{\TSP} N^{1-\frac{1}{d}}.
	\end{equation}
\end{lemma}

\begin{proof}
	Using \eqref{eq:ahlfors stuff} and the same covering argument as in \cite[Lem.~3.1]{Steele:1988ai}, we see that for every choice $x_1,\ldots,x_N\in \X$, 
	there exist $i\neq j$ such that $\dist_{\X}(x_i,x_j)\lesssim N^{-1/d}$, where the constant depends on $\X$.
	
	Let $S = \{x_1,\ldots,x_N\}$ be an arbitrary selection of $N$ points from $\X$.
	First, we choose $x_i$ and $x_j$ with $\dist_{\X}(x_i,x_j)\leq c N^{-1/d}$.
	Then, we form a minimal spanning tree $T$ of $S \setminus \{x_{i}\}$ and augment the tree by adding the edge between $x_i$ and $x_j$.
	This construction provides us with a spanning tree and hence we can estimate $\MST_{\X}(N) \leq \MST_{\X}(N-1) + c N^{-1/d}$.
	Iterating the argument, we deduce
	\begin{equation*}
		\MST _{\X}(N)\lesssim N^{1-\frac{1}{d}},
	\end{equation*}
	cf.~\cite{SS89}. 
	Finally, the standard relation $\TSP_\X(N) \leq 2 \MST_\X(N)$ for edge costs satisfying the triangular inequality  concludes the proof.
\end{proof} 

To derive a curve in $\X$ from a minimal cost tour in the graph, we require the additional assumption that $\X$ is a \emph{length space}, i.e., a metric space with 
\begin{equation*}
	\dist_\X(x,y)=\inf\bigl\{ \ell(\gamma) : \gamma \text{ a continuous curve that connects $x$ and $y$}\bigr\},
\end{equation*}
cf.~\cite{Bridson:1999qz,Burago:2001hs}. 
Thus, for the rest of this section, we are assuming that 
\begin{quotation}
	$\X$ is a compact Ahlfors $d$-regular length space.
\end{quotation}
In this case, Lemma \ref{lemma TSP} yields the next proposition.
\begin{proposition}\label{prop:ahlfors etc}
	If $\X$ is a compact Ahlfors $d$-regular length space, then $\mathcal{P}^{\atom}_N(\X)\subset \mathcal{P}^{\curve}_{C_{\TSP}  N^{1-1/d}}(\X)$.
\end{proposition}
\begin{proof}
	The Hopf-Rinow Theorem for metric measure spaces, see \cite[Chap.~I.3]{Bridson:1999qz} and \cite[Thm.~2.5.28]{Burago:2001hs}, yields that every pair of points $x,y\in\X$ 
	can be connected by a geodesic, i.e., there is $\gamma\in\Lip(\X)$ with constant speed and $\ell(\gamma|_{[s,t]})=\dist_\X(\gamma(s),\gamma(t))$ for all $0\leq s\leq t\leq 1$. 
	Thus, for any pair $x,y\in\X$, there is a constant speed curve $\gamma_{x,y}\in\Lip(\X)$ of length $\ell(\gamma_{x,y}) = \dist_\X(x, y)$ with $\gamma_{x,y}(0)=x$, $\gamma_{x,y}(1) = y$, 
	cf.~\cite[Rem.~2.5.29]{Burago:2001hs}. For $\mu_N\in\mathcal{P}^{\atom}_N(\X)$, let $\{x_1,\ldots,x_N\}=\supp(\mu_N)$. 
	The minimal cost tour in Lemma \ref{lemma TSP} leads to a curve $\gamma\in\Lip(\X)$, so that $\mu_N=\gamma{_*}\omega\in\mathcal{P}^{\curve}_L(\X)$ 
	for an appropriate measure $\omega\in\mathcal{P}^{\atom}_N([0,1])$.
\end{proof}
Proposition \ref{prop:ahlfors etc} enables us to transfer approximation rates from $\mathcal{P}^{\atom}_N(\X)$ to $\mathcal{P}^{\curve}_L(\X)$.
%The inequality \eqref{eq:Patom subset Pcurve} implies 
%$\mathcal{P}^{\atom}_N(\X)\subset \mathcal{P}^{\curve}_{C_{\TSP}  N^{1-1/d}}(\X)$, so that \eqref{eq:DiscConv} translates into an estimate for $\mathcal{P}^{\curve}_L(\X)$.

\begin{theorem}\label{thm:P curve pure no reg}
	For $\mu\in\mathcal{P}(\X)$, it holds with a constant depending on $\X$ and $K$ that
	\begin{equation}\label{eq:est Pcurve 1}
		\min_{\nu \in \mathcal P^{\curve}_L(\X)} \mathscr{D}_K (\mu,\nu)  \lesssim L^{-\frac{d}{2d-2}}.
	\end{equation}
\end{theorem}
\begin{proof}
	Choose $\alpha = \frac{d-1}{d}$. For $L$ large enough, set 
	$N \coloneqq  \lfloor (L/C_{\TSP})^{\frac{1}{\alpha}} \rfloor \in \mathbb N$, 
	so that we observe $\mathcal{P}^{\atom}_{N}(\X)\subset \mathcal{P}^{\curve}_{L}(\X)$. According to \eqref{eq:DiscConv}, we obtain
	\begin{equation*}
		\min_{\nu\in {\mathcal P}^{\curve}_{L} (\X)}\mathscr{D}_K(\mu,\nu)
		\leq	
		\min_{\nu\in {\mathcal P}^{\atom}_{N} (\X)}\mathscr{D}_K(\mu,\nu) \lesssim  N^{-\frac{1}{2}} \lesssim L^{-\frac{1}{2\alpha}}.\eqno
	\end{equation*}
	%so that \eqref{eq:est Pcurve 1} holds. 
\end{proof}

Next, we derive approximation rates for $\mathcal{P}^{\Acurve}_L(\X) $ and $\mathcal{P}^{\Lcurve}_L(\X)$.

\begin{theorem}
	For $\mu\in\mathcal{P}(\X)$, we have with a constant depending on $\X$ and $K$ that
	\begin{equation}\label{eq:est P L-curve pure}
		\min_{\nu \in \mathcal P^{\Acurve}_L(\X)} \mathscr{D}_K (\mu,\nu)\leq \min_{\nu \in \mathcal P^{\Lcurve}_L(\X)} \mathscr{D}_K (\mu,\nu)  \lesssim L^{-\frac{d}{3d-2}}.
	\end{equation}
\end{theorem}

\begin{proof}
	Let $\alpha = \frac{d-1}{d}$,  $d \ge 2$. 
	For $L$ large enough, set $N \coloneqq  \lfloor L^{\frac{2}{2\alpha + 1}} /\diam(\X) \rfloor \in \mathbb N$.
	By~\eqref{eq:DiscConv}, there is a set of points $\{ x_1, \ldots, x_{N } \} \subset \X$ such that 
	\begin{equation}\label{eq:est 0010}
		\mathscr{D}_K(\mu,\nu_N) \lesssim  N^{-\frac12} \lesssim \,L^{-\frac{1}{2\alpha+1}}, \qquad 
		\nu_N \coloneqq \frac{1}{N}\sum_{j=1}^N \delta_{x_j}.
	\end{equation}
	Let these points be ordered as a solution of the corresponding $\TSP$.
	Set 
	$x_0\coloneqq x_N$ and 
	$\tau_i \coloneqq \dist_\X(x_i,x_{i+1})/L$, $i=0, \ldots, N-1$.
	Note that
	\[N \le  L^{\frac{2}{2\alpha + 1}} /\diam(\X) \le L/\dist_\X(x_i,x_{i+1}),\]
	so that $\tau_i \le N^{-1}$ for all $i=0,\ldots, N -1$.
	We construct a closed curve $\gamma_{\uL} \colon[0,1]\rightarrow\X$ 
	that rests in each $x_i$ for a while and then rushes from $x_i$ to $x_{i+1}$. As in the proof of Proposition \ref{prop:ahlfors etc}, 
	$\X$ being a compact length space enables us to choose $\gamma_i\in\Lip(\X)$ with 
	$\gamma_i(0)=x_i$, $\gamma_i(1) = x_{i+1}$ 
	and $L(\gamma_i) = \dist_\X(x_i, x_{i+1})$.
	For $i=0,\ldots,N_{\uL}-1$, we define
	\begin{equation}
		\gamma_{\uL}(t) \coloneqq 
		\left\{
		\begin{array}{ll}
			x_i& \,\,\mathrm{ for} \; t\in  \left[\frac{i}{N},\frac{i+1}{N}-\tau_i \right),\\[0.5ex]
			\gamma_i\bigl(\frac{1}{\tau_i} \bigl(t-\tfrac{i+1}{N} + \tau_i \bigr) \bigr)
			& \,\,\mathrm{ for} \; t\in \left[\frac{i+1}{N}-\tau_i,\frac{i+1}{N}\right).
		\end{array}
		\right.
	\end{equation}
	By construction, $L(\gamma_{\uL})$ is bounded by $\min_i d(x_i,x_{i+1}) \tau_i^{-1} \leq L$.
	Defining the measure \smash{$\nu \coloneqq (\gamma_{\uL}){_*}\lambda \in {\mathcal P}^{\Lcurve}_{L} (\X)$},
	the related discrepancy can be estimated by
	\begin{align*}
		\mathscr{D}_K(\mu,\nu ) 
		&
		= \sup_{\|\varphi \|_{H_K(\X)}\leq 1 }
		\Big|\int_\X \varphi \dx \mu -\int_0^1 \varphi \circ \gamma_{\uL} \dx \lambda  \Big| \\
		& \leq 
		\mathscr{D}_K(\mu,\nu_N)
		+ 
		\sup_{\|\varphi\|_{H_K(\X)}\leq 1 }  
		\sum_{i=0}^{N-1}  \Big(\tau_i |\varphi(x_i)|
		+
		\Big|\int_{\frac{i+1}{N}-\tau_i}^{\frac{i+1}{N}} \varphi \circ \gamma_{\uL} \dx \lambda\Big|\Big).
	\end{align*}
	The relation \eqref{eq:est 0010} yields $\mathscr{D}_K(\mu,\nu_N)\leq CL^{-\frac{1}{2\alpha+1}}$ with some constant $C>0$.  
	Since for $\varphi \in H_K(\X)$ it holds 
	$\|\varphi \|_{L^\infty(\X)} \leq  C_K\|\varphi\|_{H_K(\X)}$
	with $C_K\coloneqq\sup_{x\in\X} \sqrt{K(x,x)}$,
	we finally obtain by Lemma~\ref{lemma TSP}
	\begin{align*}
		\mathscr{D}_K(\mu,\nu) 
		&\leq 
		C \, L^{-\frac{1}{2\alpha+1}}
		+ 2 \, C_K\sum_{i=0}^{N-1} \tau_i
		\leq  
		C \, L^{-\frac{1}{2\alpha+1}}+ 2C_K \, C_{\TSP}\frac{N^\alpha}{L}\\
		&\leq 
		\bigl(C + 2 C_K \, C_{\TSP}/\diam(\X) \bigr) \, L^{-\frac{1}{2\alpha+1}}.
	\end{align*}
\end{proof}

Note that many compact sets in $\R^d$ are compact Ahlfors $d$-regular length spaces with respect to the Euclidean metric 
and the normalized Lebesgue measure such as the unit ball or the unit cube. Moreover many compact connected manifolds with or without boundary satisfy these conditions. 
All assumptions in this section are indeed satisfied for $d$-dimensional connected, compact Riemannian manifolds without boundary equipped with the Riemannian metric and the normalized Riemannian measure. 
The latter setting is studied in the subsequent section to refine our investigations on approximation rates.

\begin{remark}\label{est_franz}
	For $\X = \mathbb T^d$ with $d\in\N$, the estimate 
	\begin{equation}\label{eq:Weiss ll}
		\min_{\nu \in  \mathcal{P}_L^{\Lcurve}(\X)} \mathscr{D}_K (\mu,\nu) \lesssim  L^{-\frac1d}.
	\end{equation}
	was derived in \cite{CCKW2017} provided that $K$ satisfies an additional Lipschitz condition, where the constant in \eqref{eq:Weiss ll} depends on $d$ and $K$. 
	The rate coincides with our rate in \eqref{eq:est P L-curve pure} for $d = 2$ and is worse for higher dimensions as $\frac{d}{3d-2} > \frac13$ for all $d\geq 3$.
\end{remark}

%----------------------------------------------------------------------------------------------------------------
\section{Approximation of probability measures having Sobolev densities} \label{sec:3}
%----------------------------------------------------------------------------------------------------------------
To study approximation rates in more detail, we follow the standard strategy in approximation theory and take additional smoothness properties into account. 
We shall therefore consider $\mu$ with a density satisfying smoothness requirements.
To define suitable smoothness spaces, we make additional structural assumptions on $\X$. 
Throughout the remaining part of the chapter, we suppose that 
\begin{quotation}
	$\X$ is a $d$-dimensional connected, compact Riemannian manifold without boundary equipped with the \emph{Riemannian metric} $\dist_\X$ and the \emph{normalized Riemannian measure} $\sigma_\X$.
\end{quotation}
In the first part of this section, we introduce the necessary background on Sobolev spaces and derive general lower bounds for the approximation rates.
Then, we focus on upper bounds in the rest of the section.
So far, we only have general upper bounds for $\mathcal{P}^{\curve}_L(\X)$.
In case of the smaller spaces $\mathcal{P}^{\Acurve}_L(\X)$ and $\mathcal{P}^{\Lcurve}_L(\X)$, we have to restrict to special manifolds $\X$ in order to obtain bounds.
For a better overview, all theorems related to approximation rates are named accordingly.
%-------------------------------------------------------
\subsection{Sobolev spaces and lower bounds}
In order to define a smoothness class of functions on $\X$, let $-\Delta$ denote the (negative) Laplace--Beltrami operator on $\X$. 
It is self-adjoint on $L^2(\X,\sigma_\X)$ and has a sequence of positive, non-decreasing eigenvalues 
$(\lambda_k)_{k\in \N}$ (with multiplicities) 
with a corresponding orthonormal complete system of smooth eigenfunctions $\{\phi_k : k\in \N\}$.
Every function $f \in L^2(\X,\sigma_\X)$ has a Fourier expansion
\[
f = \sum_{k=0}^\infty \hat f(k) \phi_k, \qquad \hat f(k) \coloneqq \int_{\X} f \overline{ \phi_k}\dx \sigma_\X.
\]
The \emph{Sobolev space} $H^s(\X)$, $s > 0$, is the set of all functions $f \in L^2(\X,\sigma_\X)$ 
with distributional derivative $(I-\Delta)^{s/2} f \in L^2(\X,\sigma_\X)$ and norm
\[
\|f\|_{H^s(\X)} 
\coloneqq
\| (I-\Delta)^{s/2} f \|_{L^2(\X,\sigma_{\X})}
=
\Big( \sum_{k=0}^\infty (1+ \lambda_k)^{s} |\hat f(k)|^2  \Big)^{\frac{1}{2}}.
\]
For $s>d/2$, the space $H^s(\X)$ is continuously embedded into the space of H\"older continuous functions of degree $s - d/2$, 
and every function $f \in H^s(\X)$ has a uniformly convergent Fourier series, 
see \cite[Thm.~5.7]{Roe1998}. 
Actually, $H^s(\X)$, $s>d/2$, is a RKHS with reproducing kernel 
\begin{equation} \label{sobolev_kernel}
	K(x,y) \coloneqq \sum_{k=0}^\infty (1+\lambda_k)^{-s} \phi_k(x) \overline{\phi_k(y)}.
\end{equation}
Hence, the discrepancy $\mathscr{D}_K(\mu,\nu)$ satisfies \eqref{equiv_1} with $H_K(\X)=H^s(\X)$. 
Clearly, each kernel of the above form with coefficients having the same decay as $(1+\lambda_k)^{-s}$ 
for $k \rightarrow \infty$ gives rise to a RKHS that coincides with $H^s(\X)$ with an equivalent norm. Appendix \ref{sec:examples} 
contains more details of the above discussion for the torus $\mathbb{T}^d$, the sphere $\S^d$, the special orthogonal group $\SO(3)$ and the Grassmannian~$\G_{k,d}$.

Now, we are in the position to establish lower bounds on the approximation rates.
Again, we want to remark that our results still hold if we drop the requirement that the approximating curves are closed.

\begin{theorem}[Lower bound]\label{thm:lower bounds}
	For $s > d/2$ suppose that $H_K(\X) = H^s(\X)$ holds with equivalent norms. 
	Assume that $\mu$ is absolutely continuous with respect to $\sigma_\X$ with a continuous density $\rho$. 
	Then, there are constants depending on $\X$, $K$, and $\rho$ such that 
	\begin{align*}
		N^{-\frac{s}{d}} 
		&\lesssim \min_{\nu \in \mathcal P_N^{\atom}(\X)} \mathscr{D}_K(\mu,\nu)\leq \min_{\nu \in \mathcal P_N^{\emp}(\X)} \mathscr{D}_K (\mu,\nu),\\
		L^{-\frac{s}{d-1}} 
		& \lesssim \min_{\nu \in \mathcal P_L^{\curve}(\X)} \mathscr{D}_K(\mu,\nu)
		\leq \min_{\nu \in \mathcal P_L^{\Acurve}(\X)} \mathscr{D}_K (\mu,\nu)\leq \min_{\nu \in \mathcal P_L^{\Lcurve}(\X)} \mathscr{D}_K (\mu,\nu).
	\end{align*}
\end{theorem}

\begin{proof}
	The proof is based on the construction of a suitable fooling function to be used in \eqref{equiv_1} and follows \cite[Thm.~2.16]{Brandolini:2014oz}.
	There exists a ball $B\subset \X$ with $\rho(x)\geq \epsilon = \epsilon(B,\rho)$ for all $x \in B$ and $\sigma_\X(B)>0$, which is chosen as the support of the constructed fooling functions.
	We shall verify that for every $\nu \in {\mathcal P}_N^{\atom}(\X)$ there exists $\varphi \in H^s(\X)$ such that $\varphi$ vanishes on $\supp(\nu)$ but 
	\begin{equation}\label{eq:fooling 101}
		\int_{B} \varphi \dx \mu\gtrsim \Vert \varphi \Vert_{H^s(\X)} N^{-\frac{s}{d}},
	\end{equation}
	where the constant depends on $\X$, $K$, and $\rho$.
	For small enough $\delta$ we can choose $2N$ disjoint balls in $B$ with diameters $\delta N^{-1/d}$, see also \cite{Gigante:2017pi}. 
	For $\nu \in {\mathcal P}_N^{\atom}(\X)$, there are $N$ of these balls that do not intersect with $\supp(\nu)$. 
	By putting together bump functions supported on each of the $N$ balls, we obtain a non-negative function $\varphi$ supported in $B$ 
	that vanishes on $\supp(\nu)$ and satisfies \eqref{eq:fooling 101}, with a constant that depends on $\epsilon$, cf.~\cite[Thm.~2.16]{Brandolini:2014oz}. This yields
	\begin{equation*}
		\Bigl|\int_{\X} \varphi \dx \mu - \int_{\X} \varphi \dx \nu\Bigr|
		= \int_{B} \varphi \dx \mu\gtrsim \Vert \varphi \Vert_{H^s(\X)} N^{-\frac{s}{d}}.
	\end{equation*}
	
	The inequality for $\mathcal{P}_L^{\curve}(\X)$ is derived in a similar way.
	Given a continuous curve $\gamma \colon [0,1]\rightarrow \X$ of length $ L$, choose $N$ such that $L\leq \delta  N N^{-1/d}$.
	By taking half of the radius of the above balls, there are $2N$ pairwise disjoint balls of radius \smash{$\frac{\delta}{2} N^{-1/d}$} contained in $B$ with pairwise distances at least $\delta N^{-1/d}$.
	Any curve of length $\delta N N^{-1/d}$ intersects at most $N$ of those balls.
	Hence, there are $N$ balls of radius $\frac{\delta}{2}N^{-1/d}$ 
	that do not intersect $\supp(\gamma)$. As above, this yields a fooling function $\varphi$ satisfying \eqref{eq:fooling 101}, which ends the proof.
\end{proof}

%-------------------------------------------------------------------------------------------------------
\subsection{Upper bounds for \texorpdfstring{$\mathcal{P}^{\curve}_L(\X)$}{PC(X)}}

In this section, we derive upper bounds that match the lower bounds in Theorem~\ref{thm:lower bounds} 
for $\mathcal{P}^{\curve}_L(\X)$. 
Our analysis makes use of the following theorem, which was already proved for $\X = \mathbb S^d$ in \cite{HMS2007}.

\begin{theorem}\label{thm:brandolini_2} \cite[Thm.~2.12]{Brandolini:2014oz} 
	Assume that  $\nu_r \in {\mathcal P}(\X)$ provides an exact quadrature for all eigenfunctions $\varphi_k$ 
	of the Laplace--Beltrami operator with eigenvalues $\lambda_k \leq r^2$, i.e.,
	\begin{equation} \label{eq:exact integration}
		\int_{\X} \varphi_k \dx \sigma_\X = \int_{\X} \varphi_k \dx \nu_r.
	\end{equation}
	Then, it holds for every function $f \in H^s(\X)$, $s > d/2$, that there is a constant depending on $\X$ and $s$ with 
	$$
	\Bigl| \int_{\X} f \dx \sigma_\X- \int_{\X} f \dx \nu_r \Bigr| \lesssim r^{-s} \|f\|_{H^s(\X)}.
	$$
\end{theorem}

For our estimates it is important 
that the number of eigenfunctions of the Laplace--Beltrami operator on $\X$ 
belonging to eigenvalues with $\lambda_k \leq r^2$ is of order $r^d$, 
see \cite[Chap.~6.4]{Chavel1984} and \cite[Thm.~17.5.3, Cor.~17.5.8]{Hoermander1983}.
This is known as Weyl's estimates on the spectrum of an elliptic operator.
For some special manifolds, the eigenfunctions are explicitly given in the appendix.
In the following lemma, the result from Theorem~\ref{thm:brandolini_2} is rewritten in terms of discrepancies and generalized to absolutely continuous measures with densities $\rho \in H^s(\X)$. 

\begin{lemma}\label{lemma:GeneralApproximation}
	For $s>d/2$ suppose that $H_K(\X) = H^s(\X)$ holds with equivalent norms and that $\nu_r\in\mathcal{P}(\X)$ satisfies 
	\eqref{eq:exact integration}.
	Let $\mu \in \mathcal P(\X)$ be absolutely continuous with respect to $\sigma_\X$ with density $\rho \in H^s(\X)$. 
	For sufficiently large $r$, the measures \smash{$\tilde \nu_r\coloneqq \frac{\rho}{\beta_r} \nu_r \in\mathcal{P}(\X)$} with
	$\beta_r \coloneqq \int_{\X} \rho\dx\nu_r$ are well defined and there is a constant depending on $\X$ and $K$ with
	\begin{equation}
		\mathscr{D}_K\bigl(\mu,\tilde \nu_r\bigr) \lesssim \|\rho\|_{H^s(\X)}r^{-s}.
	\end{equation}
\end{lemma}

\begin{proof}
	Note that $H^s(\X)$ is a Banach algebra with respect to addition 
	and multiplication \cite{CRT2001},  in particular, for $f,g \in H^s(\X)$ we have $fg \in H^s(\X)$ with
	\begin{equation} \label{banach_alg}
		\|fg\|_{H^s(\X)} \le \|f\|_{H^s(\X)} \, \|g\|_{H^s(\X)}.
	\end{equation}
	By Theorem \ref{thm:brandolini_2}, we obtain for all $\varphi\in H^s(\X)$ that
	\begin{equation}\label{eq:ref for Brandolini mit product}
		\Big|\int_{\X} \varphi\rho\dx \sigma_\X - \int_\X \varphi \rho\dx \nu_r \Big|
		\lesssim r^{-s}  \|\varphi \, \rho\|_{H^s(\X)} \lesssim r^{-s}  \|\varphi\|_{H^s(\X)}\| \rho\|_{H^s(\X)}.
	\end{equation}
	In particular, this implies for $\varphi \equiv 1$ that
	\begin{equation}\label{eq:1}
		\big|1 - \beta_r \big|
		\lesssim r^{-s} \|\rho\|_{H^s(\X)}.
	\end{equation}
	Then, application of the triangle inequality results in
	\begin{align*}
		&\Big|\int_{\X} \varphi \dx \mu - \int_\X \varphi \dx \tilde \nu_r\Big|
		\leq \Big|\int_{\X} \varphi \dx \mu - \int_\X \varphi \rho\dx \nu_r\Big|
		+ \Big|\int_\X  \varphi\rho \tfrac{\beta_r -1}{\beta_r}  \dx \nu_r  \Big|.
	\end{align*}
	According to \eqref{eq:ref for Brandolini mit product}, the first summand is bounded by 
	$\lesssim r^{-s}  \|\varphi\|_{H^s(\X)}\| \rho\|_{H^s(\X)}$. 
	It remains to derive matching bounds on the second term. H\"older's inequality yields 
	\begin{align*}
		\Big|\int_\X  \varphi\rho \tfrac{\beta_r -1}{\beta_r}  \dx \nu_r  \Big| \lesssim \|\varphi \|_{L^\infty(\X)}  \left|\beta_r - 1\right| \lesssim \|\varphi\|_{H^s(\X)} r^{-s}\| \rho\|_{H^s(\X)},
	\end{align*}
	where the last inequality is due to $H^s(\X) \hookrightarrow L^\infty(\X)$ and \eqref{eq:1}.
\end{proof}
Using the previous lemma, we derive optimal approximation rates for $\mathcal{P}_N^{\atom}(\X)$ and $\mathcal{P}^{\curve}_L(\X)$.

\begin{theorem}[Upper bounds]\label{thm:general no constraints on X}
	For $s > d/2$ suppose that $H_K(\X) = H^s(\X)$ holds with equivalent norms. 
	Assume that $\mu$ is absolutely continuous with respect to $\sigma_\X$ with density $\rho \in H^s(\X)$.
	Then, there are constants depending on $\X$ and $K$ such that 
	\begin{align}
		\min_{\nu \in \mathcal P_N^{\atom}(\X)} \mathscr{D}_K(\mu,\nu) &\lesssim \Vert \rho \Vert_{H^s(\X)} N^{-\frac{s}{d}},\label{eq:bound00}\\
		\min_{\nu\in\mathcal{P}^{\curve}_L(\X)}\mathscr{D}_K (\mu,\nu) &\lesssim\|\rho\|_{H^s(\X)}  L^{-\frac{s}{d-1}}.
		\label{eq:improvement 22}	
	\end{align}
\end{theorem}

\begin{proof}
	By \cite[Lem.~2.11]{Brandolini:2014oz} and since the Laplace--Beltrami has $N \sim r^d$ eigenfunctions belonging to eigenvectors $\lambda_k < r^2$, 
	there exists a measure $\nu_r\in \mathcal{P}^{\atom}_{N}(\X)$ that satisfies \eqref{eq:exact integration}.
	Hence, \eqref{eq:exact integration} is satisfied with $r\sim N^{1/d}$, where the constants depend on $\X$ and $K$. 
	Thus, Lemma \ref{lemma:GeneralApproximation} with $\tilde\nu_r\in \mathcal{P}^{\atom}_N(\X)$ leads to \eqref{eq:bound00}. 
	
	The assumptions of Lemma \ref{lemma TSP} are satisfied, so that analogous arguments as in the proof of Theorem \ref{thm:P curve pure no reg} 
	yield $\mathcal{P}^{\atom}_{N}(\X)\subset \mathcal{P}^{\curve}_{L}(\X)$ with suitable $N\sim L^{d/(d-1)}$. Hence, \eqref{eq:bound00} implies \eqref{eq:improvement 22}.
\end{proof}

\subsection{Upper bounds for \texorpdfstring{$\mathcal{P}^{\Acurve}_L(\X)$}{PA(X)} and special manifolds \texorpdfstring{$\X$}{X}}
To establish upper bounds for the smaller space $\mathcal{P}^{\Acurve}_L(\X)$, restriction to special manifolds is necessary. 
The basic idea consists in the construction of a curve and a related measure $\nu_r$ such that all eigenfunctions of the Laplace--Beltrami 
operator belonging to eigenvalues smaller than a certain value are exactly integrated by this measure and then applying Lemma~\ref{lemma:GeneralApproximation} for estimating the minimum of discrepancies.
We begin with the torus.

\begin{theorem}[Torus]\label{thm:new for torus and acurve}
	Let $\X = \mathbb T^d$ with $d\in\N$, $s>d/2$ and suppose that $H_K(\X) = H^s(\X)$ holds with equivalent norms.
	Then, for any absolutely continuous measure $\mu \in {\mathcal P}(\X)$ with Lipschitz continuous density $\rho \in H^s(\X)$, there exists a constant depending on $d$, $K$, and $\rho$ such that  
	\begin{equation}
		\min_{\nu\in\mathcal{P}^{\Acurve}_L(\X)}\mathscr{D}_K(\mu,\nu) \lesssim L^{-\frac{s}{d-1}}.
	\end{equation}
\end{theorem}

\begin{proof}
	1. First, we construct a closed curve $\gamma_r$ such that the trigonometric polynomials 
	from  $\mathrm{\Pi}_r(\mathbb T^{d})$, see \eqref{trig_torus} in the appendix, are exactly integrated along this curve.
	Clearly,  the polynomials in $\mathrm{\Pi}_r(\mathbb T^{d-1})$ are exactly integrated at equispaced nodes 
	$x_{\zb k} 
	= \frac{{\zb k}}{n}$, ${\zb k}=(k_1,\ldots,k_{d-1})  \in \mathbb N_0^{d-1}$, $0 \le k_i \le n-1$,
	with weights $1/n^{d-1}$,
	where $n \coloneqq r+1$.
	Set $z(\zb k) \coloneqq k_1 + k_2 n + \ldots + k_{d-1} n^{d-2}$ and
	consider the curves
	\[\gamma_{\zb k} \colon I_{\zb k} 
	\coloneqq 
	\bigl[ \tfrac{z(\zb k)}{n^{d-1}},\tfrac{z(\zb k)+1}{n^{d-1}}\bigr] \rightarrow \mathbb T^d \quad\text{with}\quad
	\gamma_{\zb k}(t)
	\coloneqq 
	\begin{pmatrix} x_{\zb k}\\ n^{d-1} t\end{pmatrix}.
	\]
	Then, each element in $\mathrm{\Pi}^{d}_r$ is exactly integrated along the union of these curves, i.e., using $I \coloneqq \{0,\ldots,n-1\}^{d-1}$, we have
	\[
	\int_{\mathbb T^d}  p \dx \sigma_{\mathbb{T}^d} 
	= 
	\sum_{\zb k \in  I} \int_{I_k} p \circ \gamma_{\zb k} \dx \lambda ,\quad p\in\mathrm{\Pi}^{d}_r.
	\]
	The argument is repeated for every other coordinate direction, so that we end up with $d n^{d-1}$ curves mapping from an interval of length $\frac{1}{d n^{d-1}}$ to $\mathbb T^d$.
	The intersection points of these curves are considered as vertices of a graph, where each vertex has $2d$ edges.
	Consequently, there exists an Euler path $\gamma_r\colon [0,1] \rightarrow \mathbb T^d$ trough the vertices build from all curves.
	It has constant speed $d n^{d-1}$ and the polynomials $\mathrm{\Pi}^{d}_r$ are exactly integrated along $\gamma_r$, i.e., 
	\begin{equation*}
		\int_{\mathbb{T}^d} p\dx \sigma_{\mathbb{T}^d} = \int_{\mathbb{T}^d} p\dx {\gamma_r}{_*}\lambda,\quad p\in\mathrm{\Pi}^{d}_r.
	\end{equation*}
	
	2. Next, we apply Lemma \ref{lemma:GeneralApproximation} for $\nu_r={\gamma_r}_{*}\lambda$. 
	We observe $\tilde \nu_r ={\gamma_r}{_*}((\rho\circ\gamma_r)/{ \beta_r} \lambda)$ and deduce $L(\rho\circ\gamma_r/\beta_r)\leq L(\gamma_r)L(\rho)/{\beta_r}\lesssim r^{d-1}\sim L$ as $\beta_r\sim 1$.
	Here, constants depend on $d$, $K$, and $\rho$.
	%
	%
	% and $\rho_r \equiv 1$. Since $\rho$ is positive and continuous, there is $\epsilon\geq 0$ such that $\rho\geq \epsilon$, which implies $(\gamma_r^*\rho)\rho_r \geq \epsilon$. Hence, $\tilde \gamma_r = \gamma_r \circ g_r^{-1}$ satisfies $L(\tilde \gamma_r)\lesssim L(\gamma_r)\lesssim N^{d-1}$, so that $\tilde \gamma_r \dx t$ satisfies \eqref{eq:ineq 2 2 2} and is contained in $\mathcal{P}^{\Lcurve}_L(\X)$. 
\end{proof}

Now, we provide approximation rates for $\X=\S^d$.
\begin{theorem}[Sphere]\label{thm:sphere general acurve}
	Let $\X = \mathbb S^d$ with $d \ge 2$, $s>d/2$ and suppose 
	that $H_K(\X) = H^s(\X)$ holds with equivalent norms. 
	Then, we have for any absolutely continuous measure $\mu \in {\mathcal P}(\X)$ with Lipschitz continuous density $\rho \in H^s(\X)$
	that there is a constant depending on $d$, $K$, and $\rho$ with  
	\begin{equation}\label{eq:improvement 3}
		\min_{\nu\in\mathcal{P}^{\Acurve}_L(\X)}\mathscr{D}_K (\mu,\nu) \lesssim L^{-\frac{s}{d-1}}.
	\end{equation}
\end{theorem}
\begin{proof}
	1. First, we construct a constant speed curve $\gamma_{r}\colon[0,1]\to \mathbb S^{d}$ and a probability measure $\omega_r = \rho_r \lambda$
	with Lipschitz continuous density $\rho_{r}\colon [0,1] \rightarrow \mathbb R_{\ge 0}$ such that for all $p \in \mathrm{\Pi}_{r}(\mathbb S^{d})$, it holds
	\begin{equation} 	\label{eq:quad_curve_Sd}
		\int_{\mathbb S^{d}} p \dx \sigma_{\mathbb S^{d}} = \int_{0}^{1} p \circ \gamma_{r} \dx \omega_{r}.
	\end{equation}
	Utilizing spherical coordinates 
	\begin{equation} \label{spherical}
		x_1 = \cos \theta_1, \,\, x_2 = \sin \theta_1 \cos\theta_2, \,\, \ldots \,\, ,  \,\,
		x_{d} = \prod_{j=1}^{d-1} \sin\theta_j  \cos\phi,\,\, 
		x_{d+1} = \prod_{j=1}^{d-1} \sin\theta_j \sin\phi,
	\end{equation}
	where $\theta_k \in [0,\pi]$, $k=1,\ldots,d-1$, and $\phi \in [0,2\pi)$,
	we obtain
	\begin{equation} \label{eq:ausgang}
		\int_{\mathbb S^{d}} p \dx \sigma_{\mathbb S^{d}} 
		= \int_{0}^{\pi} c_{d} \sin(\theta_1)^{d-1} \int_{\mathbb S^{d-1}} p\bigl(\cos(\theta_1),\sin(\theta_1)\tilde x\bigr) \dx \sigma_{\mathbb S^{d-1}}(\tilde x)   \dx \theta_1,
	\end{equation}
	where $c_{d} \coloneqq (\int_{0}^{\pi} \sin(\theta)^{d-1} \dx \theta)^{-1} $. 
	There exist nodes  $\tilde x_{i} \in \mathbb S^{d-1}$ and positive weights $a_{i}$, $i=1,\dots, n \sim r^{d-1}$, 
	with $\sum_{i=1}^n a_i = 1$, such that for all \smash{$p \in \mathrm{\Pi}_{r}(\mathbb S^{d-1})$} it holds
	\[
	\int_{\mathbb S^{d-1}} p \dx \sigma_{\mathbb S^{d-1}} = \sum_{i=1}^{n} a_{i} p(\tilde x_{i}).   
	\]
	To see this, substitute $u_k = \sin \theta_k$, $k=2,\ldots,d-1$, apply Gaussian quadrature with nodes  $\lceil (r+1)/2 \rceil$ and corresponding weights to exactly integrate over $u_k$, and equispaced
	nodes and weights $1/(2r+1)$ for the integration over $\phi$ as, e.g., in \cite{Wagner:1992aa}.
	Then, we define $\gamma_{r}\colon[0,1]\to\mathbb S^{d}$ for $t \in [(i-1)/n,i/n]$, $i=1,\dots,n$, by
	\[
	\gamma_{r}(t) \coloneqq \gamma_{r,i}(2\pi n t) , \qquad \gamma_{r,i}(\alpha) \coloneqq \bigl(\cos(\alpha), \sin(\alpha) \tilde x_{i}\bigr), \quad \alpha \in [0,2\pi].
	\]
	Since $(1,0,\dots,0) = \gamma_{r,i}(0) = \gamma_{r,i}(2\pi)$ for all $i=1,\dots,n$, the curve is closed.
	Furthermore, $\gamma_{r}(t)$ has constant speed since for $i=1,\dots,n$, i.e.,
	\[
	| \dot \gamma_{r}|(t) = | \dot \gamma_{r,i}|(2 \pi n t)  = 2\pi n \sim r^{d-1}.
	\]
	Next, the density $\rho_{r}\colon[0,1]\to \mathbb \R$  is defined for $t \in [(i-1)/n,i/n]$, $i=1,\dots,n$, by
	\[
	\rho_{r}(t) \coloneqq \rho_{r,i}(2\pi n t) , \qquad \rho_{r,i}(\alpha) \coloneqq a_{i} c_{d}  \pi  n|\sin(\alpha)|^{d-1}, \qquad \alpha \in [0,2\pi].
	\]
	We directly verify that $ \rho_{r}$ is Lipschitz continuous with $L(\rho_r) \lesssim \max_{i} a_i n^2$.
	By \cite{Forster:1990iw}, the quadrature weights fulfill
	$a_i \lesssim \frac{1}{r^{d-1}}$
	so that
	$L(\rho_r) \lesssim n^2 r^{-(d-1) } \sim r^{d-1}$. 
	By definition of the constant $c_{d}$ and weights $a_{i}$, we see that $\rho_r$ is indeed a probability density
	\begin{align}
		\int_{0}^{1} \rho_{r} \dx \lambda
		&= \sum_{i=1}^{n} \int_{\frac{i-1}{n}}^{\frac{i}{n}}  \rho_{r,i}(2\pi n t) \dx t =
		\frac{1}{2\pi n}  \sum_{i=1}^{n} \int_{0}^{2\pi}  \rho_{r,i}(\alpha) \dx \alpha \\
		&
		= 
		\frac{c_d}{2} \sum_{i=1}^{n} a_{i} \int_{0}^{2\pi}  |\sin(\theta)|^{d-1} \dx \theta = 1.
	\end{align}
	For $p \in \mathrm{\Pi}_{r}(\mathbb S^{d})$, we obtain
	\begin{align}
		&\int_{0}^{1} p \circ \gamma_{r}\, \rho_{r} \dx \lambda\\
		=&  \sum_{i=1}^{n} \int_{\frac{i-1}{n}}^{\frac i n} p\bigl(\gamma_{r,i}(2 \pi n t)\bigr) \rho_{r,i}(2\pi M t) \dx t
		=   \int_{0}^{2\pi}  \frac{1}{2\pi n} \sum_{i=1}^{n}  p\bigl(\gamma_{r,i}(\alpha)\bigr) \rho_{r,i}(\alpha) \dx \alpha\\
		=&  \frac{c_{d}}{2}\int_{0}^{2\pi}  |\sin(\alpha)|^{d-1} \sum_{i=1}^{n} a_{i}  p\bigl(\cos(\alpha),\sin(\alpha)\tilde x_{i}\bigr) \dx \alpha \\
		=&   \frac{c_{d}}{2}\int_{0}^{\pi}  |\sin(\alpha)|^{d-1} \sum_{i=1}^{n} a_{i}  
		\Big( p\bigl(\cos(\alpha),\sin(\alpha)\tilde x_{i}\bigr) + p\bigl(-\cos(\alpha),-\sin(\alpha)\tilde x_{i}\bigr) \Big) \dx \alpha.
	\end{align}
	Without loss of generality, $p$ is chosen as a homogeneous polynomial of degree $k \le r$, i.e., $p(t x) =t^k p(x)$. 
	Then, 
	\begin{align}
		\int_{0}^{1} p \circ \gamma_{r}\, \rho_{r} \dx \lambda
		&=   \frac{1+(-1)^{k}}{2} \int_{0}^{\pi}  c_{d} |\sin(\alpha)|^{d-1} \sum_{i=1}^n a_{i}  p\bigl(\cos(\alpha),\sin(\alpha)\tilde x_{i}\bigr) \dx \alpha,
	\end{align} 
	and regarding that for fixed $\alpha \in [0,2\pi]$ 
	the function $\tilde x \mapsto p(\cos(\alpha), \sin(\alpha) \tilde x)$ 
	is a polynomial of degree at most $r$ on $\mathbb S^{d-1}$, we conclude 
	\begin{align*}
		\int_{0}^{1} p \circ \gamma_{r}\, \rho_{r} \dx \lambda
		=  & \frac{1+(-1)^{k}}{2}\int_{0}^{\pi}  c_{d} |\sin(\alpha)|^{d-1} \int_{\mathbb S^{d-1}}  p\bigl(\cos(\alpha),\sin(\alpha)\tilde x\bigr) \dx \sigma_{\mathbb S^{d-1}}(\tilde x) \dx \alpha. 
	\end{align*}
	Now, the assertion \eqref{eq:quad_curve_Sd} follows from \eqref{eq:ausgang} and since $\int_{\mathbb S^{d}} p \dx \sigma_{\mathbb S^{d}}=0$ if $k$ is odd.
	
	2. Next, we apply Lemma \ref{lemma:GeneralApproximation} for $\nu_r={\gamma_r}_{*} \rho_r\lambda$, from which we obtain that $\tilde \nu_r={\gamma_r}{_*}((\rho\circ\gamma_r) \rho_r/{\beta_r}\lambda)$.
	As all $\rho_r$ are uniformly bounded by construction and $\rho$ is bounded due to continuity, we conclude using $L(\rho_r) \lesssim r^{d-1}$ and $L(\gamma_r) \sim r^{d-1}$ that
	\[L(\rho\circ\gamma_r \, \rho_r/\beta_r) 
	\leq 
	\bigl( L(\rho\circ\gamma_r) \Vert \rho_r \Vert_\infty + L(\rho_r) \Vert \rho \Vert_\infty \bigr)/\beta_r\lesssim \bigl(L(\rho) + \|\rho\|_\infty\bigr) r^{d-1},\]
	which concludes the proof.
\end{proof}

Finally, we derive approximation rates for $\X=\SO(3)$. 

\begin{corollary}[Special orthogonal group] \label{cor:so3}
	Let $\X = \SO(3)$, $s>3/2$ and suppose 
	$H_K(\X) = H^s(\X)$ holds with equivalent norms. 
	Then, we have for any absolutely continuous measure $\mu \in {\mathcal P}(\X)$ with Lipschitz continuous density $\rho \in H^s(\X)$
	that
	\begin{equation}\label{eq:improvement 2}
		\min_{\nu\in\mathcal{P}^{\Acurve}_L(\X)}\mathscr{D}_K (\mu,\nu) \lesssim  L^{-\frac{s}{d-1}},
	\end{equation}
	where the constant depends on $K$ and $\rho$.
\end{corollary}

\begin{proof}
	1. For fixed $L\sim r^2$, we shall construct a curve $\gamma_{r}\colon[0,1] \to \mathrm{SO(3)}$ with $L(\gamma_r)\lesssim L$  and a probability measure 
	$\omega_r = \rho_r \lambda$ with density $\rho_{r}\colon[0,1] \rightarrow \mathbb R_{\ge 0}$ and $L(\rho_{r}) \lesssim L$, 
	such that 
	\begin{equation*}
		\int_{\SO(3)} p \dx \sigma_{\SO(3)} = \int_{\SO(3)} p \dx {\gamma_r}{_*}(\rho_r\lambda). 
	\end{equation*}
	
	We  use the fact that the sphere $\mathbb S^{3}$ is a double covering of $\mathrm{SO(3)}$. 
	That is, there is a surjective two-to-one mapping $a\colon \mathbb S^{3} \to \mathrm{SO(3)}$ satisfying $a(x) = a(-x)$, $x \in \mathbb S^{3}$. 
	Moreover, we know that $a\colon \mathbb S^{3} \to \mathrm{SO(3)}$ is a local isometry, see~\cite{graef2012}, 
	i.e., it respects the Riemannian structures, implying the relations $\sigma_{\mathrm{SO(3)}} = a_{*} \sigma_{\mathbb S^{3}}$ and
	\begin{align}
		\dist_{\mathrm{SO(3)}}\bigl(a(x_{1}),a(x_{2})\bigr) &= \min\bigl( \dist_{\mathbb S^{3}}(x_{1},x_{2}),\dist_{\mathbb S^{3}}(x_{1},-x_{2}) \bigr).
	\end{align}
	It also maps $\mathrm{\Pi}_{r}(\mathrm{SO(3)})$ into $\mathrm{\Pi}_{2r}(\S^3)$, i.e., $p\in \mathrm{\Pi}_{r}(\mathrm{SO(3)})$ implies $p\circ a\in \mathrm{\Pi}_{2r}(\S^3)$. 
	Now, let $\tilde \gamma_{r}\colon [0,1]\to \mathbb S^{3}$ and $\tilde \omega_{r}$ 
	be given as in the first part of the proof of Theorem~\ref{thm:sphere general acurve} for $d=3$, 
	i.e., ${{\tilde{\gamma}}_r}{_*}\tilde \omega_{r}$ satisfies \eqref{eq:quad_curve_Sd} 
	with $L(\tilde{\gamma_r})\lesssim L$ and $\tilde{\omega}_r=\tilde{\rho_r}\lambda$ with $L(\tilde{\rho}_r)\lesssim L$. 
	
	We now define a curve $\gamma_r$ in $\SO(3)$ by 
	\[
	\gamma_{r}\colon [0,1] \to \mathrm{SO(3)},\qquad \gamma_{r}(t) \coloneqq a \circ  \tilde \gamma_{2r}(t),
	\]
	and let $\omega_r \coloneqq \tilde{\omega}_{2r}$. 
	For $p\in\mathrm{\Pi}_r(\SO(3))$, the push-forward measure ${\gamma_r}{_*} \omega_r$ leads to
	\begin{align*}
		\int_{\mathrm{SO(3)}} p \dx \sigma_{\mathrm{SO(3)}} & = \int_{\mathrm{SO(3)}} p\dx a{_*}\sigma_{\S^3} = \int_{\S^3} p \circ a\dx \sigma_{\S^3}\\
		& = \int_{\S^3} p \circ a \dx {\tilde{\gamma_{2r}}}{_*} \tilde{\omega}_{2r} =  \int_{\SO(3)} p\dx {\gamma_{r}}{_*} \omega_{r}.
	\end{align*}
	Hence, property \eqref{eq:exact integration} is satisfied for ${\gamma_r}{_*} \omega_r={\gamma_r}{_*} (\tilde{\rho}_{2r}\lambda)$. 
	
	2. The rest follows along the lines of step 2.~in the proof of Theorem~\ref{thm:sphere general acurve}.
	%Since we also have $L(\gamma_r)\leq L(\tilde{\gamma}_{2r})\lesssim L$ and $\omega_{r}=\tilde{\rho}_{2r}\dx t$ with $L(\tilde{\rho}_{2r})\lesssim L$, Lemma \ref{lemma:GeneralApproximation} implies the assertion. 
\end{proof}

%---------------------------------------------------------------------------------
\subsection{Upper bounds for \texorpdfstring{$\mathcal{P}^{\Lcurve}_L(\X)$}{PL(X)} and special manifolds \texorpdfstring{$\X$}{X}} \label{sec:approx_L}
%---------------------------------------------------------------------------------
To derive upper bounds for the smallest space $\mathcal{P}^{\Lcurve}_L(\X)$, we need the following specification of Lemma \ref{lemma:GeneralApproximation}. 

\begin{lemma}\label{lemma:zweites}
	For $s>d/2$ suppose that $H_K(\X) = H^s(\X)$ holds with equivalent norms.
	Let $\mu \in \mathcal P(\X)$ be absolutely continuous with respect to $\sigma_\X$ with positive density $\rho \in H^s(\X)$. 
	Suppose that
	$\nu_r \coloneqq {\gamma_{r}}{_*} \lambda$ with $\gamma_r\in\Lip(\X)$ satisfies \eqref{eq:exact integration} and 
	let $
	\beta_r \coloneqq \int_\X \rho\dx \nu_r 
	$. 
	Then, for sufficiently large $r$, 
	\begin{equation*}
		g \colon [0,1]\rightarrow [0,1],\qquad g(t)\coloneqq
		%\int_\X \frac{\rho(x)}{\int_\X \rho(x)\dx \nu_r(x)}\dx ({{\gamma_r}_{|_{[0,t]}}}_* \omega_r) (x) =
		\frac{1}{\beta_r} \int_0^t \rho \circ \gamma_r  \dx \lambda 
	\end{equation*}
	is well-defined and invertible. Moreover, $\tilde{\gamma}_r \coloneqq \gamma_r \circ g^{-1}$ satisfies
	$L(\tilde \gamma_r) \lesssim L( \gamma_r)$ and  
	\begin{equation}\label{eq:ineq 2 2 2}
		\mathscr{D}_K(\mu,{\tilde \gamma}_r{_*}\lambda) \lesssim  r^{-s},
	\end{equation}
	where the constants depend on $\X$, $K$, and $\rho$. 
\end{lemma}

\begin{proof}
	Since $\rho$ is continuous, there is $\epsilon>0$ with $\rho\geq \epsilon$. 
	To bound the Lipschitz constant $L(\tilde \gamma_r)$, we apply the mean value theorem together with the definition of $g$
	and the fact that $(g^{-1})'(s) = 1/g'(g^{-1}(s))$ to obtain
	\begin{align*}
		\bigl|\tilde{\gamma}_r(s)-\tilde{\gamma}_r(t)\bigr|
		\leq 
		L(\gamma_r)\bigl|g^{-1}(s)-g^{-1}(t)\bigr| 
		\leq 
		L(\gamma_r) \, \frac{\beta_r}{\epsilon}  \, |s-t|.
	\end{align*}
	Using \eqref{eq:1}, this can be further estimated for sufficiently large $r$ as
	\begin{align*}
		\bigl|\tilde{\gamma}_r(s)-\tilde{\gamma}_r(t)\bigr|
		& \lesssim 
		L(\gamma_r) \, \frac{1+\|\rho\|_{H^s(\X)}r^{-s} }{\epsilon} \, |s-t|
		\lesssim L(\gamma_r) \, \frac{2}{\epsilon} \,  |s-t|.
	\end{align*}
	To derive \eqref{eq:ineq 2 2 2}, we aim to apply Lemma \ref{lemma:GeneralApproximation} with 
	$\nu_r={\gamma_r}{_*}\lambda$. 
	We observe 
	\begin{align*}
		\tilde \nu_r =  \frac{\rho}{\beta_r}  {\gamma_r}{_*}\lambda
		= {\gamma_r}{_*} \Bigl( \frac{\rho\circ \gamma_r}{\beta_r} \lambda\Bigr)
		= {\gamma_r}{_*} (g'\lambda)
		=(\gamma_r\circ g^{-1}){_*} \lambda =\tilde{{\gamma}_r}{_*}\lambda,
	\end{align*}
	so that Lemma \ref{lemma:GeneralApproximation} indeed implies \eqref{eq:ineq 2 2 2}.
\end{proof}

In comparison to Theorem \ref{thm:new for torus and acurve}, we now trade the Lipschitz condition on $\rho$ with the positivity requirement, which enables us to cover $\mathcal{P}^{\Lcurve}_L(\X)$.

\begin{theorem}[Torus]\label{thm:torus better estimates}
	Let $\X = \mathbb T^d$ with $d\in\N$, $s>d/2$ and suppose that $H_K(\X) = H^s(\X)$ holds with equivalent norms.
	Then, for any absolutely continuous measure $\mu \in {\mathcal P}(\X)$ with positive density $\rho \in H^s(\X)$, there is a constant depending on $d$, $K$, and $\rho$ with
	\begin{equation}
		\min_{\nu\in\mathcal{P}^{\Acurve}_L(\X)}\mathscr{D}_K(\mu,\nu) \leq \min_{\nu\in\mathcal{P}^{\Lcurve}_L(\X)}\mathscr{D}_K(\mu,\nu) \lesssim L^{-\frac{s}{d-1}}.
	\end{equation}
\end{theorem}
\begin{proof}
	The first part of the proof is identical to the proof of Theorem \ref{thm:new for torus and acurve}. 
	Instead of Lemma \ref{lemma:GeneralApproximation} though, we now apply  Lemma \ref{lemma:zweites} for $\gamma_r$ and $\rho_r \equiv 1$.
	Hence, $\tilde \gamma_r = \gamma_r \circ g_r^{-1}$ satisfies $L(\tilde \gamma_r)\leq \frac{\beta_r}{\epsilon} d(2r +1)^{d-1} \lesssim r^{d-1}$, so that $\tilde{\gamma_r}{_*}\lambda$ satisfies \eqref{eq:ineq 2 2 2} and is in $\mathcal{P}^{\Lcurve}_L(\X)$ with $L \sim r^{d-1}$. 
\end{proof}

The construction on $\X=\S^d$ for $\mathcal{P}^{\Acurve}_L(\X)$ in the proof of Theorem \ref{thm:sphere general acurve} 
is not compatible with $\mathcal{P}^{\Lcurve}_L(\X)$. Thus, the situation is different from the torus, 
where we have used the same underlying construction and only switched from Lemma \ref{lemma:GeneralApproximation} 
to Lemma \ref{lemma:zweites}. 
Now, we present a new construction for $\mathcal{P}^{\Lcurve}_L(\X)$, which is tailored to $\X=\S^2$.
In this case, we can transfer the ideas of the torus, but with Gauss-Legendre quadrature points. 

\begin{theorem}[2-sphere]\label{thm:sphere}
	Let $\X = \mathbb S^2$, $s>1$ and suppose 
	$H_K(\X) = H^s(\X)$ holds with equivalent norms. 
	Then, we have 
	for any absolutely continuous measure $\mu \in {\mathcal P}(\X)$ with positive density $\rho \in H^s(\X)$
	that there is a constant depending on $K$ and $\rho$ with 
	\begin{equation}\label{eq:improvement 2.1}
		\min_{\nu\in\mathcal{P}^{\Acurve}_L(\X)}\mathscr{D}_K(\mu,\nu)\leq \min_{\nu\in\mathcal{P}^{\Lcurve}_L(\X)}\mathscr{D}_K (\mu,\nu) \lesssim  L^{-s}.
	\end{equation}
\end{theorem}
\begin{proof}
	1. We construct closed curves such that the spherical polynomials from $\mathrm{\Pi}_r(\mathbb S^2)$, see \eqref{trig_sphere} in the appendix,
	are exactly integrated along this curve.
	It suffices to show this for the polynomials 
	$p(x) = x^{k_1} x^{k_2} x_3^{k_3} \in \mathrm{\Pi}_r(\S^2)$ with $k_1+k_2+k_3 \le r$
	restricted to $\mathbb S^2$.
	We select $n = \lceil (r+1)/2 \rceil$ Gauss-Legendre quadrature points 
	$u_j = \cos(\theta_j)\in [-1,1]$ 
	and corresponding weights $2\omega_j$, $j=1, \ldots,n$. Note that $\sum_{j=1}^n \omega_j = 1$.
	Using spherical coordinates 
	$x_1=\cos(\theta)$, $x_2=\sin(\theta)\cos(\phi)$, and $x_3=\sin(\theta)\sin(\phi)$
	with $(\theta, \phi) \in[0,\pi] \times [0,2\pi]$, we obtain
	\begin{align*}\label{eq:S2 Wagner}
		\int_{\S^2} p \dx \sigma_{\mathbb S^2} 
		&=
		\frac{1}{4\pi} 
		\int_0^{2\pi} \cos(\phi)^{k_2} \sin(\phi)^{k_3} 
		\int_0^{\pi} \cos(\theta)^{k_1} \sin(\theta)^{k_2+k_3} \sin(\phi) \dx \theta \dx \phi\\
		&=
		\frac{1}{4\pi} \int_0^{2\pi} \cos(\phi)^{k_2} \sin(\phi)^{k_3} \int_{-1}^1 u^{k_1} (1-u^2)^{\frac{k_2+k_3}{2}} \dx u \dx \phi,
	\end{align*}
	see also \cite{Wagner:1991vl}. 
	If $k_2+k_3$ is odd, then the integral over $\phi$ becomes zero.
	If $k_2+k_3$ is even, the inner integrand is a polynomial of degree $\le r$.
	In both cases we get
	\begin{align*}
		\int_{\S^2} p \dx \sigma_{\mathbb S^2} 
		&= 
		\frac{1}{2\pi} \sum_{j=1}^n \omega_j
		\int_0^{2\pi} p\bigl(\cos(\theta_j),\sin(\theta_j) \cos(\phi), \sin(\theta_j)\sin(\phi)\bigr) \dx \phi.
	\end{align*}
	Substituting in each summand $\phi = 2\pi t /\omega_j$, $j=1,\ldots,n$, yields
	\begin{equation*}
		\int_{\S^2} p\dx \sigma_{\mathbb S^2} = \sum_{j=1}^n  \int_0^{\omega_j} p \circ \gamma_j \dx \lambda,
	\end{equation*}
	where $\gamma_j\colon[0,\omega_j] \rightarrow \S^2$ is defined by
	\begin{equation*}
		\gamma_j(t) \coloneqq \bigl(\cos(\theta_j),\sin(\theta_j)\cos(2\pi t/\omega_j),\sin(\theta_j)\sin(2\pi t/\omega_j) \bigr),
	\end{equation*}
	and has constant speed $L(\gamma_j) = 2\pi \sin(\theta_j)/\omega_j$. 
	The lower bound $\omega_j \gtrsim \frac{1}{n}\sin(\theta_j)$,
	cf.~\cite{Forster:1990iw}, implies that $L(\gamma_j)\lesssim n$. 
	Defining a curve $\tilde{\gamma}\colon[0,1]\rightarrow\S^2$ piecewise via 
	\begin{equation*}
		\tilde{\gamma}|_{[0,s_1]} = \gamma_1,\quad 
		\tilde{\gamma}|_{[s_1,s_2]} = \gamma_2(\cdot-s_1), \quad \ldots \quad, \quad
		\tilde{\gamma}|_{[s_{n-1},1]} = \gamma_n(\cdot-s_{n-1}),
	\end{equation*}
	where $s_j \coloneqq \omega_1 + \ldots + \omega_j$, we obtain 
	\begin{equation*}
		\int_{\S^2} p\dx \sigma_{\mathbb S^2}  = \int_0^{1} p \dx \tilde{\gamma}{_*}\lambda,  \quad p \in \mathrm{\Pi}_r(\S^2).
	\end{equation*}
	Further, the curve satisfies $L(\tilde{\gamma})\lesssim r$. 
	
	As with the torus, we now ``turn'' the sphere (or switch the position of $\phi$) so that we get circles along orthogonal directions. 
	This large collection of circles is indeed connected. 
	As with the torus, each intersection point has an incoming and outgoing part of a circle, so that all this corresponds to a graph, 
	where again each vertex has an even number of ``edges''. Hence, there is an Euler path inducing our final curve $\gamma_r\colon[0,1]\rightarrow\S^2$ 
	with piecewise constant speed $L(\gamma_r)\lesssim r$ satisfying 
	\begin{equation*}
		\int_{\S^2} p\dx \sigma_{\mathbb S^2}  = \int_0^{1} p \dx ({\gamma_r}{_*}\lambda),  \quad p \in \mathrm{\Pi}_r(\S^2).
	\end{equation*}
	
	2. Let $r \sim L$. Analogous to the end of the proof of Theorem \ref{thm:torus better estimates}, Lemma \ref{lemma:zweites} now yields the assertion.
\end{proof}

To get the approximation rate for $\X=\G_{2,4}$, we make use of its double covering $\X=\S^2\times\S^2$, cf.~Remark \ref{rem:Grassi}.

\begin{theorem}[Grassmannian]\label{thm:grassi}
	Let $\X = \G_{2,4}$, $s>2$ and suppose 
	$H_K(\X) = H^s(\X)$ holds with equivalent norms. 
	Then, we have 
	for any absolutely continuous measure $\mu \in {\mathcal P}(\X)$ with positive density $\rho \in H^s(\X)$
	that there exists a constant depending on $K$ and $\rho$ with
	\begin{equation}
		\min_{\nu\in\mathcal{P}^{\Acurve}_L(\X)}\mathscr{D}_K(\mu,\nu)\leq \min_{\nu\in\mathcal{P}^{\Lcurve}_L(\X)}\mathscr{D}_K (\mu,\nu) \lesssim  L^{-\frac{s}{3}}.
	\end{equation}
\end{theorem}

\begin{proof}
	By Remark \ref{rem:Grassi} in the appendix, we know that $\G_{2,4} \cong \S^2\times\S^2/ \{\pm 1\}$ so that is remains to prove the assertion for
	$\X = \S^2 \times \S^2$.
	
	There exist pairwise distinct points $\{x_1,\ldots,x_N\}\subset\S^2$ such that \smash{$\frac{1}{N}\sum_{j=1}^N \delta_{x_j}$}
	satisfies \eqref{eq:exact integration} on $\S^2$ with $N\sim r^2$, cf.~\cite{Bondarenko:2011kx,Bondarenko:2015eu}. 
	On the other hand, let $\tilde{\gamma}$ be the curve on $\S^2$ constructed in the proof of Theorem \ref{thm:sphere}, 
	so that $\tilde{\gamma}{_*}\lambda$ satisfies \eqref{eq:exact integration} on $\S^2$ with $\ell(\tilde{\gamma})\leq L(\tilde{\gamma})\sim r$. 
	Let us introduce the virtual point $x_{N+1}\coloneqq x_1$. 
	%Without loss of generality, we may assume that $\{x_1,\ldots,x_N\}$ are ordered, so that neighboring points $x_j$ and $x_{j+1}$ lie in one hemisphere, $j=1\ldots,N$. Here, the hemisphere may depend on $j$. 
	The curve $\tilde{\gamma}([0,1])$ contains a great circle.
	Thus, for each pair $x_j$ and $x_{j+1}$ there is $O_j\in \OOO(3)$ such that $x_j,x_{j+1}\in \Gamma_j\coloneqq O_j\tilde{\gamma}([0,1])$. %
	%
	%
	%
	%
	%There is $O_j\in \OOO(3)$ such that $x_{j+1}=O_j \tilde{\gamma}(0)$. The curves $\tilde{\gamma}_j:=O_j \tilde{\gamma}$ satisfy $\tilde{\gamma}_j(0)=x_j$. Since $(\tilde{\gamma}_j)_*\dx t$ also satisfies \eqref{eq:exact integration} on $\S^2$ with $\ell(\tilde{\gamma})\leq L(\tilde{\gamma})\sim r$, the curve $\tilde{\gamma}_j$ cannot stay within one hemisphere around $x_j$ for sufficiently large $r$. Thus, there is a rotation $R_j\in \OOO(3)$ with axis $x_j$ such that $x_{j+1}$ is in the image of $R_j \tilde{\gamma}_j$. Hence, the curves $\gamma_j:=R_j\tilde{\gamma}_j$ satisfy $\gamma_j(0)=x_j$ and $x_{j+1}\in \Gamma_j:=\gamma_j([0,1])$. 
	%
	%We repeat the above procedure in an analogous fashion to derive another set of curves $\gamma^2_j$ such that $\gamma^2_j(0)=x_j$ and $x_{j-1}\in \Gamma^2_{j}:=\gamma^2_{j}([0,1])$ with the convention $x_{0}:=x_N$. 
	%
	It turns out that the set on $\S^2 \times \S^2$ given by
	$
	\bigcup_{j=1}^N (\{x_j\}\times \Gamma_j)\cup (\Gamma_j \times \{x_{j+1}\})
	$
	is connected.
	We now choose $\gamma_j\coloneqq O_j\tilde{\gamma}$ and know that the union of the trajectories of the set of curves
	\begin{equation*}
		t\mapsto \bigl(x_j,\gamma_j(t)\bigr),\qquad t\mapsto \bigl(\gamma_{j}(t),x_{j+1}\bigr),\quad j=1,\ldots,N,
	\end{equation*}
	is connected.
	Combinatorial arguments involving Euler paths, see Theorems \ref{thm:new for torus and acurve} 
	and~\ref{thm:sphere}, lead to a curve $\gamma$ with $\ell(\gamma)\leq L(\gamma)\sim N L(\tilde{\gamma}) \sim r^3$, 
	so that $\gamma{_*} \lambda$ satisfies  \eqref{eq:exact integration}.
	The remaining part follows along the lines of the proof of Theorem~\ref{thm:sphere general acurve}.
\end{proof}

Our approximation results can be extended to diffeomorphic manifolds, e.g., from $\S^2$ to ellipsoids, see also the 3D-torus example in Section \ref{sec:numerics_dith}. To this end, recall that
we can describe the Sobolev space $H^s(\X)$ using local charts, 
see \cite[Sec.~7.2]{Triebel:1992aa}.
The exponential maps $\exp_{x} \colon  T_{x}\X \to \X$ 
give rise to local charts $(\mathring{B}_{x}(r_0), \exp_x^{-1})$, 
where $\mathring{B}_{x}(r_0) \coloneqq \{y \in \X: \dist_\X(x,y) < r_0\}$ 
denotes the geodesic balls around $x$ with the injectivity radius $r_0$.
If $\delta < r_0$ is chosen small enough, there exists a uniformly locally finite covering of $\X$ by a sequence of balls 
$(\mathring{B}_{x_j}(\delta))_j$ with a corresponding smooth resolution of unity $(\psi_j)_j$ with 
\smash{$\supp(\psi_j) \subset \mathring{B}_{x_j}(\delta)$}, see  
\cite[Prop.~7.2.1]{Triebel:1992aa}.
Then, an equivalent Sobolev norm is given by
\begin{equation}\label{eq:EquivalentNorm}
	\|f\|_{H^s(\X)} \coloneqq \Big( \sum_{j=1}^\infty \Vert (\psi_j f) \circ \exp_{x_j} \Vert^2_{H^s(\R^d)} \Big)^{\frac{1}{2}},
\end{equation}
where $(\psi_j f) \circ \exp_{x_j}$ is extended 
%from $\vert x \vert \leq \delta$ 
to $\mathbb R^d$ by zero, see \cite[Thm.~7.4.5]{Triebel:1992aa}.
Using Definition~\eqref{eq:EquivalentNorm}, 
we are able to pull over results from the Euclidean setting.

\begin{proposition}\label{prop:diffeo}
	Let $\mathbb X_1$, $\mathbb X_2$ be two $d$-dimensional connected, compact Riemannian manifolds without boundary, which are $s+1$ diffeomorphic with $s>d/2$.
	Assume that for $H_K(\mathbb X_2)=H^s(\mathbb X_2)$ and every absolutely continuous measure $\mu$ with positive density $\rho\in H^s(\mathbb X_2)$ it holds
	\[
	\min_{\nu \in  \mathcal{P}_L^{\Lcurve} } \mathscr{D}_K(\mu,\nu)\lesssim L^{-\frac{s}{d-1}},
	\]
	where the constant depends on $\X_2$, $K$, and $\rho$. 
	Then, the same property holds for $\mathbb X_1$, where the constant additionally depends on the diffeomorphism.
\end{proposition}

\begin{proof}
	Let $f \colon \mathbb X_2 \to \mathbb X_1$ 
	denote such a diffeomorphism and $\rho \in H^s(\mathbb X_1)$ 
	the density of the measure $\mu$ on $\mathbb X_1$. 
	Any curve $\tilde \gamma \colon[0,1] \to \mathbb X_2$ gives rise to a curve $\gamma \colon[0,1] \to \mathbb X_1$ via $\gamma = f \circ \tilde \gamma$, which for every $\varphi \in H^s(\mathbb X_1)$ satisfies
	\begin{align}
		\Big| \int_{\X_1} \varphi \rho \dx \sigma_{X_1} - \int_{0}^1 \varphi \circ \gamma \dx \lambda \Big|
		= \Big| \int_{\mathbb X_2} (\varphi \rho)\circ f \vert \det(J_f) \vert \dx \sigma_{X_2} -\int_{0}^1 \varphi\circ f\circ\tilde \gamma \dx \lambda \Big|,
	\end{align}
	where $J_f$ denotes the Jacobian of $f$. 
	Now, note that $\varphi \circ f, \rho \circ f \vert \det(J_f) \vert \in H^s(\mathbb X_2)$, see \eqref{banach_alg} and \cite[Thm.~4.3.2]{Triebel:1992aa}, 
	which is lifted to manifolds using \eqref{eq:EquivalentNorm}.
	Hence, we can define a measure $\tilde \mu$ on $\X_2$ through the probability density $\rho \circ f \vert \det(J_f) \vert$.
	Choosing $\tilde \gamma_L$ as a realization for some minimizer of 
	$\inf_{\nu \in  \mathcal{P}_L^{\Lcurve}} \mathscr{D}(\tilde \mu,\nu)$, 
	we can apply the approximation result for 
	$\mathbb X_2$ and estimate for $\gamma_L = f \circ \tilde \gamma_L$ that
	\begin{align}
		\Big| \int_{\mathbb X_1} \varphi \rho \dx \sigma_{X_1} - \int_{0}^1 \varphi \circ \gamma_L \dx \lambda \Big| 
		\lesssim L^{-\frac{s}{d-1}} \|\varphi \circ f\|_{H^s(\X_2)} \lesssim L^{-\frac{s}{d-1}} \|\varphi \|_{H^s(\X_1)},
	\end{align}
	where the second estimate follows from \cite[Thm.~4.3.2]{Triebel:1992aa}.
	Now, $L(\gamma_L) \leq L(f)L$ implies
	\[
	\inf_{\nu \in  \mathcal{P}_L^{\Lcurve}} \mathscr{D}_K (\mu,\nu)
	\lesssim L^{-\frac{s}{d-1}}.\eqno
	\]
\end{proof}

\begin{remark}\label{rem:andere}
	Consider a probability measure $\mu$ on $\X$ such that the dimension $d_\mu$ of its support is smaller than the dimension $d$ of $\X$. Then, $\mu$ does not have any density with respect to $\sigma_\X$. If $\supp(\mu)$ is itself a $d_\mu$-dimensional connected, compact Riemannian manifold $\Y$ without boundary, we switch from $\X$ to $\Y$. Sobolev trace theorems and reproducing kernel Hilbert space theory imply that the assumption $H_K(\X)=H^s(\X)$ leads to $H_{K'}(\Y)=H^{s'}(\Y)$, where $K'\coloneqq K|_{\Y\times\Y}$ is the restricted kernel and $s'=s-(d-d_\mu)/2$, cf.~\cite{Fuselier:2012jt}. If, for instance, $\Y$ is diffeomorphic to $\T^{d_\mu}$ (or $\S^{d_{\mu}}$ with $d_\mu=2$), and $\mu$ has a positive density $\rho\in H^{s'}(\Y)$ with respect to $\sigma_{\Y}$, then Theorem \ref{thm:torus better estimates} (or \ref{thm:sphere}) and Proposition \ref{prop:diffeo} eventually yield
	\begin{equation*}
		\min_{\nu \in  \mathcal{P}_L^{\Lcurve} } \mathscr{D}_K(\mu,\nu)\lesssim L^{-\frac{s'}{d_\mu-1}}.
	\end{equation*}
	
	If $\supp(\mu)$ is a proper subset of $\Y$, we can analyze approximations with $\mathcal{P}_L^{\Acurve}(\Y)$.
	First, we observe that the analogue of Proposition \ref{prop:diffeo} also holds for $\mathcal{P}_L^{\Acurve}(\X_1)$ and $\mathcal{P}_L^{\Acurve}(\X_2)$ when the positivity assumption on $\rho$ is replaced with the Lipschitz requirement as in Theorems~\ref{thm:new for torus and acurve} and  \ref{thm:sphere general acurve}. 
	If, for instance, $\Y$ is diffeomorphic to $\T^{d_\mu}$ or $\S^{d_\mu}$ and $\mu$ has a Lipschitz continuous density $\rho\in H^{s'}(\Y)$ with respect to $\sigma_{\Y}$, then Theorems~\ref{thm:new for torus and acurve} and \ref{thm:sphere general acurve}, and Proposition \ref{prop:diffeo} eventually yield
	\begin{equation*}
		\min_{\nu \in  \mathcal{P}_L^{\Acurve} } \mathscr{D}_K(\mu,\nu)\lesssim L^{-\frac{s'}{d_\mu-1}}.
	\end{equation*}
\end{remark}

%--------------------------------------------------------------
\section{Discretization} \label{sec:discretization_dith}
%--------------------------------------------------------------
In our numerical experiments, we are interested in determining minimizers of 
\begin{equation} \label{eq:min_PL}
	\min_{\nu\in\mathcal P_{L}^{\Lcurve}(\mathbb X)} \mathscr{D}^2_K (\mu,\nu). 
\end{equation}
Defining 
$
A_L \coloneqq \{\gamma \in \Lip(\X): L(\gamma) \le L\}$
and using the indicator function 
$$
\iota_{A_L} (\gamma) \coloneqq 
\left\{
\begin{array}{ll}
	0&\mathrm{if} \; \gamma \in A_L,\\
	+ \infty&\mathrm{otherwise},
\end{array}
\right.
$$
we can rephrase problem \eqref{eq:min_PL} as a minimization problem over curves
\begin{equation}
	\min_{\gamma \in \mathcal{C}([0,1],\X)} \mathcal J_L(\gamma),
\end{equation}
where $\mathcal J_L(\gamma)\coloneqq \mathscr{D}^2_K(\mu,\gamma{_*}\lambda) + \iota_{A_L}(\gamma)$.
As $\X$ is a connected Riemannian manifold, we can approximate curves in $A_L$ by piecewise shortest geodesics with $N$ parts, i.e., by curves from
\[
A_{L,N} \coloneqq \left\{\gamma \in A_L \colon \gamma|_{[(i-1)/N,i/N]}\text{ is a shortest geodesic for } i = 1,\ldots, N\right\}.
\]
Next, we approximate the Lebesgue measure on $[0,1]$ by $e_N \coloneqq \frac{1}{N} \sum_{i=1}^{N} \delta_{i/N}$ 
and consider the minimization problems
\begin{equation} \label{eq:nextmin}
	\min_{\gamma \in \mathcal{C}([0,1],\X)} \mathcal J_{L,N}(\gamma) ,
\end{equation}
where $\mathcal J_{L,N}(\gamma)\coloneqq \mathscr{D}^{2}_K (\mu, \gamma{_*}e_N) + \iota_{A_{L,N}}(\gamma)$.
Since $\mathrm{ess}\sup_{t \in [0,1]} |\dot \gamma| (t) = L(\gamma)$, the constraint $L(\gamma) \leq L$ can be reformulated as 
\smash{$\int_0^1 (|\dot \gamma| (t) - L)_+^2 \dx t = 0$}.\footnote[1]{For $r\in\R$, we use the notation $r_+=\begin{cases}
		r,& r\geq 0,\\ 0,& 
		\text{otherwise.} \end{cases}$}
Hence, using $x_i = \gamma (i/N)$, $i=1,\ldots,N$, $x_0 = x_N$ 
and regarding that $|\dot \gamma| (t) = N \dist_\X(x_{i-1},x_{i})$
for $t \in \left( \frac{i-1}{N},\frac{i}{N} \right)$,
problem \eqref{eq:nextmin} is rewritten in the computationally more suitable form  
\begin{equation} \label{eq:full}
	\min_{(x_1,\ldots,x_N)\in \X^N} \mathscr{D}^{2}_K\Big(\mu, \frac{1}{N} \sum_{i=1}^{N} \delta_{x_{i}}\Big) 
	\quad \text{s.t.} \quad \frac{1}{N}\sum_{i=1}^{N} \big(N \dist_\X(x_{i-1},x_{i}) - L \big)_+^{2} = 0.
\end{equation}
This discretization is motivated by the next proposition.
To this end, recall that  a sequence $(f_N)_{N\in\N}$ of functions $f_N\colon {\X} \rightarrow (-\infty,+\infty]$
is said to $\Gamma$-converge to $f \colon  {\X} \rightarrow (-\infty,+\infty]$
if the following two conditions are fulfilled for each $x \in {\X}$, see~\cite{Braides02}:
\begin{enumerate}
	\item[i)] $f(x) \leq \liminf_{N \rightarrow \infty} f_N(x_N)$ whenever  $x_N \rightarrow x$, 
	\item[ii)] there is a sequence $(y_N)_{N\in\N}$ with $y_N \rightarrow x $ and $\limsup_{N \to \infty} f_N(y_N) \le f(x)$.
\end{enumerate}
The importance of $\Gamma$-convergence relies in the fact that 
every cluster point of minimizers of $(f_N)_{N\in\N}$ is a minimizer of $f$. Note that for non-compact manifolds $\X$ an additional equi-coercivity condition would be required.

\begin{proposition}\label{prop:gamma-conv}
	The sequence $(\mathcal J_{L,N})_{N\in\N}$ is $\Gamma$-convergent with limit $\mathcal J_L$.
\end{proposition}

\begin{proof}
	1. First, we verify the $\liminf$-inequality. Let $(\gamma_N)_{N\in\N}$ with $\lim_{N\rightarrow \infty} \gamma_N = \gamma$, i.e., the sequence satisfies $\sup_{t \in [0,1]} \dist_\X(\gamma(t),\gamma_N(t)) \rightarrow 0$.
	By excluding the trivial case $\liminf_{N \to \infty} \mathcal J_{L,N}(\gamma_N) = \infty$ 
	and restricting to a subsequence $(\gamma_{N_k})_{k \in \N}$, we may assume $\gamma_{N_k} \in A_{L,N_k} \subset A_L$.
	Since $A_L$ is closed, we directly infer $\gamma \in A_L$.
	It holds $e_N \weakly \lambda$,  which is equivalent to the convergence of Riemann sums for $f \in C[0,1]$,
	and hence also ${\gamma_N}_{*} e_N \weakly \gamma_{*}\!\dx r$. 
	By the weak continuity of $\mathscr{D}^2_K$, we obtain
	\begin{equation}\label{eq:liminf}
		\mathcal J_L(\gamma) =\mathscr{D}^2_K(\mu,\gamma_{*}\lambda) 
		= \lim_{N \to \infty} \mathscr{D}^{2}_K (\mu, {\gamma_N}_{*}e_N)= \liminf_{N \to \infty} \mathcal J_{L,N}(\gamma_N).
	\end{equation}
	
	2. Next, we prove the $\limsup$-inequality, i.e., we are searching for a sequence $(\gamma_N)_{N\in \N}$ 
	with $\gamma_N \to \gamma$ and $\limsup_{N \to \infty} \mathcal J_{L,N}(\gamma_N) \leq \mathcal J_L(\gamma)$.
	First, we may exclude the trivial case $\mathcal J_L(\gamma) = \infty$.
	Then, $\gamma_N$ is defined on every interval $[(i-1)/N,i/N]$, $i=1,\ldots,N$, as a shortest geodesic 
	from $\gamma((i-1)/N)$ to $\gamma(i/N)$. By construction we have $\gamma_N \in A_{L,N}$.
	From $\gamma,\gamma_N \in A_L$ we conclude
	\begin{align}
		&\sup_{t \in [0,1]}  \dist_\X\bigl(\gamma(t), \gamma_N(t) \bigr) = 
		\max_{i=1,\ldots N} \sup_{t \in [(i-1)/N,i/N]} \dist_\X \bigl(\gamma(t), \gamma_N(t) \bigr)\\
		\leq& \max_{i=1,\ldots N} \sup_{t \in [(i-1)/N,i/N]} \dist_\X \bigl(\gamma(t), \gamma(i/N) \bigr) + \dist_\X \bigl(\gamma_N(i/N), \gamma_N(t) \bigr)\leq \frac{2L}{N},
	\end{align}
	implying $\gamma_N \to \gamma$.
	Similarly as in \eqref{eq:liminf}, we infer $\limsup_{N \to \infty} \mathcal J_{L,N}(\gamma_N) \leq \mathcal J_L(\gamma)$.
\end{proof}

In the numerical part, we use the %{\color{blue} exact, SEBASTIAN, da wolltest du noch mal gucken} 
penalized form of \eqref{eq:full} and minimize
\begin{equation}   \label{eq:constant_speed_min}
	\min_{(x_1,\ldots,x_N) \in \X^N} \mathscr{D}^{2} _K\Big(\mu, \frac{1}{N} \sum_{i=1}^{N} \delta_{x_{i}}\Big) + \frac{\lambda}{N} \sum_{i=1}^{N} \big(N \dist_\X(x_{i-1},x_{i}) - L \big)_+^{2} ,\qquad \lambda >0.
\end{equation} 

%--------------------------------------------------------------
\section{Numerical algorithm} \label{sec:algs}
%--------------------------------------------------------------
For a detailed overview on Riemannian optimization we refer to  \cite{RiWi12} and the books~\cite{Absil:2008qr,Uriste94}.
In order to minimize \eqref{eq:constant_speed_min}, we have a closer look at the discrepancy term.
By \eqref{mercer_1} and \eqref{mercer_2}, the discrepancy can be represented as follows
\begin{align}  
	\mathscr{D}_K^{2} \Big(\mu, \frac{1}{N} \sum_{i=1}^{N} \delta_{x_{i}}\Big)
	=&
	\frac{1}{N^{2}}\sum_{i,j=1}^{N} K(x_{i},x_{j}) - 2\sum_{i=1}^{N} \int_{\mathbb X} K(x_{i},x) \dx \mu(x) + \iint \limits_{\mathbb X \times \mathbb X} K \dx \mu \dx \mu\\[-1ex]
	=&\sum_{k=0}^{\infty} \alpha_{k} \Big| \hat\mu_{k}
	- \frac 1N \sum_{i=1}^{N} \varphi_{k}(x_{i}) \Big|^{2}.
\end{align}
Both formulas have pros and cons:
The first formula allows for an exact evaluation only if the expressions
$\Phi(x) \coloneqq \int_{\mathbb X} K(x,y) \dx \mu(y)$ and $\int_{\mathbb X} \Phi \dx \mu$
can be written in closed forms.
In this case the complexity scales quadratically in the number of points  $N$.
The second formula allows for exact evaluation only if the kernel has a finite expansion \eqref{mercer}. 
In that case the complexity scales linearly in $N$. 

Our approach is to use kernels fulfilling $H_K(\X) = H^s(\X)$, $s > d/2$, and approximating them
by their truncated representation with respect to the eigenfunctions of the Laplace--Beltrami operator
\[
K_r(x,y) \coloneqq \sum_{k\in {\mathcal I}_r} \alpha_{k} \varphi_{k}(x) \overline{\varphi_{k}(y)}, 
\quad 
{\mathcal I}_r \coloneqq \bigl\{k : \varphi_k \in \mathrm{\Pi}_r(\X) \bigr\}.
\]
Then, we finally aim to minimize
\begin{equation}   \label{eq:final}
	\min_{x \in \X^N}  F(x) \coloneqq \sum_{k\in {\mathcal I}_r} \alpha_{k} \Big( \hat\mu_{k}
	- \frac 1N \sum_{i=1}^{N} \varphi_{k}(x_{i}) \Big)^{2}+ \frac{\lambda}{N} \sum_{i=1}^{N} \big(N \dist_\X(x_{i-1},x_{i}) - L \big)_+^{2},
\end{equation} 
where $\lambda >0$.
Our algorithm of choice is the nonlinear conjugate gradient (CG) method with Armijo line search as outlined in Algorithm~\ref{alg:cg} with notation and implementation details described in the comments after Remark~\ref{rem:gehtdoch}, see \cite{Da67} for Euclidean spaces.
Note that the notation is independent of the special choice of $\X$ in our comments.
The proposed method is of ``exact conjugacy'' 
and uses the second order derivative information provided by the Hessian.
For the Armijo line search itself, the sophisticated initialization in Algorithm~\ref{alg:armijo} is used, which also incorporates second order information via the Hessian.
The main advantage of the CG method is its simplicity together with fast convergence at low computational cost. 
Indeed, Algorithm~\ref{alg:cg}, together with Algorithm~\ref{alg:armijo} replaced by an exact line search, 
converges under suitable assumptions superlinearly, more precisely $dN$-step quadratically
towards a local minimum, cf.~\cite[Thm.~5.3]{Sm94} and \cite[Sec.~3.3.2, Thm.~3.27]{Graf:2013zl}.

\begin{remark}\label{rem:gehtdoch}
	The objective in \eqref{eq:final} violates the smoothness requirements whenever 
	$x_{k-1}=x_{k}$ or $\dist_\X(x_{k-1},x_{k}) = L/N$. 
	However, we observe numerically that local minimizers of \eqref{eq:final} 
	do not belong to this set of measure zero. 
	This means in turn, if a local minimizer has a positive definite Hessian, then there is a local neighborhood where the CG method (with exact line search) permits a superlinear convergence rate.
	We do indeed observe this behavior in our numerical experiments.
\end{remark}
%%%%%%%%%%%%%%%%%%%%%%%%%%%%%%%%%%%%%%%%%%%%%%%%%%%%%%%%%%%%%%%%%%%%%%%%%%%%%%%%%%%%%%%%%%%%%%%%%%%%%%%%%%
\begin{algorithm}[t]
	\caption{(\textbf{CG Method with Restarts})}
	\begin{algorithmic}
		\State \textbf{Parameters:} maximal iterations $k_{\max} \in \mathbb N$
		\State \textbf{Input:} twice differentiable function $F\colon\mathbb X^{N} \to [0,\infty)$,
		initial point $x^{(0)} \in \mathbb X^{N}$
		\State \textbf{Initialization:} $g^{(0)} \coloneqq \nabla_{\mathbb X^{N}}F\bigl(x^{(0)}\bigr)$, 
		$d^{(0)}\coloneqq - g^{(0)}$, $r\coloneqq0$
		\For{$k\coloneqq0,\dots,k_{\max}$}
		%    \State $\tau^{(k)}_{0} \coloneqq
		%    \begin{cases}
		%      \left| \frac{ \langle d^{(k)}, g^{(k)} \rangle}{ \langle d^{(k)}, \mathrm{H}_{\mathbb X^{N}}f(x^{(k)}) d^{(k)}\rangle } \right|,
		%      & \langle d^{(k)}, \mathrm{H}_{\mathbb X^{N}}f(x^{(k)}) d^{(k)} \rangle \ne 0,\\
		%      1, & \text{ else}
		%    \end{cases}
		%    $
		%   % $ \tau^{(k)}\coloneqq \mathrm{min} \big\{ \frac{\mathrm d}{\mathrm d t} f \circ \gamma_{x^{(k)},d^{(k)}}(t) = 0 \big\}$  (perform a line search)
		\State $x^{(k+1)} \coloneqq \gamma_{x^{(k)}, d^{(k)}}\bigl(\tau^{(k)}\bigr)$  where $\tau^{(k)}$ is determined by Algorithm~\ref{alg:armijo}
		\State
		$
		\tilde d^{(k)}\coloneqq  \dot \gamma_{x^{(k)}, d^{(k)}}\bigl(\tau^{(k)}\bigr)
		%\in \Ts_{\zb  x^{(k+1)}}\M.
		$
		\State  $g^{(k+1)} \coloneqq \nabla_{\mathbb X^{N}}F\bigl(x^{(k+1)}\bigr)$
		\State
		$
		\beta^{(k)}  \coloneqq  
		\begin{cases}
			\frac{ \bigl\langle \tilde d^{(k)}, \mathrm H_{\mathbb X^{N}} F(x^{(k+1)}) g^{(k+1)}\bigr\rangle }{\bigl\langle \tilde d^{(k)}, \mathrm H_{\mathbb X^{N}} F(x^{(k+1)}) \tilde d^{(k)} \bigr\rangle},
			& \bigl\langle \tilde d^{(k}, \mathrm H_{\mathbb X^{N}}F\bigl(x^{(k+1)}\bigr) \tilde d^{(k)} \bigr\rangle \ne 0,\\
			0, & \text{ else}
		\end{cases}
		$
		\State $d^{(k+1)}\coloneqq- g^{(k+1)} + \beta^{(k)} \tilde d^{(k)}$
		\If{$\bigl\langle d^{(k+1)}, g^{(k+1)} \bigr\rangle > 0$ or $(k+1) \equiv r \mod N \mathrm{dim}(\mathbb X)$}
		\State $d^{(k+1)} = - g^{(k+1)}$
		\State $r \coloneqq k+1$
		\EndIf
		\EndFor
		\State \textbf{Output:} iteration sequence $x^{(0)},x^{(1)}, \dots \in \mathbb X^{N}$
	\end{algorithmic}
	\label{alg:cg}
\end{algorithm}
%-----------------------------------------------------------------
\begin{algorithm}[t]
	\caption{(\textbf{Armijo Line Search})}
	\begin{algorithmic}
		\State \textbf{Parameters:} $0 < a < \tfrac12$, $0 < b < 1$, maximal iterations $k_{\max} \in \mathbb N$
		\State \textbf{Input:} smooth function $F\colon\mathbb X^{N} \to [0,\infty)$, start point $x \in \mathbb X^{N}$, 
		descent direction $d \in \mathrm T_{x}\mathbb X^{N}$
		\State \textbf{Initialization:}
		$k\coloneqq0$,
		\[\tau^{(0)} \coloneqq
		\begin{cases}
			\left| \frac{ \bigl\langle d, \nabla_{\mathbb X^{N}} F(x) \bigr\rangle}{ \bigl\langle d, \mathrm{H}_{\mathbb X^{N}}F(x) d\bigr\rangle } \right|,
			& \bigl\langle d, \mathrm{H}_{\mathbb X^{N}}F(x) d \bigr\rangle \ne 0,\\
			1, & \text{ else}
		\end{cases}
		\]
		\While{$f \circ \gamma_{x,d} \bigl(\tau^{(k)}\bigr) - F(x) \ge a \tau^{(k)} \bigl\langle \nabla_{\mathbb X^{N}} F(x), d \bigr\rangle$ and $k < k_{\max}$}
		\State $\tau^{(k+1)} \coloneqq b \tau^{(k)}$
		\State $k\coloneqq k+1$
		\EndWhile
		\State \textbf{Output:} $\tau^{(k)}$ (success if $k \le k_{\max}$)
	\end{algorithmic}
	\label{alg:armijo}
\end{algorithm}
%%%%%%%%%%%%%%%%%%%%%%%%%%%%%%%%%%%%%%%%%%%%%%%%%%%%%%%%%%%%%%%%%%%%%%%%%%%%%%%%%%%%%%%%%%%%%%%%%%%%%%%%%
Let us briefly comment on Algorithm \ref{alg:cg}
for  $\X \in \{\mathbb T^2, \mathbb T^3, \S^2,\SO(3),\mathcal G_{2,4}\}$ which are considered in our numerical examples. 
For additional implementation details we refer to \cite{Graf:2013zl}.
By $\gamma_{x, d}$ we denote the geodesic with $\gamma_{x, d}(0) = x$ and $\dot \gamma_{x, d}(0) = d$.
Besides evaluating the geodesics $\gamma_{x^{(k)}, d^{(k)}}(\tau^{(k)})$ in the first iteration step, 
we have to compute the parallel transport of \smash{$d^{(k)}$} along the geodesics in the second step.
%Furthermore, we need to compute the Riemannian gradient $\nabla_{\X^N}F$, the Hessian  $H_{\X^N} F$ and perform matrix-vector multiplications with the Hessian, which is approximated by the finite difference
Furthermore, we need to compute the Riemannian gradient $\nabla_{\X^N}F$ and products of the Hessian $H_{\X^N} F$ with vectors $d$, which are approximated by the finite difference
\[
\mathrm H_{\mathbb X^{N}}F(x) d \approx \tfrac{\|d\|}{h} \Bigl(\nabla_{\mathbb X^{N}}F\bigl(\gamma_{x,h d/\|d\|}\bigr) - \nabla_{\mathbb X^{N}}F(x)\Bigr), \qquad h\coloneqq 10^{-8}.
\]
%{\color{blue} not! detailed in  the papers cited in Table \ref{tab:1}}.
The computation of the gradient of the \emph{penalty term} in (30) is done by applying the chain rule
and noting that for $x \mapsto \dist_\X(x,y)$, 
we have $\nabla_\X \dist_\X(x,y) = \log_x y/\dist_\X(x,y)$, $x\not = y$ 
with the logarithmic map $\log$ on $\X$, while the distance
is not differentiable for $x=y$. Concerning the later point, see Remark~5.
The evaluation of the gradient of the penalty term at a point in $\X^N$ requires only $\mathcal O(N)$ arithmetic operations.
The computation of the Riemannian gradient of the \emph{data term} in (30) is done analytically via
the gradient of the
eigenfunctions $\varphi_k$ of the Laplace--Beltrami operator. 
Then, the evaluation of the gradient of the whole data term at given points 
can be done efficiently by \emph{fast Fourier transform}~(FFT) techniques at non-equispaced nodes using the NFFT software package of Potts et al.~\cite{Keiner:2009cv}.
The overall complexity of the algorithm and references for the computation details for the above manifolds are given in Table \ref{tab:1}.

\begin{table}
	\begin{center}
		\begin{tabular}{l|l|l}
			$\X$ &Reference &Complexity\\
			\hline
			$\T^d$   &\cite{Graf:2013fk}, \cite[Sec.~5.2.1]{Graf:2013zl}           &$\mathcal O(r^{d} \log(r) + N)$  \\
			$\S^2$   &\cite{Graf:2011lp,Graf:2013fk}, \cite[Sec.~5.2.2]{Graf:2013zl}&$\mathcal O(r^{2} \log^{2}(r) + N)$\\
			$\SO(3)$ &\cite{graef2012,Graf:2009ye}, \cite[Sec.~5.2.3]{Graf:2013zl}  &$\mathcal O(r^{3} \log^{2}(r) + N)$\\
			$\G_{2,4}$& \cite{Dick:2019sy}                                              & $\mathcal O(r^{4} \log^{2}(r) + N)$
		\end{tabular}
	\end{center}
	\caption{References for implementation details of Alg.~\ref{alg:cg} (left) and arithmetic complexity for the evaluations per iteration
		for the different manifolds (right). }\label{tab:1}
\end{table}

%--------------------------------------------------------------
\section{Numerical results} \label{sec:numerics_dith}
%--------------------------------------------------------------
In this section, we underline our theoretical results by numerical examples.
We start by studying the parameter choice in our numerical model.
Then, we provide examples for the approximation of absolutely continuous measures with densities in $H^s(\X)$, $s > d/2$, by push-forward measures of the Lebesgue measure on $[0,1]$ by Lipschitz curves for the manifolds  $\X \in \{\mathbb T^2, \mathbb T^3, \S^2,\SO(3),G_{2,4}\}$.
Supplementary material can be found on our webpage.

\subsection{Parameter choice}
We like to emphasize that the optimization problem \eqref{eq:final} is highly nonlinear and the objective function has a large number of local minimizers, which appear to increase exponentially in N. In order to find for fixed $L$ reasonable (local) solutions of \eqref{eq:min_PL}, we carefully adjust the parameters in problem~\eqref{eq:final}, namely the number of points $N$, the polynomial degree $r$ in the kernel truncation, and the penalty parameter $\lambda$.
In the following, we suppose that $\mathrm{dim}(\supp(\mu)) = d \ge 2$.
\begin{itemize}
	\item[i)] \textbf{Number of points $N$}: 
	Clearly, $N$ should not be too small compared to $L$. However, from a computational perspective it should also be not too large since the optimization procedure is hampered by the vast number of local minimizers.
	From the asymptotic of the path lengths of TSP in Lemma \ref{lemma TSP}, we conclude that
	$N \gtrsim \mathcal \ell(\gamma)^{d/(d-1)}$
	is a reasonable choice, where $\mathcal \ell(\gamma)  \le L$ is the length of the resulting curve $\gamma$ going through the points.
	%So we should keep in mind $N \sim L^{\frac{d}{d-1}}$.
	\item[ii)] \textbf{Polynomial degree} $r$: 
	Based on the proofs of the theorems in Subsection \ref{sec:approx_L} it is reasonable to choose 
	\begin{equation} \label{param2}
		r \sim L^{\frac{1}{d-1}} \sim N^{\frac{1}{d}}.
	\end{equation}
	\item[iii)] \textbf{Penalty parameter} $\lambda$:
	If $\lambda$ is too small, we cannot enforce that the points approximate a regular curve, i.e., $L/N \gtrsim \dist_\X(x_{k-1},x_{k})$. Otherwise, if $\lambda$ is too large the optimization procedure is hampered by the rigid constraints. 
	Hence, to find a reasonable choice for $\lambda$ in dependence on $L$, we assume that the minimizers of \eqref{eq:final}  treat both terms proportionally, i.e., for $N\to \infty$ both terms are of the same order.
	Therefore, our heuristic is to choose the parameter $\lambda$ 
	such that 
	\begin{equation}   \label{eq:penalty_behavior}
		\min_{x_{1},\dots,x_{N}} \mathscr{D}_K^{2} \Big(\mu, \frac{1}{N} \sum_{k=1}^{N} \delta_{x_{k}}\Big) \sim N^{-\frac{2s}{d}} \sim
		\frac{\lambda}{N} \sum_{k=1}^{N} \big(N \dist_\X(x_{k-1},x_{k}) - L \big)_{+}^{2}  .
	\end{equation}
	On the other hand,  assuming that for the length 
	$\ell(\gamma) = \sum_{k=1}^{N}\dist_\X(x_{k-1},x_{k})$ of a minimizer $\gamma$ 
	we have $\ell(\gamma) \sim L \sim N^{(d-1)/d}$, 
	so that $N \dist_\X(x_{k-1},x_{k}) \sim L$, the value of the penalty term behaves like
	\[
	\frac{\lambda}{N} \sum_{k=1}^{N} \big(N \dist_\X(x_{k-1},x_{k}) - L \big)_{+}^{2} \sim  \lambda L^{2} \sim \lambda N^{\frac{2d-2}{d}}.
	\]
	Hence, a reasonable choice is
	\begin{equation}   \label{param3}
		\lambda \sim L^{\frac{-2s-2(d-1)}{d-1}} \sim N^{\frac{-2s-2(d-1)}{d}}.
	\end{equation}
\end{itemize}

\begin{remark}
	\label{rem:param_choice}
	In view of Remark~\ref{rem:andere} %we use for  $2 \le d_{\mu} \coloneqq \mathrm{dim} (\supp(\mu)) < d$
	the relations in i)-iii) become
	$$
	N \sim L^{\frac{d_{\mu}}{d_{\mu}-1}}, \quad
	r \sim N^{\frac{1}{d_{\mu}}} \sim L^{\frac{1}{d_{\mu}-1}}, \quad
	\lambda \sim L^{\frac{-2s-3d_{\mu}+d+2}{d_{\mu}-1}} \sim N^{\frac{-2s-3d_{\mu}+d+2}{d_{\mu}}}.
	$$
\end{remark} 

In the rest of this subsection, we aim to provide some numerical evidence for the parameter choice above.
We restrict our attention to the torus $\X = \mathbb T^{2}$ and 
the kernel $K$ given in \eqref{kernel_Td} with $d=2$ and $s = 3/2$.
Choose $\mu$ as the Lebesgue measure on $\mathbb T^{2}$.
From \eqref{param3}, we should keep in mind 
$\lambda \sim N^{-5/2} \sim L^{-5}$. 

\textbf{Influence of $N$ and $\lambda$.}
We fix $L=4$ and a large polynomial degree $r=128$ for truncating the kernel.
For any $\lambda_{i}=0.1\cdot 2^{-5i/2}$, $i=1,\dots,4$, we compute local minimizers with $N_{j}=10 \cdot 2^{j}$, $j=1,\dots,4$. 
More precisely, keeping $\lambda_{i}$ fixed we start with $N_{1}=20$ and refine successively the curves 
by inserting the midpoints of the line segments connecting consecutive points 
and applying a local minimization with this initialization.
The results are depicted in Fig.~\ref{fig:Nexperiment}. 
For fixed $\lambda$ (fixed row) we can clearly notice that the local minimizers converge 
towards a smooth curve for increasing $N$. 
Moreover, the diagonal images correspond to the choice  $\lambda =0.1 (N/10)^{-5/2}$, 
where we can already observe good approximation of the curves emerging to the right of it. 
This should provide some evidence that the choice of the penalty parameter $\lambda$ and the number of points $N$ discussed above is reasonable. 
Indeed, for $\lambda \to \infty$  we observe $L(\gamma) \to \ell(\gamma) \to L = 4$.

%\begin{figure}
%\centering
%  \begin{tabular}{cccc}
%  $N=20$   & $N=40$  & $N=80$  & $N=160$ \\
%  \includegraphics[width=0.22\textwidth]{Images/N20_L4_a1000} & 
%  \includegraphics[width=0.22\textwidth]{Images/N40_L4_a1000} &
%  \includegraphics[width=0.22\textwidth]{Images/N80_L4_a1000} &
%  \includegraphics[width=0.22\textwidth]{Images/N160_L4_a1000}
%  \end{tabular}
%  \caption{ Local minimizers for $\alpha=1000$ have length $\ell(\gamma)=4$. }    
%    \label{fig:l4}
%\end{figure}
% 
\begin{figure}
	\centering
	\begin{tabular}{cccc}
		$N=20$   & $N=40$  & $N=80$  & $N=160$ \\
		\includegraphics[width=0.21\textwidth]{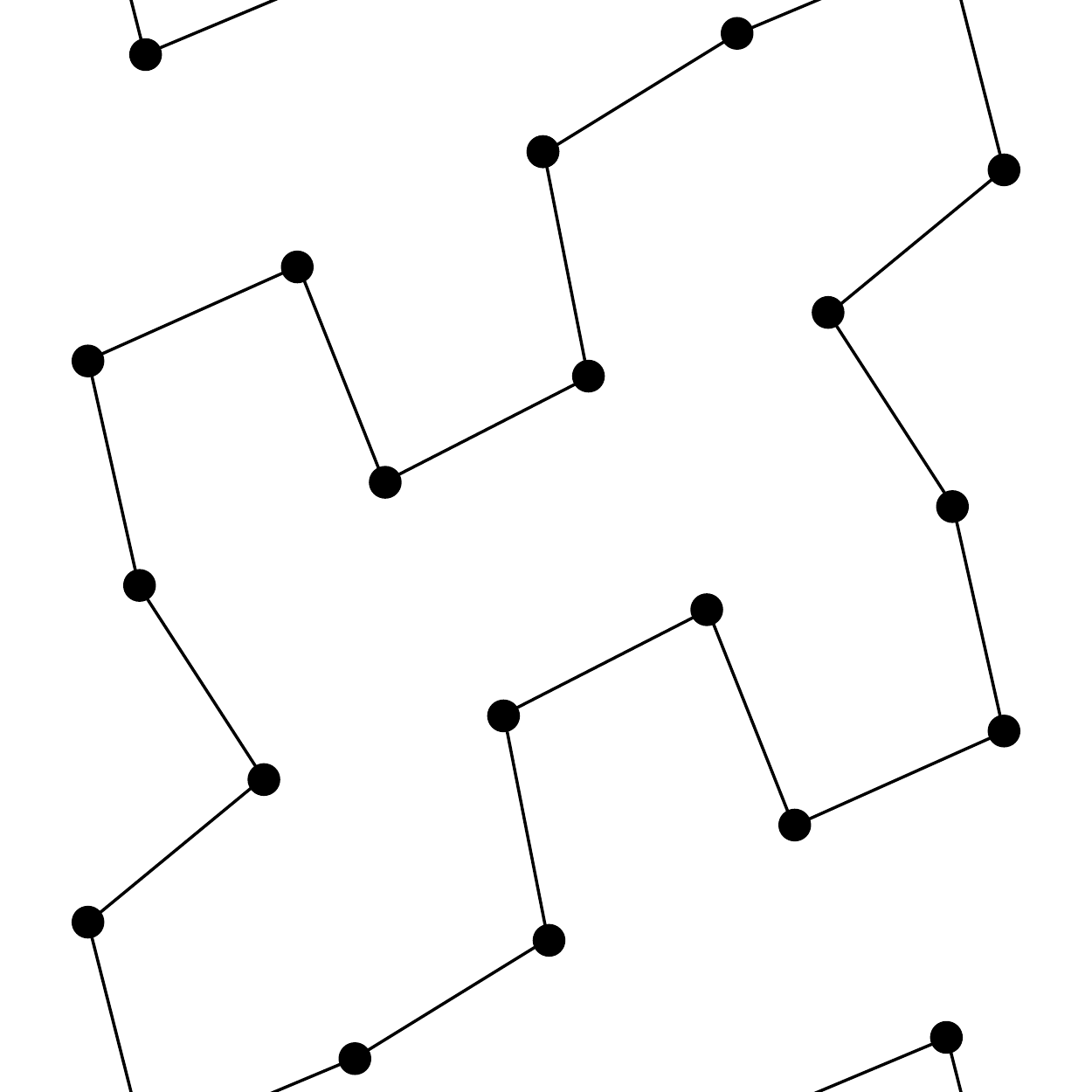} & 
		\includegraphics[width=0.21\textwidth]{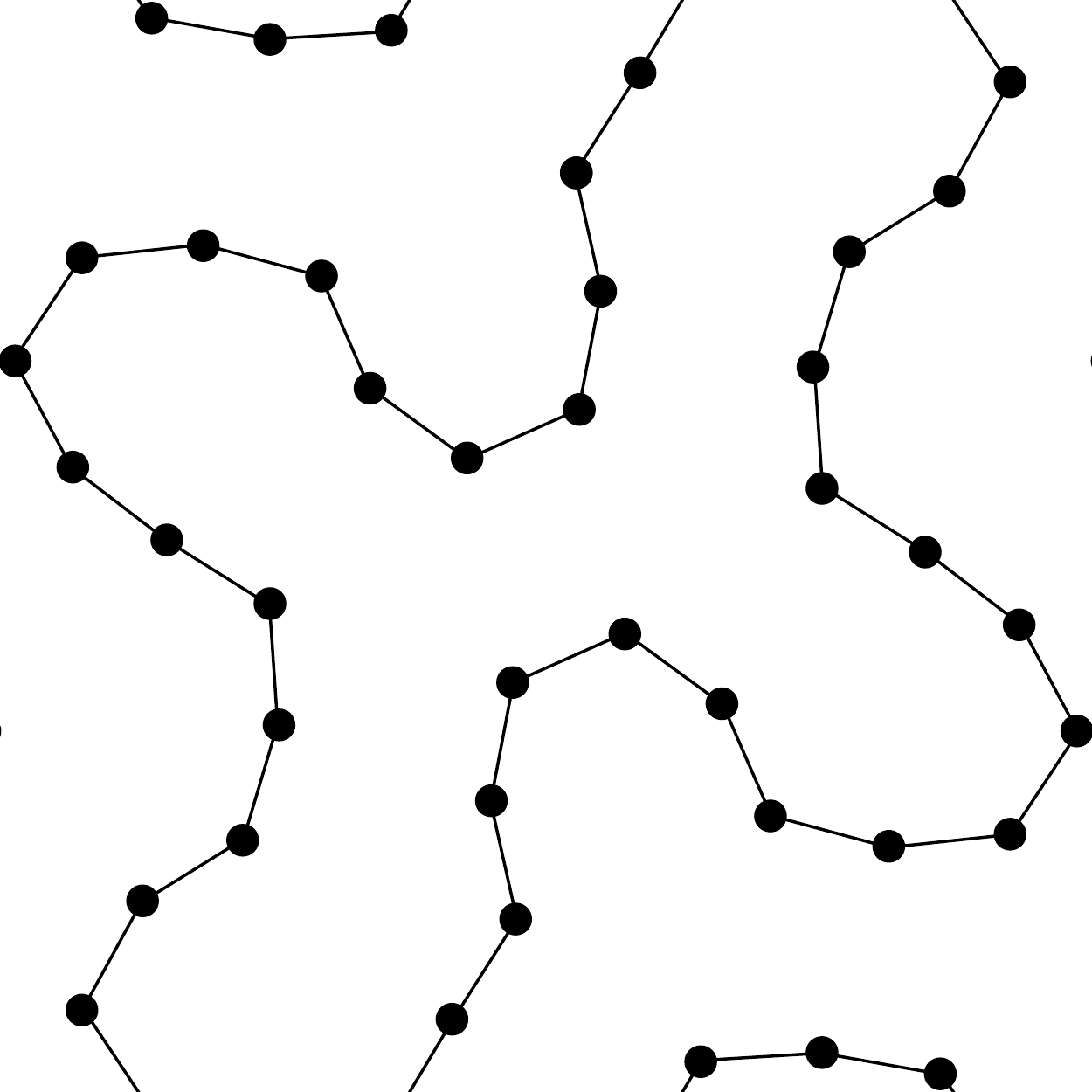} &
		\includegraphics[width=0.21\textwidth]{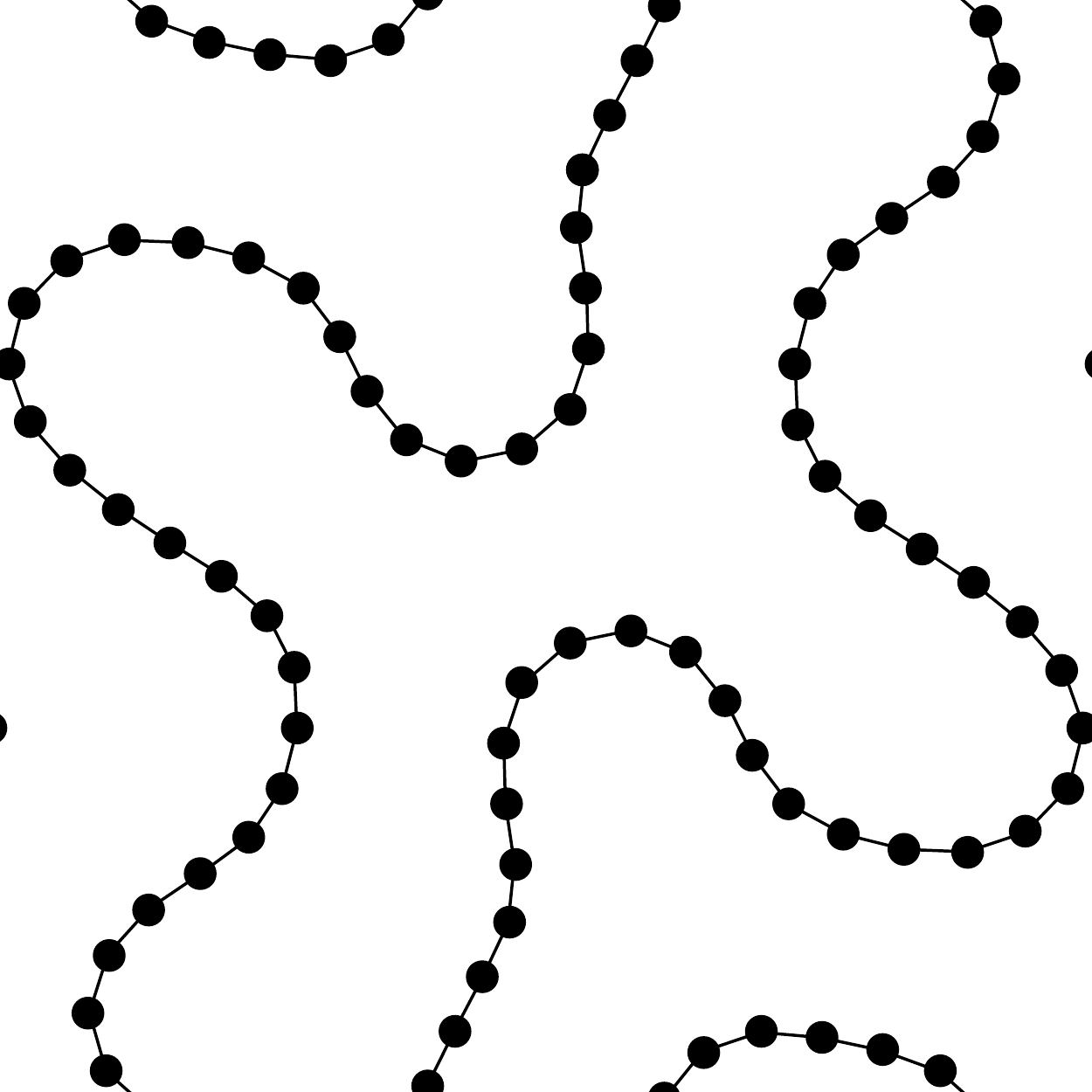} &
		\includegraphics[width=0.21\textwidth]{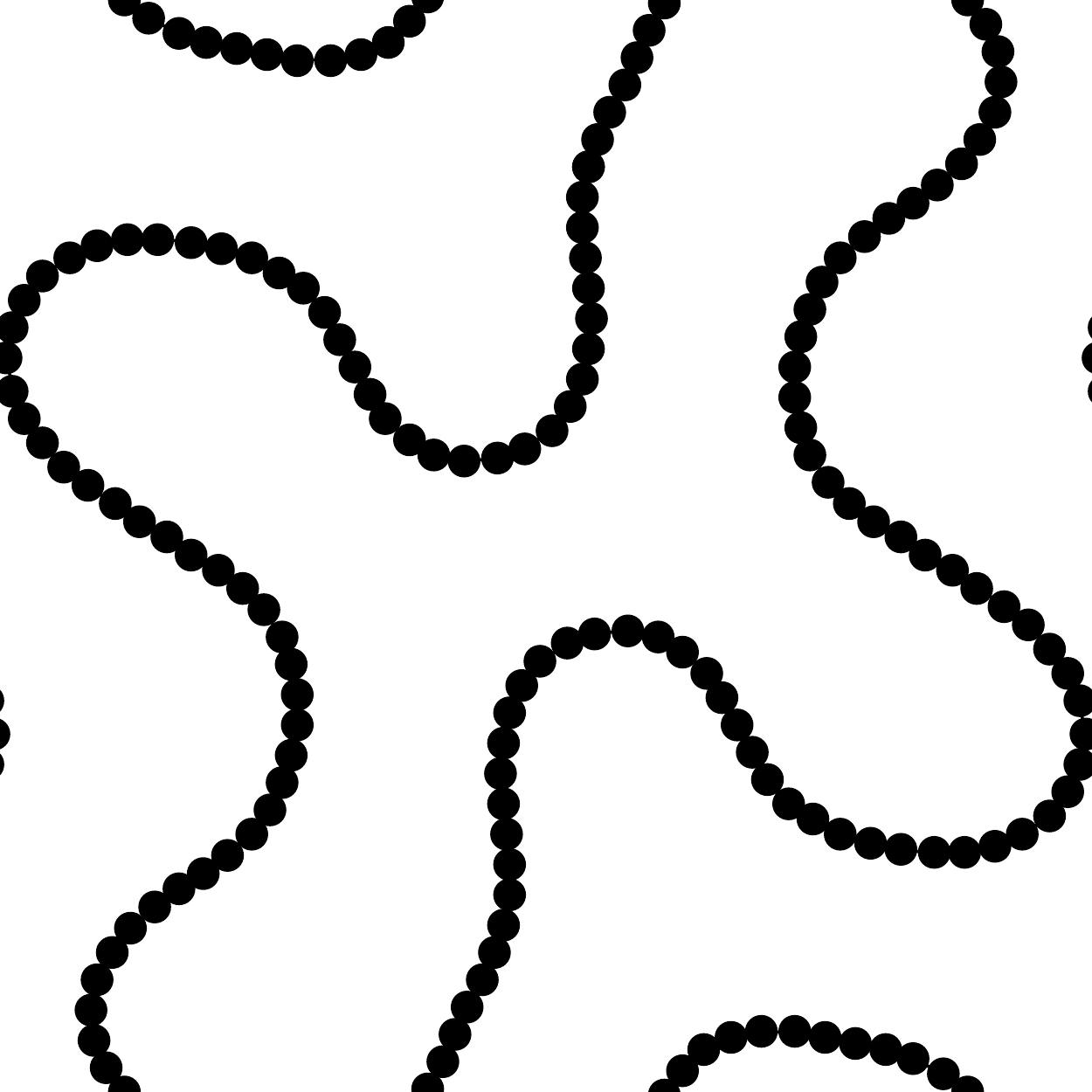}\\
		$\ell(\gamma)\approx 4.20$ &
		$\ell(\gamma)\approx 4.43$ &
		$\ell(\gamma)\approx 4.49$ &
		$\ell(\gamma)\approx 4.50$\\ 
		\includegraphics[width=0.21\textwidth]{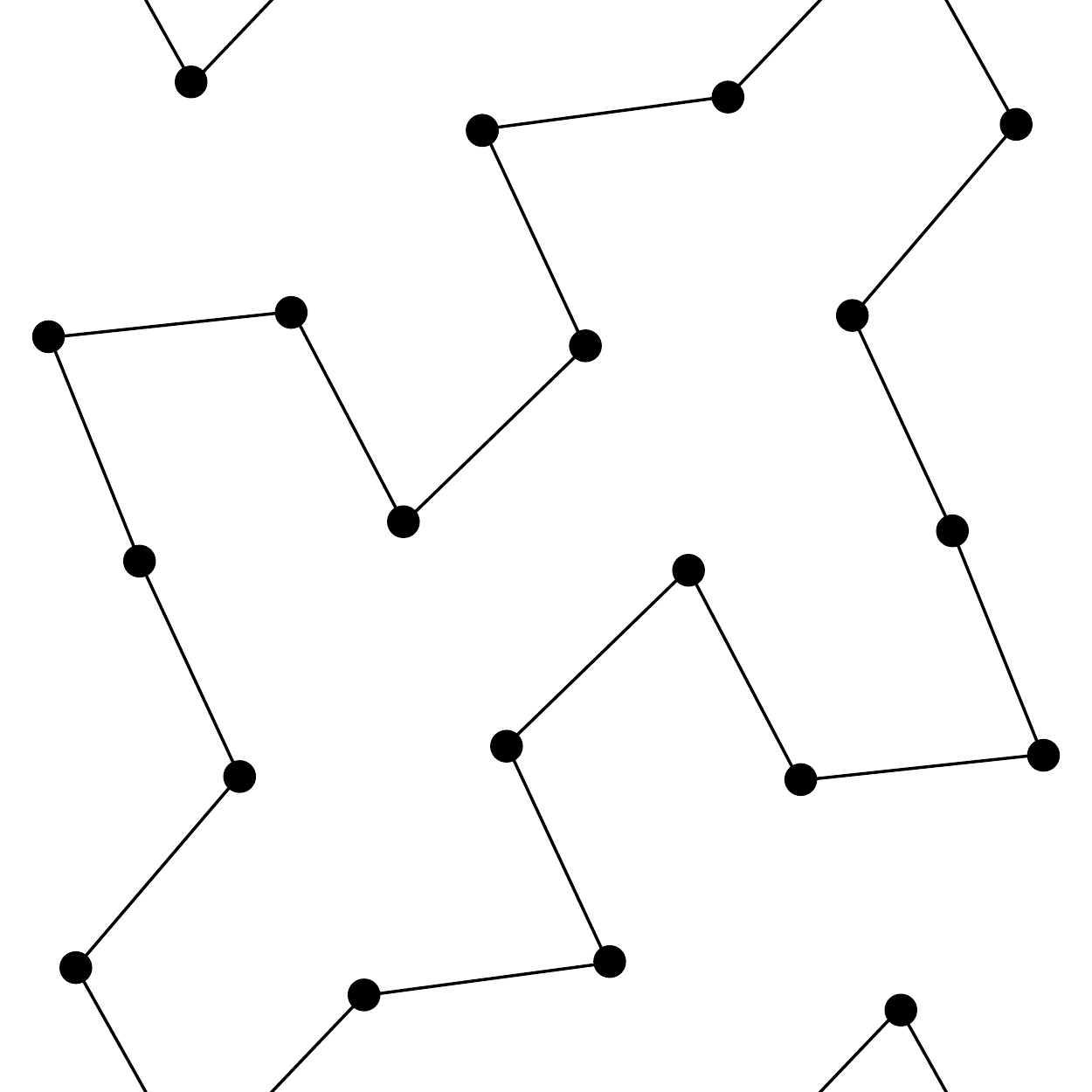} & 
		\includegraphics[width=0.21\textwidth]{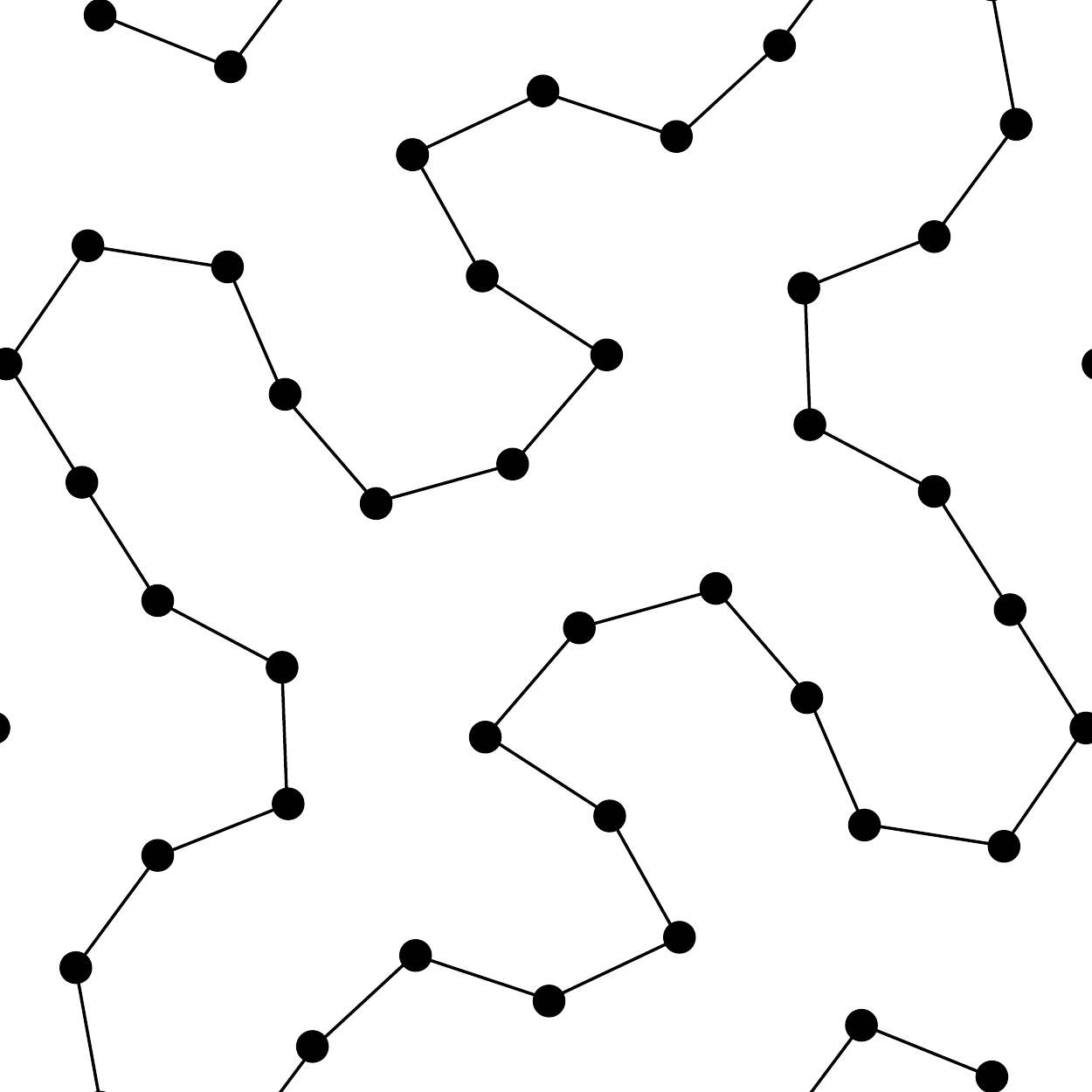} &
		\includegraphics[width=0.21\textwidth]{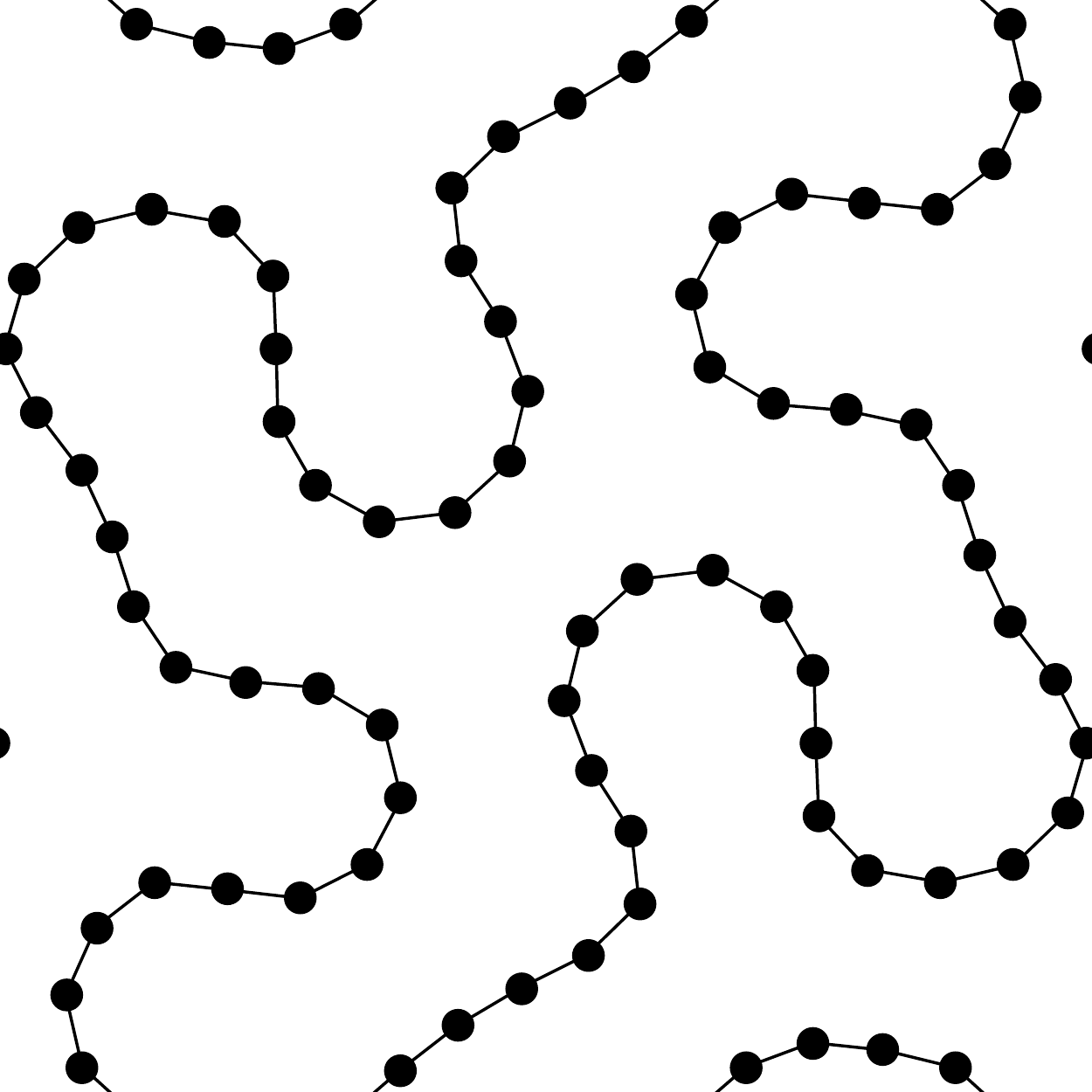} &
		\includegraphics[width=0.21\textwidth]{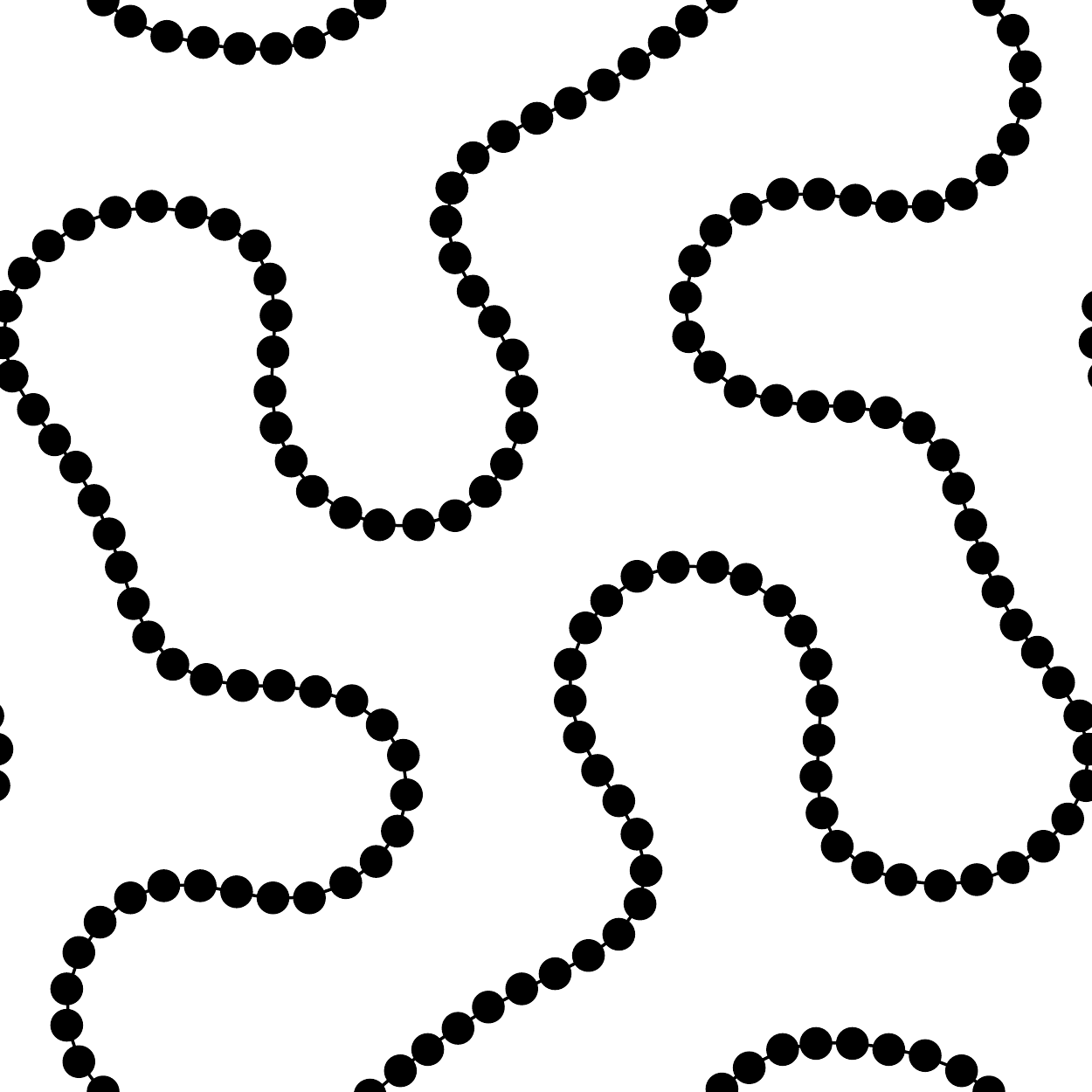}\\
		$\ell(\gamma)\approx 4.47$ &
		$\ell(\gamma)\approx 5.16$ &
		$\ell(\gamma)\approx 5.38$ &
		$\ell(\gamma)\approx 5.44$\\ 
		\includegraphics[width=0.21\textwidth]{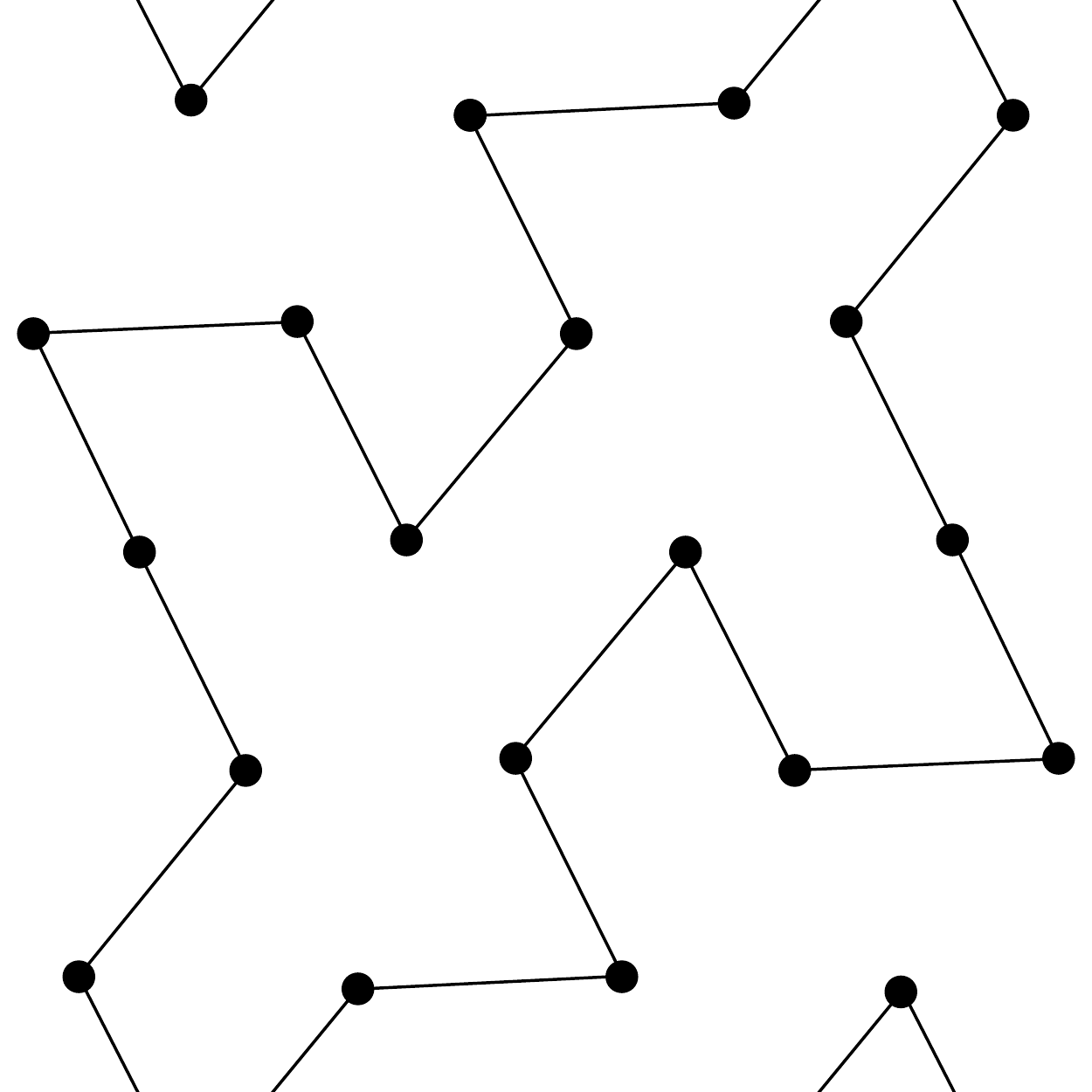} & 
		\includegraphics[width=0.21\textwidth]{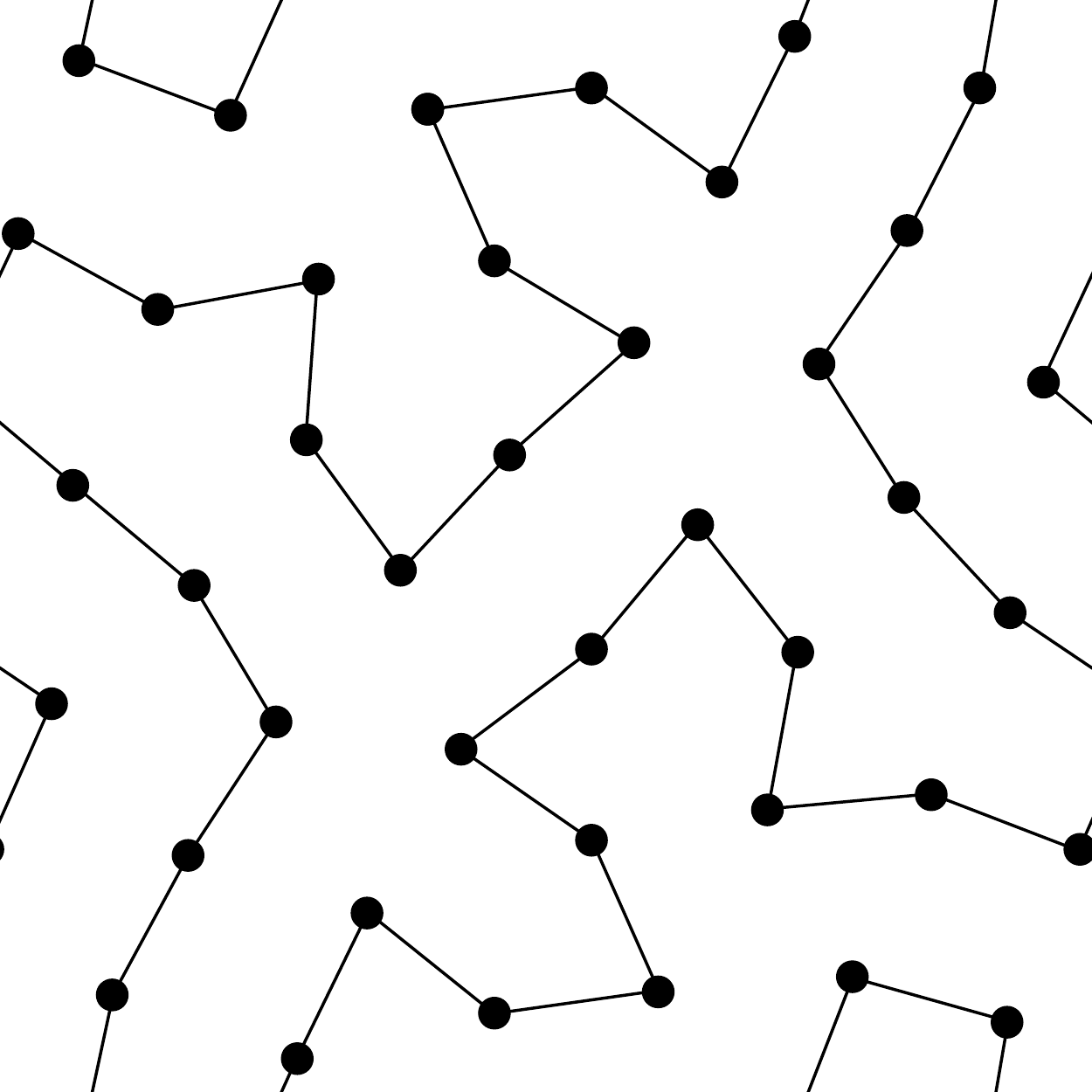} &
		\includegraphics[width=0.21\textwidth]{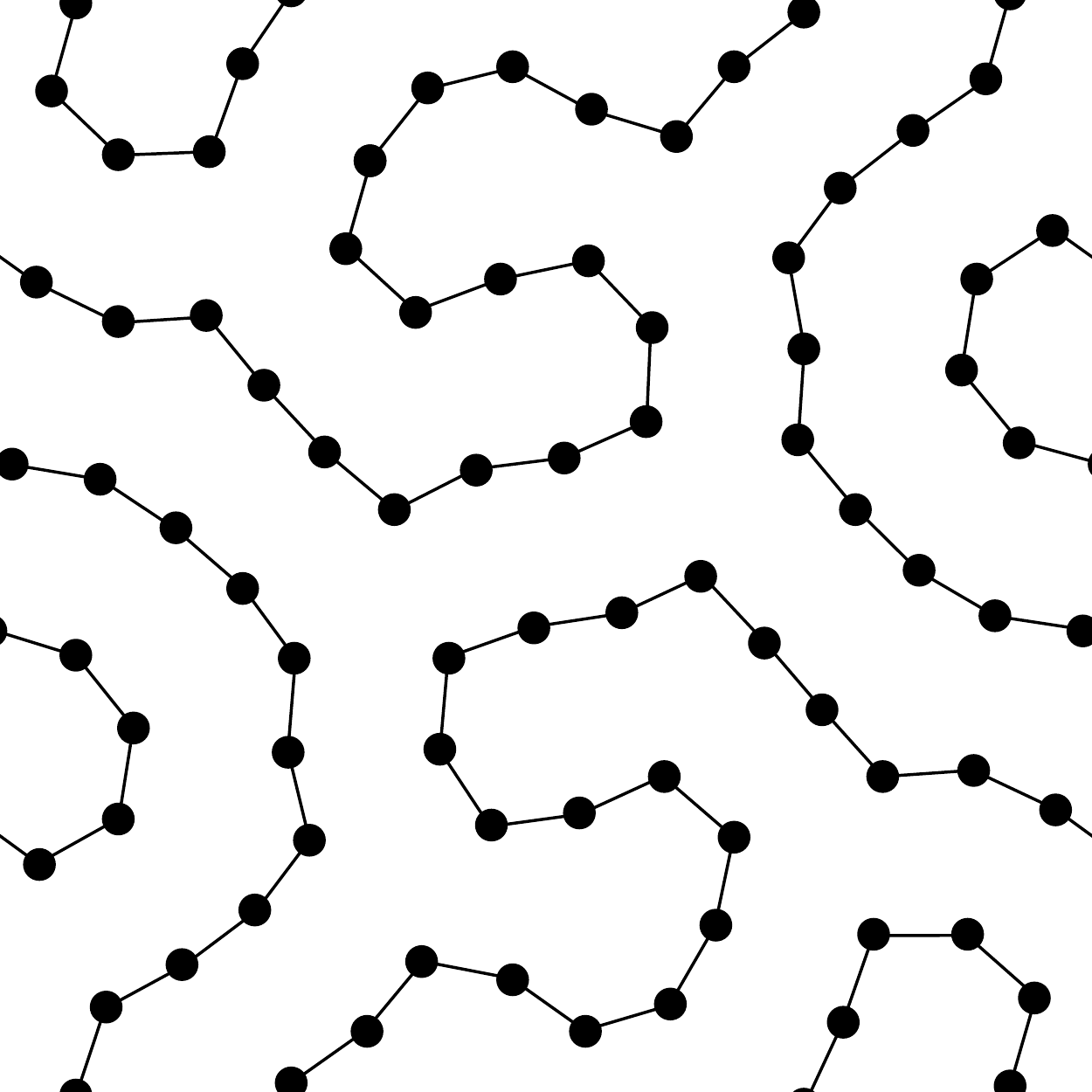} &
		\includegraphics[width=0.21\textwidth]{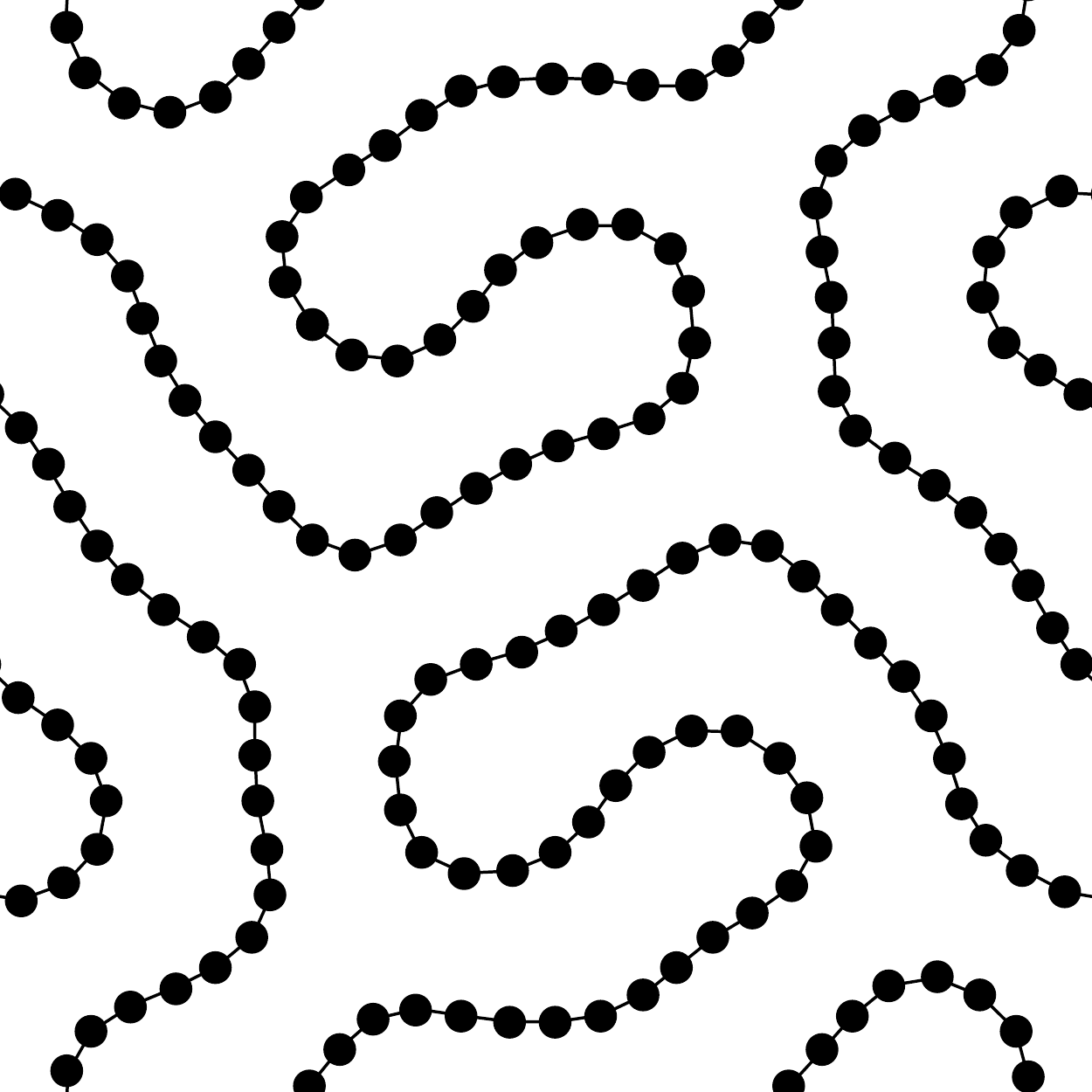}\\
		$\ell(\gamma)\approx 4.66$ &
		$\ell(\gamma)\approx 5.91$ &
		$\ell(\gamma)\approx 6.64$ &
		$\ell(\gamma)\approx 6.87$\\ 
		\includegraphics[width=0.21\textwidth]{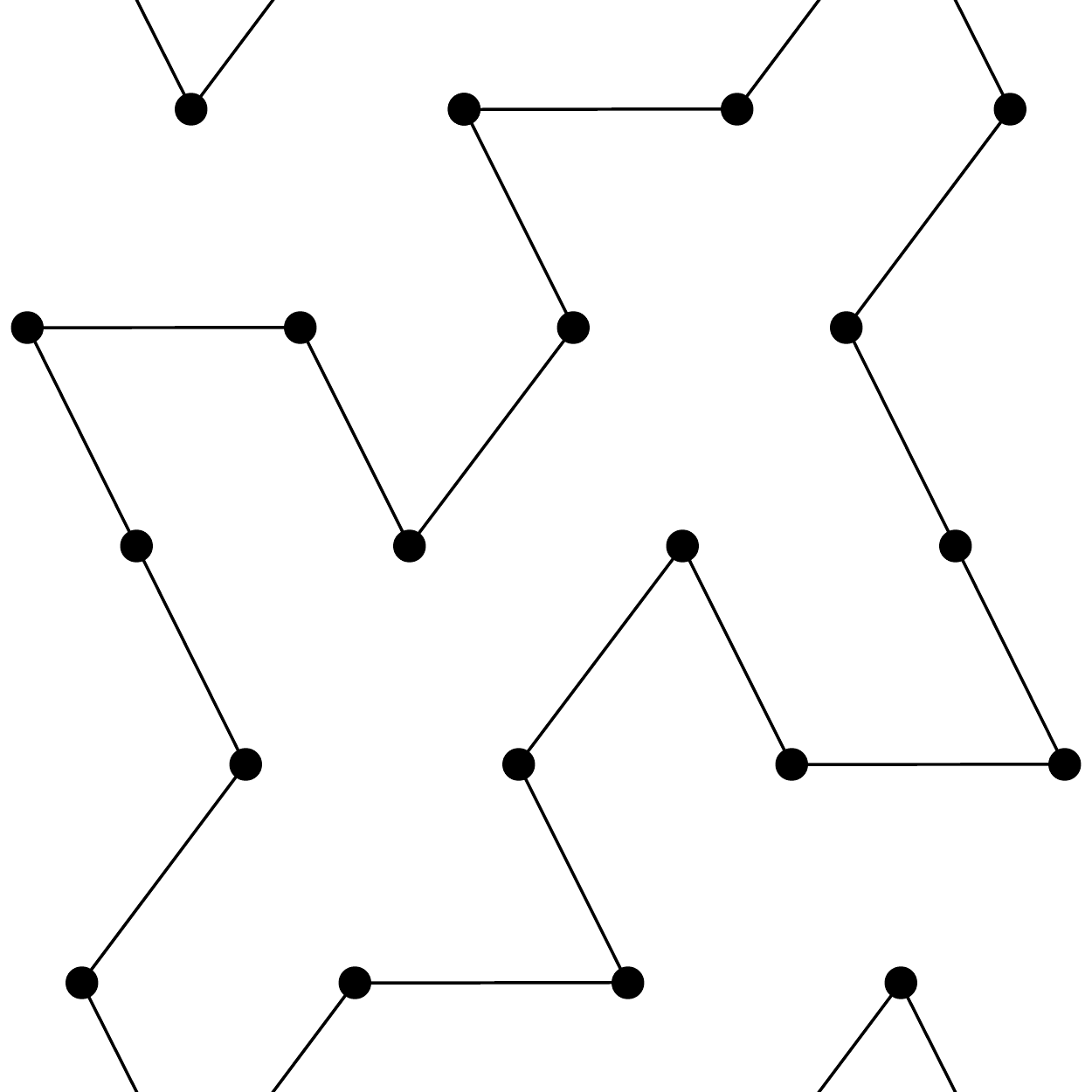} & 
		\includegraphics[width=0.21\textwidth]{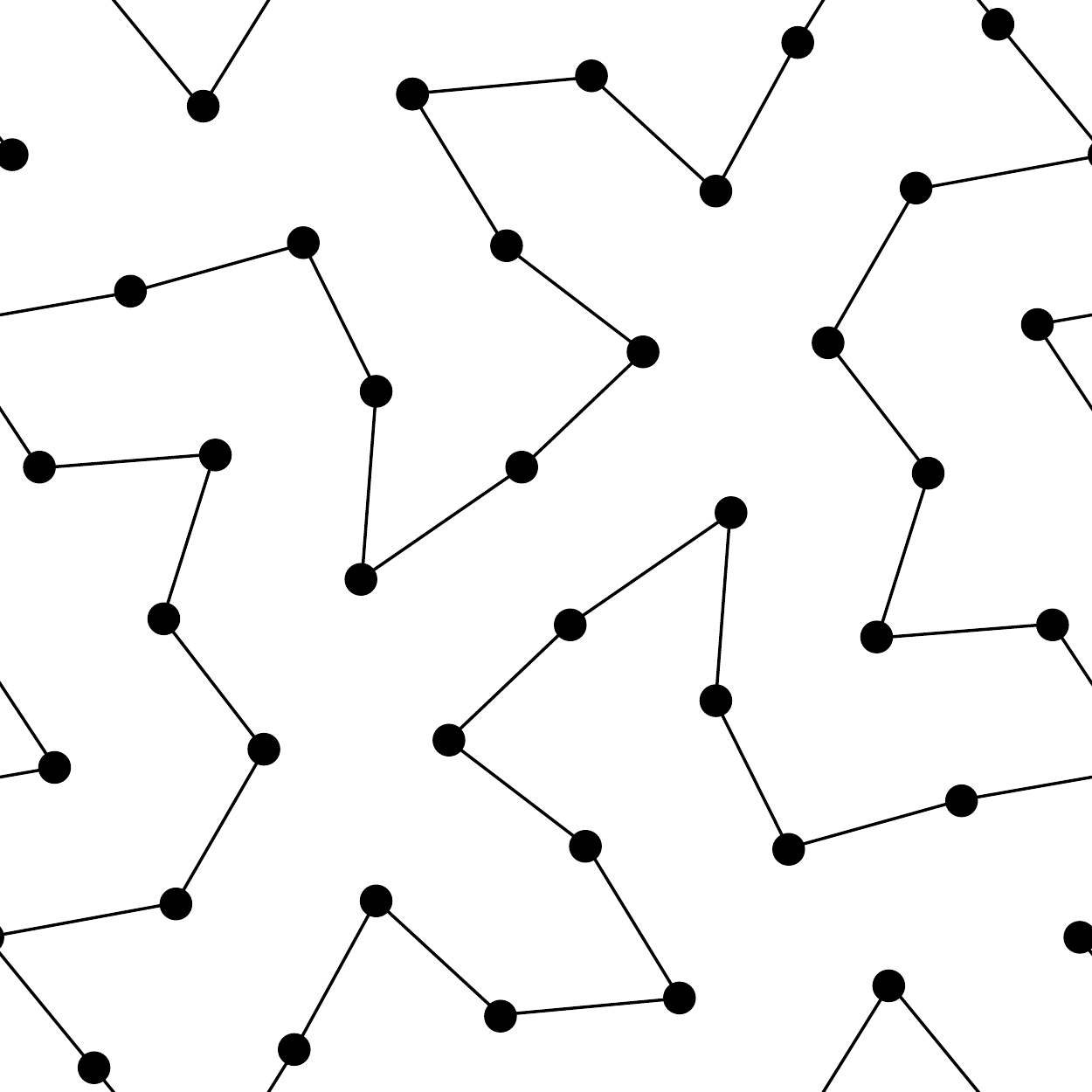} &
		\includegraphics[width=0.21\textwidth]{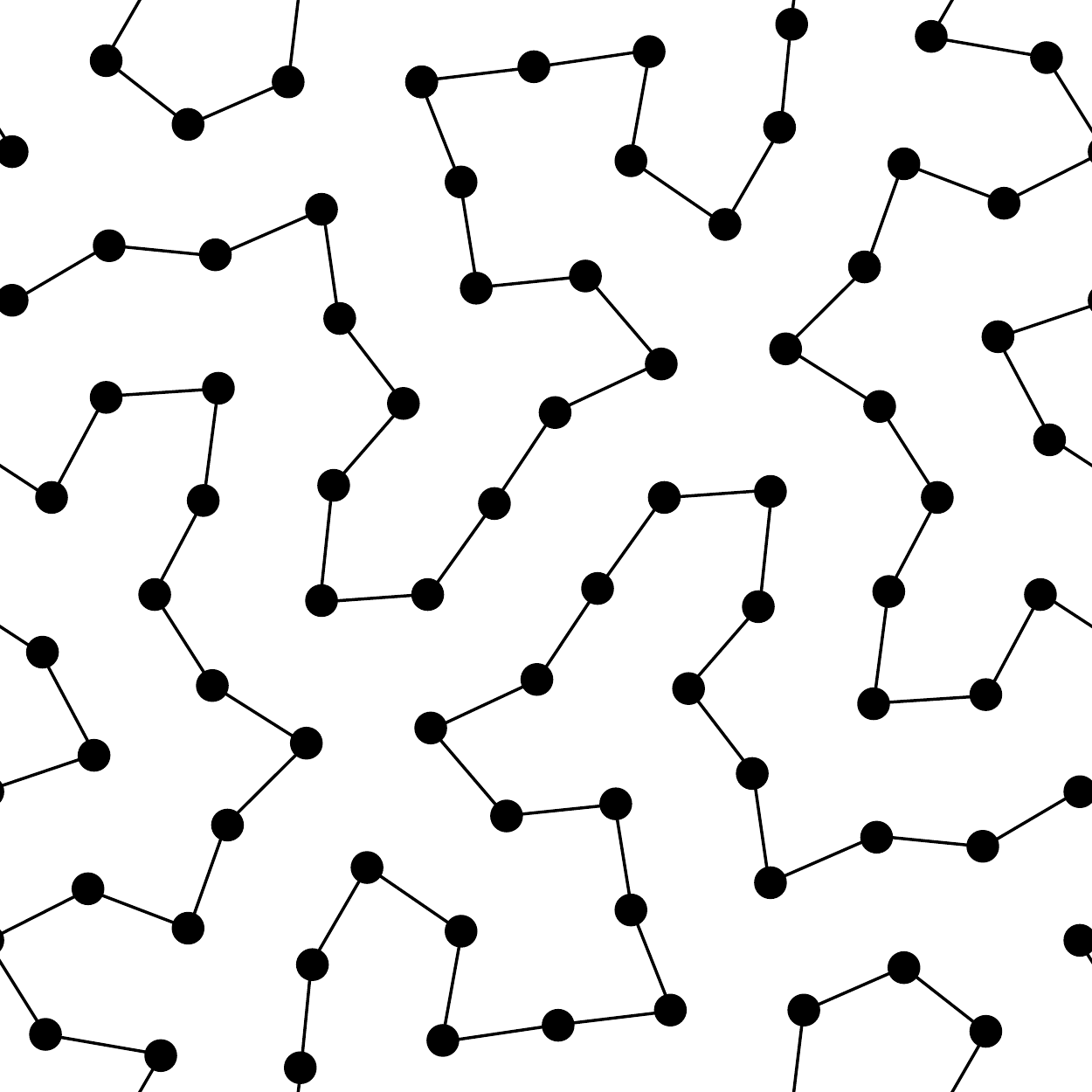} &
		\includegraphics[width=0.21\textwidth]{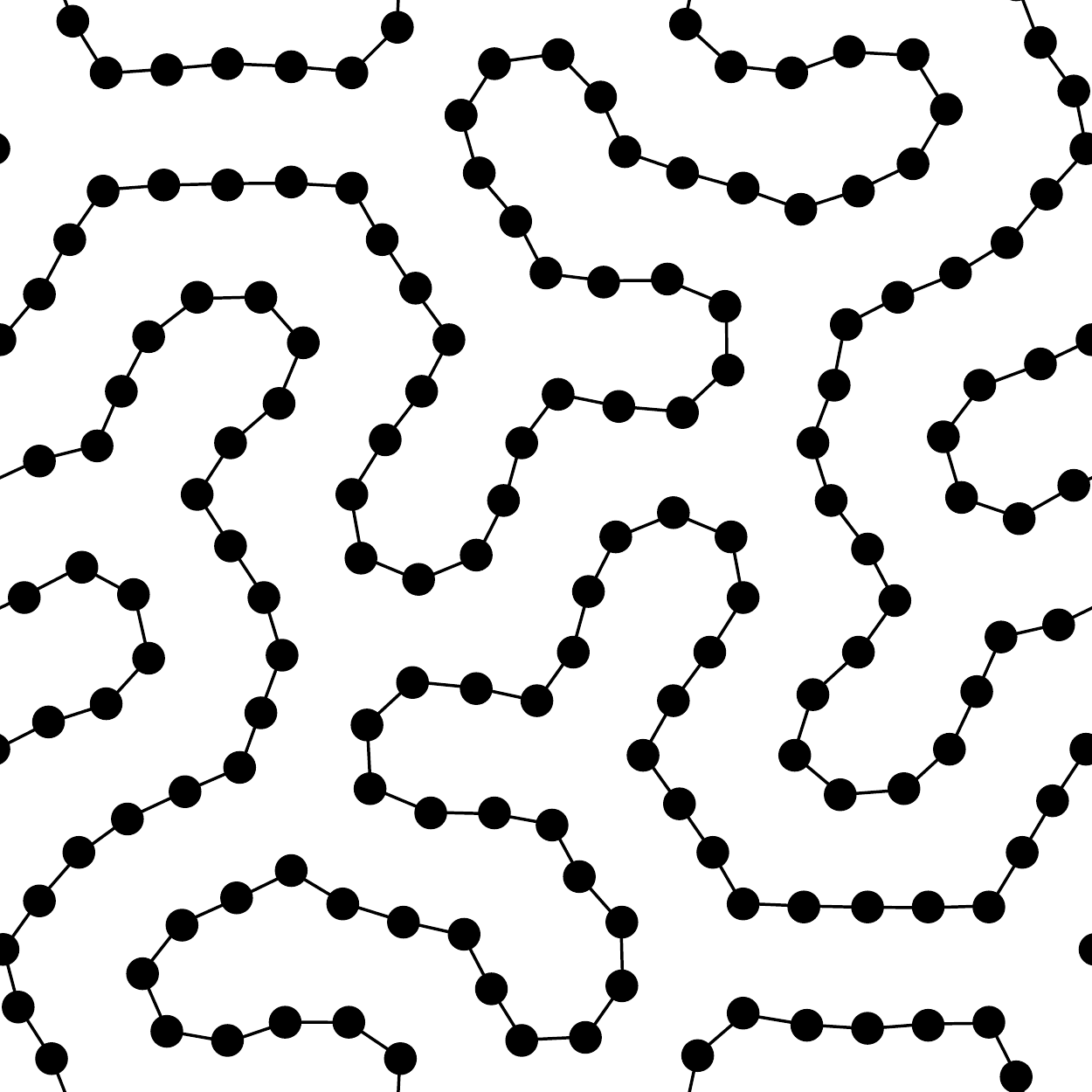}\\
		$\ell(\gamma)\approx 4.73$ &
		$\ell(\gamma)\approx 6.45$ &
		$\ell(\gamma)\approx 8.15$ &
		$\ell(\gamma)\approx 9.03$
	\end{tabular}
	\caption{Influence of $N$ and $\lambda$ on local minimizers of \eqref{eq:final} for the Lebesgue measure on $\T^2$, $L=4$ and $r=128$. 
		Results for increasing $N$ (column-wise) and decreasing $\lambda = 0.1\cdot 2^{-5i/2}$, $i=1,\dots,4$, (row-wise).
		Here, the curve length increases for decreasing $\lambda$ or increasing $N$, until stagnation for sufficient small $\lambda$ or large $N$. 
		For all minimizer the distance between consecutive points is around $\ell(\gamma)/N$.}
	\label{fig:Nexperiment}
\end{figure}

\textbf{Influence of the polynomial degree $r$.}
In Fig.~\ref{fig:t_experiment} we illustrate the local minimizers of \eqref{eq:final} 
for fixed Lipschitz parameters 
$L_{i}=2^{i}$ and corresponding regularization weights $\lambda_{i} = 0.2 \cdot L_{i}^{-5}$, $i=1,\dots,4$, (rows) in dependence on the polynomial degrees $r_{j}=8\cdot 2^{j}$, $j=1,\dots,5$ (columns). 
According to the previous experiments, it seems reasonable to choose $N = 20 L^{2}$.  
Note, that the (numerical) choice of $\lambda$ leads to curves with length $\ell(\gamma) \approx 2L$.
In Fig.~\ref{fig:t_experiment} we observe that for $r=c L$ the corresponding local minimizers have common features.
For instance, if $c=4$ (i.e., $r \approx \ell(\gamma)$) the minimizers have mostly vertical and horizontal line segments. 
%{\color{blue} However, for sufficient large $c$, say $c\ge 4$, the minimizers 
%appear to provide {\color{blue} honeycomb patterns}. GABI: Sehe ich nicht. Streichen?}
Furthermore, for fixed $r$ 
it appears that the length of the curves increases linearly with $L$ until $L$ exceeds $2r$, from where it remains unchanged.
This observation can be explained by the fact that there are curves of bounded length $c r$ which provide exact quadratures for degree $r$. 
\begin{figure}
	\centering
	\begin{tabular}{ccccc}
		%$t=8$ & $t=16$  &  $r=32$   & $r=64$  & $r=128$  \\  
		$r=16$  &  $r=32$   & $r=64$  & $r=128$  & $r=256$ \\
		\includegraphics[width=0.16\textwidth]{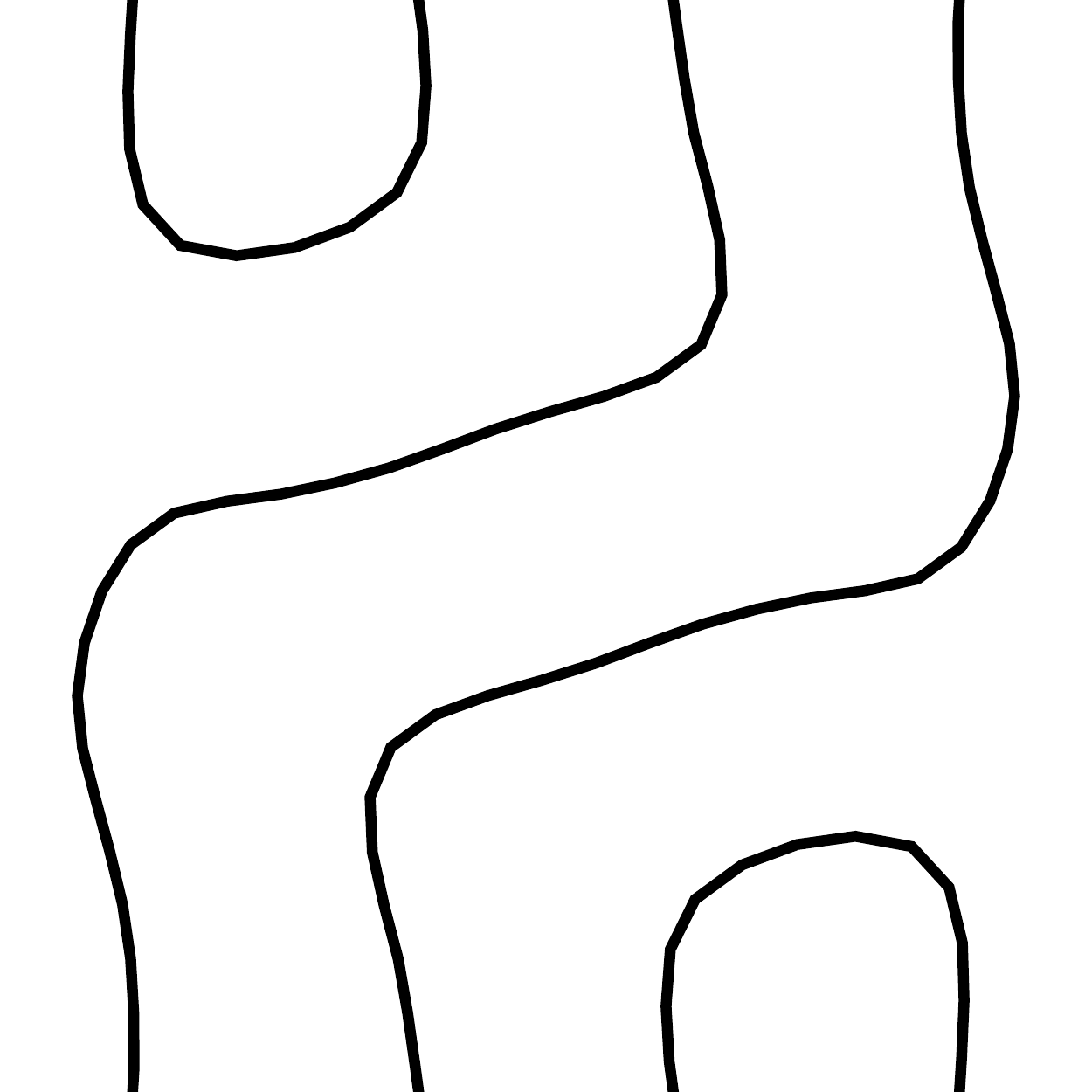} & 
		\includegraphics[width=0.16\textwidth]{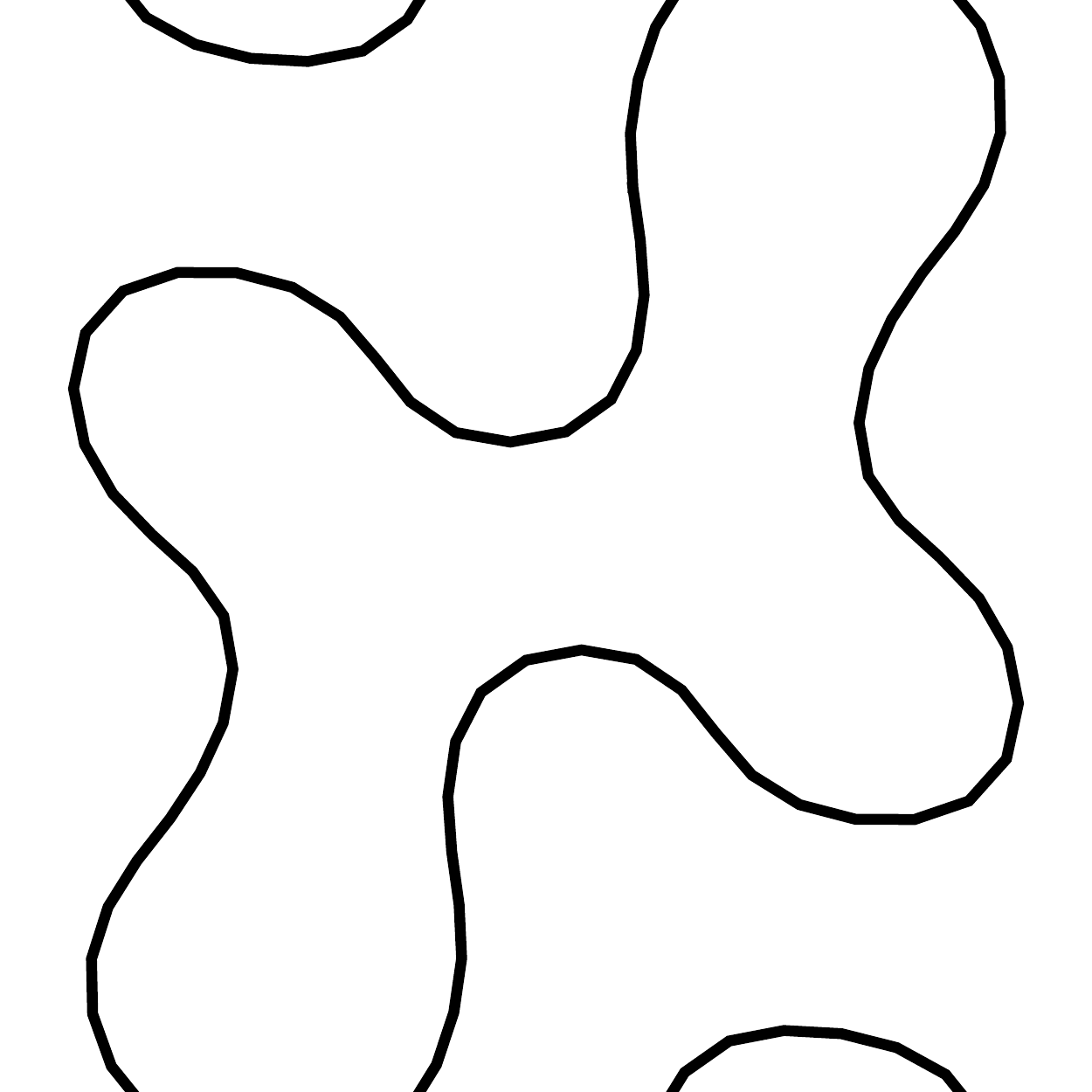} & 
		\includegraphics[width=0.16\textwidth]{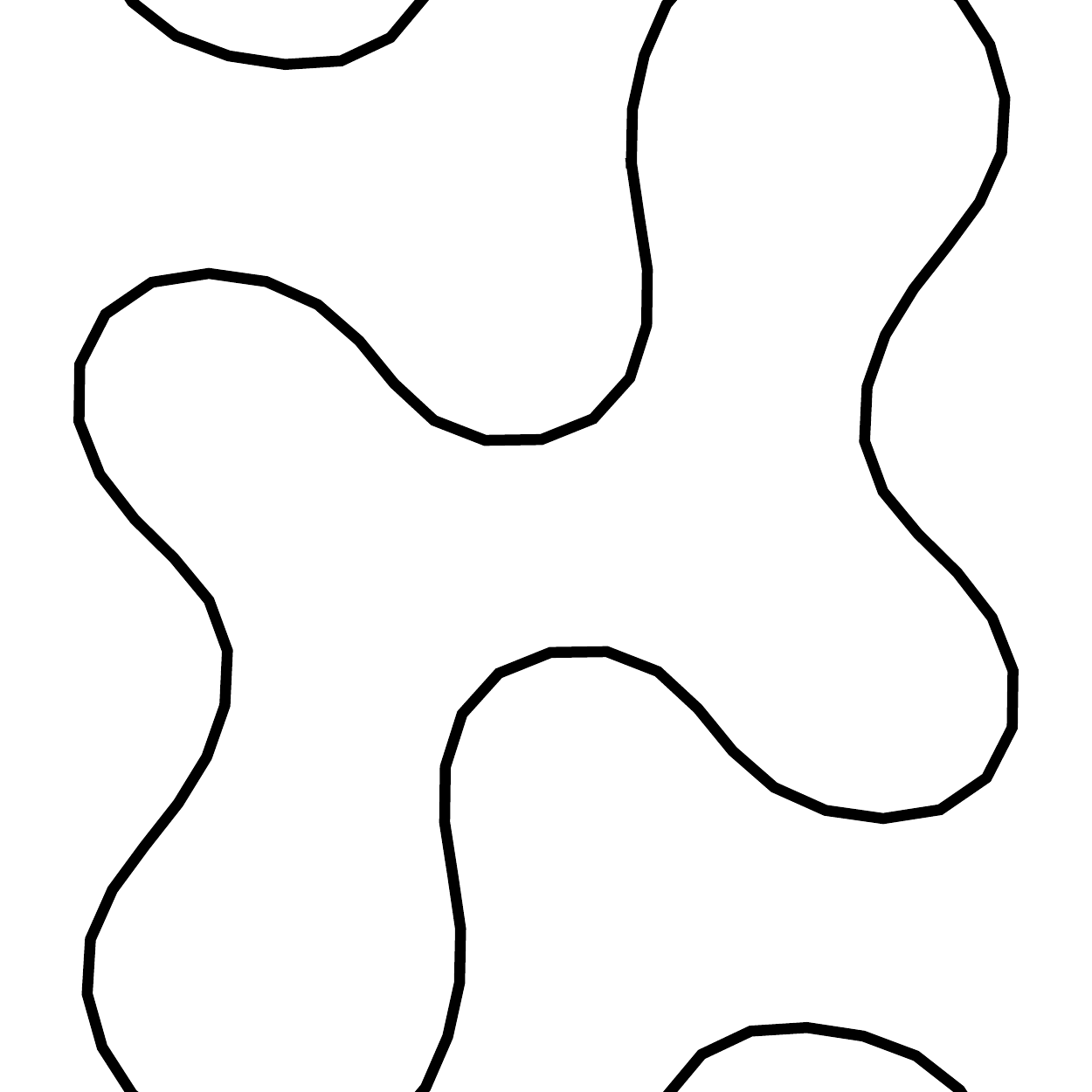} & 
		\includegraphics[width=0.16\textwidth]{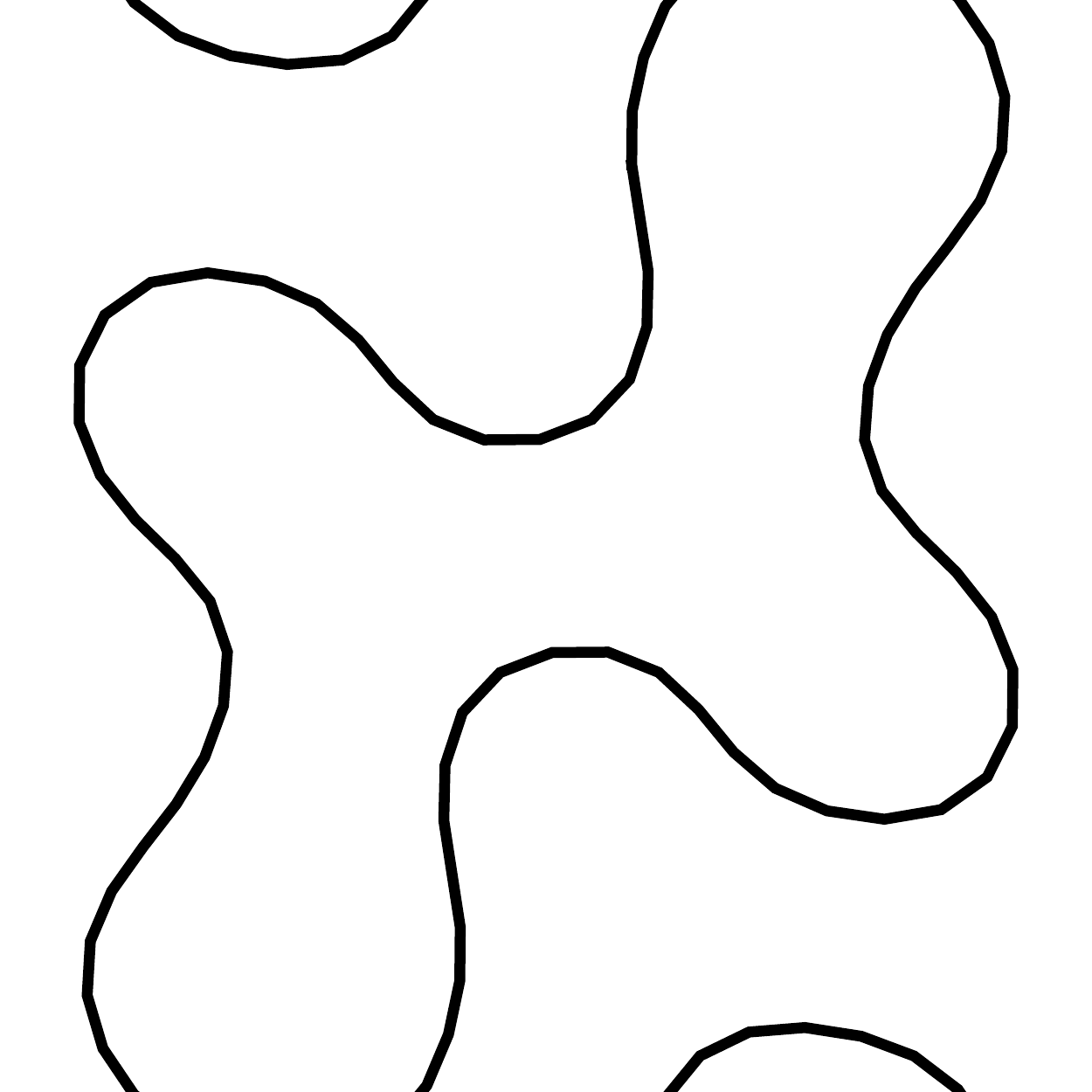} & 
		\includegraphics[width=0.16\textwidth]{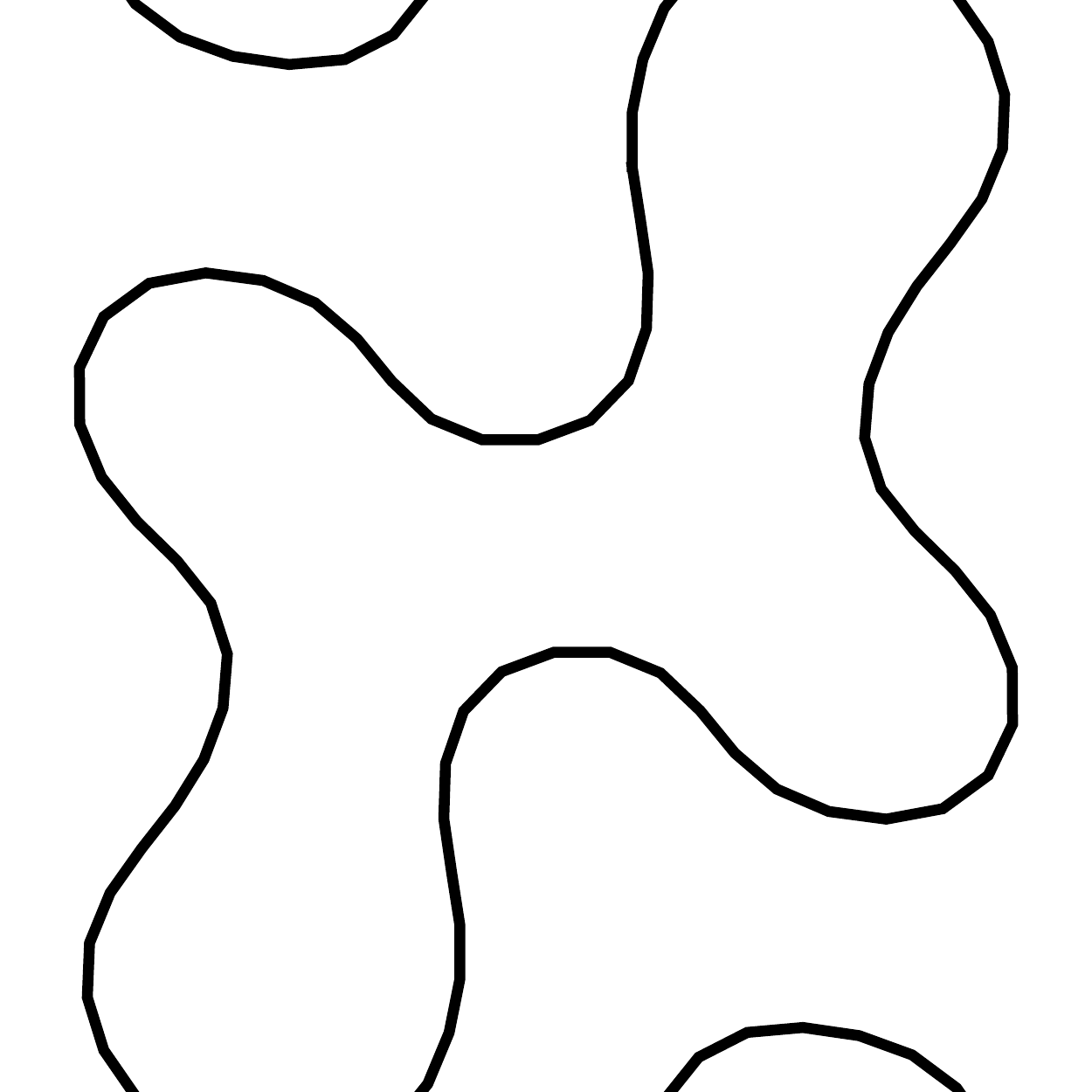} \\
		$\ell(\gamma)\approx 4.07$ &
		$\ell(\gamma)\approx 4.07$ &
		$\ell(\gamma)\approx 4.06$ &
		$\ell(\gamma)\approx 4.06$ &
		$\ell(\gamma)\approx 4.05$\\ 
		\includegraphics[width=0.16\textwidth]{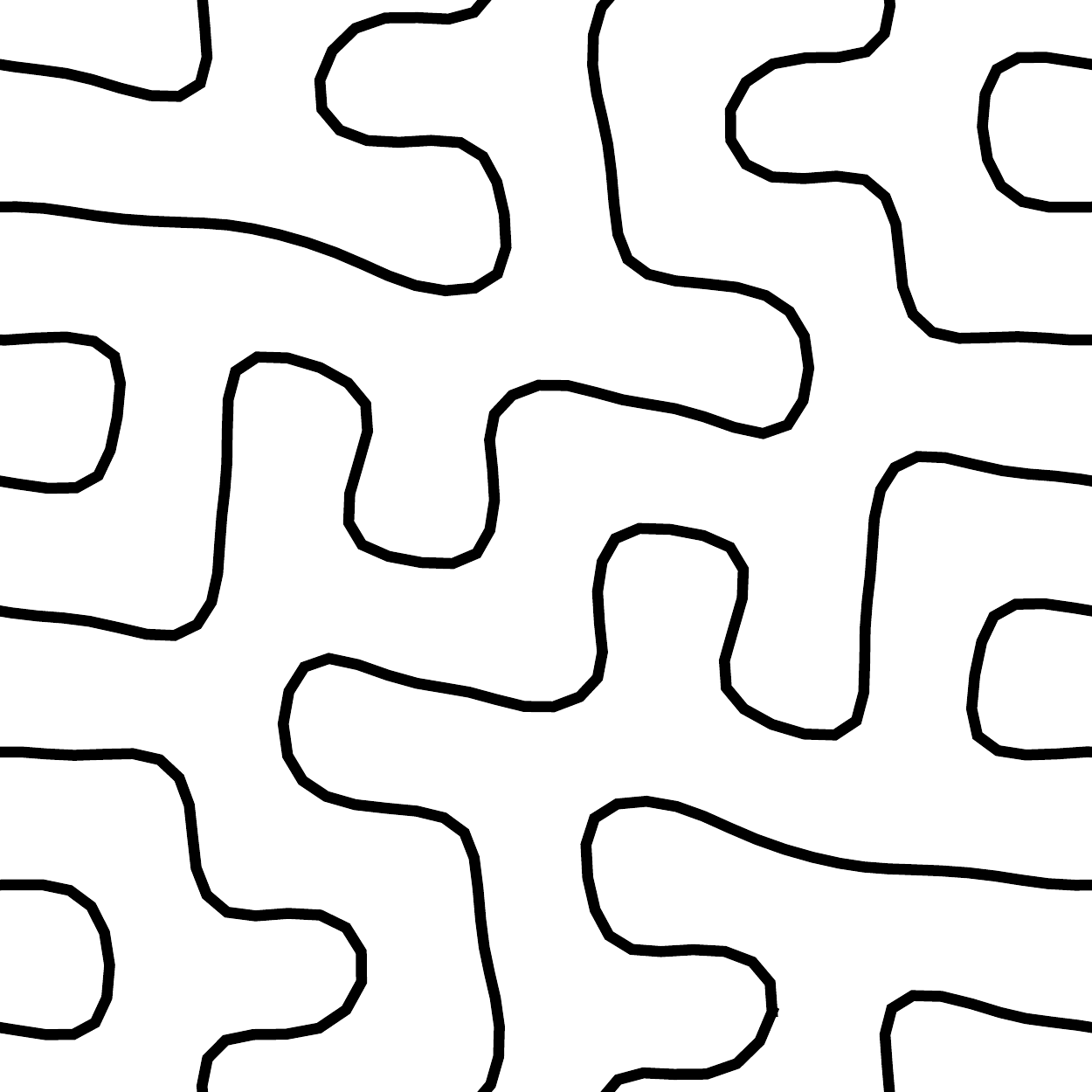} & 
		\includegraphics[width=0.16\textwidth]{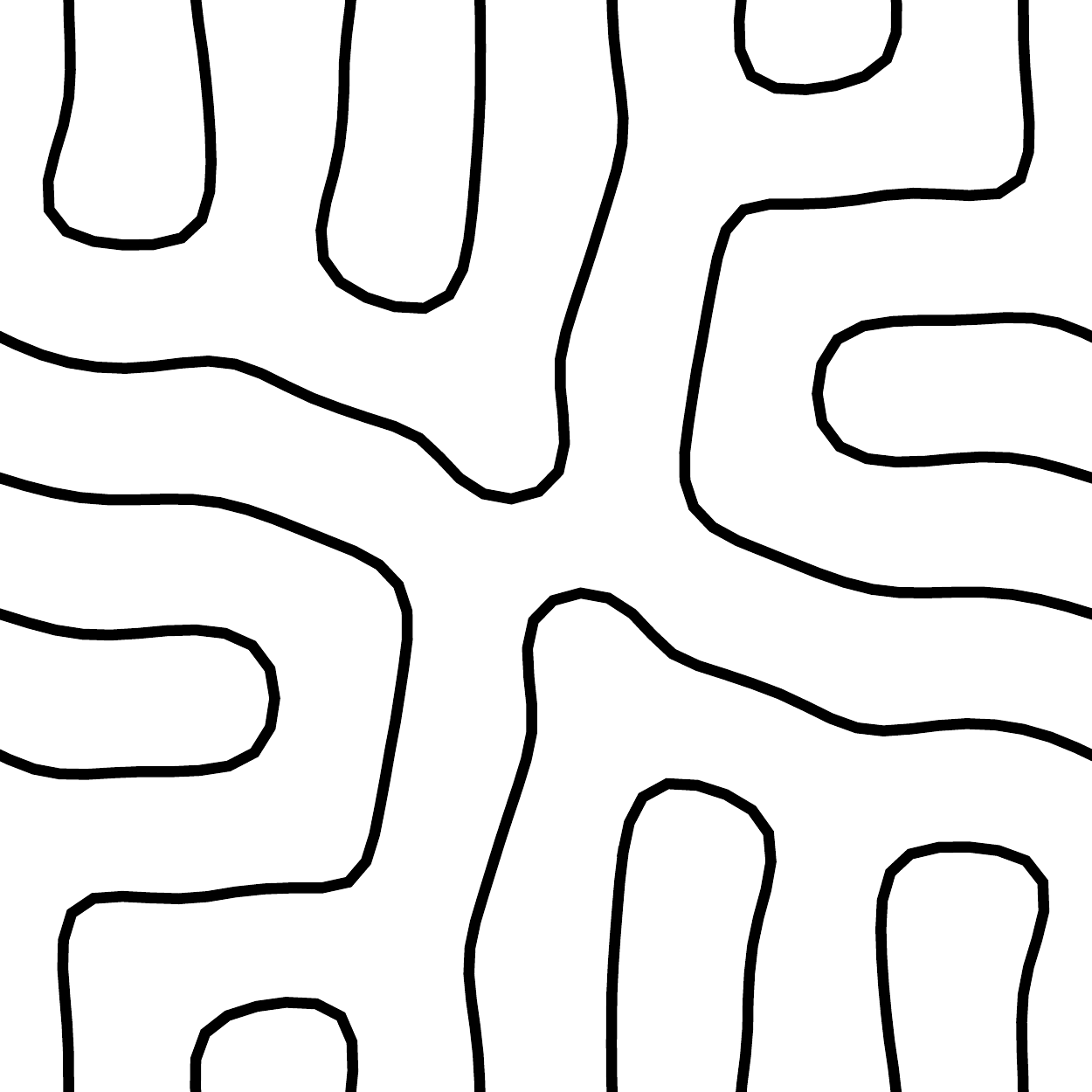} & 
		\includegraphics[width=0.16\textwidth]{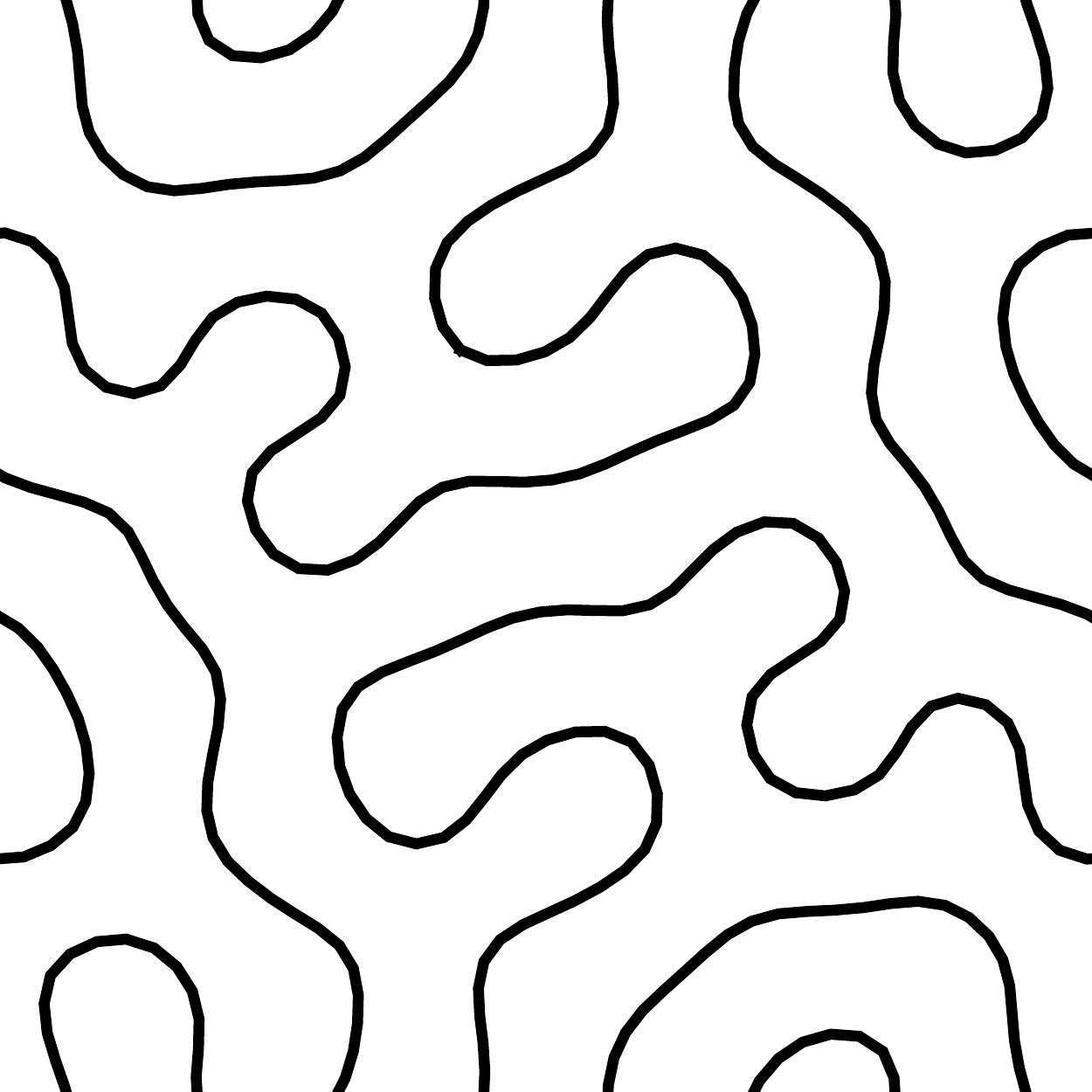} & 
		\includegraphics[width=0.16\textwidth]{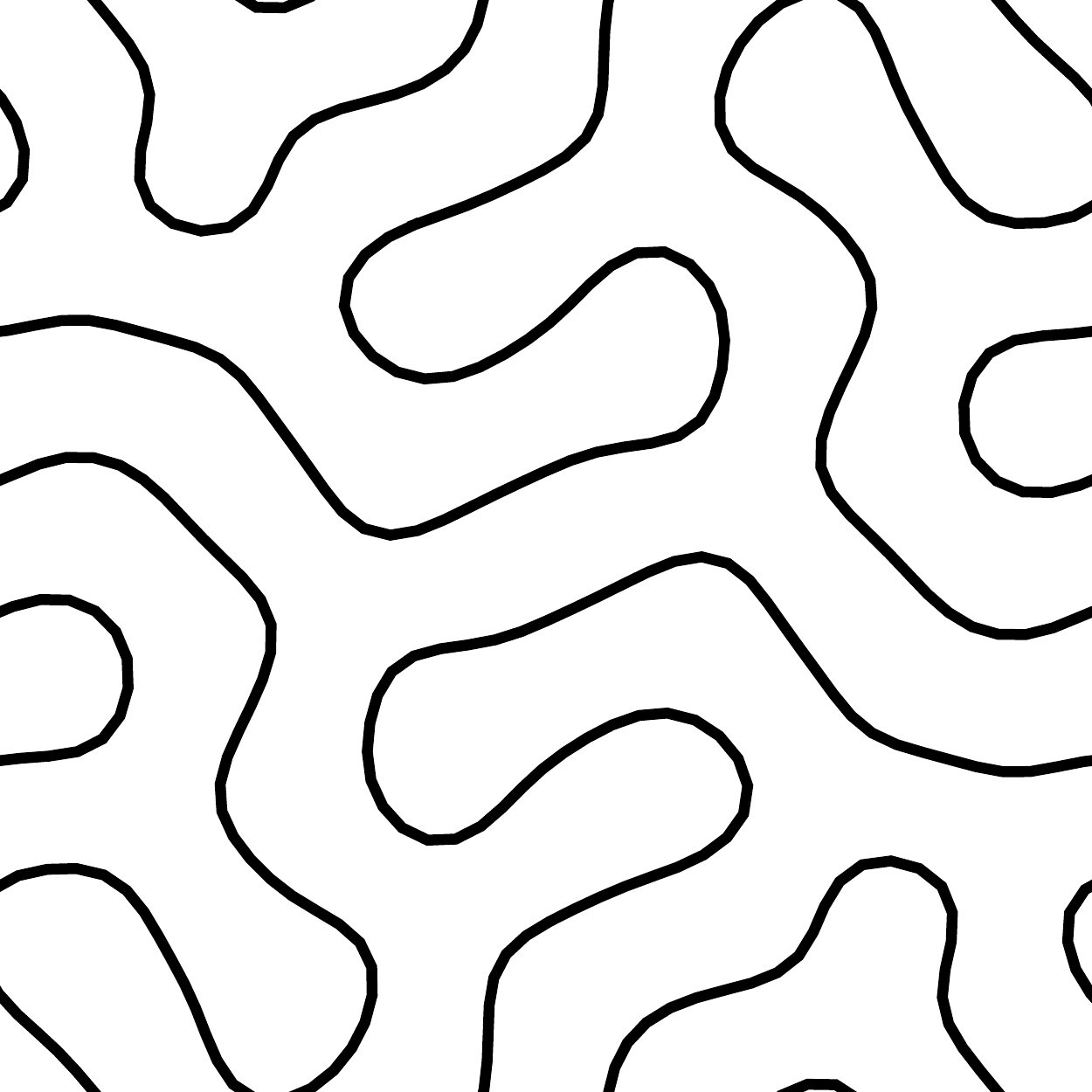} & 
		\includegraphics[width=0.16\textwidth]{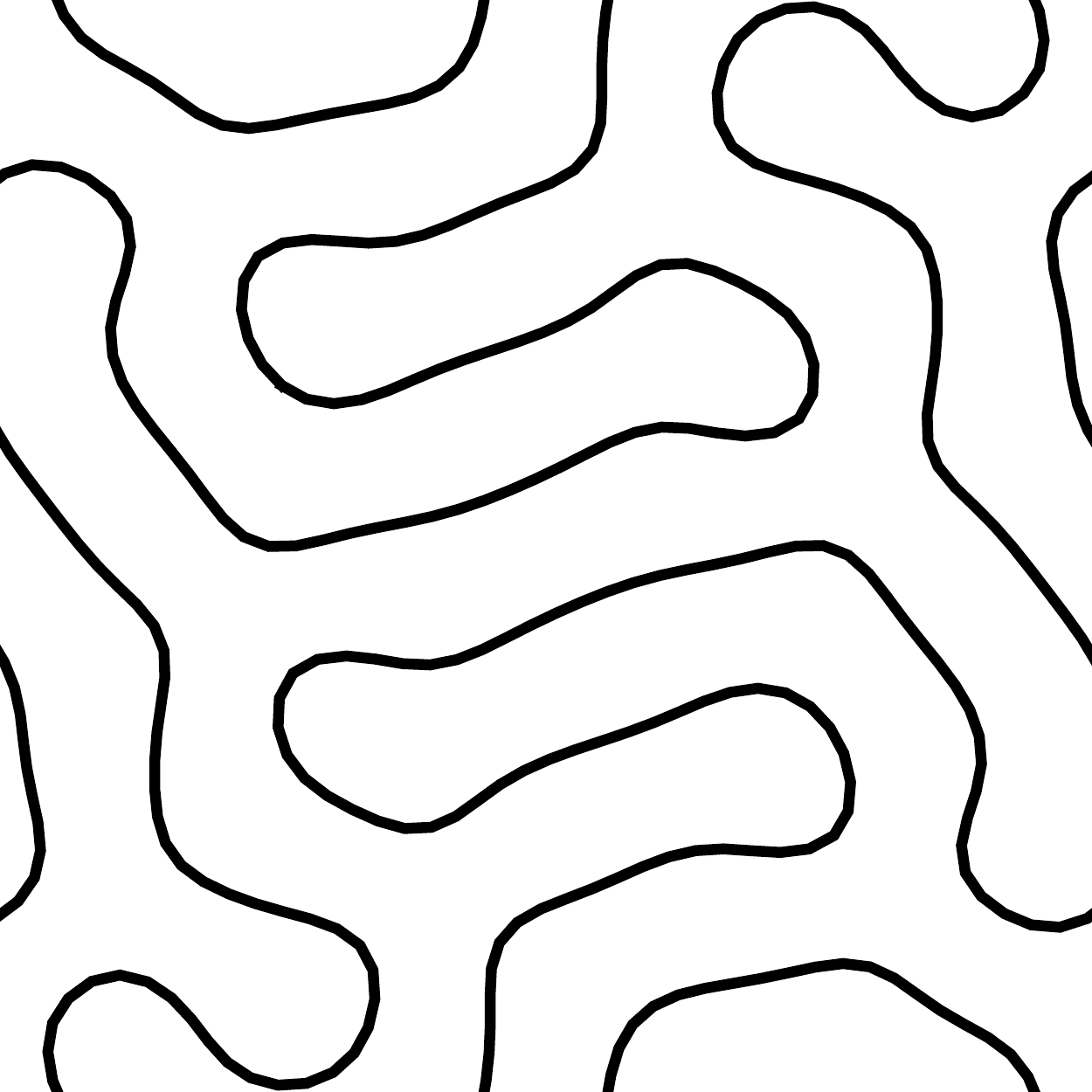} \\
		$\ell(\gamma)\approx 8.48$ &
		$\ell(\gamma)\approx 8.28$ &
		$\ell(\gamma)\approx 8.32$ &
		$\ell(\gamma)\approx 8.23$ &
		$\ell(\gamma)\approx 8.22$\\ 
		\includegraphics[width=0.16\textwidth]{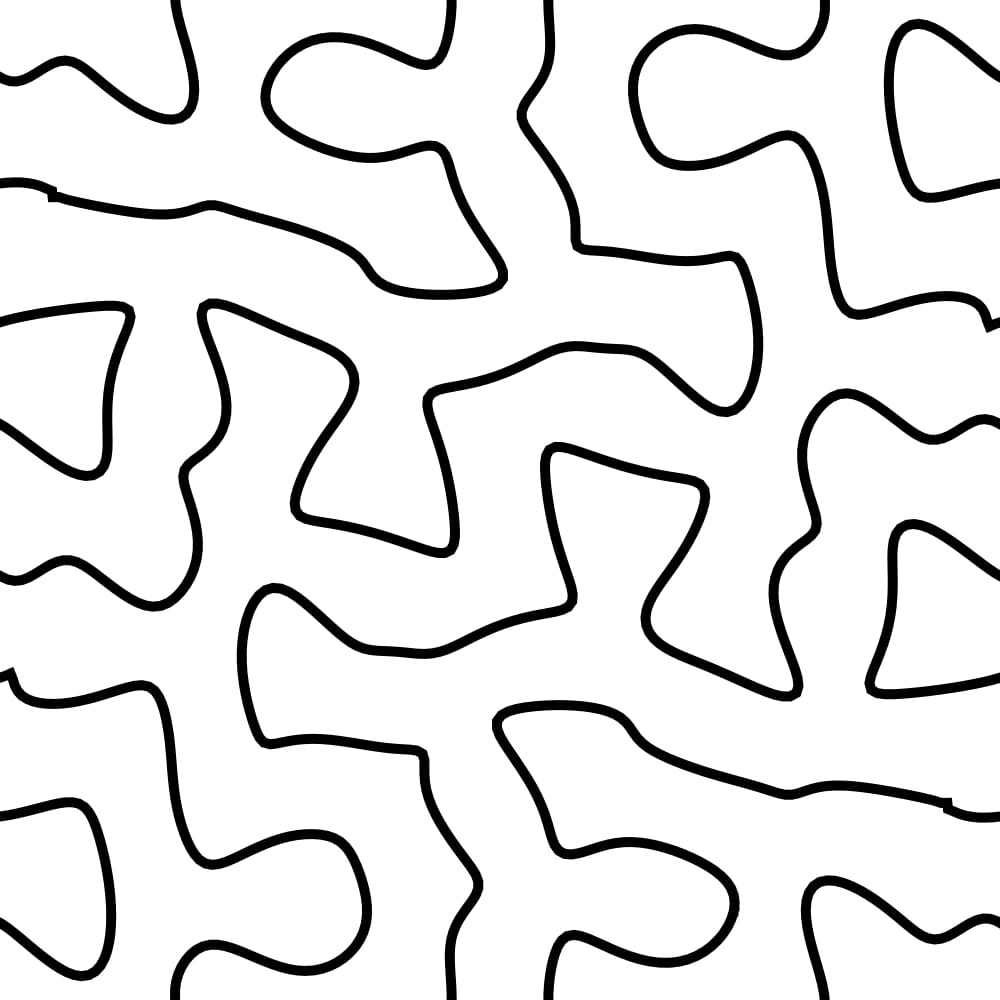} & 
		\includegraphics[width=0.16\textwidth]{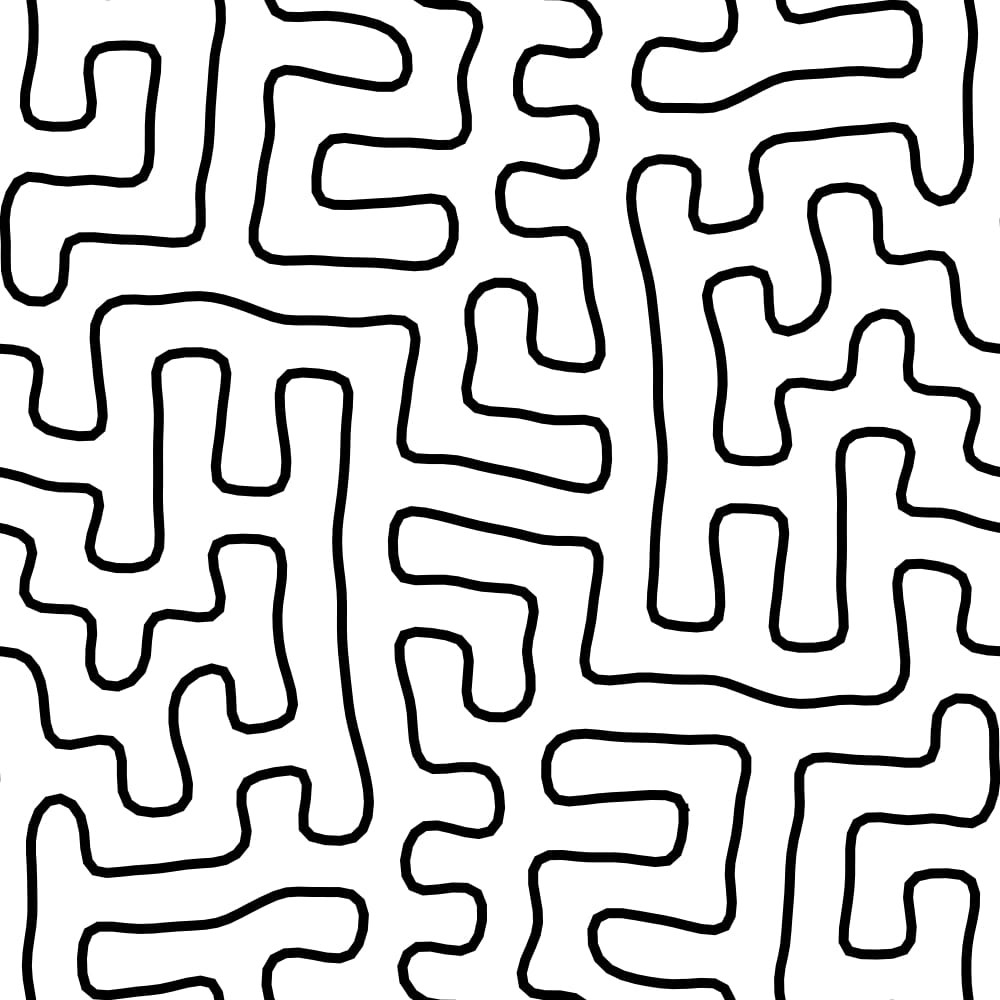} & 
		\includegraphics[width=0.16\textwidth]{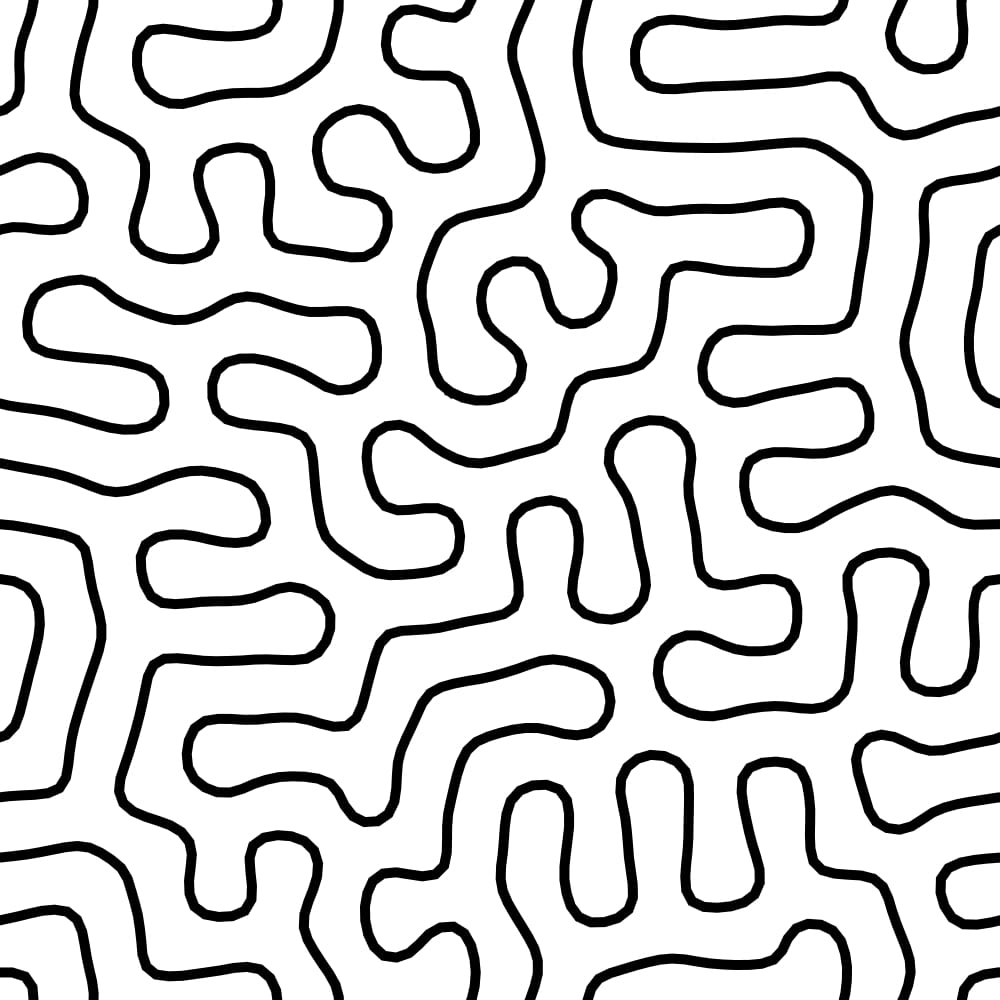} & 
		\includegraphics[width=0.16\textwidth]{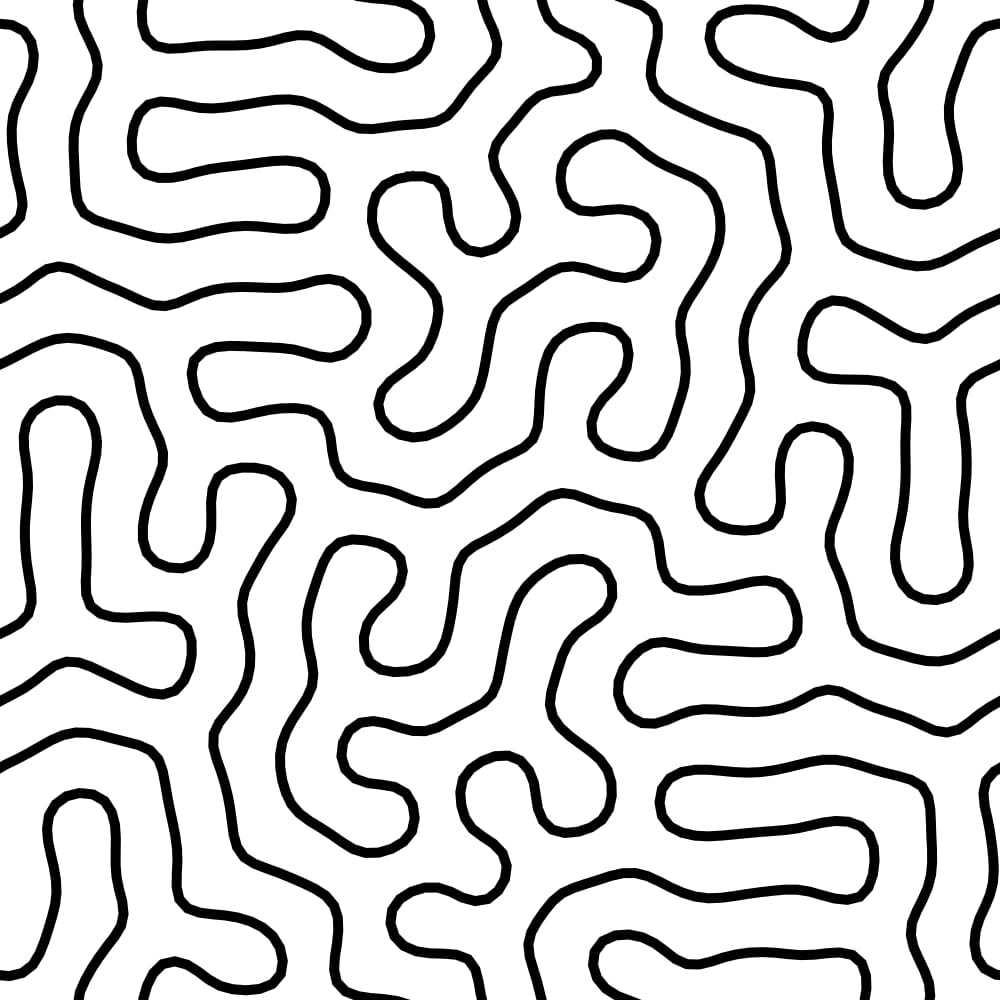} & 
		\includegraphics[width=0.16\textwidth]{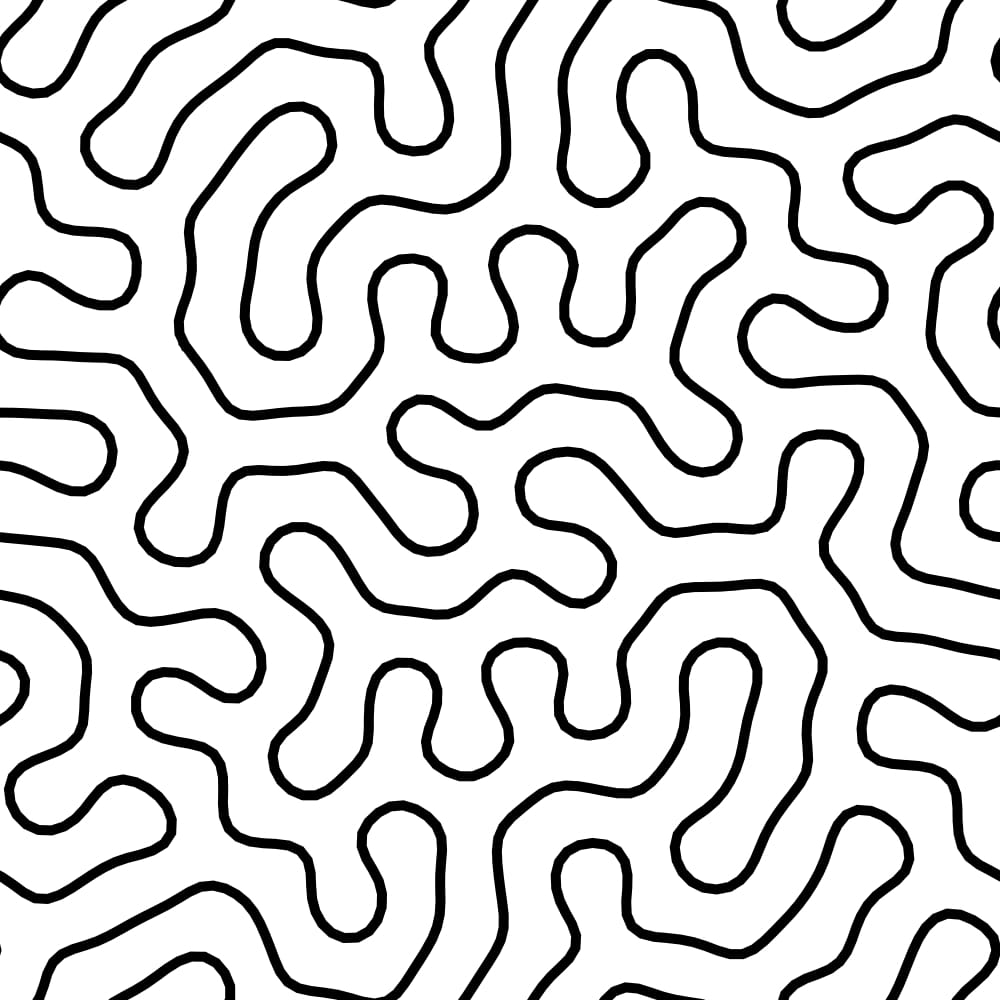} \\
		$\ell(\gamma)\approx 10.42$ &
		$\ell(\gamma)\approx 16.96$ &
		$\ell(\gamma)\approx 16.77$ &
		$\ell(\gamma)\approx 16.63$ &
		$\ell(\gamma)\approx 16.4$\\ 
		\includegraphics[width=0.16\textwidth]{Images/Dith/T8_L8} & 
		\includegraphics[width=0.16\textwidth]{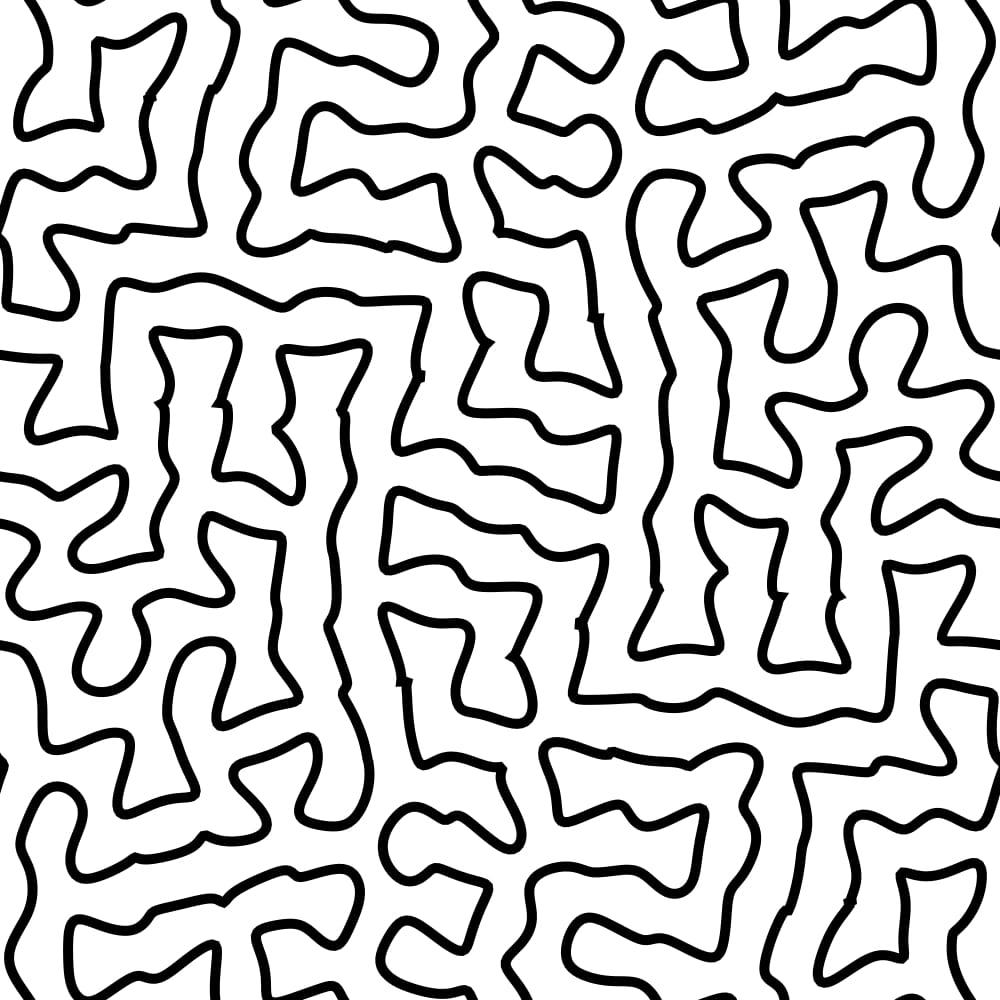} & 
		\includegraphics[width=0.16\textwidth]{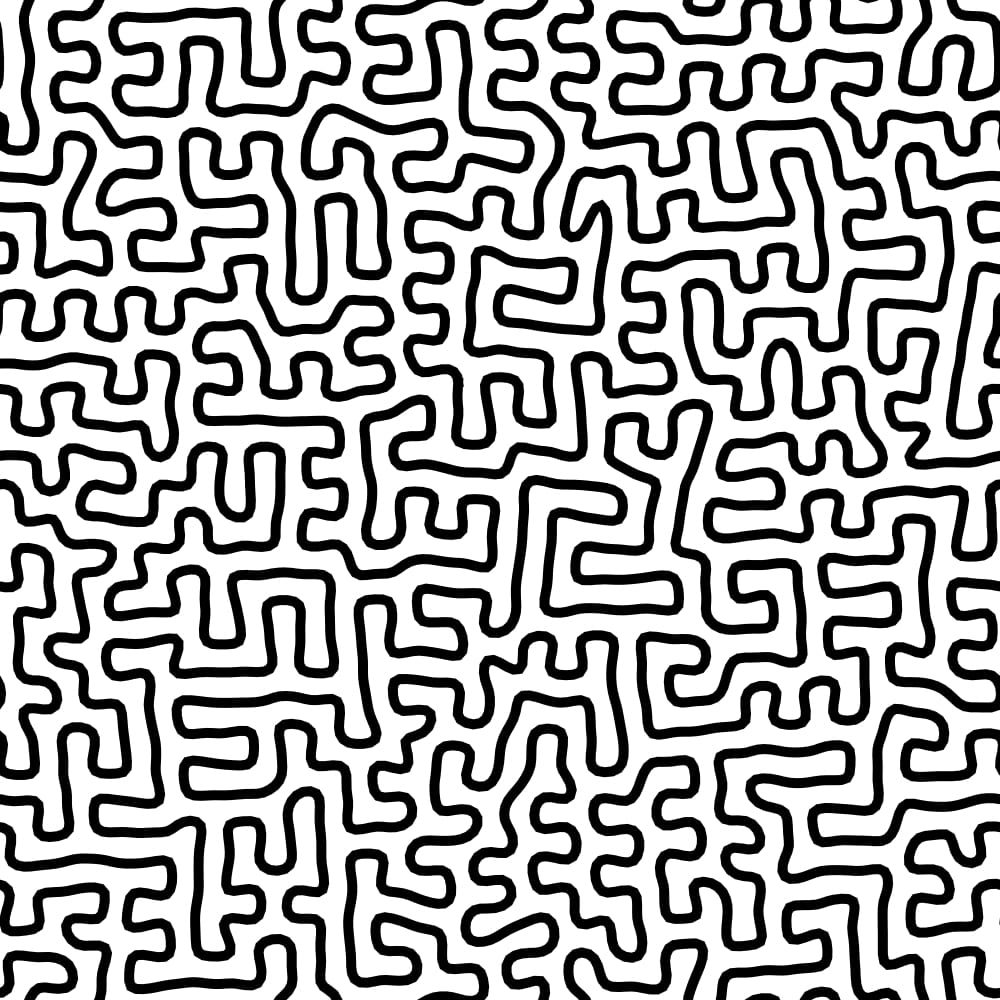} & 
		\includegraphics[width=0.16\textwidth]{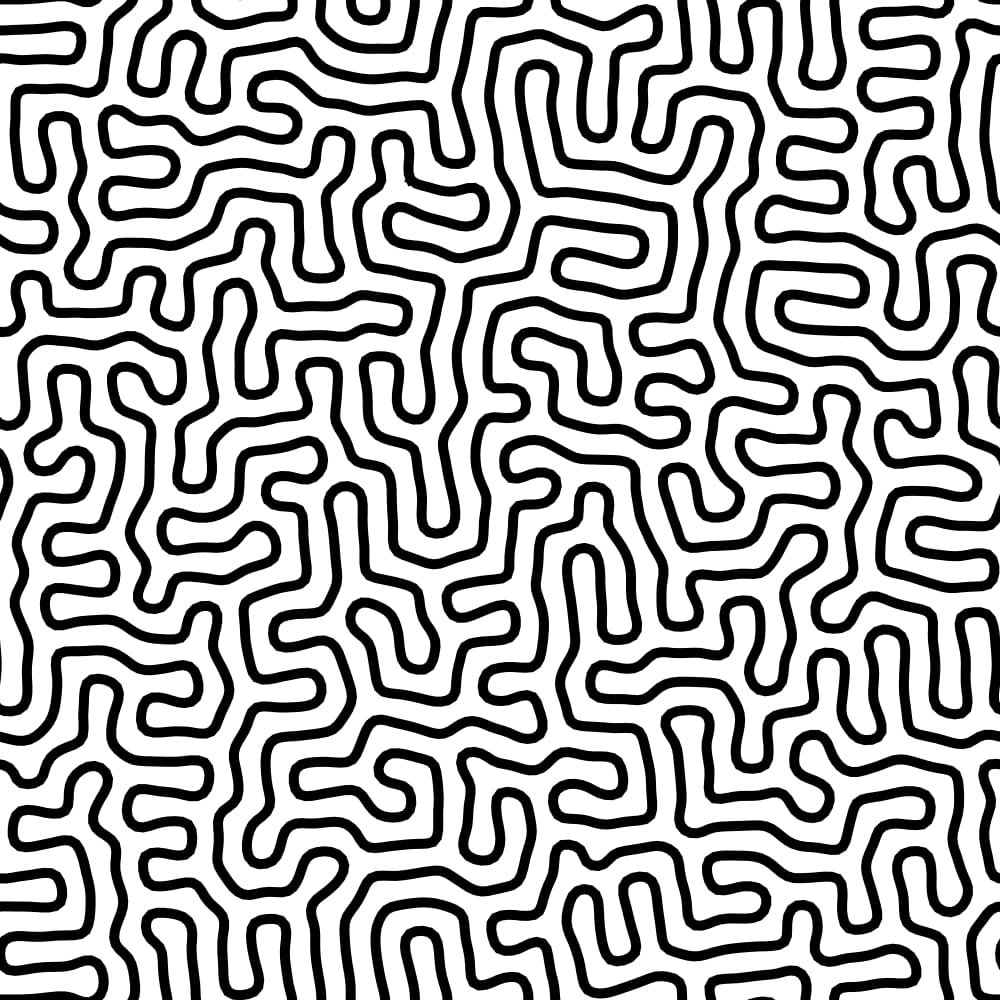} & 
		\includegraphics[width=0.16\textwidth]{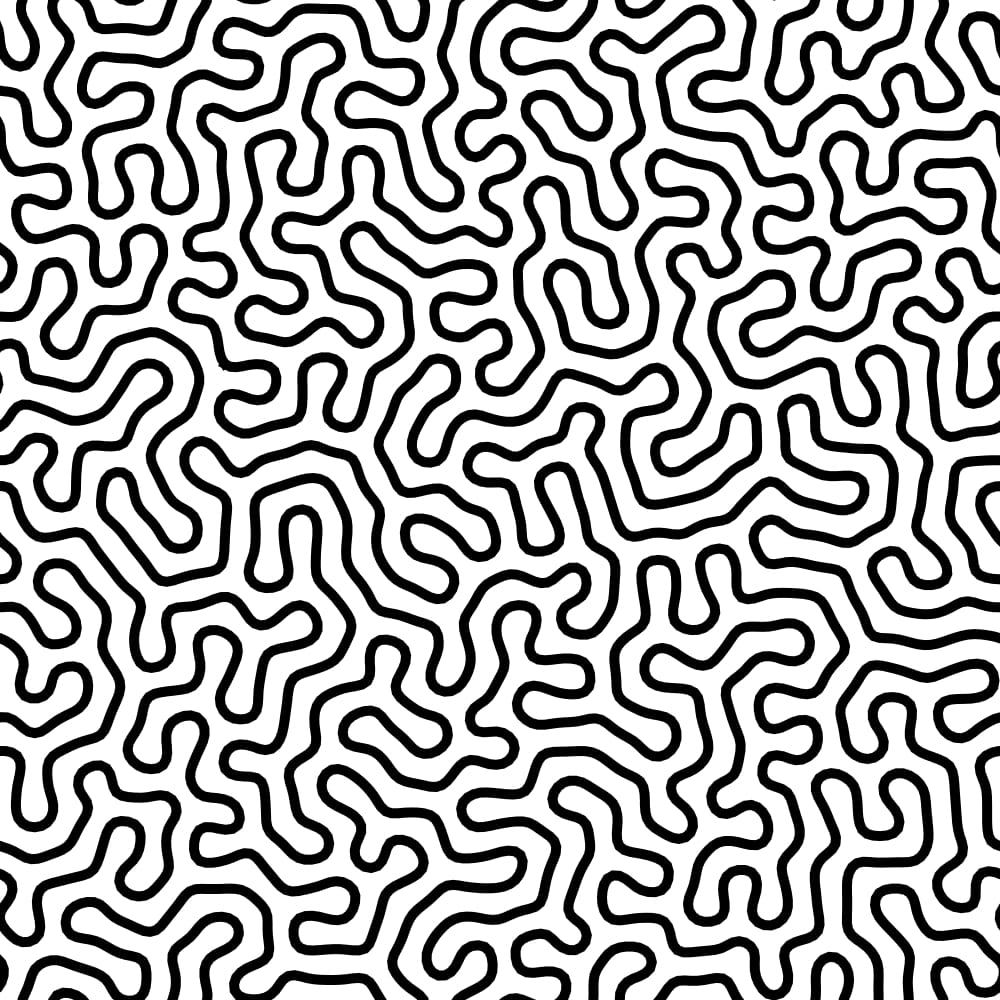} \\
		$\ell(\gamma)\approx 10.48$ &
		$\ell(\gamma)\approx 20.83$ &
		$\ell(\gamma)\approx 34.09$ &
		$\ell(\gamma)\approx 33.52$ &
		$\ell(\gamma)\approx 33.35$\\ 
	\end{tabular}
	\caption{Influence of $r$ on the local minimizer of \eqref{eq:final} for the Lebesgue measure on $\T^2$.
		Column-wise we increase $r=16,32,64,128,256$ and row-wise we increase $L=2,4,8,16$, where $\lambda = 0.2 L^{-5}$ and $N=20 L^{2}$. 
		Note that the degree $r$  steers the resolution of the curves. 
		It appears that the spacing of the curves is bounded by $r^{-1}$.}
	\label{fig:t_experiment}
\end{figure}

\subsection{Quasi-optimal curves on special manifolds}

In this subsection, we give numerical examples for $\X \in \{\T^2,\T^3,\S^2, \SO(3), \mathcal G_{2,4} \}$.
Since the objective function in \eqref{eq:final} is highly non-convex, the main problem is to find nearly optimal curves $\gamma_{L} \in  \mathcal P_{L}^{\Lcurve}(\mathbb X)$ for increasing $L$. 
Our heuristic is as follows:
\begin{itemize}
	\item[i)] We start with a curve $\gamma_{L_{0}}\colon[0,1]\to \mathbb X$ of small length 
	$\ell(\gamma) \approx L_{0}$ and solve the problem \eqref{eq:final} 
	for increasing $L_{i} = c L_{i-1}$, $c>1$, 
	where we choose the parameters $N_{i}$, $\lambda_{i}$ and $r_{i}$ in dependence on $L_{i}$ as described in the previous subsection.
	In each step a local minimizer is computed using the CG method with 100 iterations. 
	Then, the obtained minimizer $\gamma_{i}$ serves as the initial guess in the next step, which is obtained by inserting the midpoints. 
	\item[ii)] In case that the resulting curves $\gamma_i$ have non-constant speed, each is refined by increasing $\lambda_{i}$ and $N_{i}$.
	Then, the resulting problem is solved with the CG method and $\gamma_i$ as initialization.
	Details on the parameter choice are given in the according examples. 
\end{itemize}

The following examples show that this recipe indeed  enables us to compute ``quasi-optimal'' curves, 
meaning that the obtained minimizers have optimal decay in the discrepancy.
\smallskip

\textbf{2d-Torus $\T^2$.}
In this example we illustrate how well a gray-valued image (considered as probability density) may be approximated by an almost constant speed curve. 
The original image of size 170x170 is depicted in the bottom-right corner of Fig.~\ref{fig:tiger}.
Its Fourier coefficients $\hat \mu_{k_{1},k_{2}}$ are computed by a discrete Fourier transform (DFT) 
using the FFT algorithm and normalized appropriately. The kernel $K$ is given by \eqref{kernel_Td} with $d=2$ and $s = 3/2$.

We start with $N_{0}=96$ points on a circle given by the formula
\[
x_{0,k} = \Bigl( \tfrac15 \cos(2\pi k/N_{0}), \tfrac15 \sin(2\pi k/N_{0})\Bigr), \qquad k=0,\dots,N_{0}.
\]
Then, we apply our procedure for  $i=0,\dots,11$ with parameters
\[
L_{i} = 0.97\cdot 2^{\frac{i+5}{2}},\quad \lambda_{i} = 100\cdot L_{i}^{-5},\quad N_{i} =  96\cdot 2^{i} \sim L_{i}^{2}
\quad r_{i}= \lfloor 2^{\frac{i+11}{2}} \rfloor \sim L_{i},
\]
chosen such that the length
of the local minimizer $\gamma_{i}$ satisfies 
$\ell(\gamma_{i}) \approx 2^{(i+5)/2}$ and the maximal speed  is close to $L_{i}$. 

To get nearly constant speed curves $\gamma_i$, see ii), we increase 
$\lambda_{i}$ by a factor of 100, $N_{i}$ by a factor of 2 and set $L_{i} \coloneqq 2^{(i+5)/2}$. 
Then, we apply the CG method with maximal 100 iterations and $i$ restarts.
The results are depicted in Fig.~\ref{fig:tiger}. 
Note that the complexity for the evaluation of the function in \eqref{eq:final} scales roughly as $N \sim L^{2}$.
In Fig.~\ref{fig:err_tiger} we observe that the decay-rate of the squared discrepancy $\mathscr{D}_K^{2}(\mu,\nu)$ 
in dependence on the Lipschitz constant $L$ matches indeed the theoretical findings of Theorem \ref{thm:torus better estimates}.
\begin{figure}
	\centering
	\begin{tabular}{ccc}
		\includegraphics[width=0.23\textwidth]{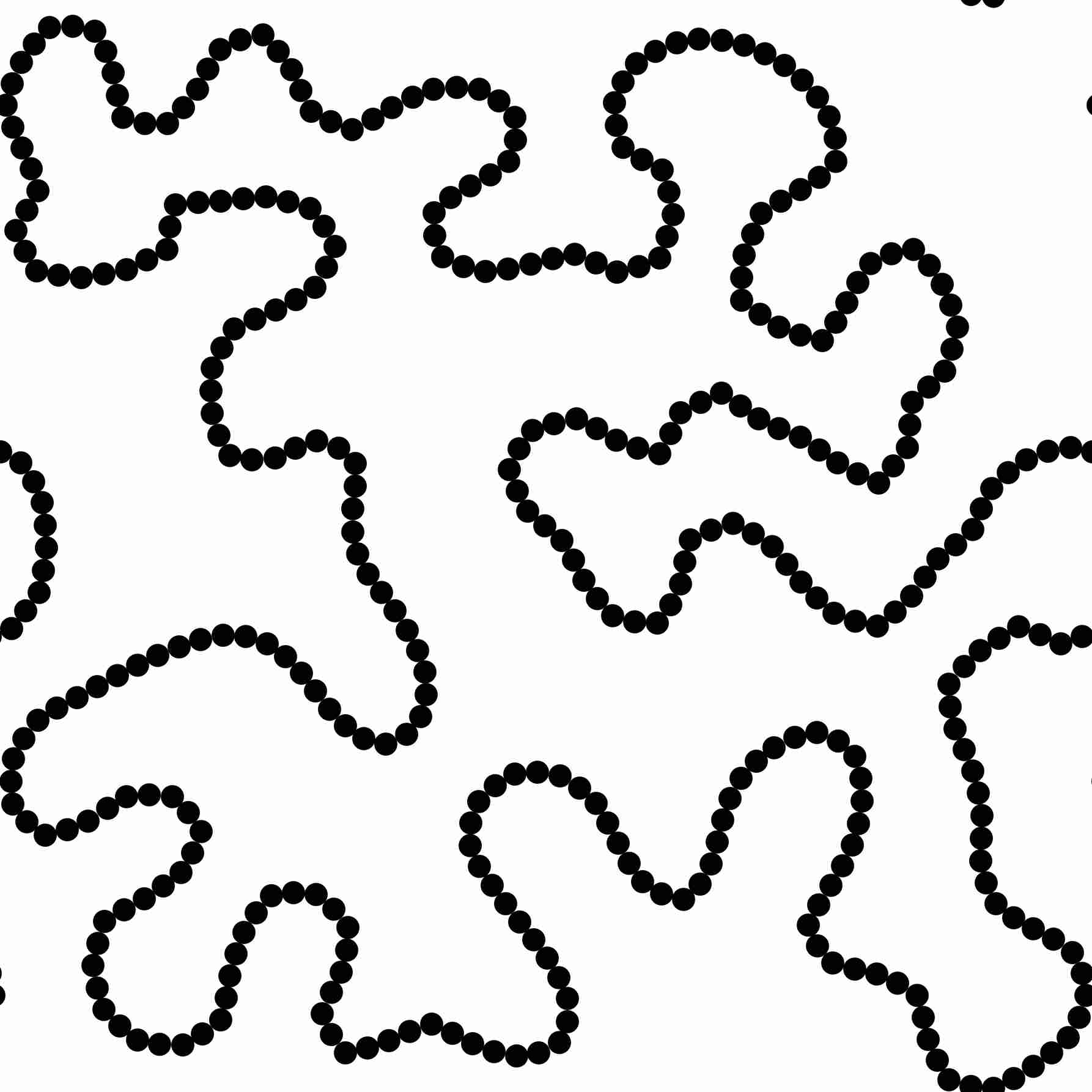} & 
		\includegraphics[width=0.23\textwidth]{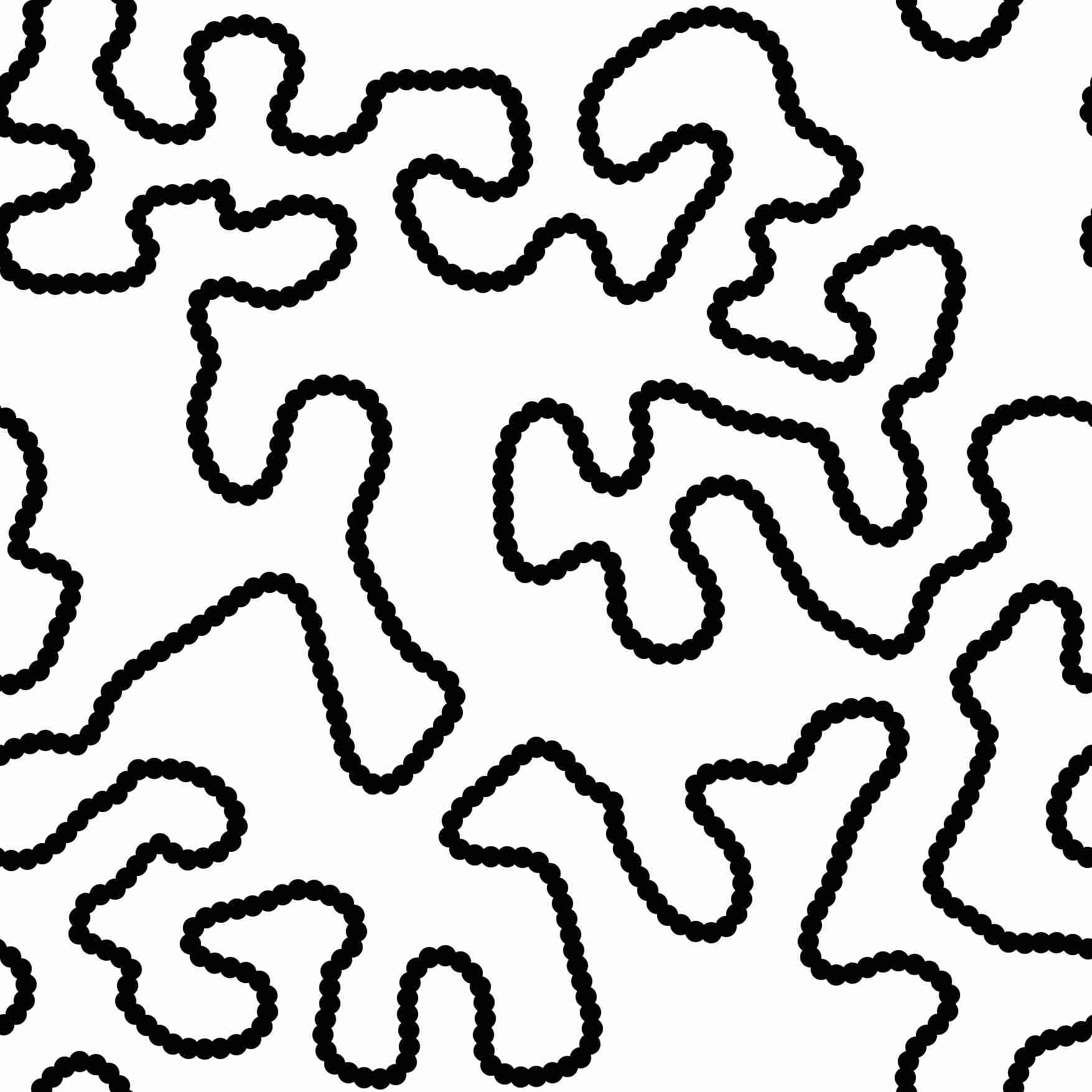} &
		\includegraphics[width=0.23\textwidth]{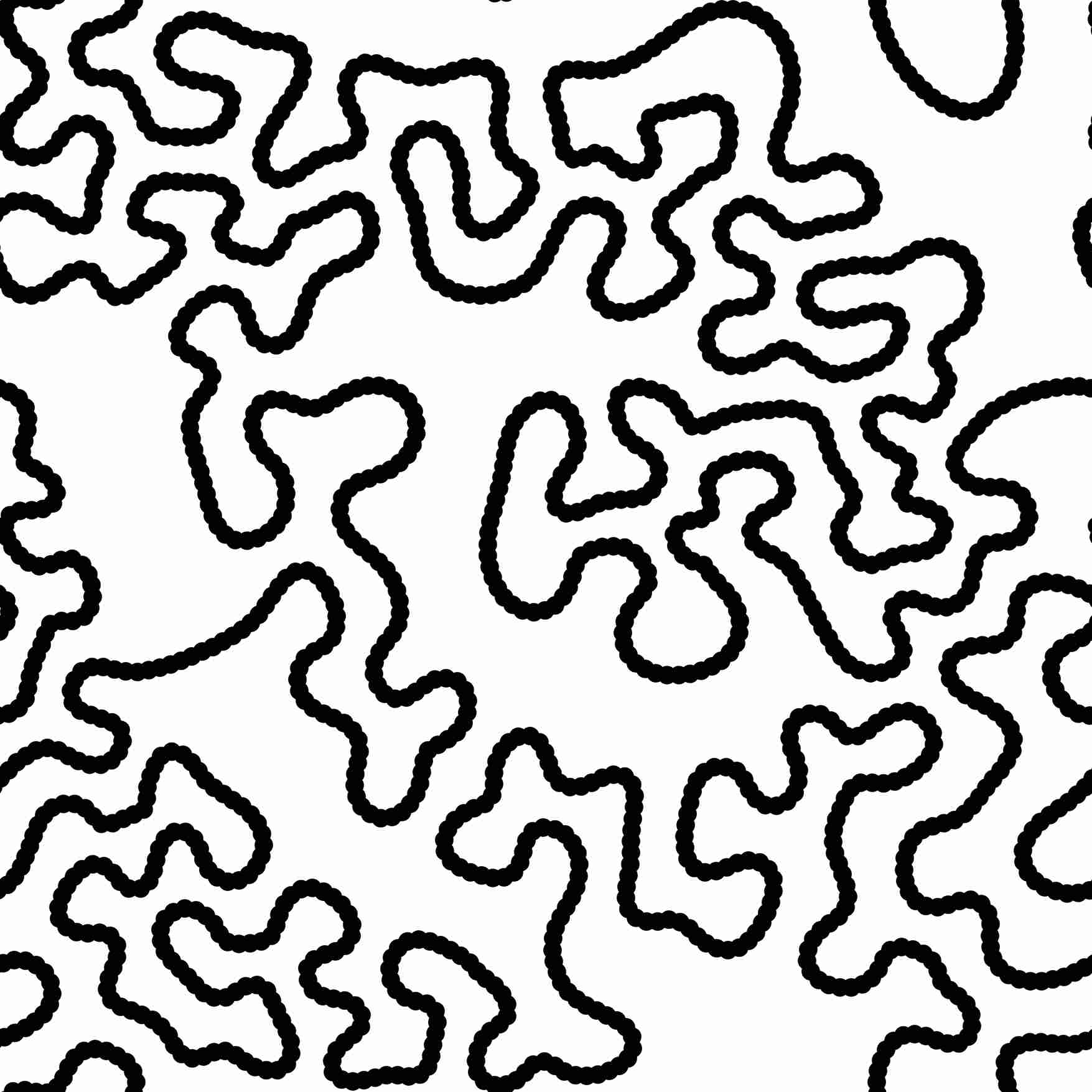} \\ 
		\includegraphics[width=0.23\textwidth]{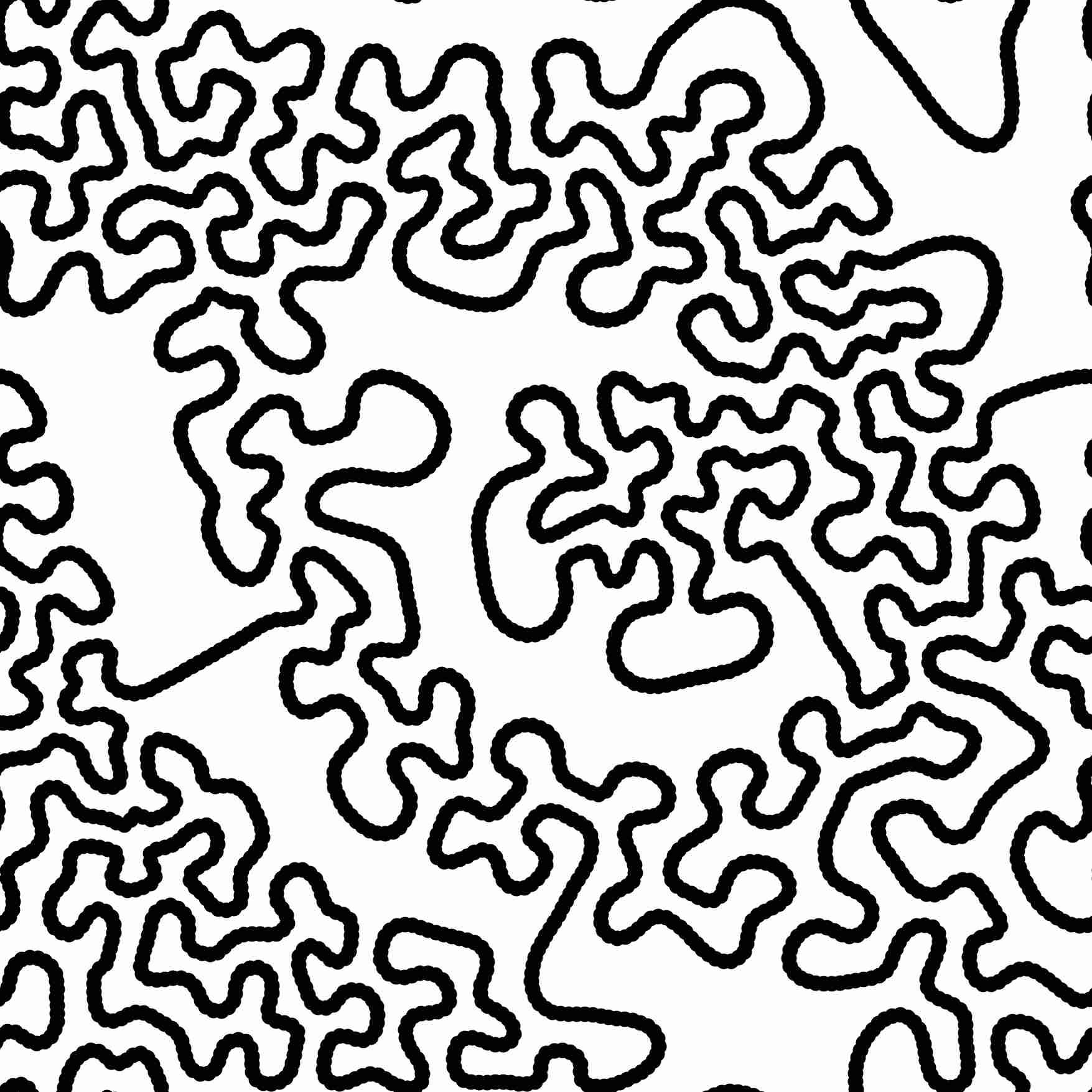} & 
		\includegraphics[width=0.23\textwidth]{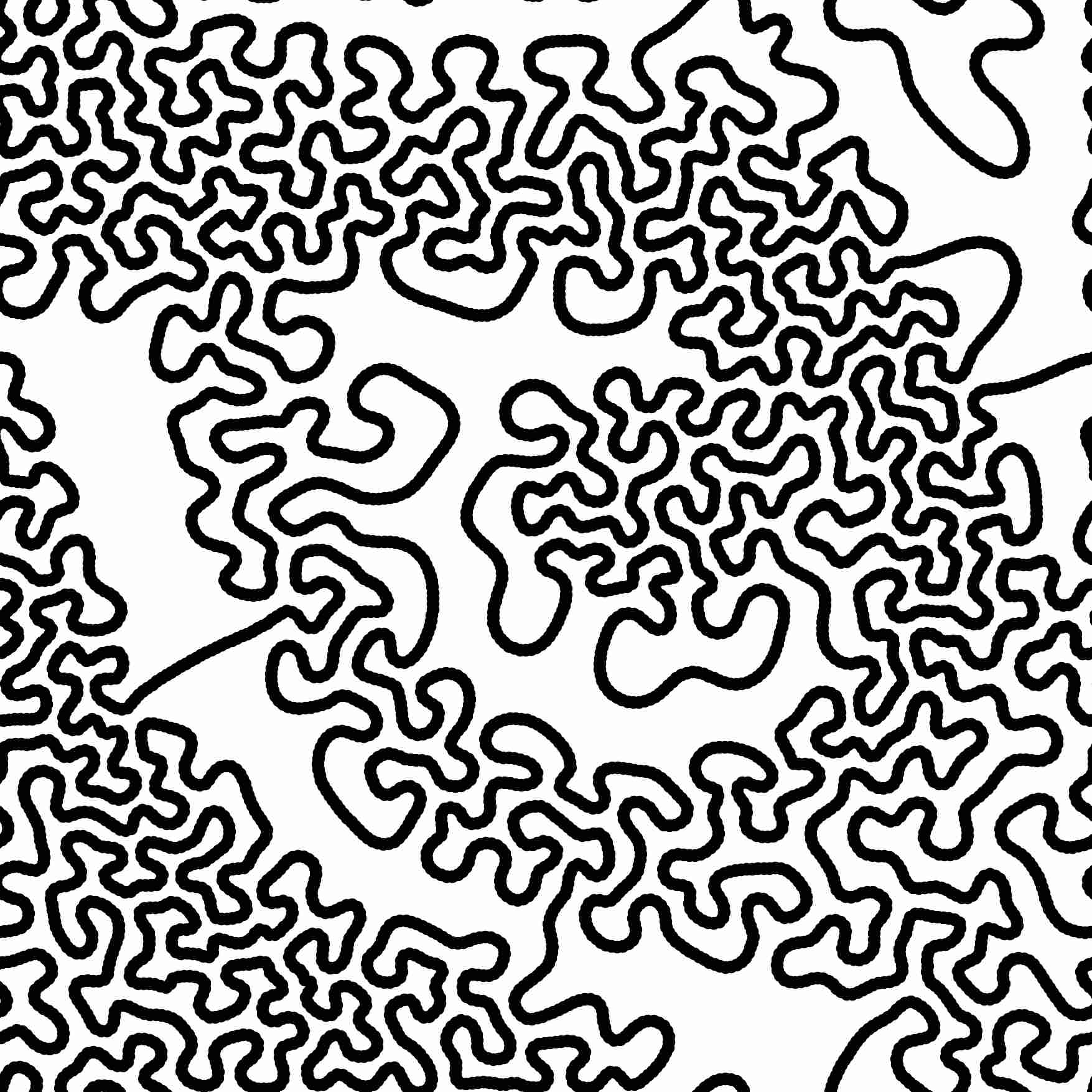} &
		\includegraphics[width=0.23\textwidth]{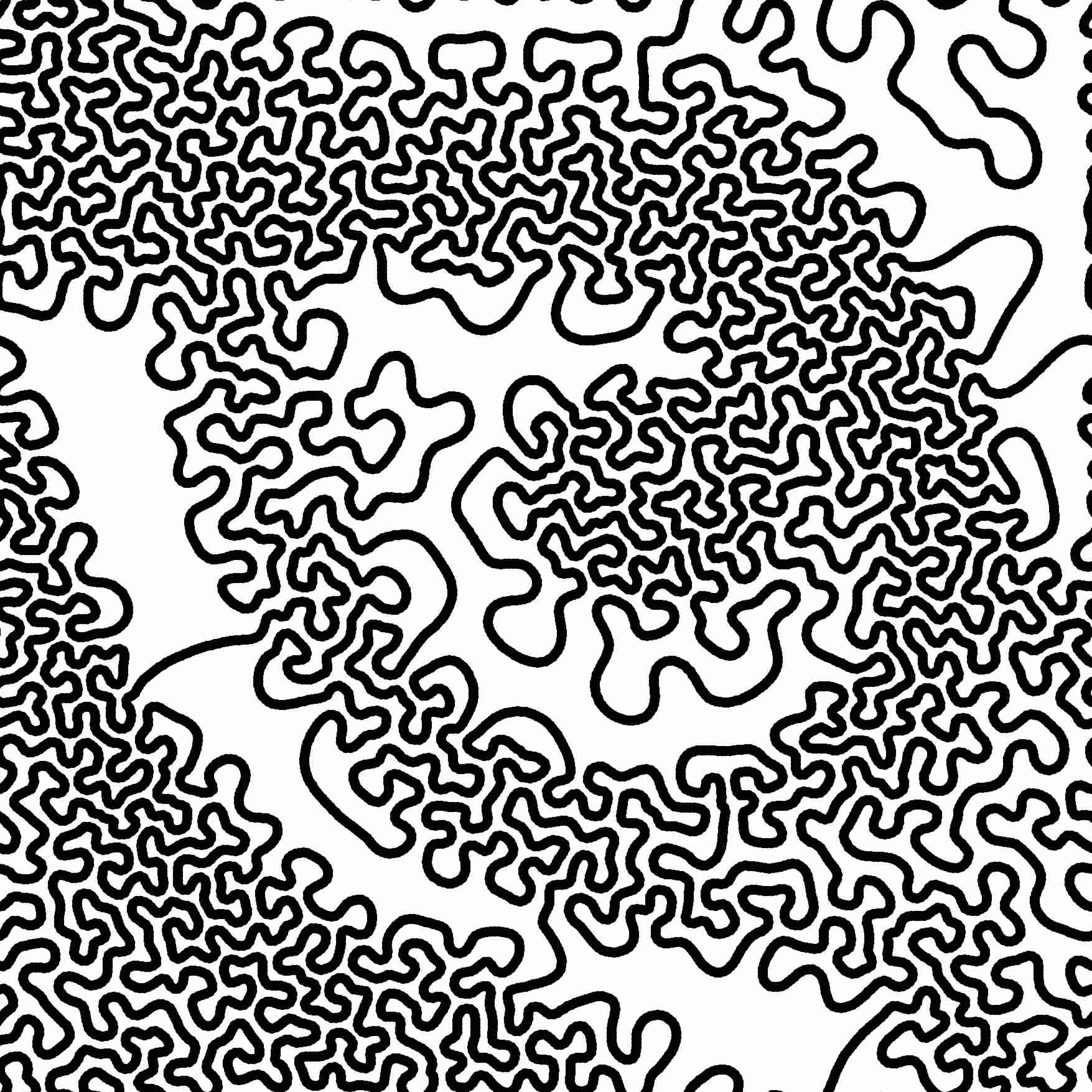} \\ 
		\includegraphics[width=0.23\textwidth]{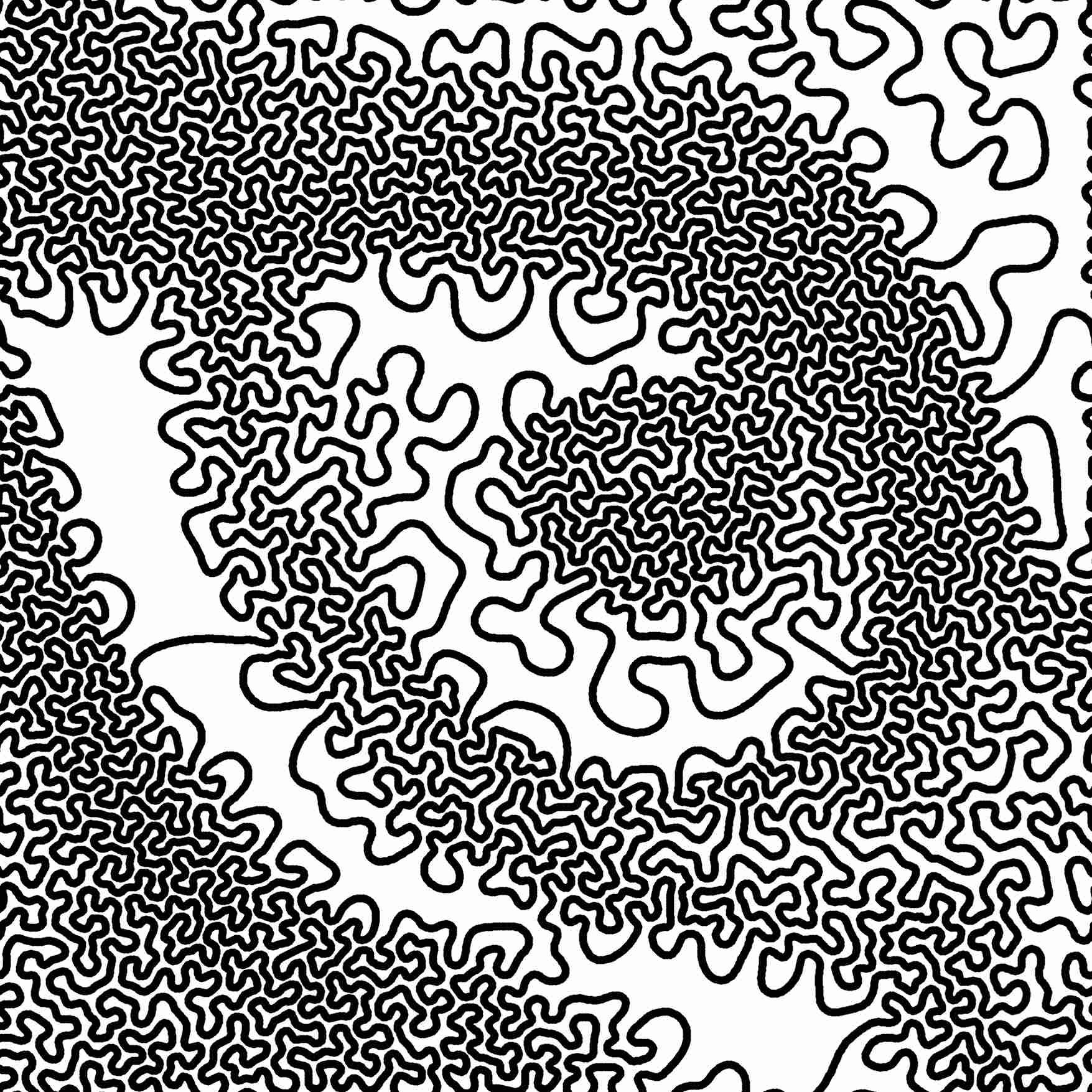} & 
		\includegraphics[width=0.23\textwidth]{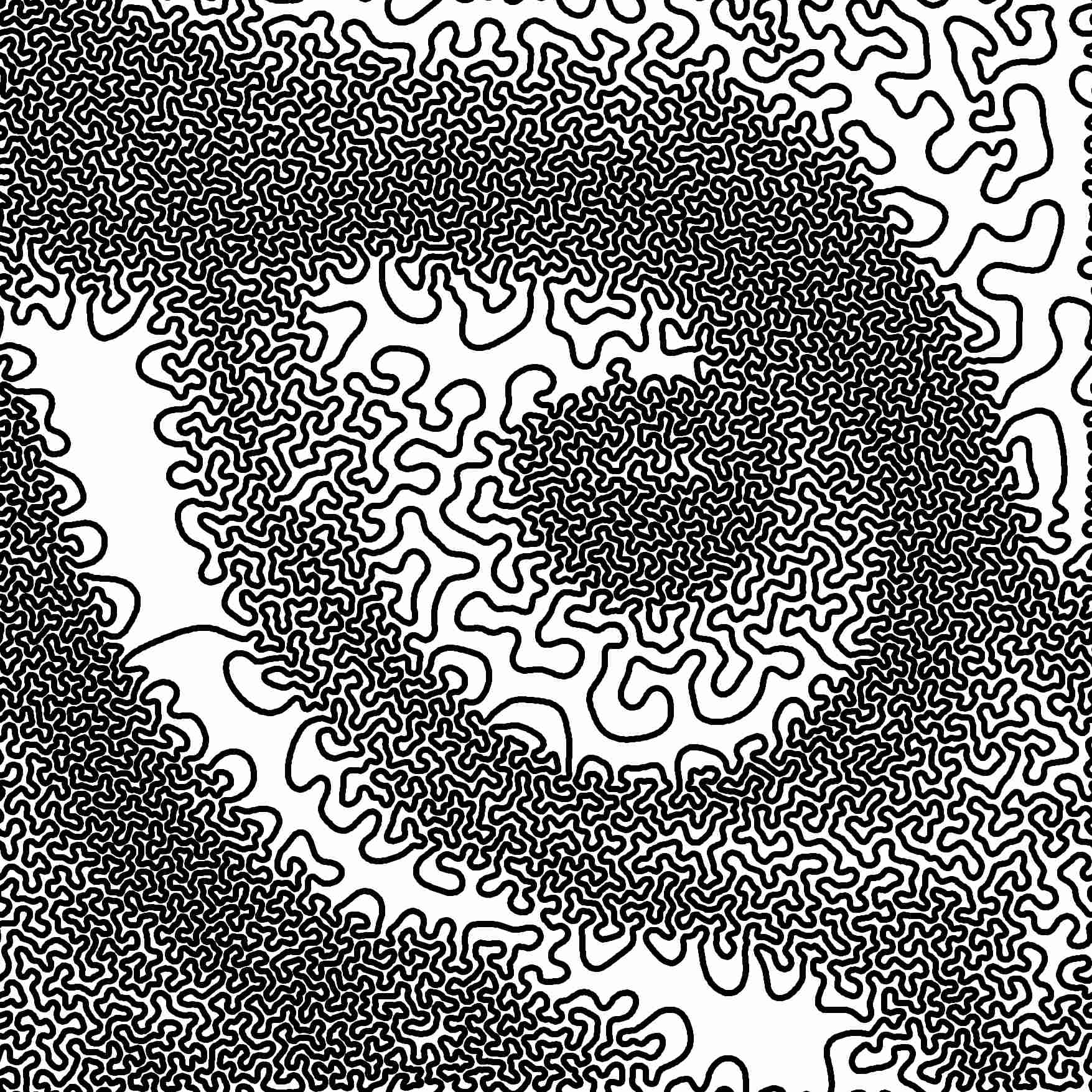} &    
		\includegraphics[width=0.23\textwidth]{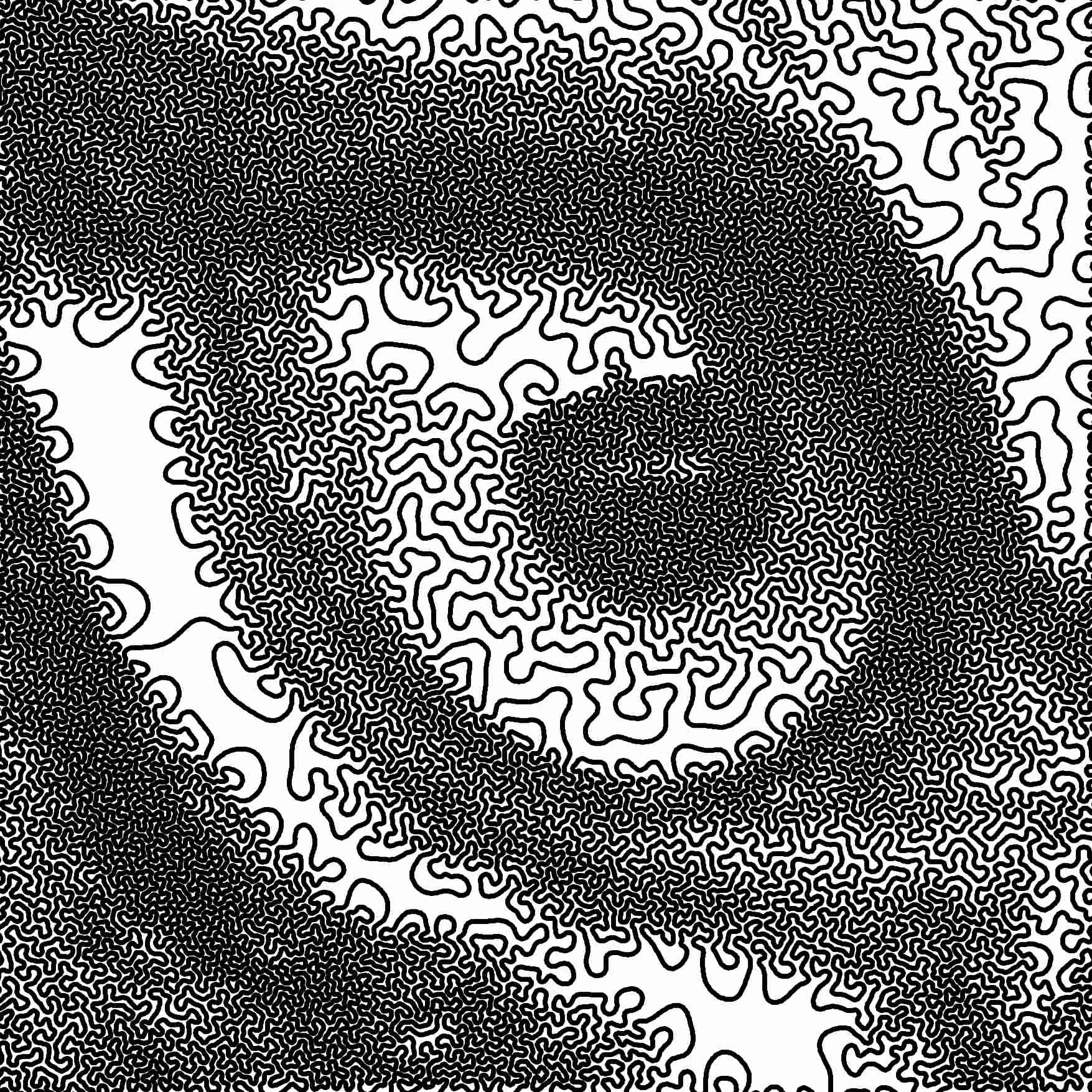} \\ 
		\includegraphics[width=0.23\textwidth]{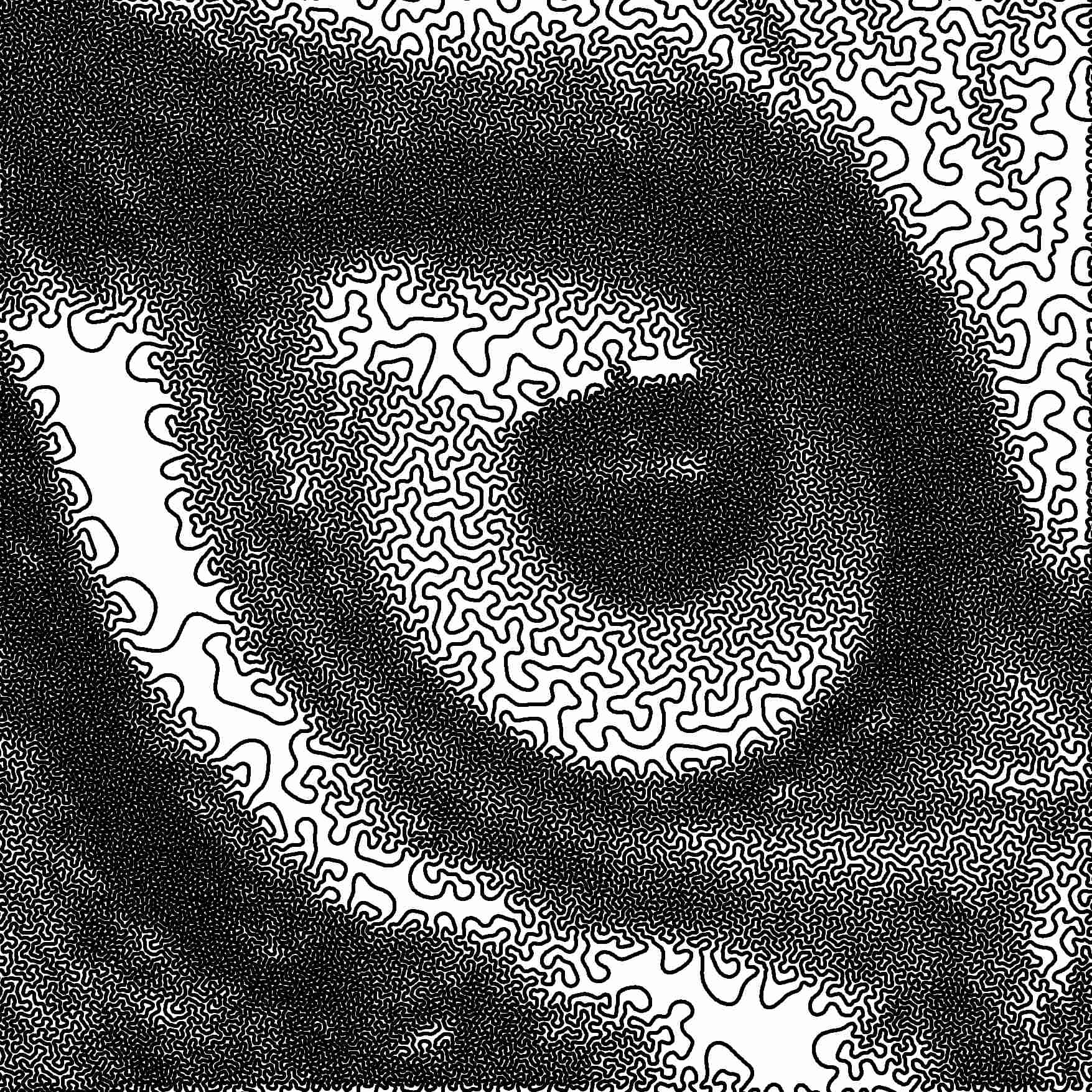} & 
		\includegraphics[width=0.23\textwidth]{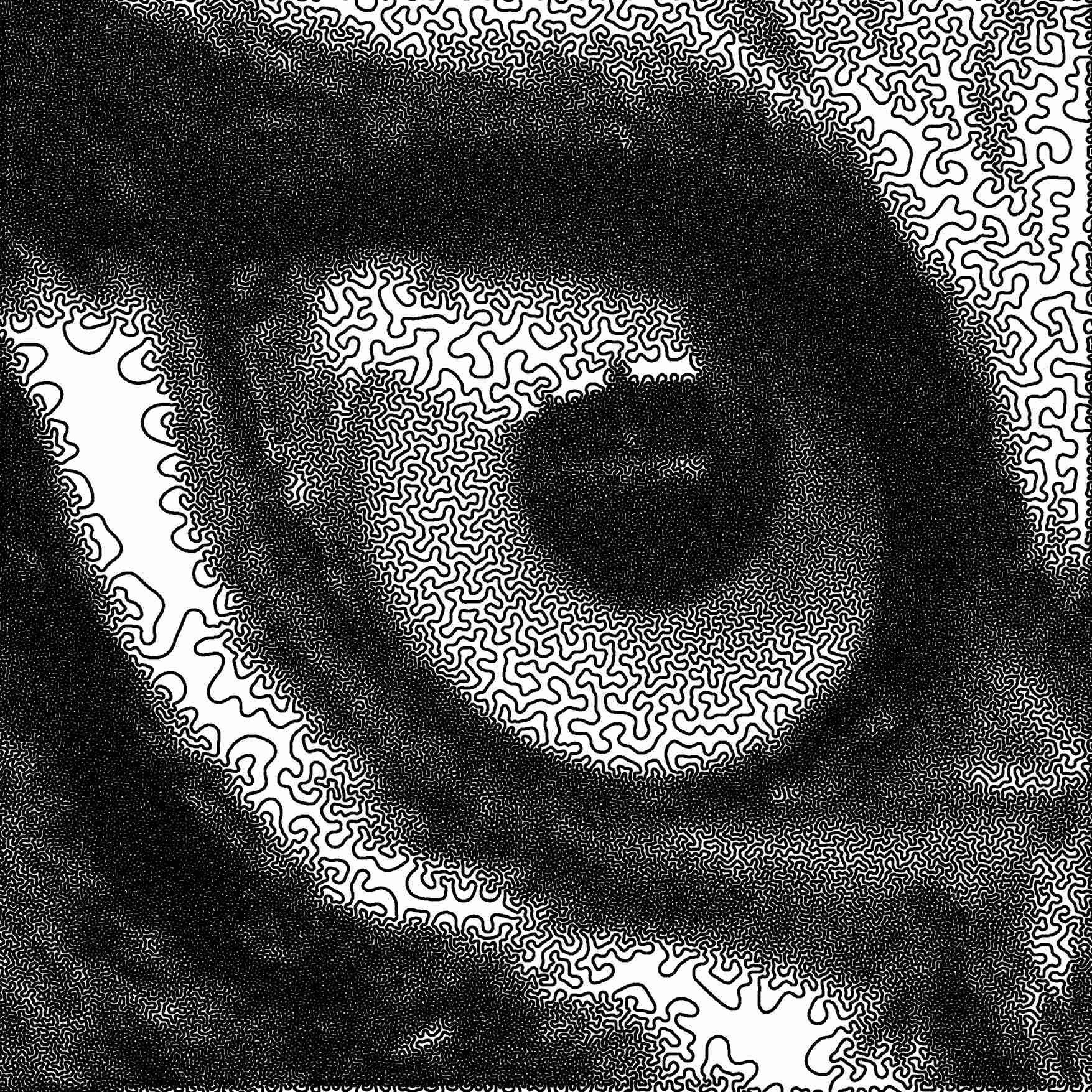} &    
		\includegraphics[width=0.23\textwidth]{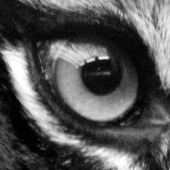} \\   
	\end{tabular}
	\caption{Local minimizers of \eqref{eq:final} for the image at bottom right.}
	\label{fig:tiger}
	\vspace{.73cm}
	\centering
	\includegraphics[width=0.7\textwidth]{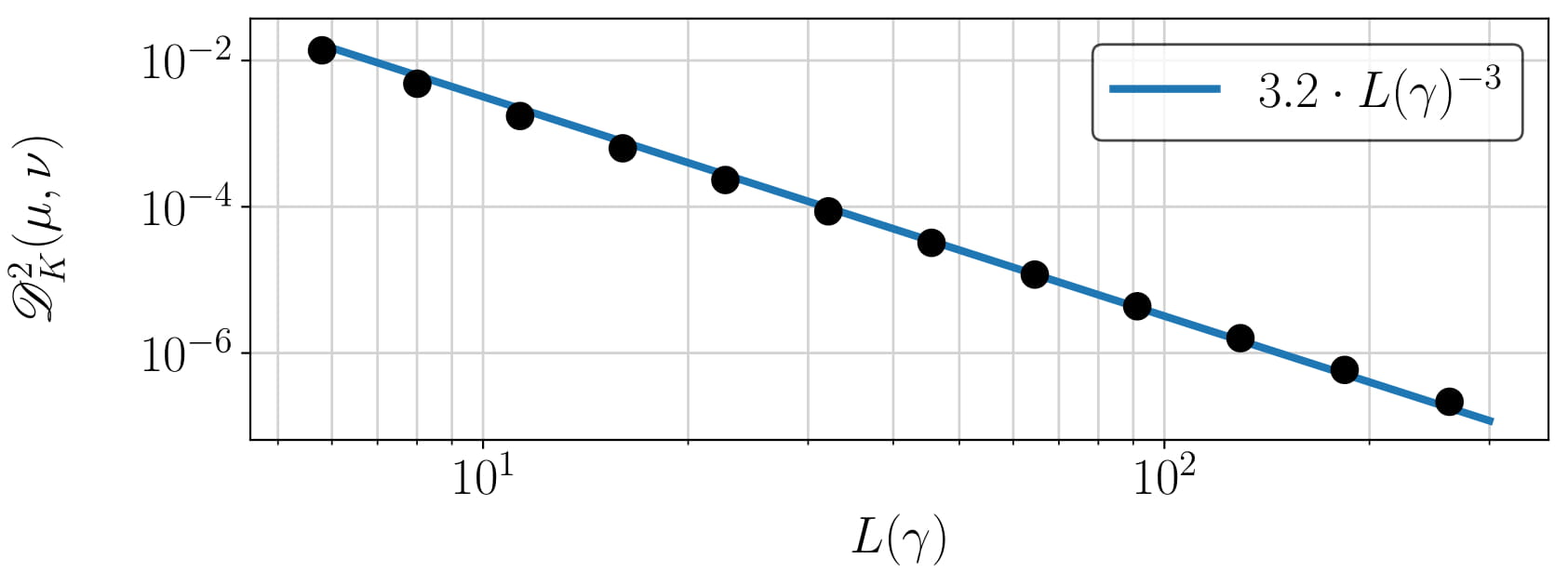}
	\caption{Squared discrepancy between the measure $\mu$ given  by the image in Fig.~\ref{fig:tiger} and the computed local minimizers (black dots) on $\mathbb T^{2}$ in log-scale.
		The blue line corresponds to the optimal decay-rate in  Theorem \ref{thm:torus better estimates}.}
	\label{fig:err_tiger}
\end{figure}
\smallskip

%--------------------------------------------------------------------
\textbf{3D-Torus $\T^3$.} 
The aim of this example is two-fold. First, it shows that the algorithm works pretty well in three dimensions. 
Second, we are able to approximate any compact surface in the three-dimensional space by a curve. 
We construct a measure $\mu$ supported around a two-dimensional surface by taking samples from Spock's head\footnote{http://www.cs.technion.ac.il/$\sim$vitus/mingle/} 
and placing small Gaussian peaks at the sampling points, i.e., the density is given for $x \in [-\tfrac12,\tfrac12]$ by
\[
\rho(x) \coloneqq c^{-1} \sum_{p \in S} \mathrm{e}^{-30000 \|p-x\|_2^2}, \qquad c \coloneqq \int_{[-\tfrac12,\tfrac12]^3} \sum_{p \in S} \mathrm{e}^{-30000 \|p-x\|_2^2} \dx x,
\]
where $S\subset [-\tfrac12,\tfrac12]^{3}$ is the discrete sampling set.
From a numerical point of view it holds $\mathrm{dim}(\supp (\mu)) = 2$.
The Fourier coefficients are again computed by a DFT and the kernel $K$ is given by \eqref{kernel_Td} with $d=3$ and $s = 2$ so that $H_{K} = H^{2}(\mathbb T^{3})$.

We start with $N_{0}=100$ points on a smooth curve given by the formula
\[
x_{0,k} =  \Bigl( \tfrac{3}{10} \cos(2\pi k/N_{0}), \tfrac{3}{10} \sin(2\pi k/N_{0}), \tfrac{3}{10} \sin(4\pi k/N_{0})\Bigr), \qquad k=0,\dots,N_{0}.
\]
Then, we apply our procedure for  $i=0,\dots,8$ with parameters, cf.~Remark~\ref{rem:param_choice},
\[
L_{i} = 2^{\frac{i+5}{2}},\quad \lambda_{i} = 10\cdot L_{i}^{-5},\quad N_{i} = 100 \cdot 2^{i} \sim L_{i}^{2}, \quad r_{i}= \lfloor 2^{\frac{i+5}{2}} \rfloor \sim L_{i}.
\]
To get nearly constant speed curves  $\gamma_i$, we increase $\lambda_{i}$ by a factor of 100, $N_{i}$ by a factor of 2 and set $L_{i} \coloneqq 2^{(i+6)/2}$. 
Then, we apply the CG method with maximal 100 iterations and one restart to the previously found curve $\gamma_{i}$.
The results are illustrated in Fig.~\ref{fig:spock}.  
Note that the complexity of the function evaluation in \eqref{eq:final} scales roughly as $N^{3/2} \sim L^{3}$.
In Fig.~\ref{fig:err_spock} we depict the squared discrepancy $\mathscr{D}_K^{2}(\mu,\nu)$ of the computed curves. For small Lipschitz constants, say $L(\gamma) \le 50$,  
we observe a decrease of approximately $L(\gamma)^{-3}$, which matches the optimal decay-rate for measures supported on surfaces as discussed in Remark \ref{rem:andere}.
\begin{figure}
	\centering
	\begin{tabular}{ccc}
		\includegraphics[width=0.29\textwidth]{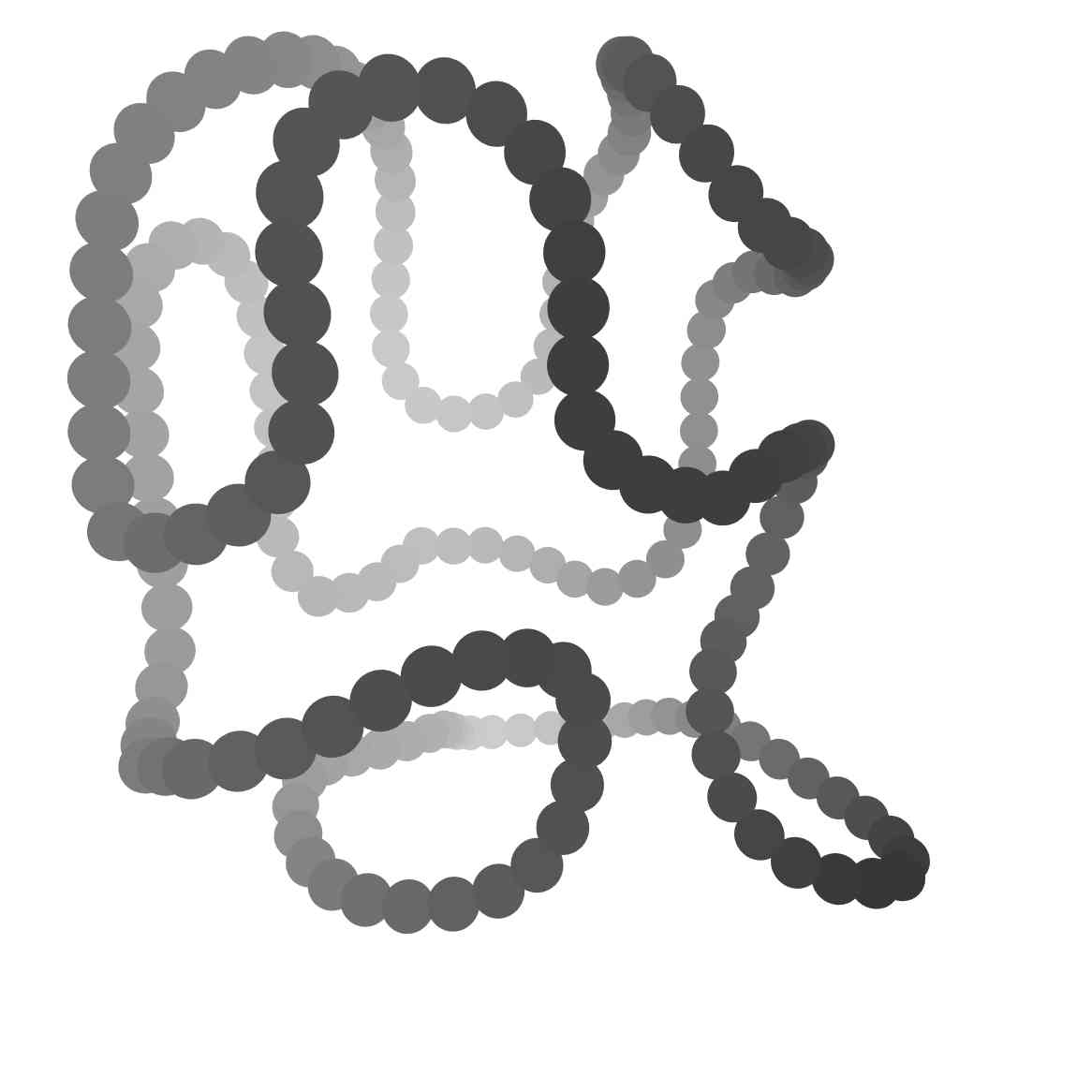} & 
		\includegraphics[width=0.29\textwidth]{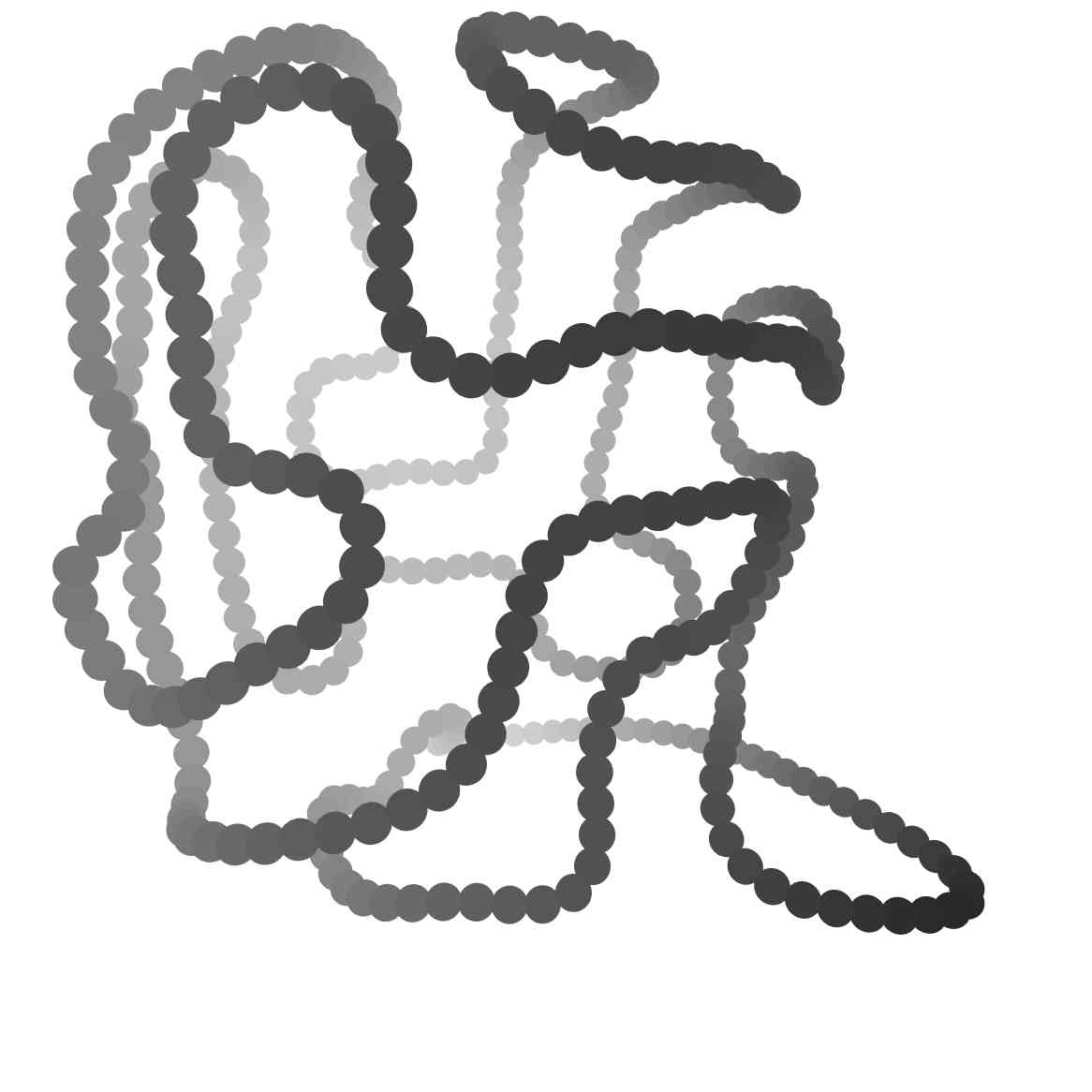} & 
		\includegraphics[width=0.29\textwidth]{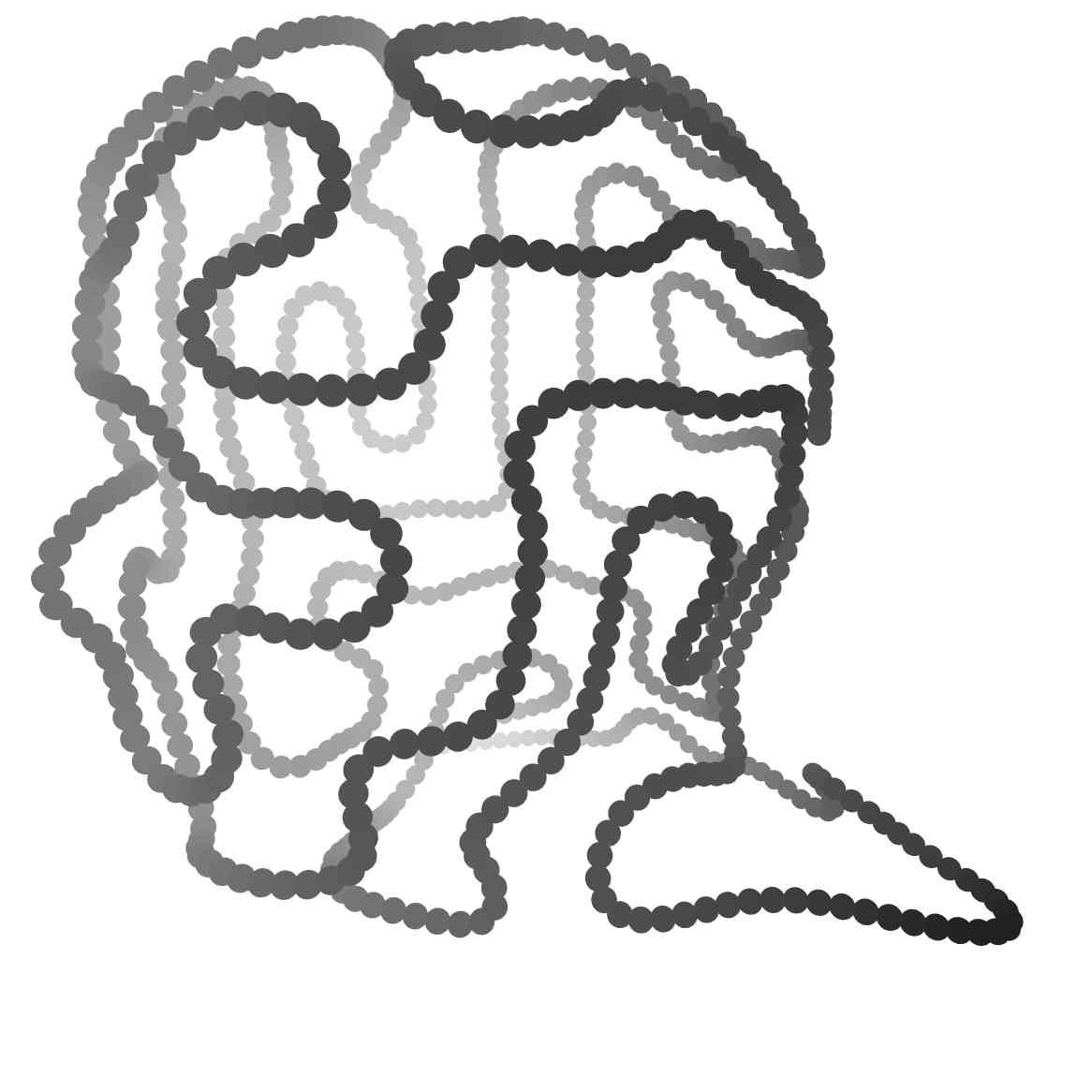} \\
		\includegraphics[width=0.29\textwidth]{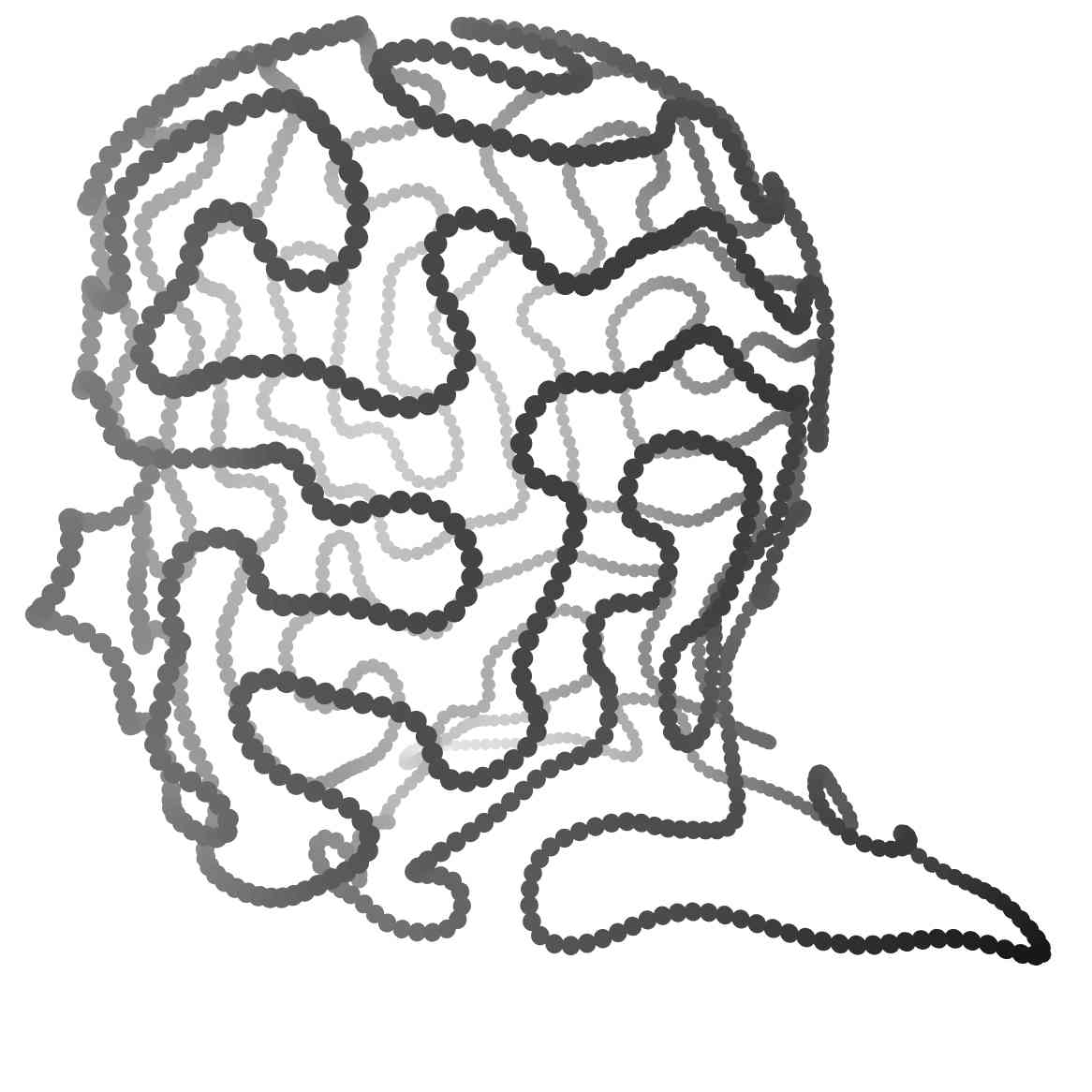} & 
		\includegraphics[width=0.29\textwidth]{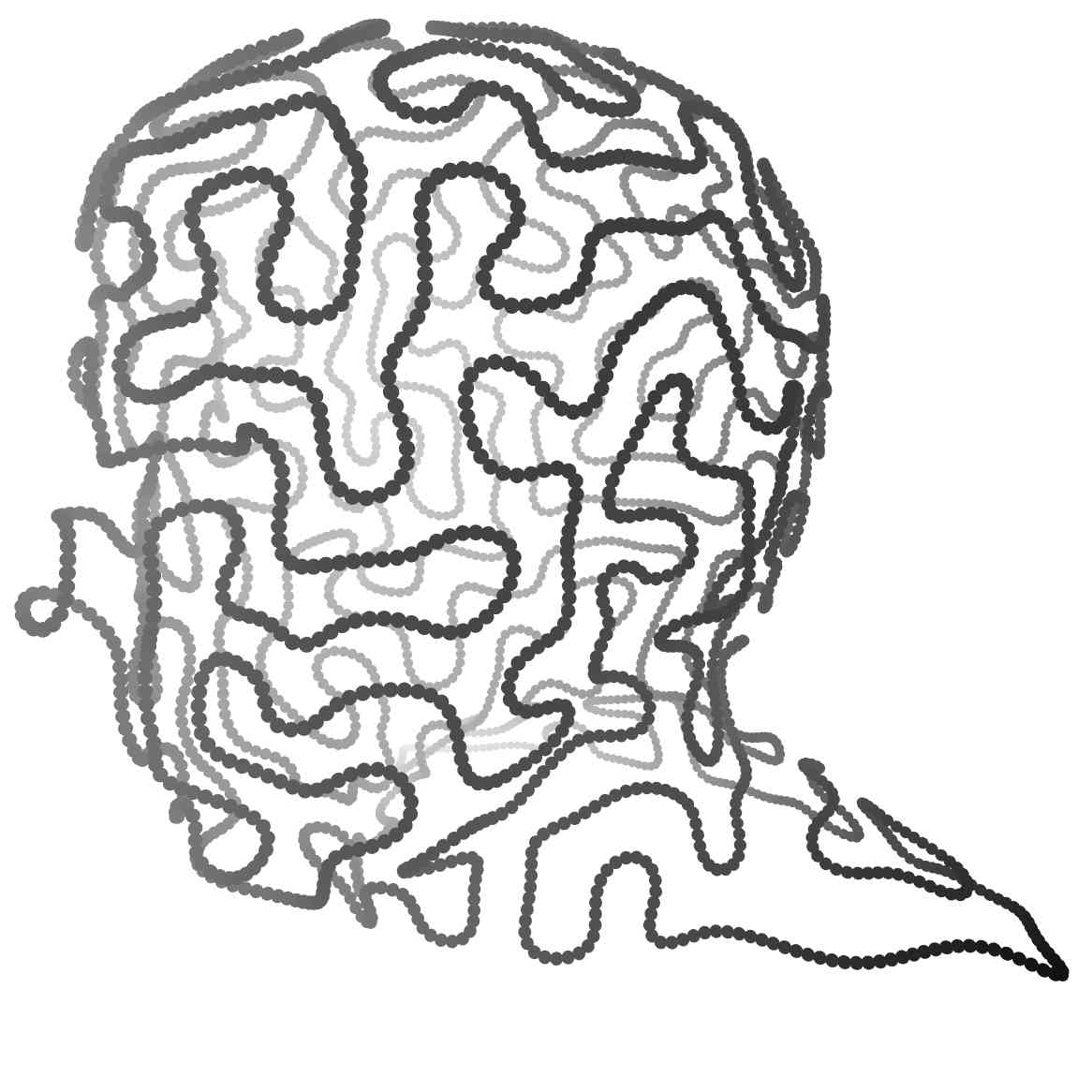} & 
		\includegraphics[width=0.29\textwidth]{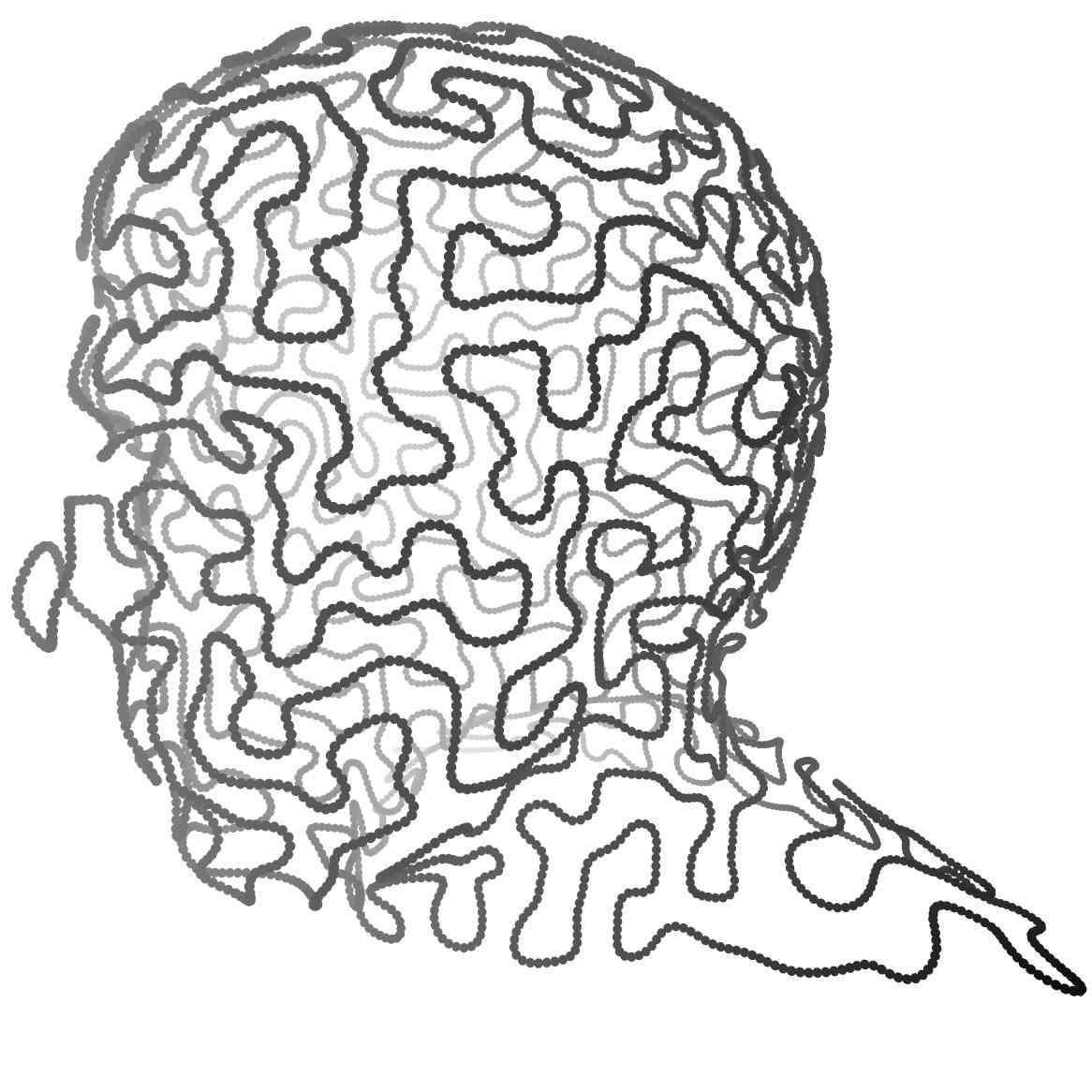} \\ 
		\includegraphics[width=0.29\textwidth]{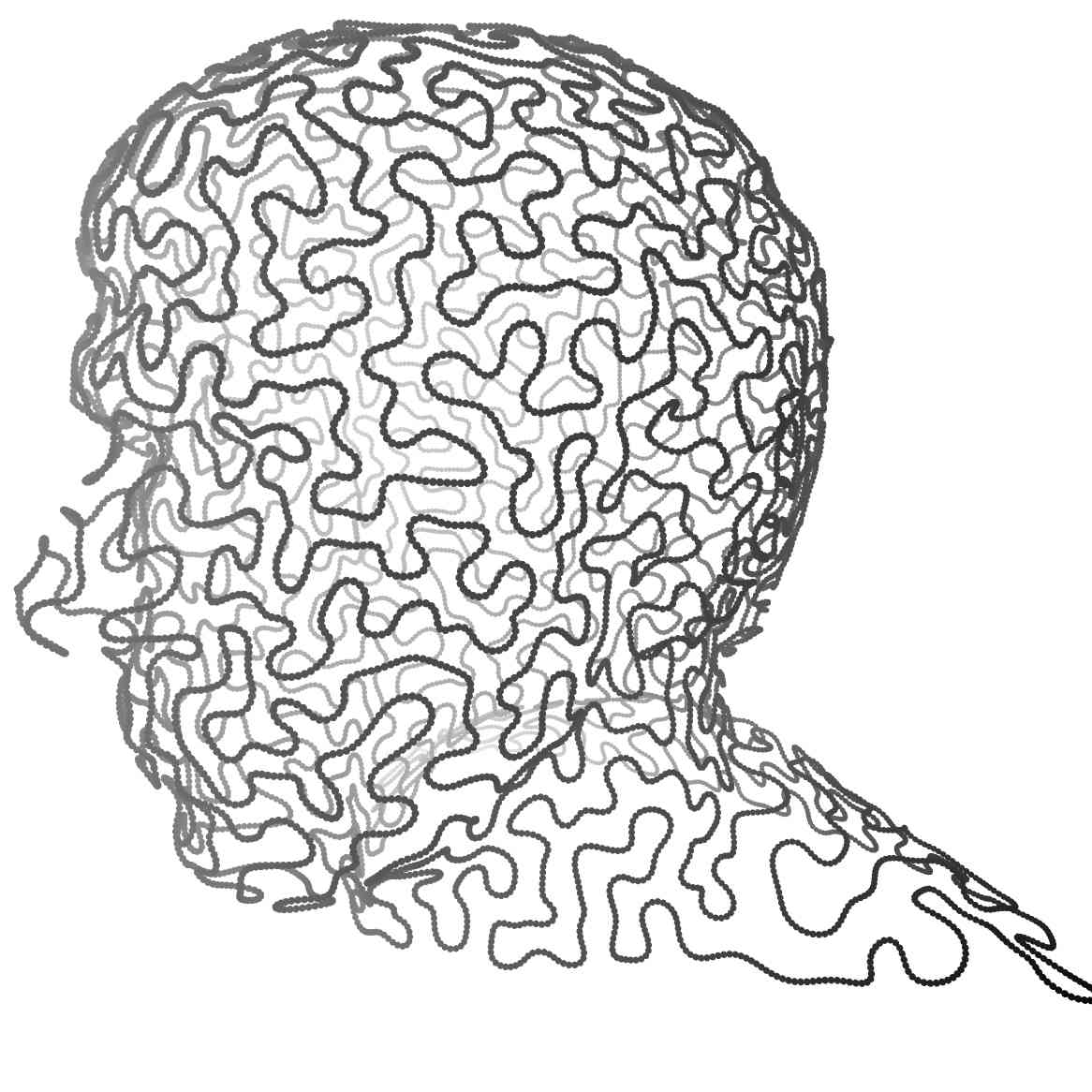} & 
		\includegraphics[width=0.29\textwidth]{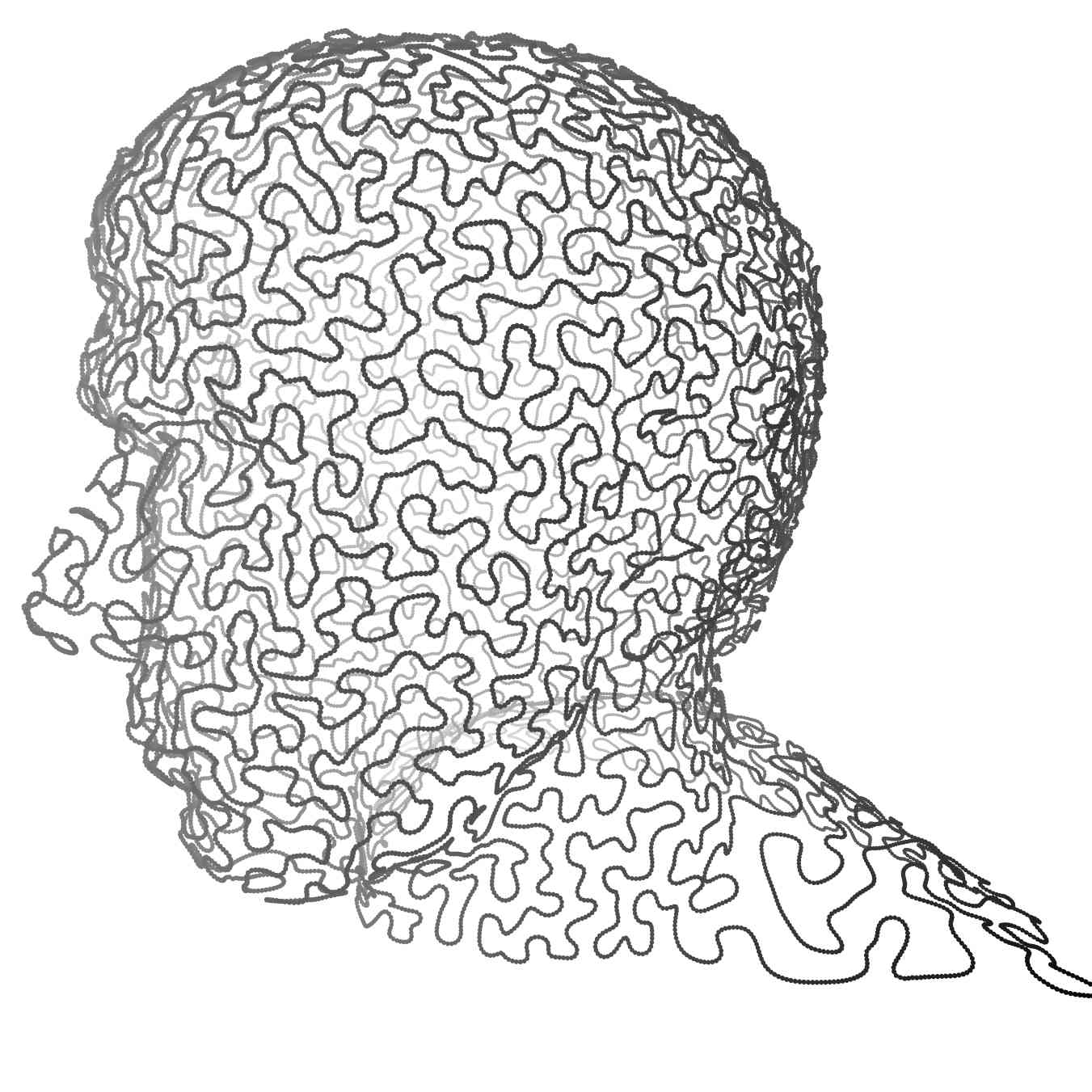} & 
		\includegraphics[width=0.29\textwidth]{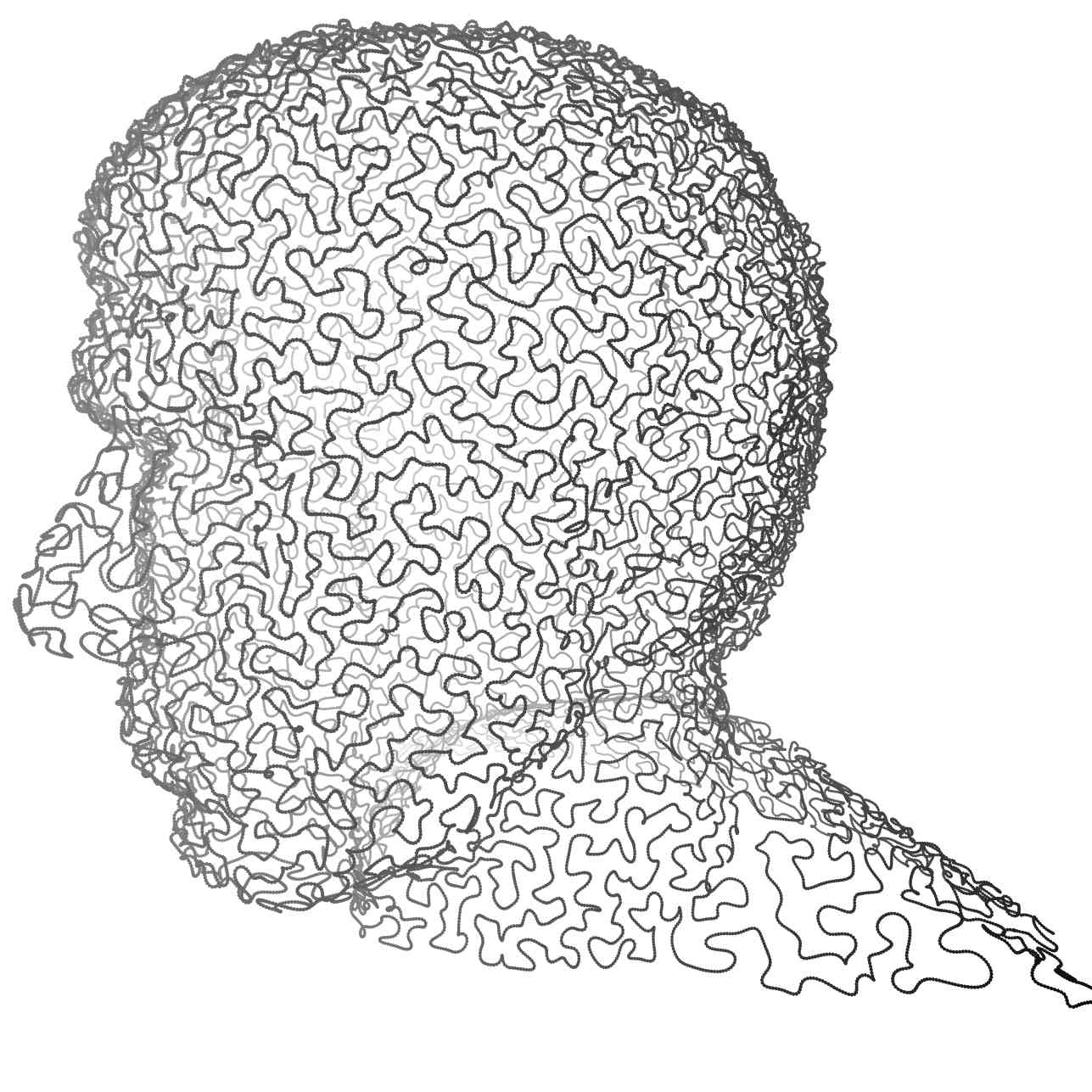} \\ 
	\end{tabular}
	\caption{Local minimizers of \eqref{eq:final} for a measure $\mu$ concentrated on a surface (head of Spock) in~$\T^3$.}
	\label{fig:spock}
	\vspace{.75cm}
	\centering
	\includegraphics[width=0.73\textwidth]{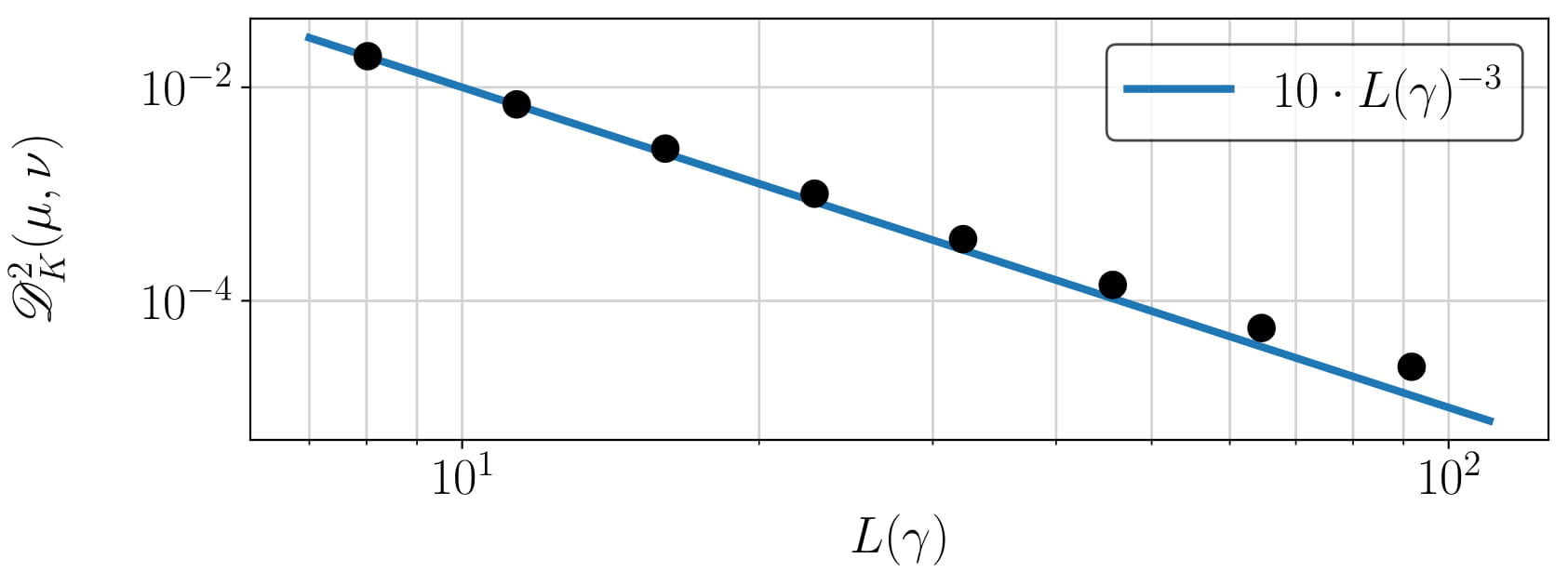}
	\caption{Squared discrepancy between the measure $\mu$ given  by the surface in Fig.~\ref{fig:spock} 
		and the computed local minimizers (black dots) on $\mathbb T^{3}$ in log-scale.
		The blue line corresponds to the optimal decay-rate in  Theorem \ref{thm:torus better estimates}.}
	\label{fig:err_spock}
\end{figure}
\smallskip

%--------------------------------------------
\textbf{2-Sphere $\S^2$.}
Next, we approximate a gray-valued image on the sphere $\mathbb S^{2}$ by an almost constant speed curve. 
The image represents the earth's elevation data provided by MATLAB, given by samples $\rho_{i,j}$, $ i=1,\dots,180,\;j=1,\dots,360$, on the grid
\[
x_{i,j} \coloneqq \Bigl(\sin\bigl( i \tfrac{\pi}{180}\bigr) \sin\bigl(j \tfrac{\pi}{180}\bigr), \sin\bigl( i \tfrac{\pi}{180}\bigr) \cos\bigl(j \tfrac{\pi}{180}\bigr), \cos\bigl(i \tfrac{\pi}{180}\bigr)\Bigr).
\]
The Fourier coefficients are computed by discretizing the Fourier integrals, i.e.,
\[
\hat \mu_{k}^{m} \coloneqq \begin{cases}
	\frac{1}{180 \cdot 360} \sum_{i=1}^{180} \sum_{j=1}^{360} \rho_{i,j} \overline{Y_{k}^{m}(x_{i,j})} \sin\bigl(i \tfrac{\pi}{180}\bigr),& 1\leq k\le 2m+1, m \le 180,\\
	0, & \text{else},
\end{cases}
\]
followed by a normalization such that $\hat\mu_{0}^{0}=1$. 
The corresponding sums are efficiently computed by an adjoint non-equispaced fast spherical Fourier transform (NFSFT), see \cite{PPST2019}.
The kernel $K$ is given by \eqref{kernel_Sd}.
Similar to the previous examples, we apply our procedure for  $i=0,\dots,12$ with parameters
\[
L_{i} = 9.7\cdot 2^{\frac{i}{2}}, \quad \lambda_{i} = 100\cdot L_{i}^{-5},\quad N_{i} = 100\cdot 2^{i} \sim L_{i}^{2}, \quad r_{i}= \lfloor L_{i} \rfloor \sim L_{i}.
\]

To get nearly constant speed curves, we increase $\lambda_{i}$ by a factor of 100, $N_{i}$ by a factor of 2 
and set $L_{i} \coloneqq L_{0}  2^{i/2}$. 
Then, we apply the CG method with maximal 100 iterations and one restart to the previously constructed curves $\gamma_{i}$.
The results for $i=6,8,10,12$ are depicted in Fig.~\ref{fig:earth}. 
Note that the complexity of the function evaluation in \eqref{eq:final} scales roughly as $N \sim L^{2}$.
In Fig.~\ref{fig:err_earth} we observe that the decay-rate of the squared discrepancy 
$\mathscr{D}_K^{2}(\mu,\nu)$ in dependence on the Lipschitz constant matches indeed the theoretical findings in Theorem \ref{thm:sphere}.
\begin{figure}
	\centering
	\begin{tabular}{cccc}
		\includegraphics[width=0.21\textwidth]{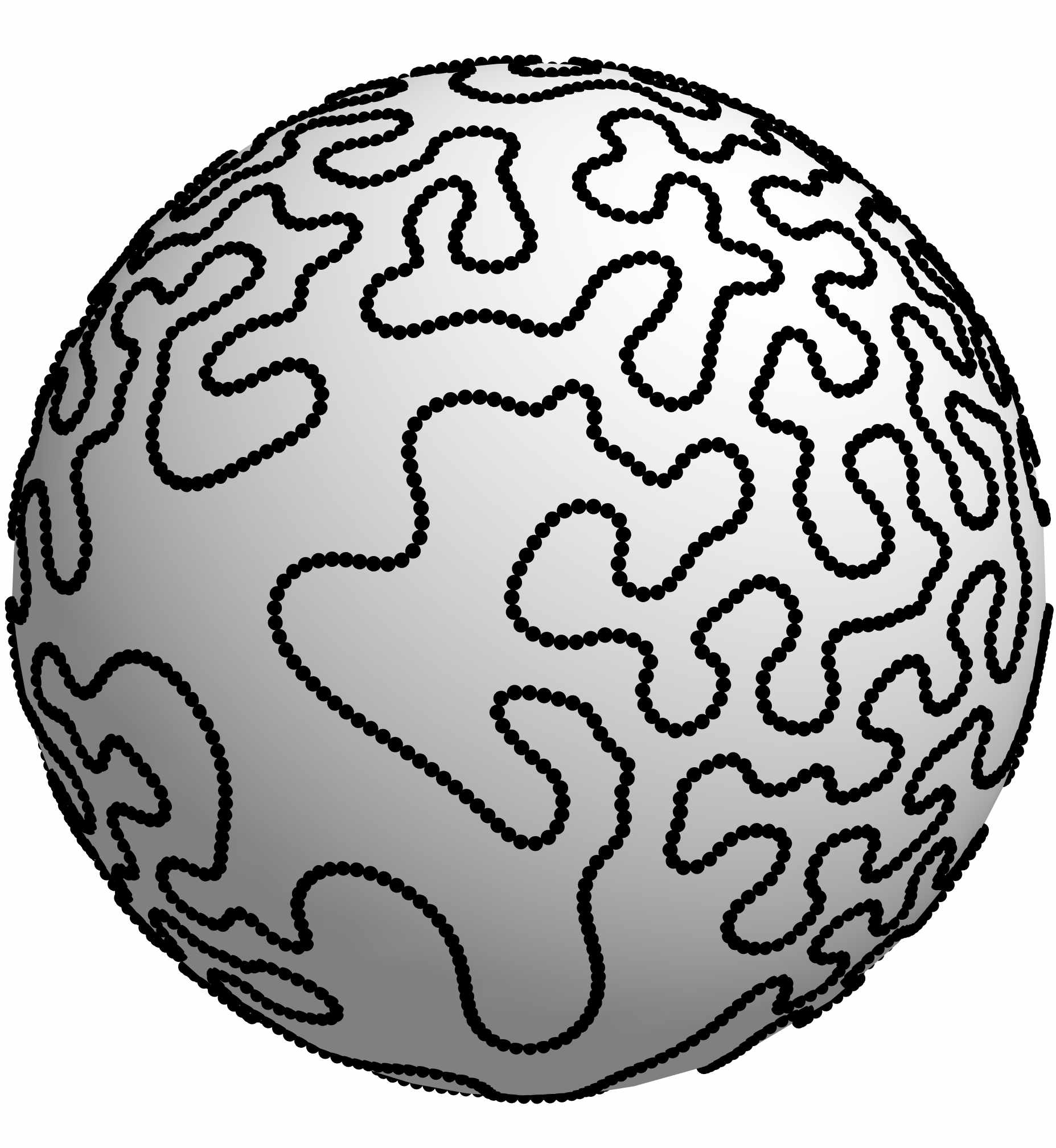} & 
		\includegraphics[width=0.21\textwidth]{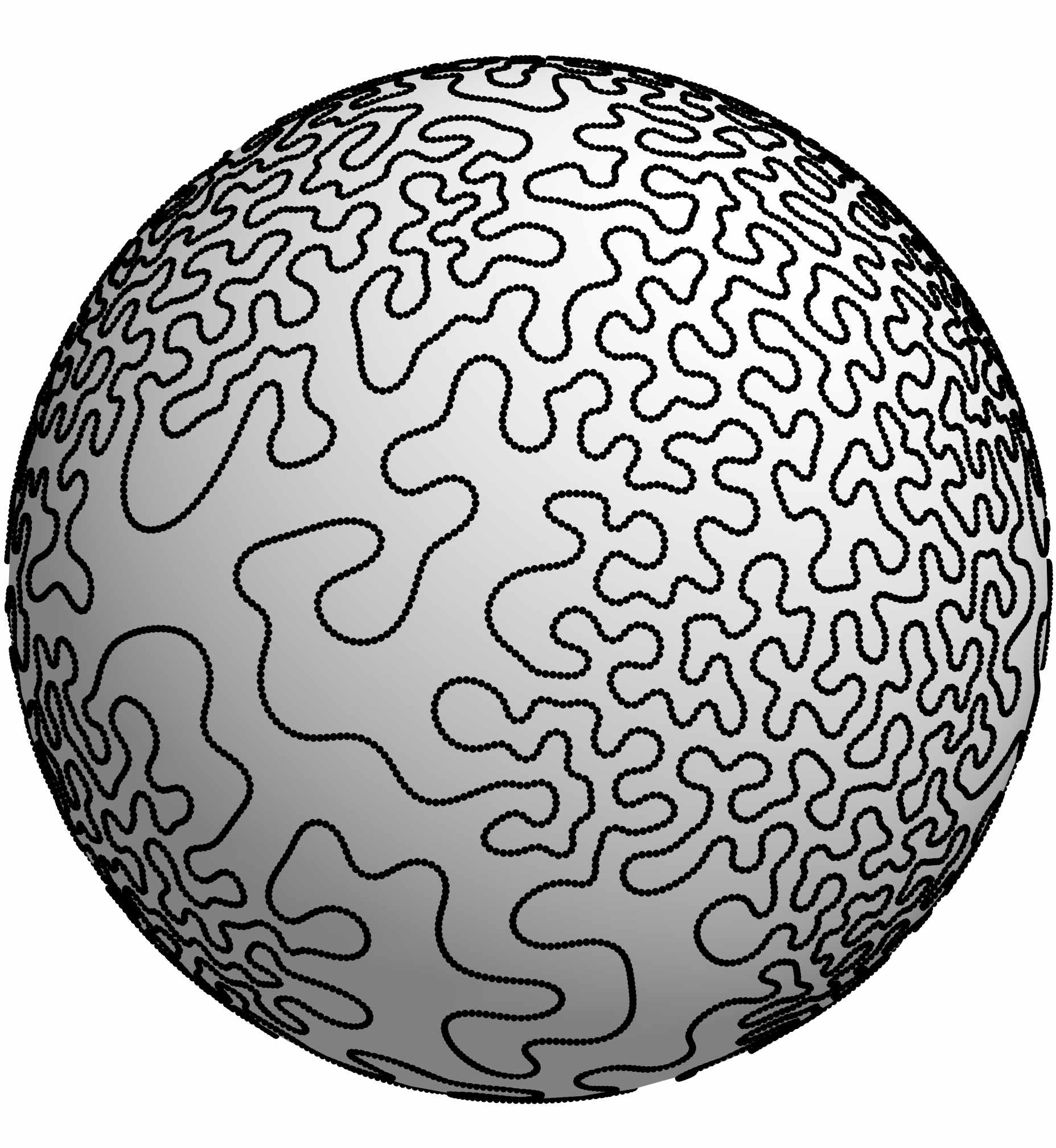} &
		\includegraphics[width=0.21\textwidth]{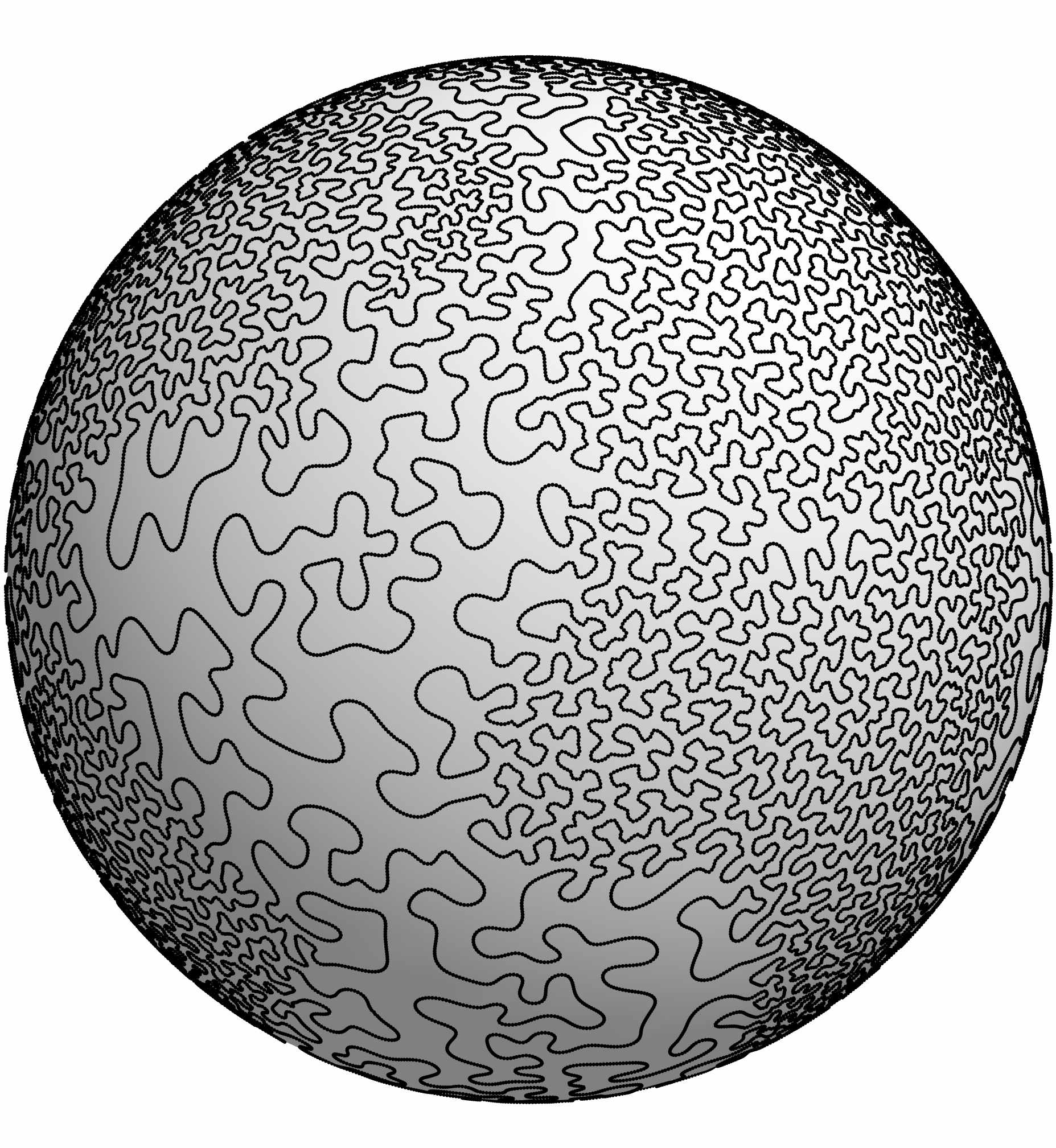} & 
		\includegraphics[width=0.21\textwidth]{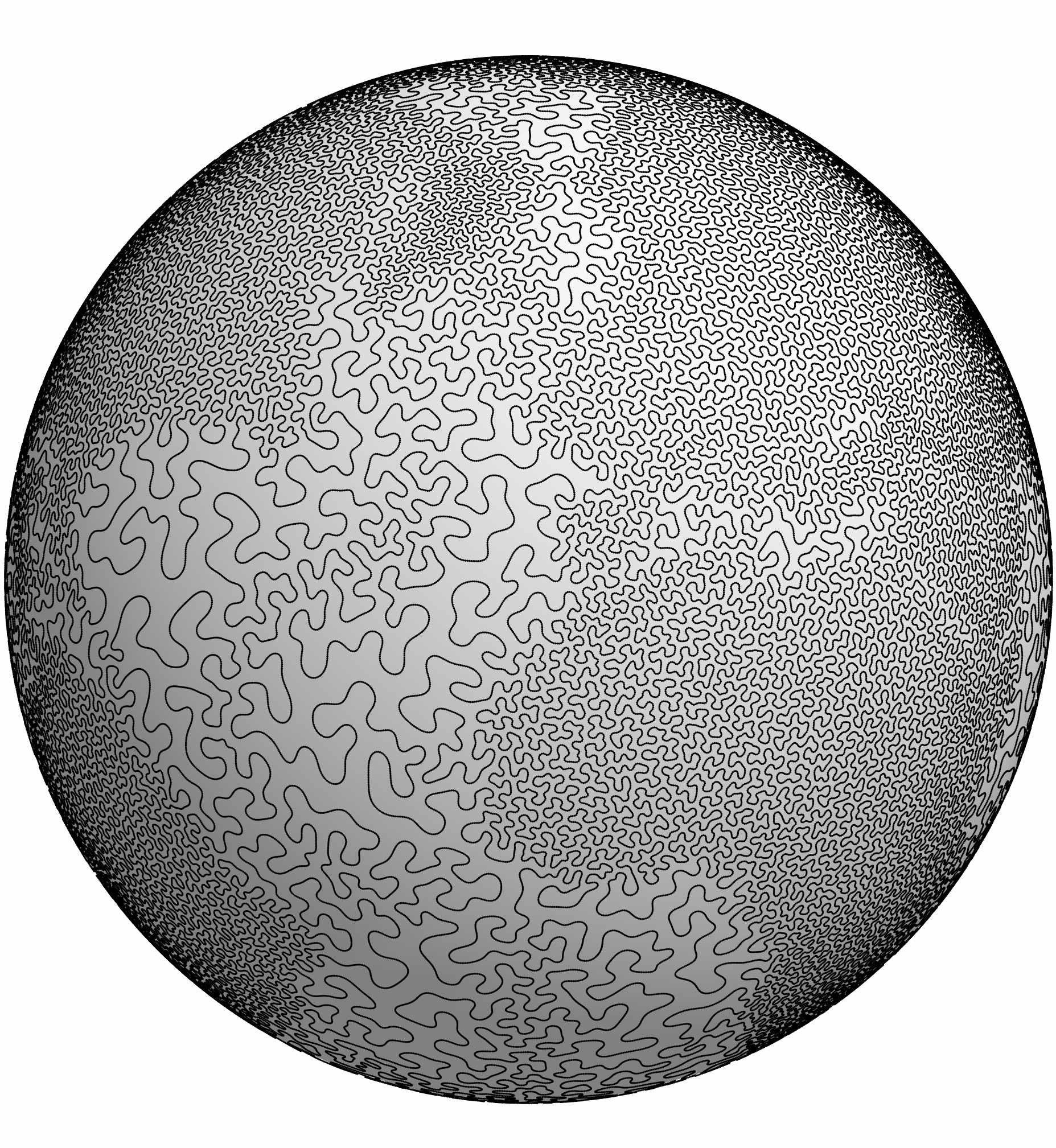} \\
	\end{tabular}
	\caption{Local minimizers of \eqref{eq:final} for $\mu$ given by the earth's elevation data on the sphere $\mathbb S^{2}$.}
	\label{fig:earth}
	\vspace{.5cm}
	\centering
	\includegraphics[width=0.75\textwidth]{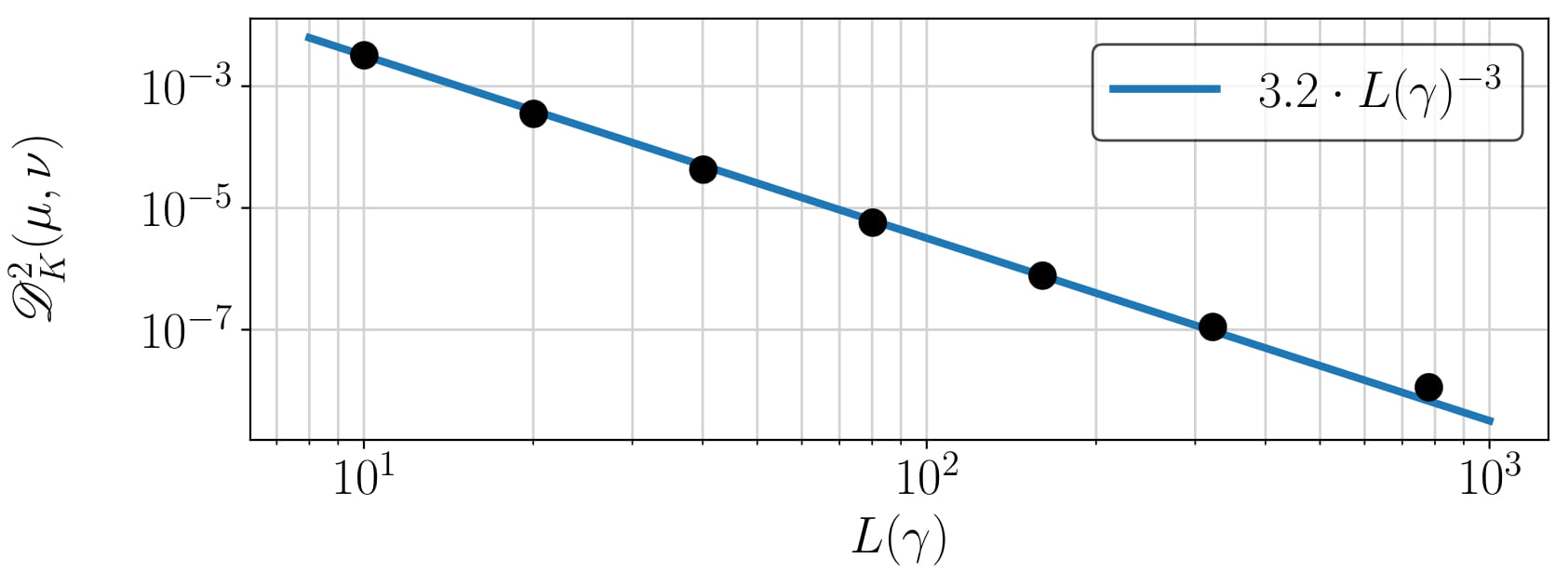}
	\caption{Squared discrepancy between the measure $\mu$ and the computed local minimizers (black dots) in log-scale.
		The blue line corresponds to the optimal decay-rate in Theorem \ref{thm:sphere}.}
	\label{fig:err_earth}
\end{figure}
\smallskip

%--------------------------------------------------------------------------------
\textbf{3D-Rotations $\SO(3)$.} There are several possibilities to parameterize the rotation group $\SO(3)$.
We apply those by Euler angles and an axis-angle representation for visualization. 
Euler angles 
$(\varphi_{1}, \theta, \varphi_{2}) \in [0,2\pi) \times [0,\pi] \times [0,2\pi)$ 
correspond to rotations $\mathrm{Rot} ( \varphi_{1} , \theta, \varphi_{2} )$ in $\SO(3)$
that are the successive rotations around the axes $e_3,e_2,e_3$ by the respective angles.
%\begin{equation}   \label{def:EulerAngles}
%    \mathrm{Rot}( \varphi_{1} , \theta, \varphi_{2} ) = \left(
%    \begin{smallmatrix}
%        \cos (\varphi_1)\cos (\theta)\cos (\varphi_{2}) - \sin (\varphi_1)\sin
%        (\varphi_{2}) & \; - \cos(\varphi_1)\cos(\theta)\sin(\varphi_{2})
%        - \sin (\varphi_1) \cos (\varphi_{2})& \; \cos (\varphi_1)\sin (\theta)  \\
%        \sin (\varphi_1)\cos (\theta)\cos (\varphi_{2}) + \cos (\varphi_1)\sin
%        (\varphi_{2}) & \; - \sin(\varphi_1)\cos (\theta) \sin (\varphi_{2}) + \cos
%        (\varphi_1)\cos (\varphi_{2}) & \; 
%        \sin (\varphi_1)\sin (\theta)   \\
%        -\sin (\theta) \cos (\varphi_{2}) & \; \sin (\theta) \sin (\varphi_{2}) & \; 
%        \cos (\theta)
%      \end{smallmatrix}
%    \right).
% \end{equation}
Then, the Haar measure of $\SO(3)$ is determined by
\[
\dx \mu_{\mathrm{SO(3)}}(\varphi_{1},\theta,\varphi_{2}) 
= \tfrac{1}{8\pi^{2}} \sin(\theta) \dx \varphi_{1} \dx \theta \dx \varphi_{2}.
\]
We are interested in the  full three-dimensional doughnut
\[
D = \bigl\{     \mathrm{Rot} ( \varphi_{1} , \theta, \varphi_{2} )  \;:\; 0 \le \theta \le \tfrac{\pi}{2},\; 0\le \varphi_{1},\varphi_{2} \le 2\pi \bigr\} \subset \mathrm{SO(3)}.
\]
Next, we want to approximate the Haar measure $\mu = \mu_D$ restricted to $D$, i.e., with normalization we consider the measure defined for $f \in C(\SO(3))$ by
\[
\int_{\mathrm{SO(3)}} f \dx \mu_{D}
= \frac{1}{4\pi^{2}}\int_{0}^{2\pi}\int_{0}^{\frac{\pi}{2}}\int_{0}^{2\pi} f (\varphi_{1},\theta,\varphi_{2}) \sin(\theta)  \dx \varphi_{1} \dx \theta  \dx \varphi_{2}.
\]
The Fourier coefficients of $\mu_D$ can be explicitly computed by
\[
\hat \mu_{l,l'}^k =
\begin{cases}
	P_{k-1}(0) - P_{k+1}(0), & l,l' = 0,\; k\ge 0,\\
	0, & l,l' \ne 0,
\end{cases}  
\]
where $P_{k}$ are the Legendre polynomials. 
The kernel $K$ is given by \eqref{kernel_SO3} with $d=3$ and $s=2$.
For $i=0,\dots,8$ the parameters are chosen as
\[
L_{i} = 0.93 \cdot  2^{\frac{2i + 12}{3}}, \quad \lambda_{i} = 10\cdot L_{i}^{-4},\quad N_{i} = 64 \cdot 2^{i} \sim L_{i}^{2}, 
\quad r_{i}= \lfloor 2^{\frac{i+9}{3}} \rfloor \sim L_{i}^{\frac12}.
\]
Here, we use a CG method with 100 iterations and one restart. Step ii) appears to be not necessary.
Note that the complexity for the function evaluations in~\eqref{eq:final} scales roughly as $N \sim L^{3/2}$.

The constructed curves are illustrated in Fig.~\ref{fig:so3_cylinder}, where we 
utilized the following visualization: 
Every rotation $R(\alpha,r) \in \SO(3)$ is determined by a rotation axis $r = (r_1,r_2,r_3) \in \S^2$ and a rotation angle $\alpha \in [0,\pi]$, 
i.e.,
\[R(\alpha,r) x = r(r^\tT x) + \cos(\alpha) \left( (r \times x) \times r \right)  + \sin(\alpha) (r \times x).\]
Setting $q \coloneqq (\cos(\tfrac \alpha 2), \sin (\tfrac \alpha 2) r) \in \S^3$ with $r \in \S^2$ and $\alpha \in [0, 2\pi]$,
see \eqref{spherical},  we observe that the same rotation is generated by $-q = (\cos(\tfrac{2\pi- \alpha}{ 2}), \sin (\tfrac{2\pi- \alpha}{ 2} (-r)) \in \S^3$, in other words
$\SO(3) \cong \mathbb S^3 / \{\pm 1\}$.
Then, by applying the stereographic projection
$\pi(q) = (q_2,q_3,q_4)/(1+q_1)$, we map the upper hemisphere onto the three dimensional unit ball. 
Note that the equatorial plane of $\mathbb S^3$ is mapped onto the sphere $\mathbb S^2$, 
hence on the surface of the ball antipodal points have to be identified. 
In other words, the rotation $R(\alpha, r)$ is plotted as  the point \[\pi(q) = \frac{\sin\bigl(\tfrac \alpha 2\bigr)}{1 + \cos\bigl(\tfrac \alpha 2\bigr)}r = \tan\bigl(\tfrac \alpha 4\bigr) r \in \mathbb R^3.\]

In Fig.~\ref{fig:err_so3_cylinder} we observe that the decay-rate of $\mathscr D_K^{2}(\mu,\nu)$ in dependence on the Lipschitz constant $L$ matches the theoretical findings in Corollary \ref{cor:so3}.
\begin{figure}
	\centering
	\begin{tabular}{ccc}
		\includegraphics[width=0.29\textwidth]{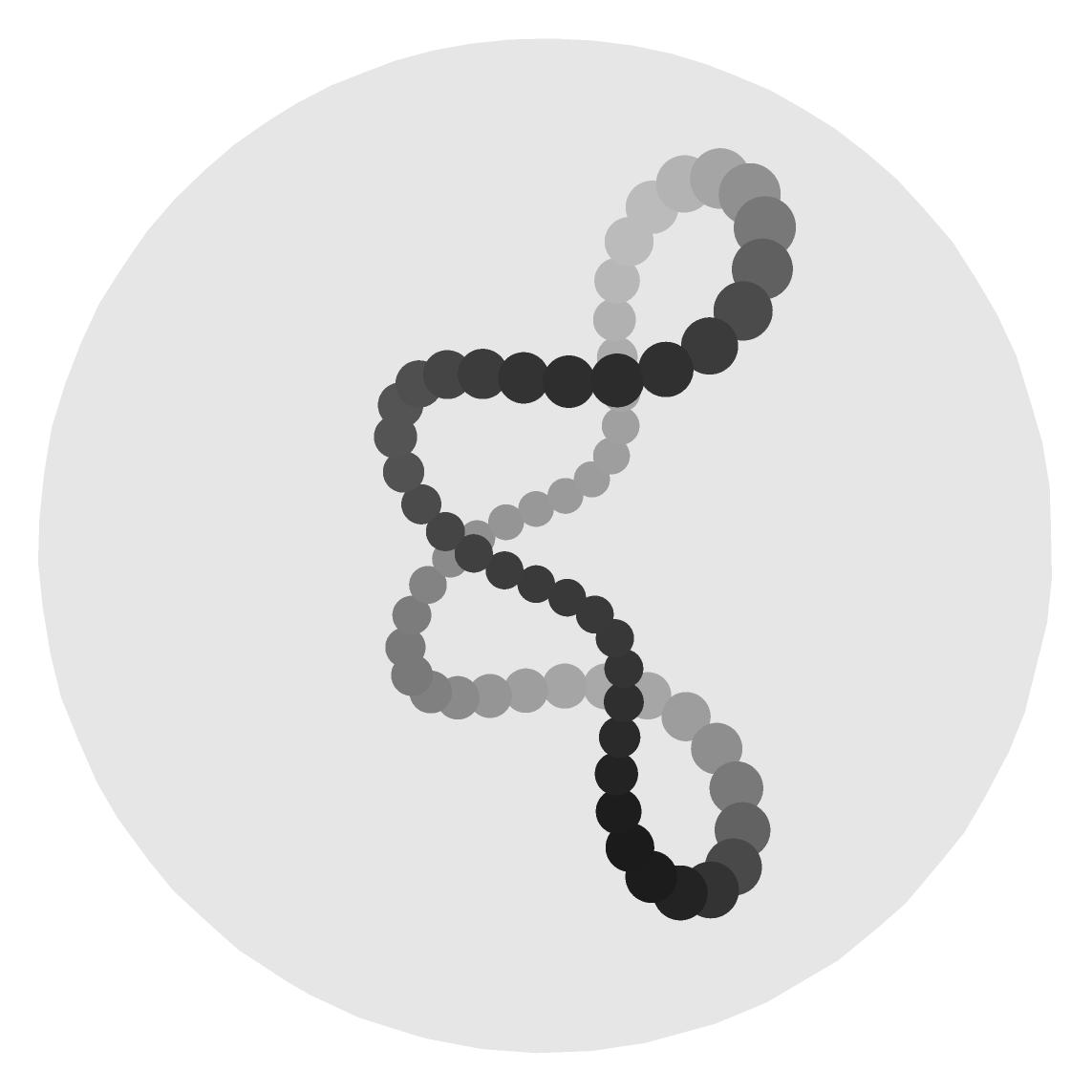} & 
		\includegraphics[width=0.29\textwidth]{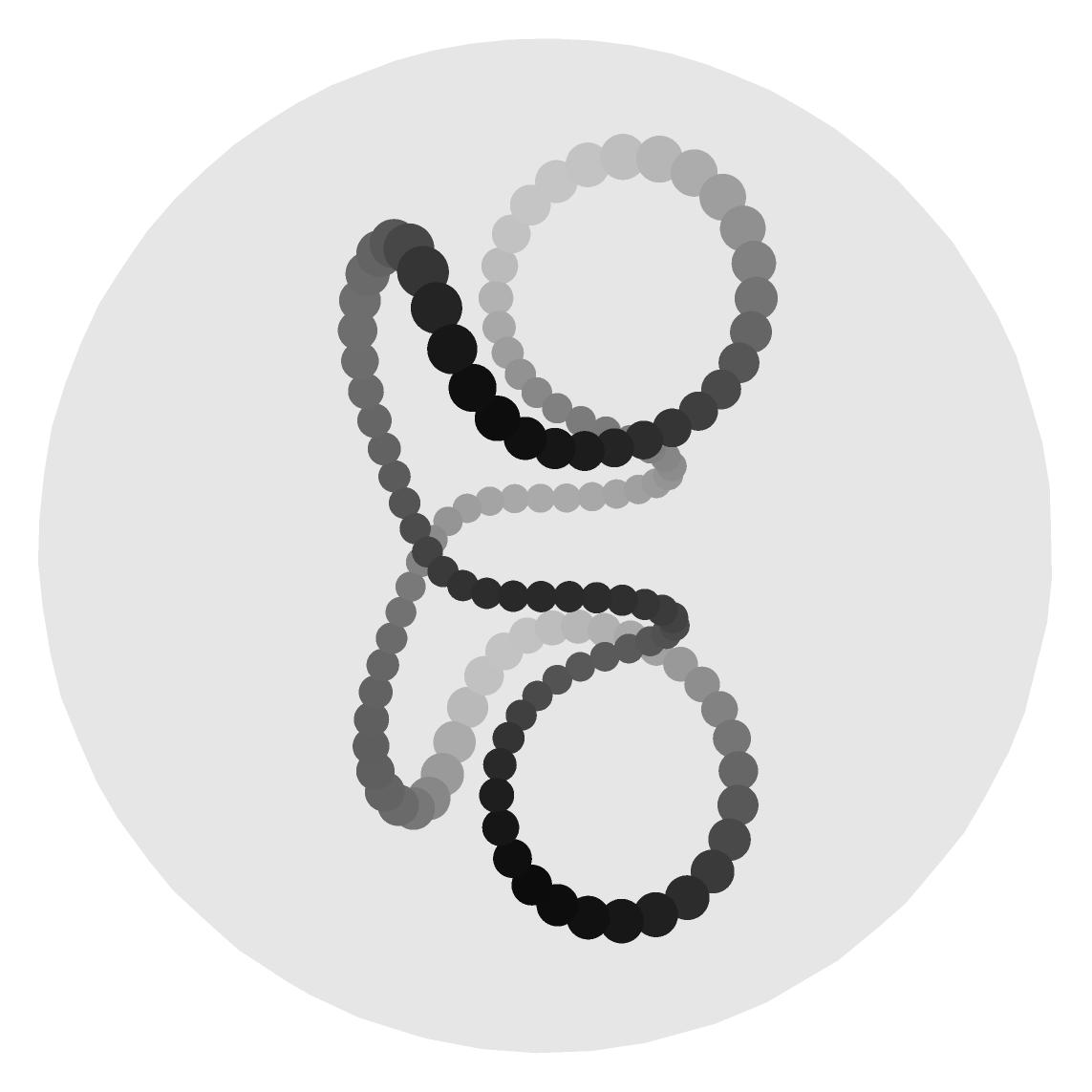} & 
		\includegraphics[width=0.29\textwidth]{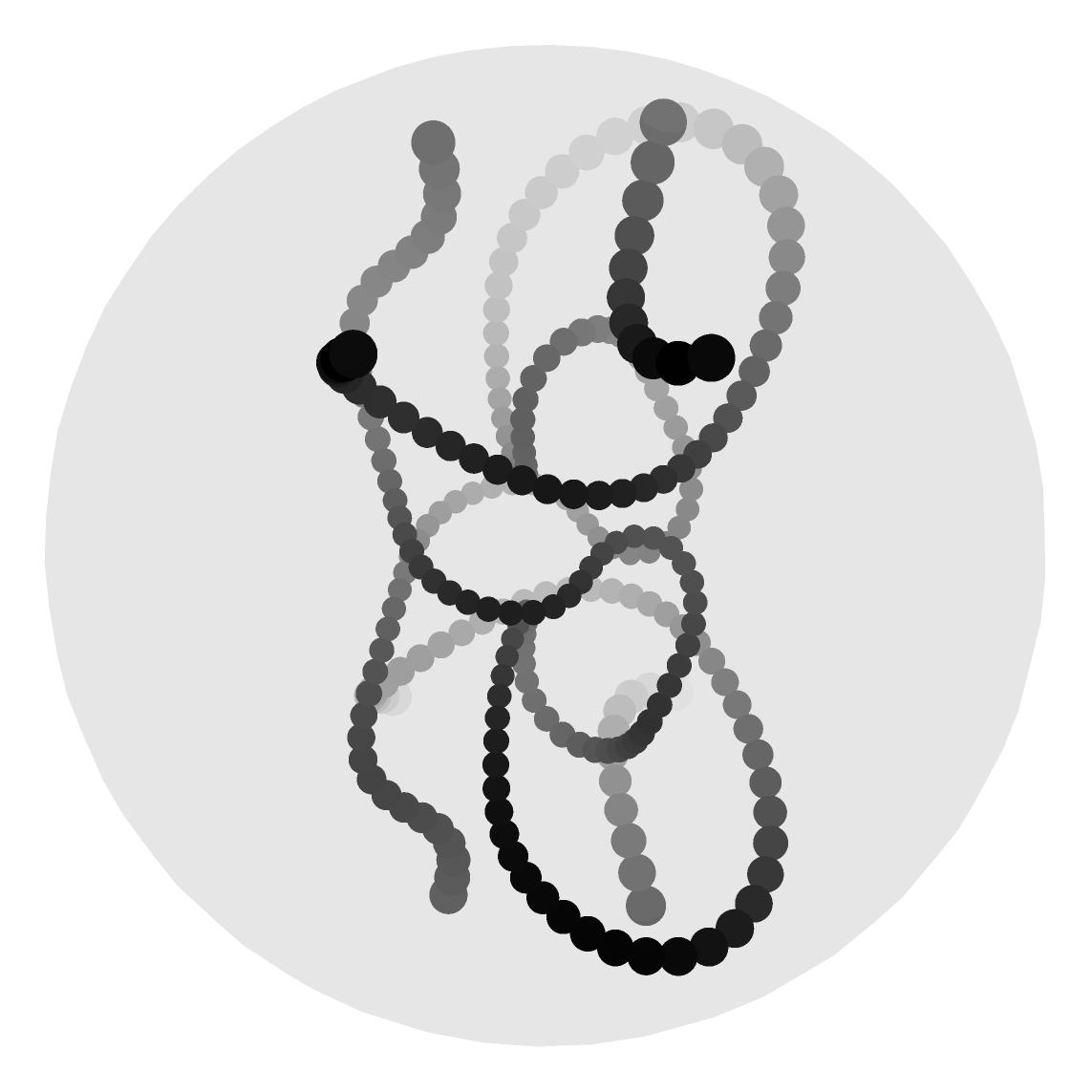} \\ 
		\includegraphics[width=0.29\textwidth]{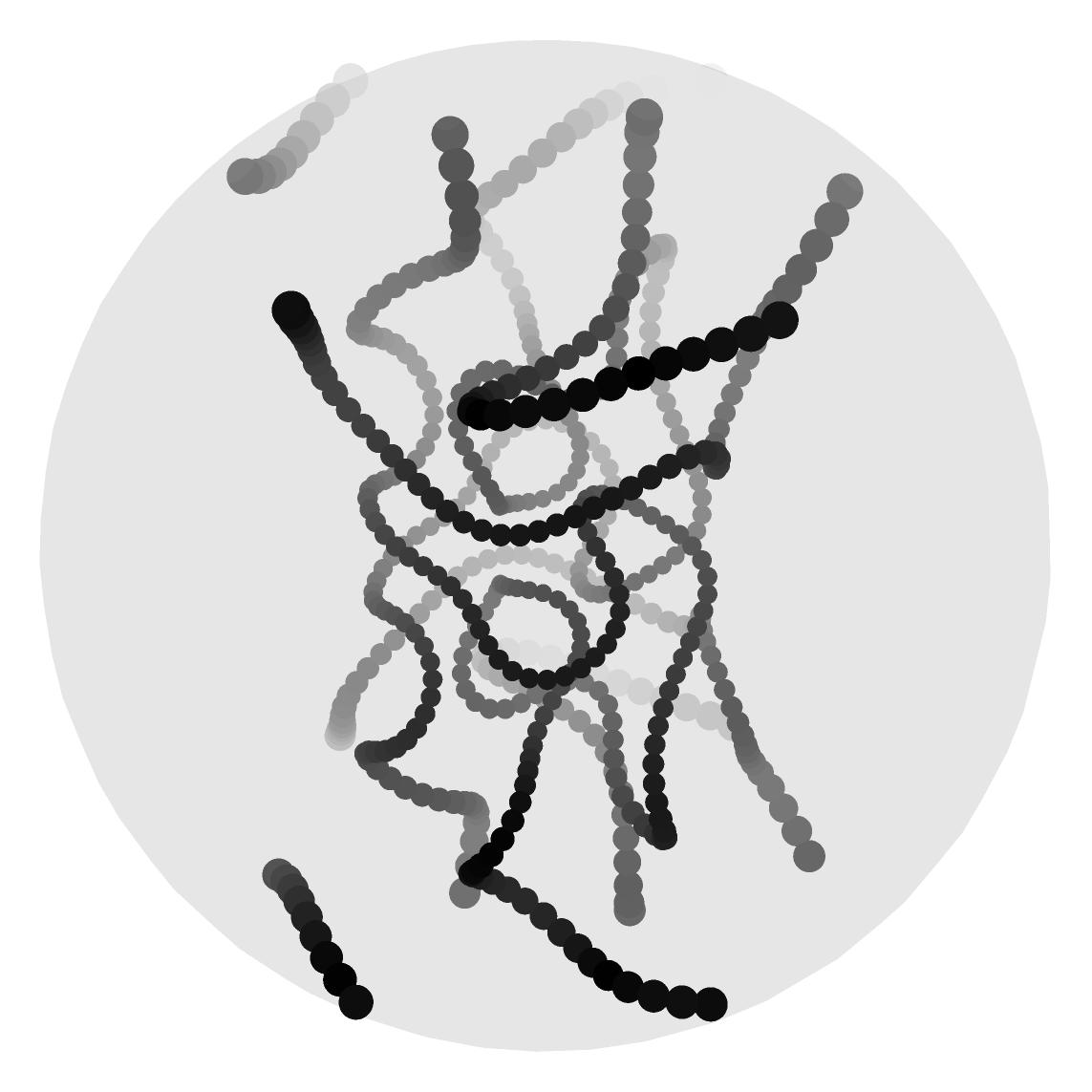} & 
		\includegraphics[width=0.29\textwidth]{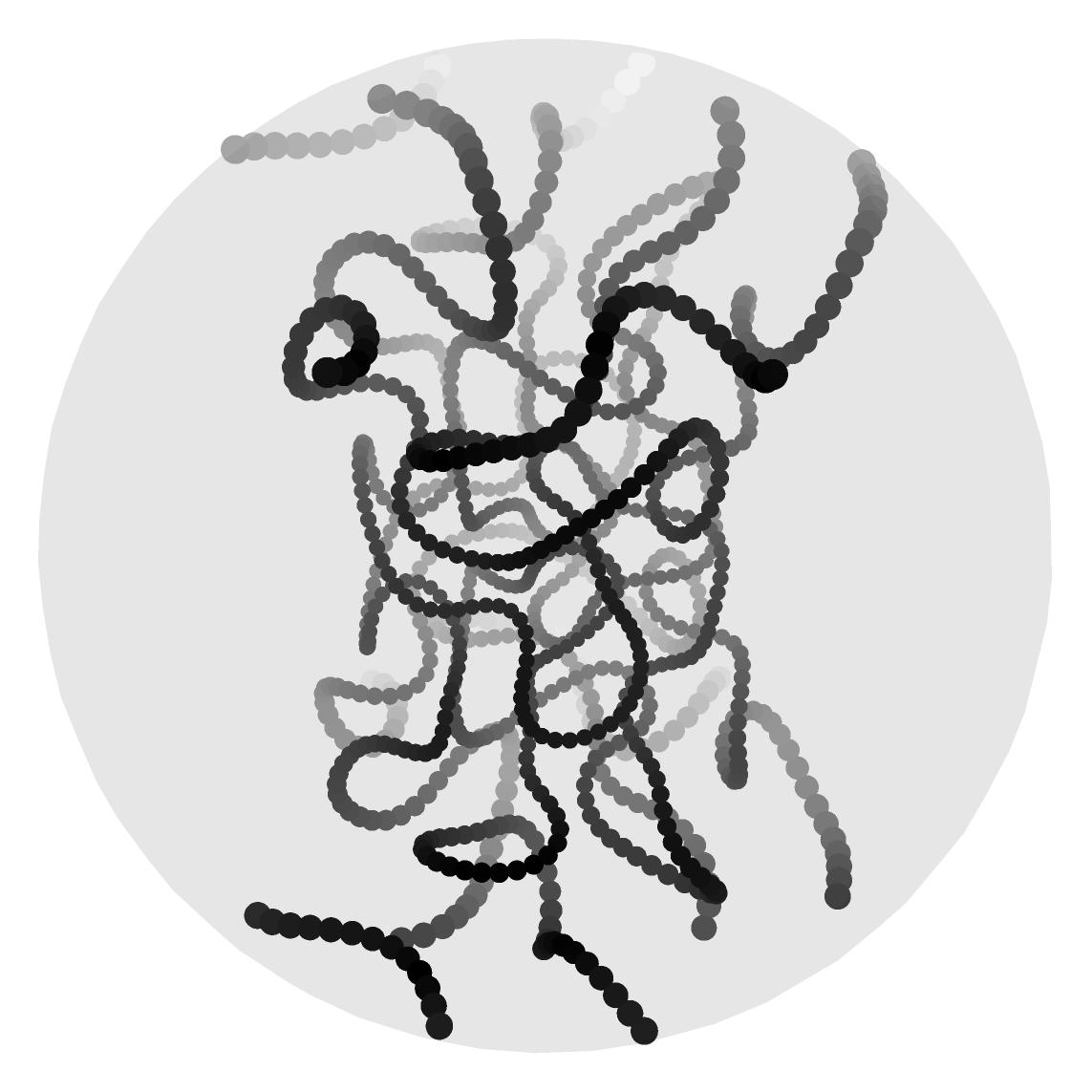} & 
		\includegraphics[width=0.29\textwidth]{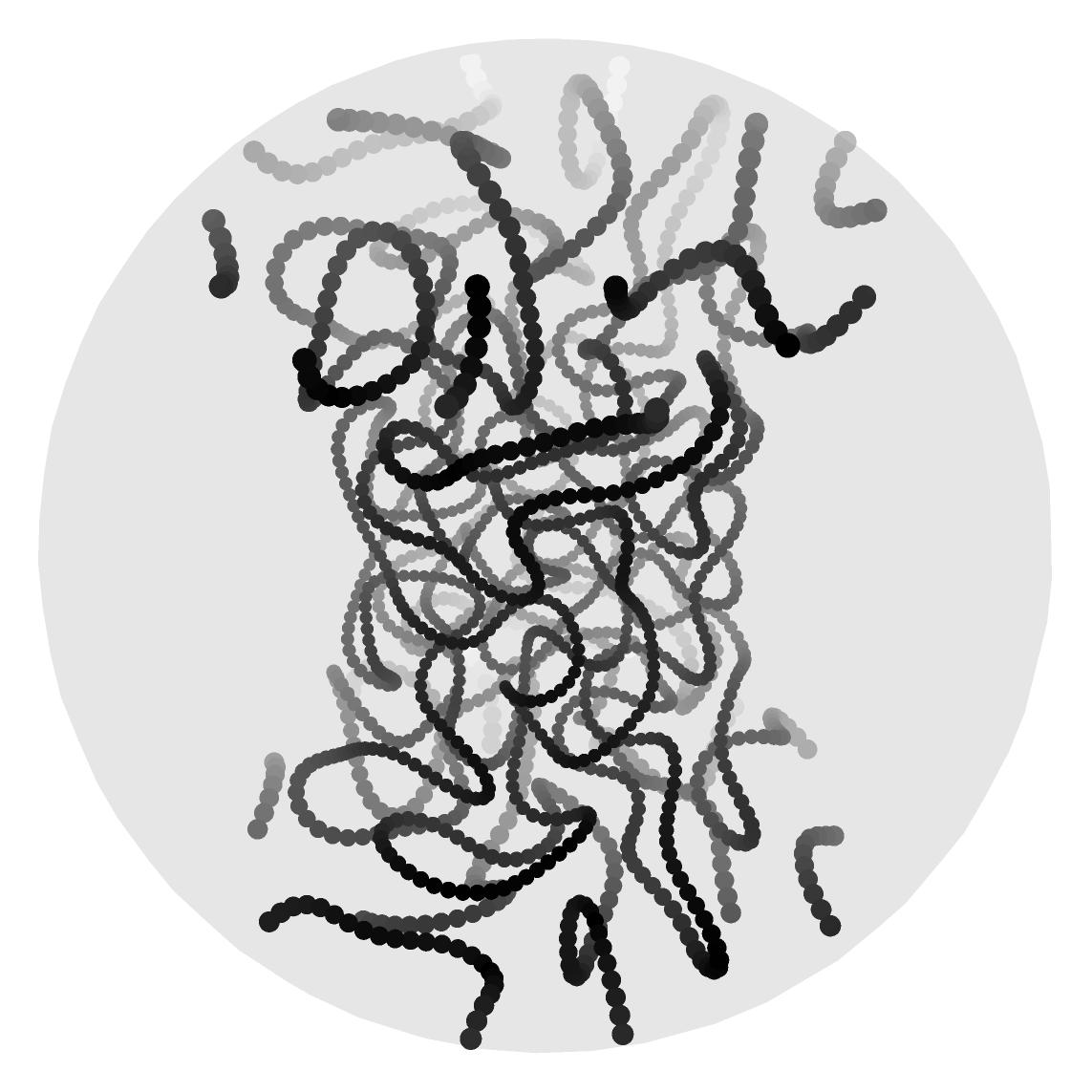} \\ 
		\includegraphics[width=0.29\textwidth]{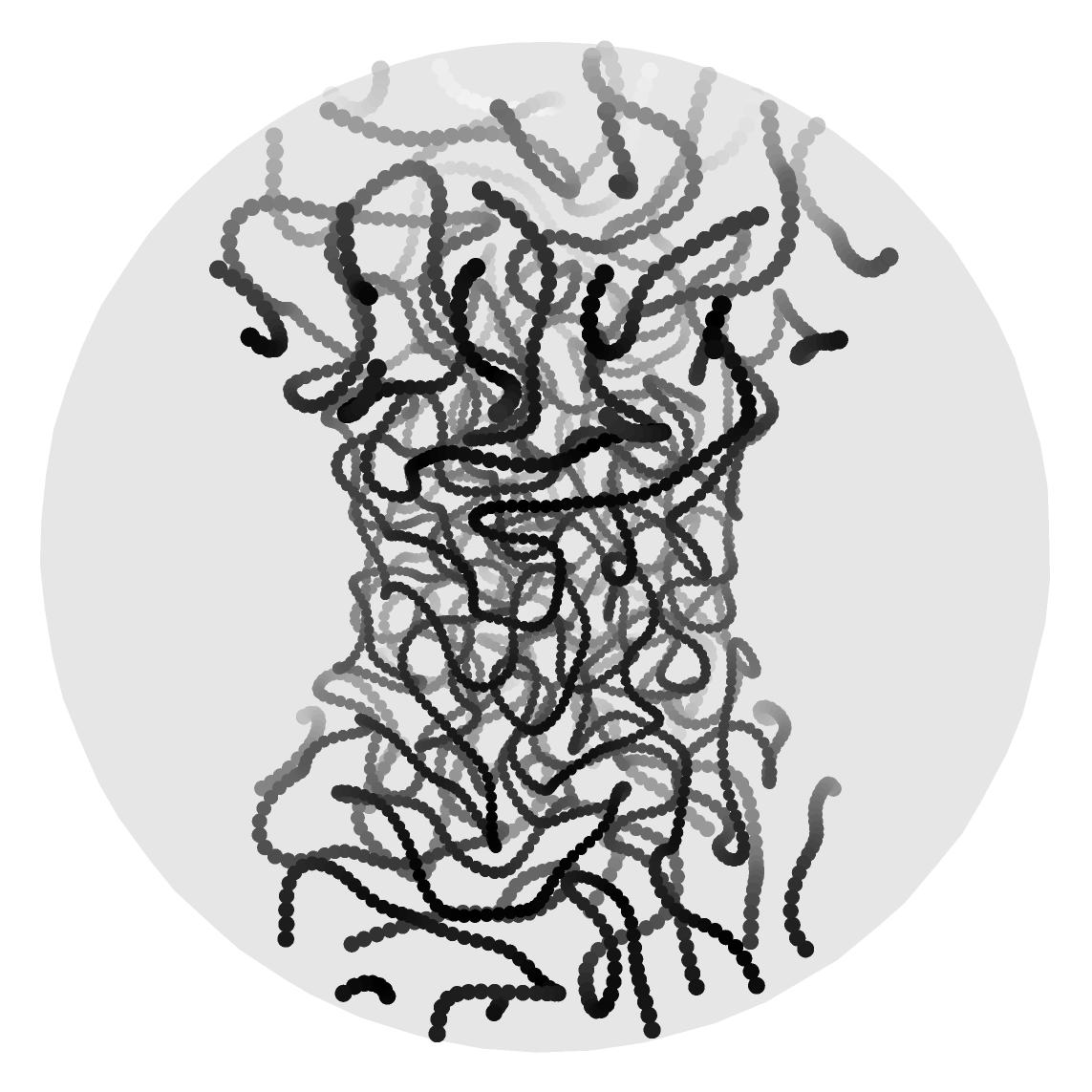} & 
		\includegraphics[width=0.29\textwidth]{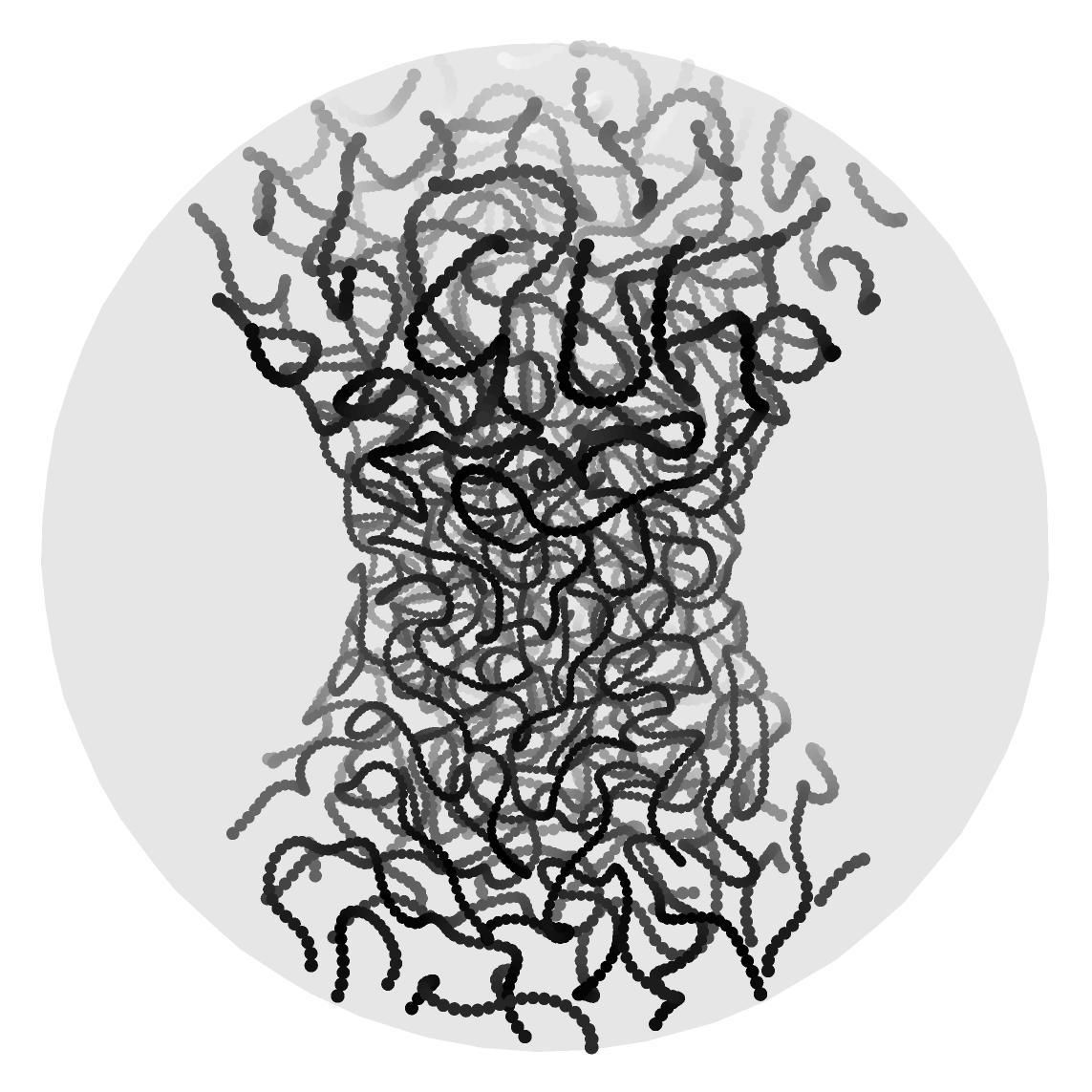} & 
		\includegraphics[width=0.29\textwidth]{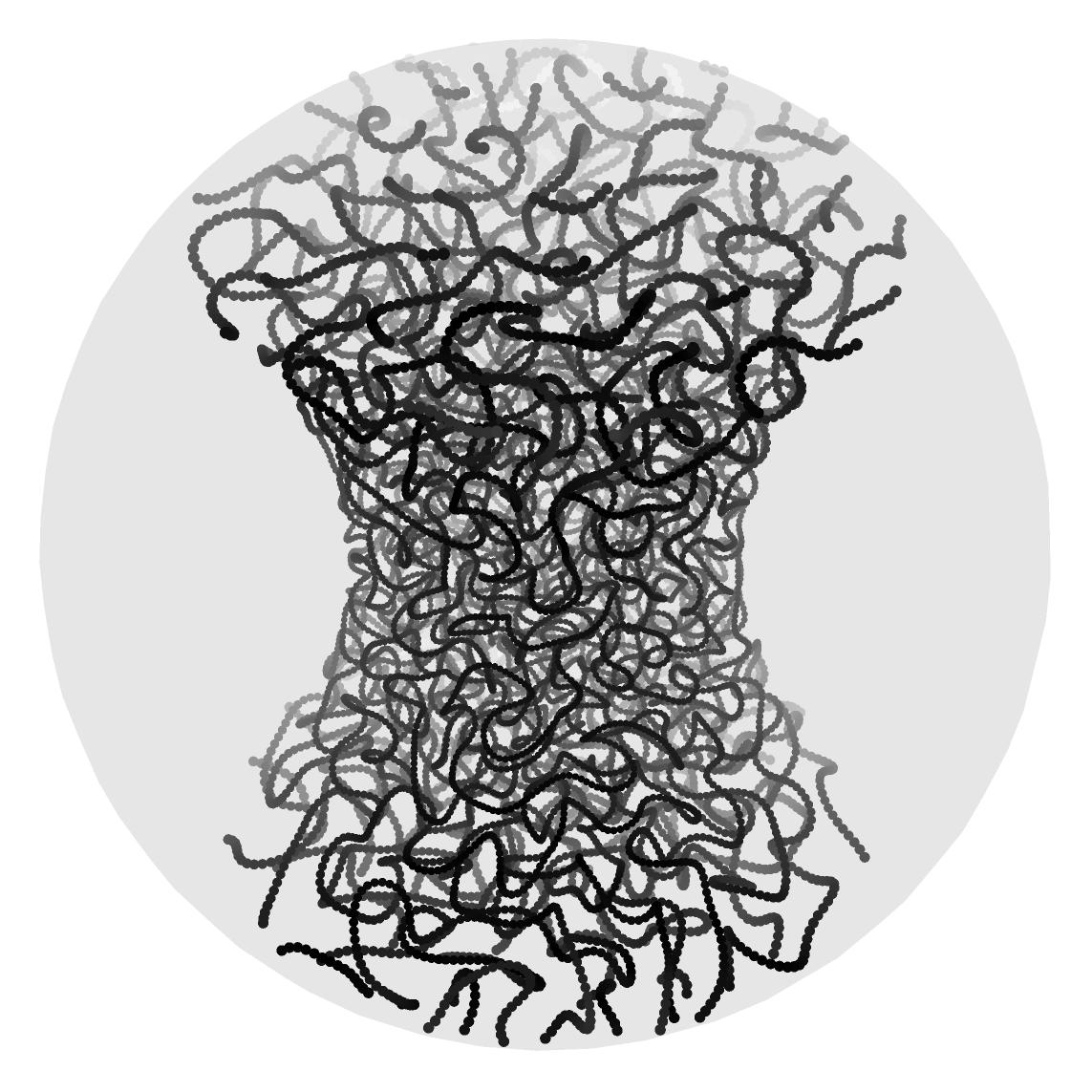} \\ 
	\end{tabular}
	\caption{Local minimizers of \eqref{eq:final} for Haar measure $\mu_{D}$ of three-dimensional doughnut $D$ in the rotation group $\mathrm{SO(3)}$. Color scheme for better visibility of 3D structure.}
	\label{fig:so3_cylinder}
	\vspace{.5cm}
	\centering
	\includegraphics[width=0.73\textwidth]{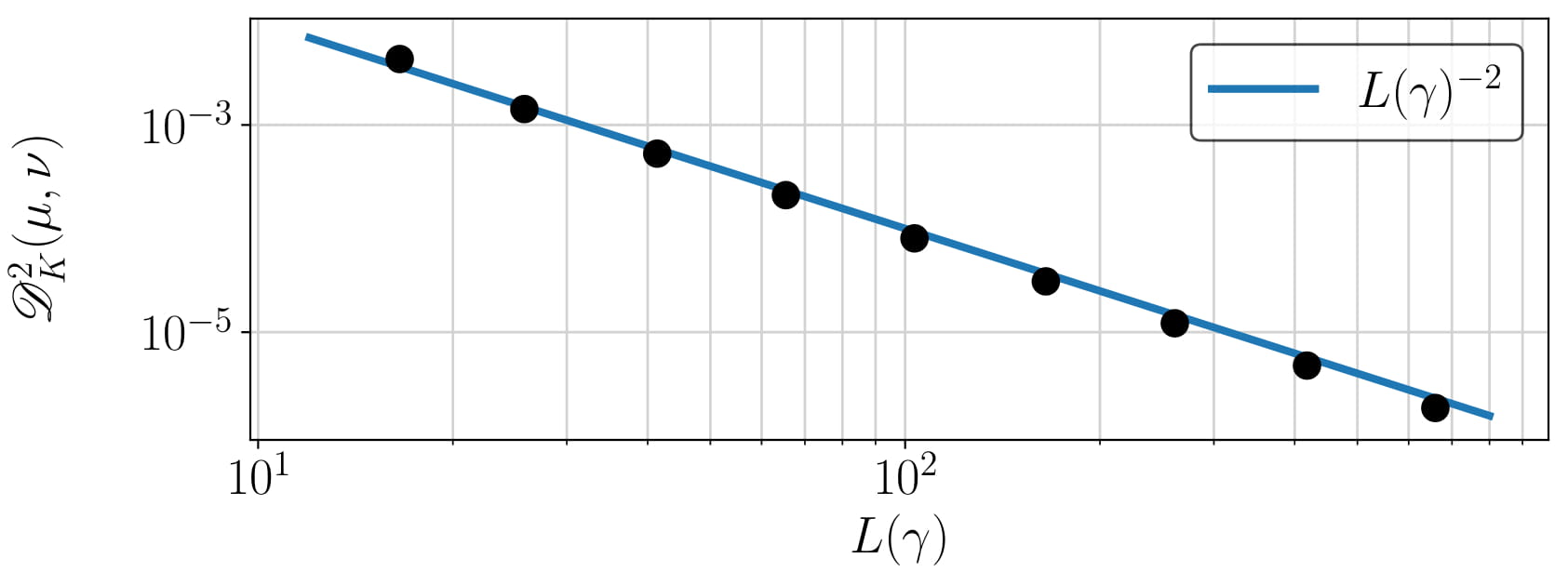}
	\caption{Squared discrepancy between the measure $\mu_D$ and the computed local minimizers (black dots) in log-scale.
		The blue line corresponds to the optimal decay-rate in Corollary \ref{cor:so3}.}
	\label{fig:err_so3_cylinder}
\end{figure}
\smallskip

%---------------------------------------------------------------------
\textbf{The $4$-dimensional Grassmannian $\mathcal{G}_{2,4}$.}
Here, we aim to approximate the Haar measure of the Grassmannian $\mathcal G_{2,4}$ by a curve of almost constant speed.
As this curve samples the space $\mathcal G_{2,4}$ quite evenly, it could be used for the grand tour, a technique to analyze high-dimensional data by their projections onto two-dimensional subspaces, cf.~\cite{Asimov:1985aa}.

%Finally, we approximate the Haar measure of the Grassmannian $\mathcal G_{2,4}$ by a curve of almost constant speed. 
%In this way we should obtain a curve which samples the space of $\mathcal G_{2,4}$ quite evenly, also known as the grand tour, cf.~\cite{Asimov:1985aa}.
%This technique can be used to analyze a given high-dimensional data set from projections onto two-dimensional subspaces. 

%In our test-case the data set will be four-dimensional.
The kernel $K$ of the Haar measure is given by \eqref{kernel:grass} and the Fourier coefficients are given by \smash{$\hat\mu_{m,m'}^{k,k'} = \delta_{m,0}\delta_{m',0}\delta_{k,0}\delta_{k',0}$}.
For $i=0,\dots,8$ the parameters are chosen as
\[
L_{i} = 0.91 \cdot 2^{\frac{3i+16}{4}}, \,\,\, \lambda_{i} = 100\cdot L_{i}^{-\frac{11}{3}}, \,\,\, N_{i} = 128 \cdot 2^{i} \sim L_{i}^{2},\,\,\, r_{i}= \lfloor 2^{\frac{3i+16}{12}} \rfloor +1 \sim L_{i}^{\frac13}.
\]
Here, we use a CG method with 100 iterations and one restart.
Our experiments suggest that step ii) is not necessary.
Note that the complexity for the function evaluation in~\eqref{eq:final} scales roughly as $N \sim L^{3/2}$.

The computed curves are illustrated in Fig.~\ref{fig:uniform_g24}, where we use the following visualization.
By Remark \ref{rem:Grassi},
there exists an isometric one-to-one mapping 
$P \colon \S^2\times\S^2/ \{\pm 1\} \to \G_{2,4}$.
Using this relation, we plot the point $P(u,v) \in \mathcal G_{2,4}$ 
by two antipodal points $z_{1}=u+v,\, z_{2}=-u-v \in \mathbb R^{3}$ together with the RGB color-coded vectors $\pm u$.\footnote{
	Note that the decomposition of $z \in \mathbb R^3$ with  $0 < \|z\| < 2$ into $u$ and $v$ is not unique.
	There is a one-parameter family of points $u_{s},v_{s} \in \mathbb S^{2}$ such $z=u_{s} + v_{s}$. 
	The point $z=0$ has a two-dimensional ambiguity $v=-u$, $u \in \mathbb S^{2}$ and the point $z \in 2\mathbb S^{2}$ has a unique pre-image $v=u=\tfrac 12 z$. }
More precisely, $R=(1\mp u_{1})/2$, $G=(1\mp u_{2})/2$, $B=(1\mp u_{3})/2$.
This  means a curve $\gamma(t) \in \mathcal G_{2,4}$ 
only intersects itself if the corresponding curve $z(t) \in \mathbb R^3$ intersects 
and has the same colors at the intersection point.
In Fig.~\ref{fig:err_uniform_g24} we observe that the decay-rate of the squared discrepancy $\mathscr D_K^{2}(\mu,\nu)$ 
in dependence on the Lipschitz constant $L$ matches indeed the theoretical findings in Theorem~\ref{thm:grassi}.
\begin{figure}
	\centering
	\begin{tabular}{ccc}
		\includegraphics[width=0.29\textwidth]{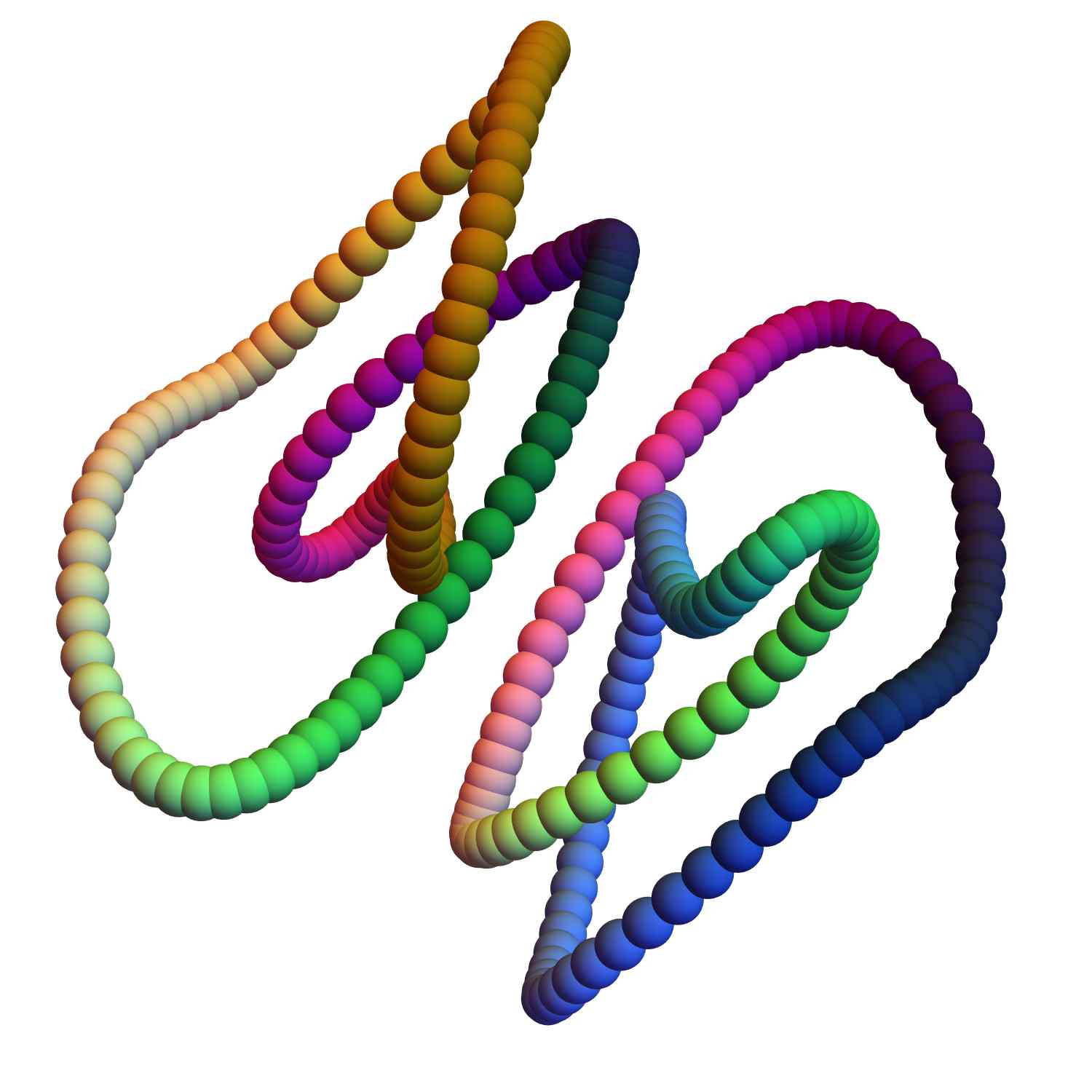} & 
		\includegraphics[width=0.29\textwidth]{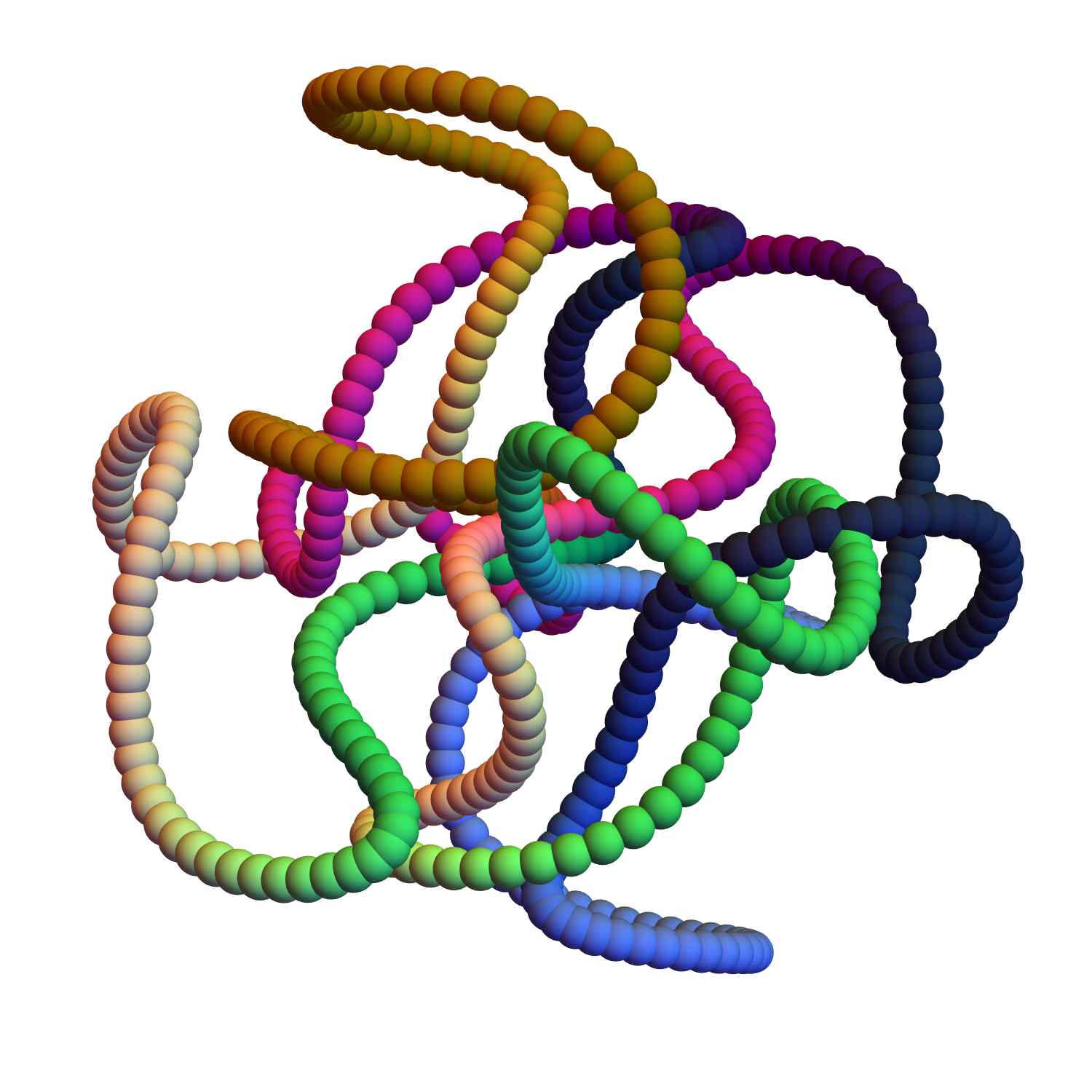} & 
		\includegraphics[width=0.29\textwidth]{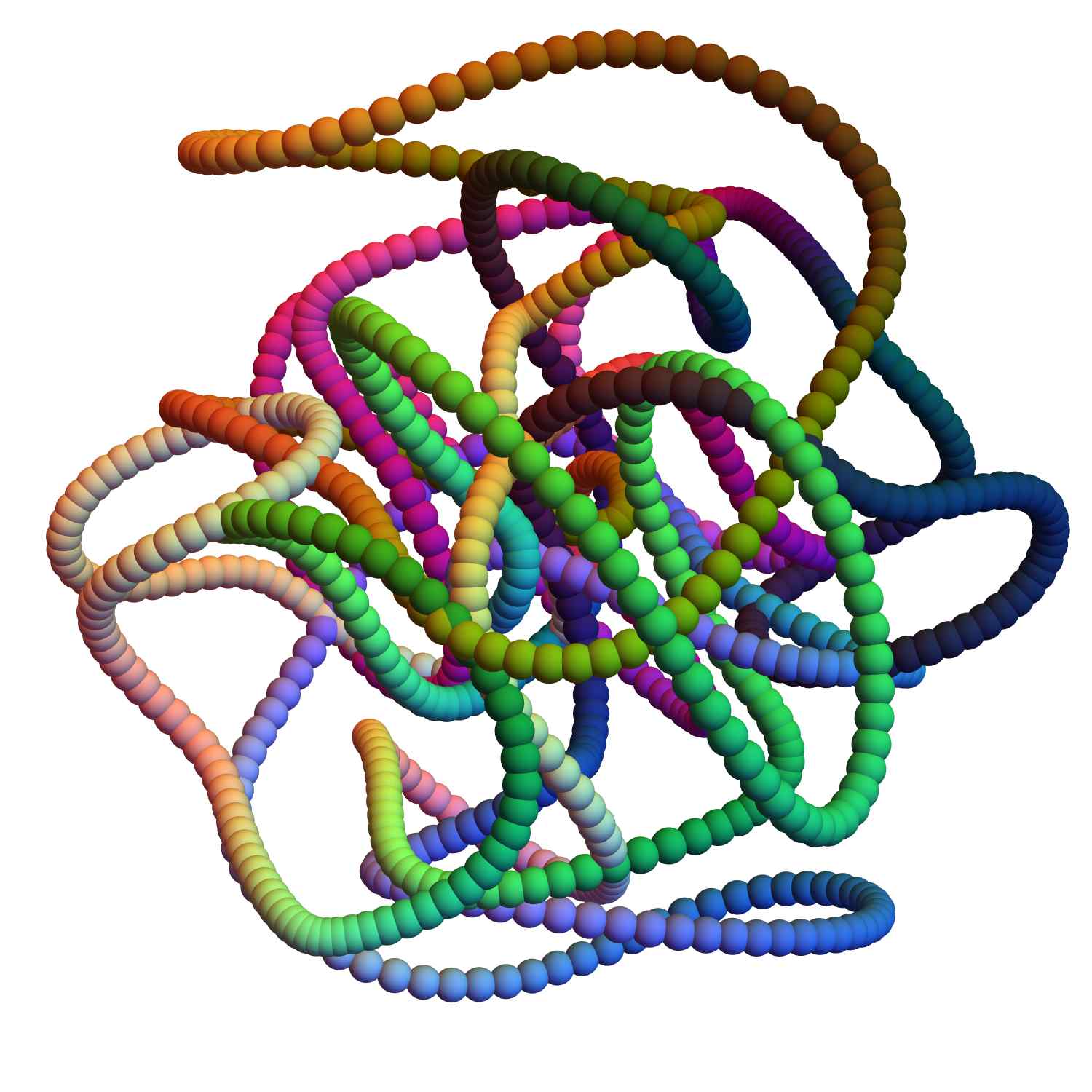} \\ 
		\includegraphics[width=0.29\textwidth]{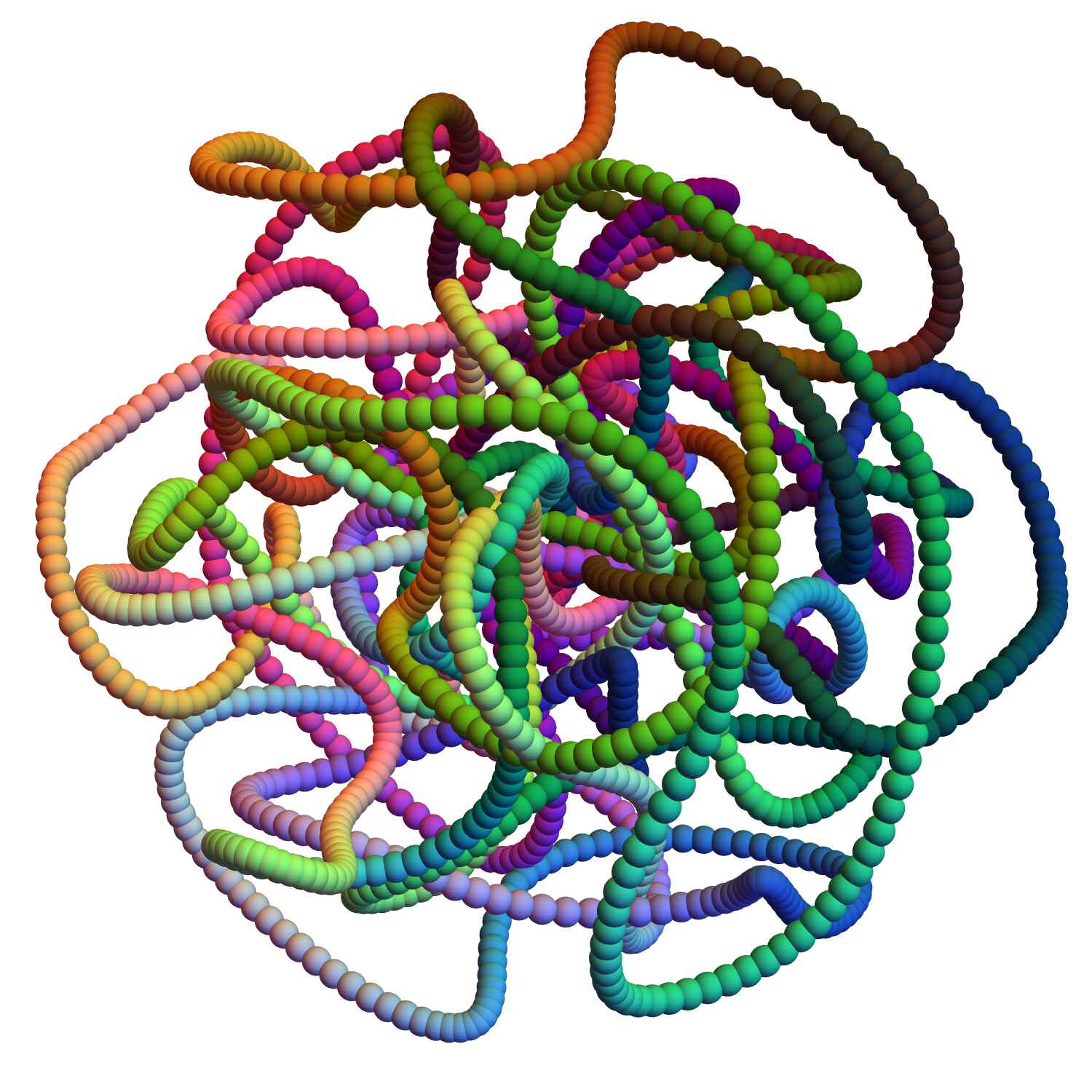} & 
		\includegraphics[width=0.29\textwidth]{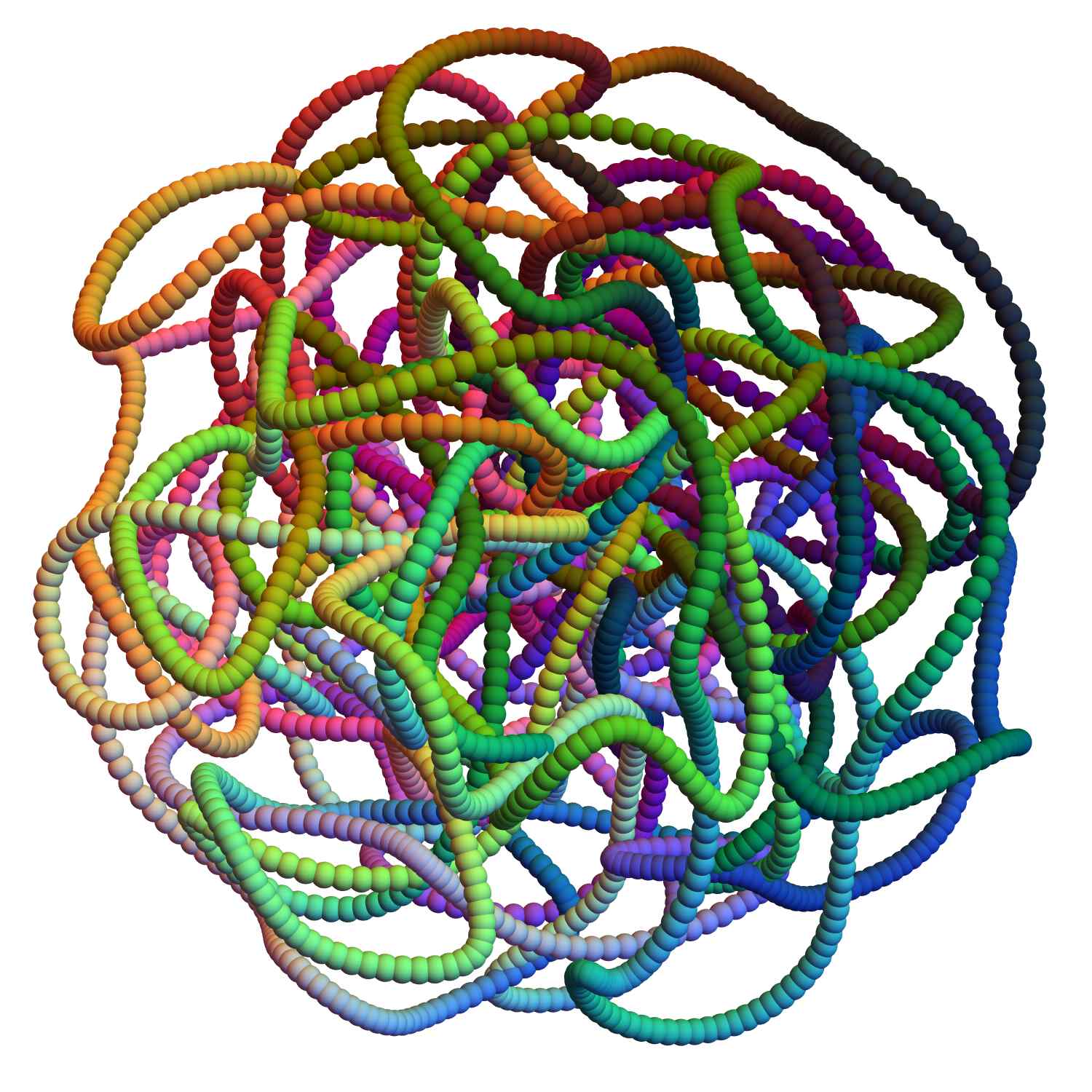} & 
		\includegraphics[width=0.29\textwidth]{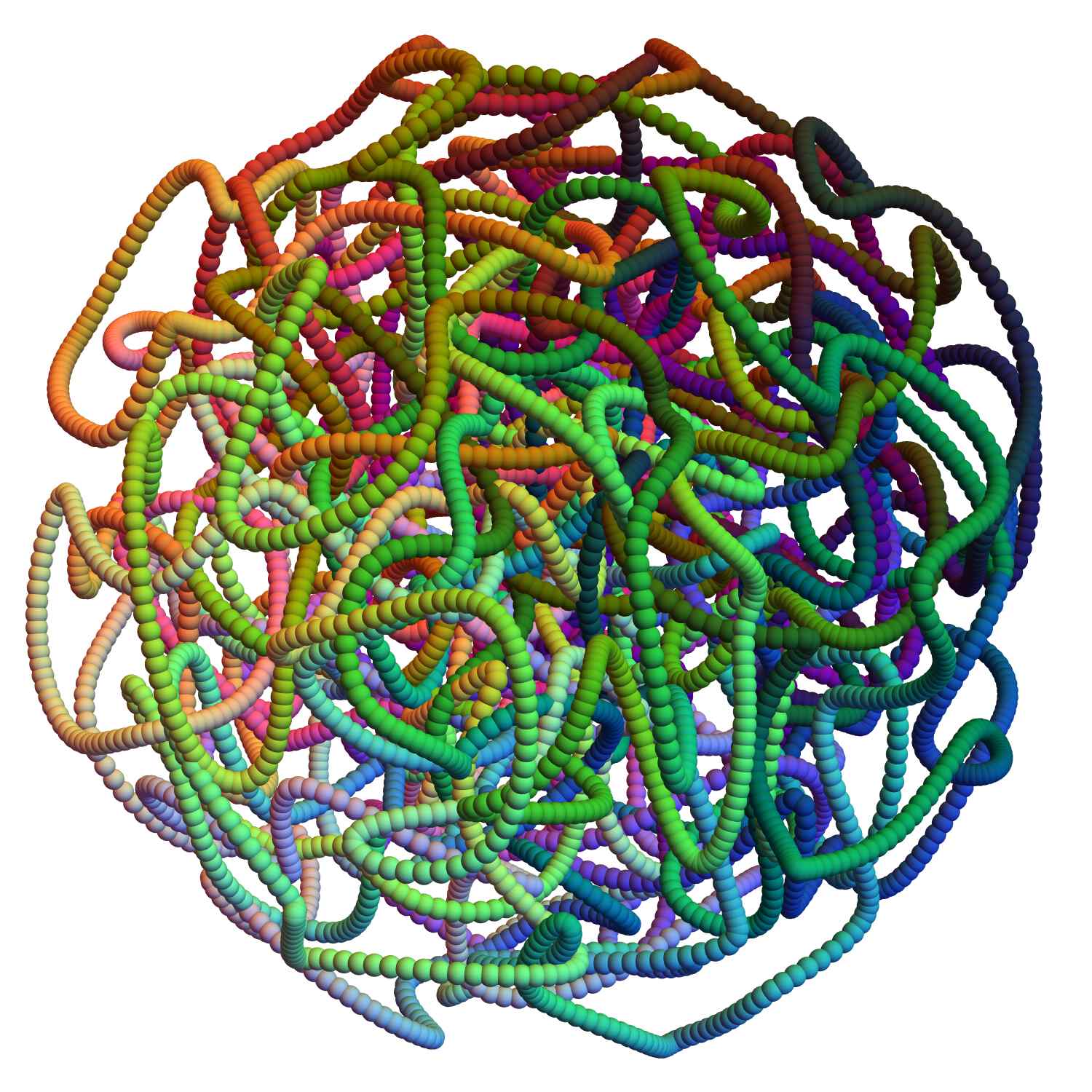} \\ 
		\includegraphics[width=0.29\textwidth]{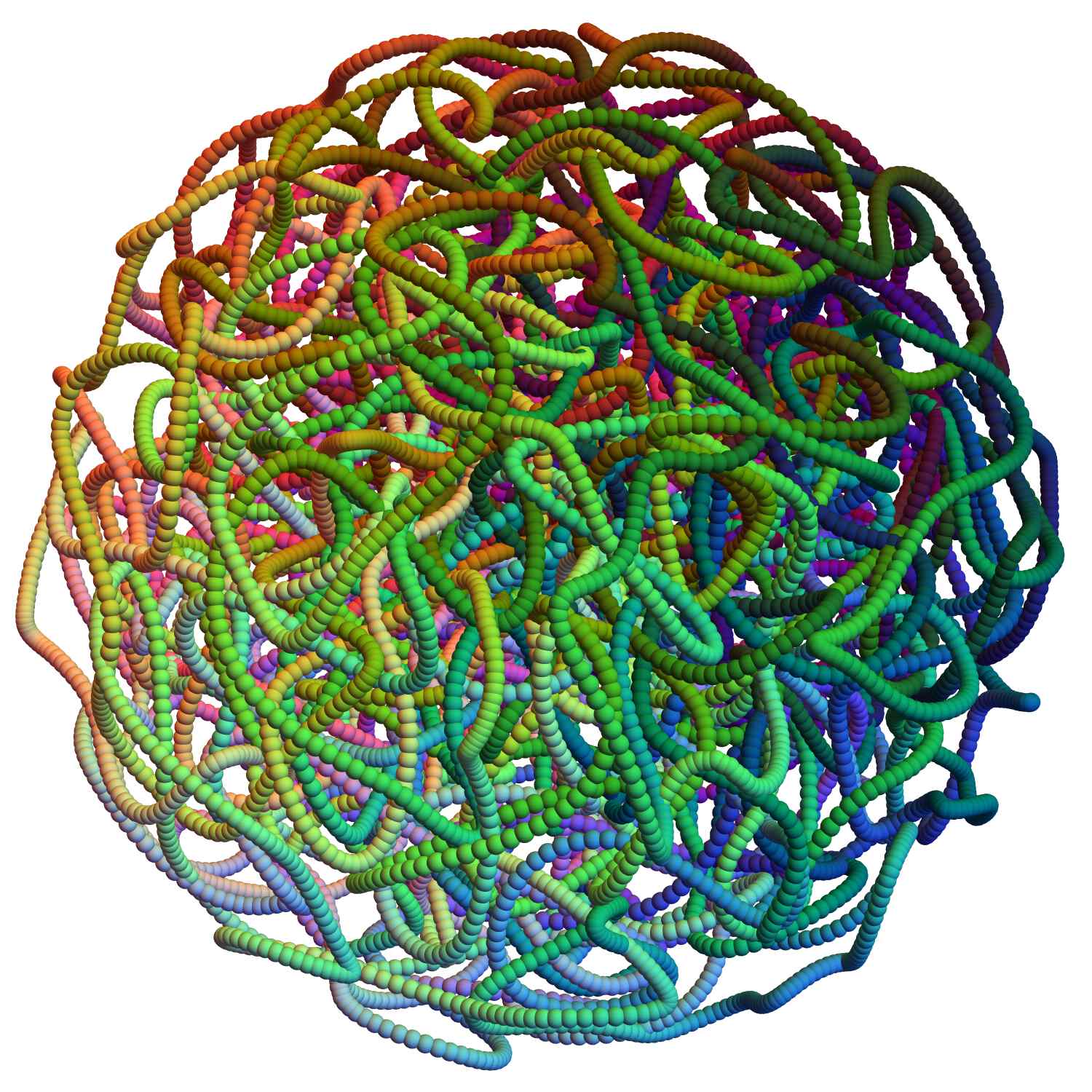} & 
		\includegraphics[width=0.29\textwidth]{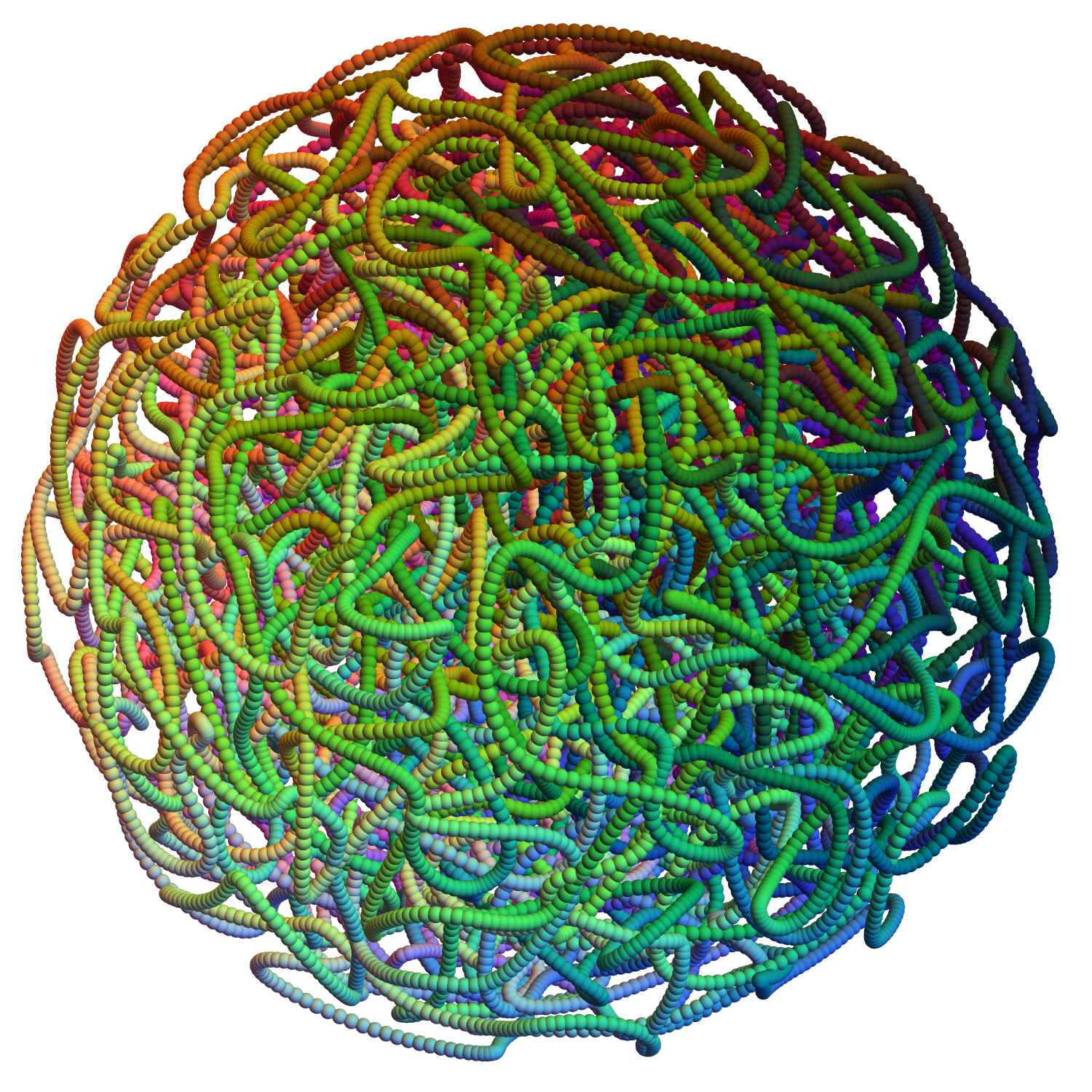} & 
		\includegraphics[width=0.29\textwidth]{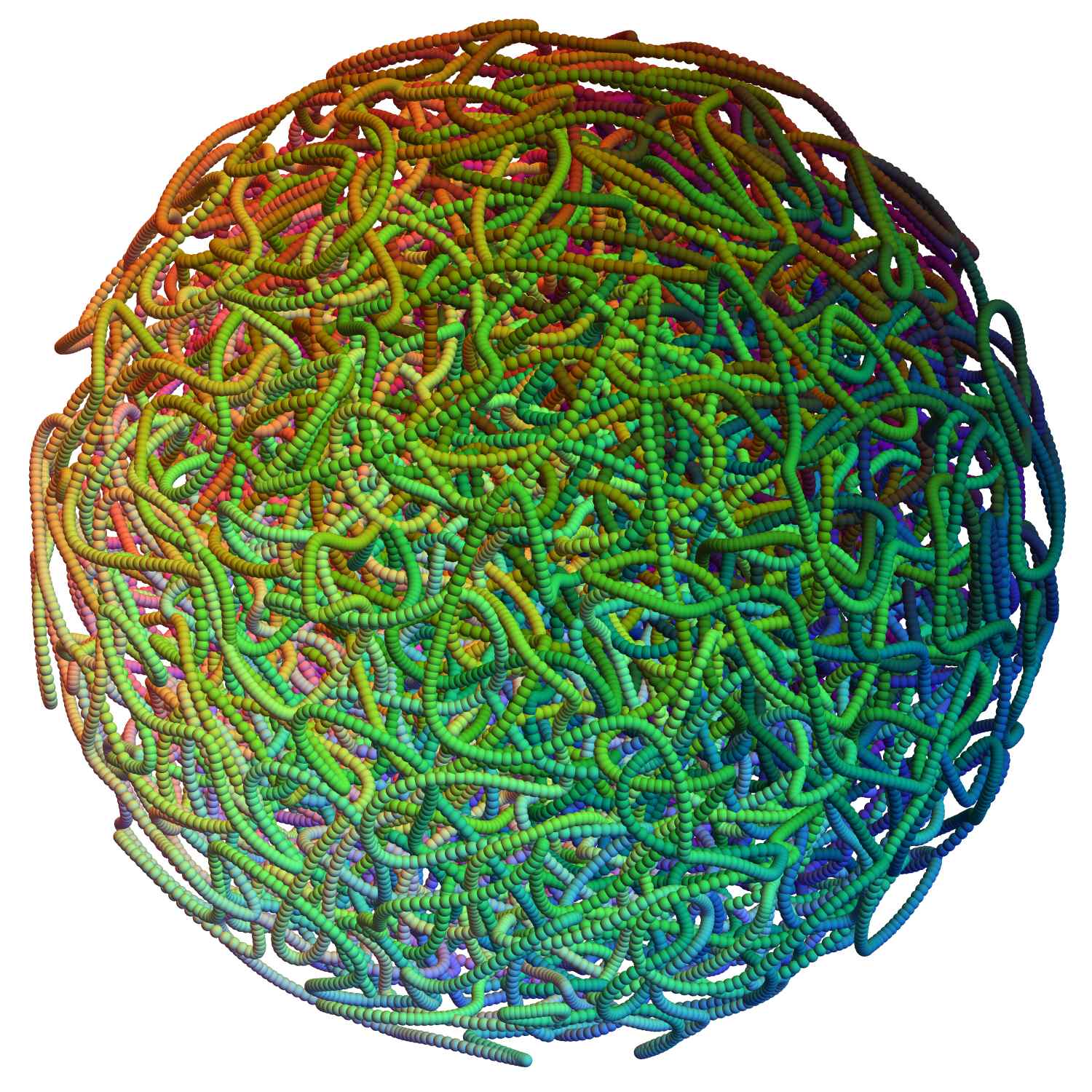} \\ 
	\end{tabular}
	\caption{Local minimizers of \eqref{eq:final} for the Haar measure of the Grassmannian $\mathcal G_{2,4}$.}
	\label{fig:uniform_g24}
	\vspace{.75cm}
	\centering
	\includegraphics[width=0.73\textwidth]{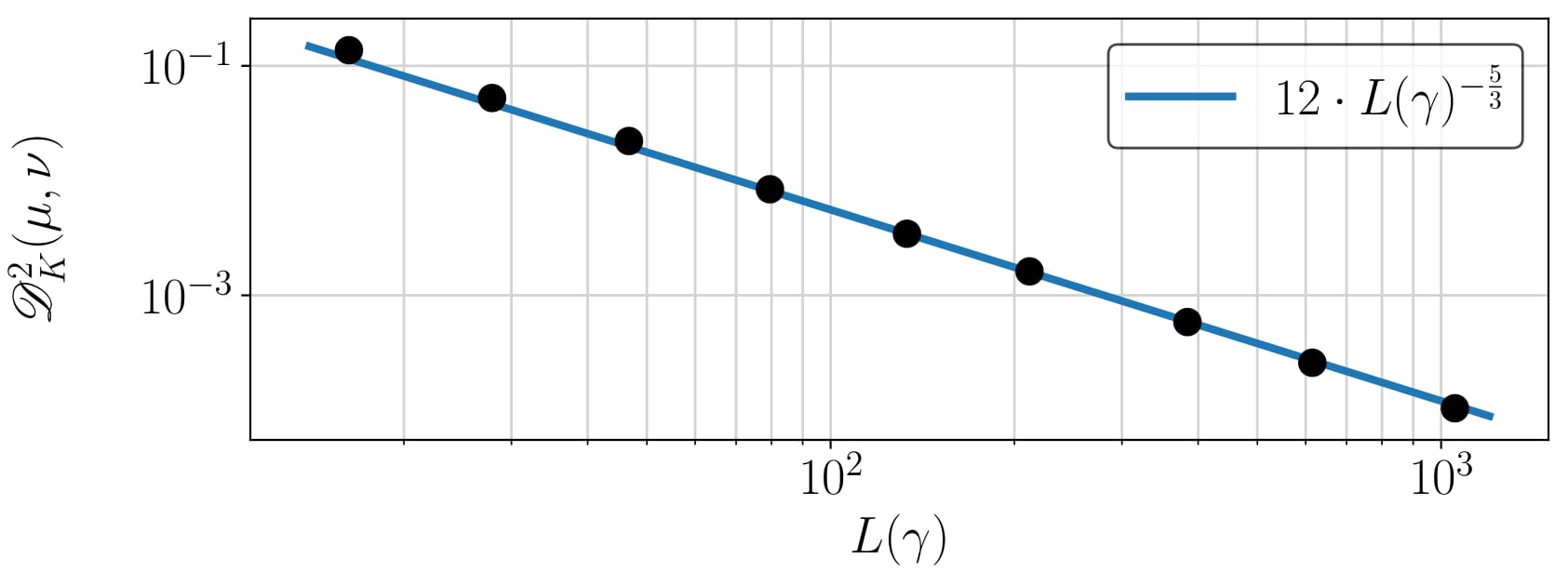}
	\caption{The squared discrepancy between the Haar measure $\mu$ and the computed local minimizers (black dots) in log-scale.
		Here, the blue line corresponds to the optimal decay-rate, cf.~Theorem~\ref{thm:grassi}.}
	\label{fig:err_uniform_g24}
	
\end{figure}

%----------------------------------------------
\section{Conclusions} \label{sec:conclusions_dith}
%----------------------------------------------

In this chapter, we provided approximation results
for general probability measures on compact Ahlfors $d$-regular metric spaces $\X$ by 
\begin{itemize}
	\item[i)]
	measures supported on continuous curves of finite length,
	which are actually push-forward measures of probability measures on $[0,1]$ by Lipschitz curves;
	\item[ii)] push-forward measures of absolutely continuous probability measures on  $[0,1]$ by Lipschitz curves;
	\item[iii)] push-forward measures of the Lebesgue measure on $[0,1]$ by Lipschitz curves.
\end{itemize}
Our estimates rely on discrepancies between measures.
In contrast to Wasserstein distances, these estimates do not reflect the curse of dimensionality. 

In approximation theory, a natural question is how the approximation rates improve as the ``measures become smoother''.
Therefore, we considered absolutely continuous probability measures with densities in
Sobolev spaces, where we have to restrict ourselves to compact Riemannian manifolds $\X$.
We proved lower estimates for all three approximation spaces i)-iii).
Concerning upper estimates, we gave a result for the approximation space i).
Unfortunately, we were not able to show  similar results for the smaller approximation spaces ii) and iii).
Nevertheless, for these cases, we could provide results for the $d$-dimensional torus, the $d$-sphere, the three-dimensional
rotation group and the Grassmannian $\mathcal{G}_{2,4}$, which are all of interest on their own.
Numerical examples on these manifolds underline our theoretical findings.

Our results can be seen as starting point for future research.
Clearly, we want to have more general results also for the approximation spaces ii) and iii).
We hope that our research leads to further practical applications.
It would be also interesting to consider approximation spaces of measures supported on higher
dimensional submanifolds as, e.g., surfaces.

Recently, results on the principal component analysis (PCA) on manifolds were obtained.
It may be interesting to see if some of our approximation results can be also
modified for the setting of principal curves, cf.~Remark~\ref{rem:principalcurves}. In contrast to \cite[Thm.~1]{KKLZ00} that bounds the discretization error for fixed length, we were able to provide precise error bounds for the discrepancy in dependence on the Lipschitz constant $L$ of $\gamma$ and the smoothness of the density $\mathrm d \mu$.

%--------------------------------------------------------------
\appendix
	%----------------------------------------------
	\section{Special manifolds} \label{sec:examples}
	%----------------------------------------------
	Here, we introduce the main examples that are addressed in the numerical part. 
	The measure $\sigma_\X$ is always the normalized Riemannian measure on the manifold $\X$.
	Note that for simplicity of notation all eigenspaces are complex in this section.
	We are interested in the following special manifolds.
	
	\paragraph{Example 1: $\X = \mathbb T^d$.} For $\zb k \in \mathbb Z^d$, set $|\zb k|^2 \coloneqq k_1^2+ \ldots+k_d^2$ and
	$|\zb k|_\infty  \coloneqq \max\{|k_1|, \ldots, |k_d|\}$. 
	Then $-\Delta$ has
	eigenvalues $\{4 \pi^2|\zb k|^2\}_{\zb k \in \mathbb Z^d}$%(counted with multiplicities)
	with eigenfunctions
	$\{ \e^{2 \pi \ii \langle \zb k, \cdot\rangle} \}_{\zb k \in \mathbb Z^d}$.
	The \emph{space of $d$-variate trigonometric polynomials of degree $r$},
	\begin{equation} \label{trig_torus}
		\mathrm{\Pi}_r(\mathbb T^d) \coloneqq \mathrm{span} \big\{\mathrm{e}^{2\pi \mathrm i \langle \zb k,x \rangle}: |\zb k|_\infty \le r \big\}
	\end{equation}
	has dimension $(2r+1)^d$ and contains the eigenspaces belonging to eigenvalues smaller than $4\pi^2 r^2$.
	As kernel for $H^s$, $s =(d+1)/2$, we use in our numerical examples
	\begin{equation}\label{kernel_Td}
		K(x,y) = \sum_{\zb k\in \mathbb Z ^d} (1+ |\zb k|_2^2)^{-\frac{d+1}{2}} \mathrm{e}^{2\pi \mathrm i \langle \zb k,x-y \rangle} 
		= \sum_{\zb k\in \mathbb Z ^d} (1+ |\zb k|_2^2)^{-\frac{d+1}{2}} \cos\bigl(2\pi \langle \zb k,x-y \rangle\bigr).
	\end{equation}
	
	\paragraph{Example 2: $\X = \mathbb S^d \subset \mathbb R^{d+1}$, $d \ge 1$.} 
	We use distance
	$\dist_{\mathbb S^d}(x,z)=\arccos(\left\langle x,z \right\rangle)$.
	The Laplace--Beltrami operator $-\Delta$ on $\mathbb S^d$ has the
	eigenvalues $\{k(k+d-1)\}_{k\in \N}$ 
	% \textcolor{blue}{Previous example mentions multiplicities, this one not. Why?} 
	with 
	the \emph{spherical harmonics of degree} $k$,
	$$
	\big\{Y^k_l\colon l=1, \ldots, Z(d, k)\big\}, 
	\quad Z(d, k) \coloneqq (2k+d-1)\tfrac{\Gamma(k+d-1)}{\Gamma(d)\Gamma(k+1)}
	$$
	as corresponding orthonormal eigenfunctions \cite{Mueller1966}. The span of eigenfunctions with eigenvalues smaller than $r(r+d-1)$ is
	given by
	\begin{align} \label{trig_sphere}
		\mathrm{\Pi}_r(\mathbb S^d) 
		&\coloneqq  \mathrm{span} \big\{Y^k_l\colon k=0,\ldots,r,\, l=1, \ldots, Z(d, k) \big\}.
		%&=  \mathrm{span} \big\{\prod\limits_{i=1}^{d+1} x_i^{k_i}: k_i \in \mathbb N_0,\, k_1+\ldots+k_{d+1} \le r \big\}
	\end{align}
	It has dimension 
	$\sum_{k=0}^r Z(d,k) = \frac{(d+2r)\Gamma(d+r)}{\Gamma(d+1)\Gamma(r+1)} \sim r^{d}$ and coincides with the space of polynomials of total degree $r$ in $d$ variables restricted to the sphere. 
	As kernel for
	$H^s(\mathbb S^2)$, $s = 3/2$, we use
	\begin{align}\label{kernel_Sd}
		K(x,y) 
		&= \frac13 + \sum_{k=1}^\infty \frac{2}{(2k-1)(2k+1)(2k+3)} \sum_{l=1}^{2k+1} Y^k_l (x) \overline{Y^k_l (y)}\\
		&= \frac13 + \sum_{k=1}^\infty \frac{2}{(2k-1)(2k+3)} P_k\bigl(\langle x,y \rangle\bigr) = 1 - \frac12\|x-y\|_{2}  
	\end{align}
	with the Legendre polynomials $P_k$. Note that the coefficients decay as $\left(k(k+1) \right)^{-3/2}$.
	
	\paragraph{Example 3: $\X = \SO(3)$.} This $3$-dimensional manifold is equipped with the distance $\dist_{\SO(3)}(x,y) = \arccos((\operatorname{trace}(x^\tT y)-1)/2)/2$. 
	The eigenvalues of $-\Delta$ are $\{k(k+1)\}_{k=0}^{\infty}$
	and the (normalized) \emph{Wigner}-$\mathcal{D}$ \emph{functions}
	$\{\mathcal{D}^k_{l,l'}\colon l, l' = -k, \ldots, k\}$ provide an orthonormal basis for $L^2(\textup{SO}(3))$, cf.~\cite{VMK1988}.
	The span of eigenspaces belonging to eigenvalues smaller than $r(r+1)$ is 
	\[
	\mathrm{\Pi}_r(\SO(3)) 
	\coloneqq  \mathrm{span} \bigl\{\mathcal D_{l,l'}^k: k=0,\ldots,r,\, l,l' = -k,\ldots,k\bigr\}
	\]
	and has dimension $(r+1)(2r+1)(2r+3)/3$.
	In the numerical part, we use the following kernel for $H^s\left( \textup{SO}(3) \right)$, $s = 2$,
	\begin{align}\label{kernel_SO3}
		K(x,y) & = \frac{\pi}{8} - \frac13 +
		\sum_{k=1}^\infty \frac{1}{(2k-1)(2k+1)^2(2k+3)} 
		\sum_{l = -k}^k \sum_{l'=-k}^k \mathcal D^k_{l,l'} (x) \overline{\mathcal D^k_{l,l'} (y)}\hspace{1.2cm}\\
		& = \frac{\pi}{8} - \frac13 + 
		\sum_{k=1}^\infty \frac{1}{(2k-1)(2k+1)(2k+3)} U_{2k}\Big(\tfrac12\sqrt{\mathrm{tr}(x^\top y)+1}\Big) \\
		& = \frac{\pi}{8} - \pi \frac{\sqrt{2}}{16} \| x - y\|_{\mathrm F},
	\end{align}
	where $U_k$ are the Chebyshev polynomials of the second kind.
	
	\paragraph{Example 4: $\X = \G_{2,4}$.}
	For integers $1\leq s<r$, the $(s,r)$-Grassmannian is the collection of all $s$-dimensional linear subspaces of $\R^r$ and carries the structure of a closed Riemannian manifold. 
	By identifying a subspace with the orthogonal projector onto this subspace, the Grassmannian becomes 
	\begin{equation*}
		\G_{s,r}\coloneqq\bigl\{x\in\R^{r\times r} : x^\top = x,\; x^2=x,\; \mathrm{rank}(x)=s\bigr\}.
	\end{equation*}
	In our context, the cases $\G_{1,2}$, $\G_{1,3}$, and $\G_{2,3}$ can essentially be treated by the spheres $\S^1$ and $\S^2$. 
	The simplest Grassmannian that is algebraically different is $\G_{2,4}$.
	It is a $4$-dimensional manifold and the geodesic distance between $x,y\in\mathcal{G}_{2,4}$ is given by 
	\begin{equation*}
		\dist_{\G_{2,4}}(x,y)=\sqrt{2}\sqrt{\theta_1^2(x,y)+\theta_2^2(x,y)},
	\end{equation*}
	where $\theta_1(x,y)$ and $\theta_2(x,y)$ are the principal angles between the subspaces associated to $x$ and $y$, respectively.
	The terms $\cos(\theta_1(x,y))^{2}$ and $\cos(\theta_2(x,y))^{2}$ correspond to the two largest singular values of the product $xy$. 
	The eigenvalues of $-\Delta$ on $\G_{2,4}$ are $4(\lambda_1^2+\lambda_2^2+\lambda_1)$, where $\lambda_1$ and $\lambda_2$ 
	run through all integers with $\lambda_1 \geq \lambda_2\geq 0$, cf.~\cite{Bachoc:2006aa,Bachoc:2004fk,Bachoc:2002aa,Ehler:2014zl,James:1974aa,Roy:2010fk}. 
	The associated  eigenfunctions are denoted by $\varphi^\lambda_l$ with $l=1,\ldots,Z(\lambda)$, where 
	$Z(\lambda) = (1+\lambda_1+\lambda_2) \eta(\lambda_2)$ and $\eta(\lambda_2) = 1$ if $\lambda_{2}=0$ and $2$ if $\lambda_{2}>0$ 
	cf.~\cite[(24.29) and (24.41)]{Fulton:1991fk} as well as \cite{Bachoc:2004fk,Bachoc:2002aa}. 
	
	The space of polynomials of total degree $r$ on $\R^{16}\cong \R^{4\times 4}$ restricted to $\G_{2,4}$ is 
	\begin{equation*}
		\mathrm{\Pi}_r(\G_{2,4})\coloneqq\spann\bigl\{ \varphi^\lambda_l : \lambda_1+\lambda_2\leq r,\; l=1,\ldots,Z(\lambda) \bigr\}.
	\end{equation*}
	It contains all eigenfunctions $\varphi^\lambda_l$ with $4(\lambda_1^2+\lambda_2^2+\lambda_1)<2(r+1)(r+2)$, cf.~\cite[Thm.~5]{Breger:2016vn}.
	
	For $H^s(\G_{2,4})$ with $s= 5/2$, we chose the kernel
	\begin{equation} \label{kernel:grass}
		K(x,y) = \sum_{\lambda_{1}\ge \lambda_{2}\ge 0}  \bigl(1+\lambda_{1}^{2} + \lambda_{2}^{2}\bigr)^{-\frac{5}{2}} \sum_{l=1}^{Z(\lambda)}\varphi_{l}^{\lambda}(x) \overline{\varphi_{l}^{\lambda}(y)}.
	\end{equation}

	\begin{remark}\label{rem:Grassi}
		It is well-known that $\S^2\times\S^2$ is a double covering of $\G_{2,4}$. 
		More precisely, there is an isometric one-to-one mapping 
		$P \colon \S^2\times\S^2/ \{\pm 1\} \to \G_{2,4}$ given by
		\[
		P(u,v) = P(-u,-v)
		\coloneqq \frac12 
		\begin{pmatrix}
			1+ u^\tT v & -( u \times v)^\tT\\
			- u \times v & u v^\tT + v u^\tT + (1-u^\tT v) I_3
		\end{pmatrix},
		\]
		cf.~\cite{Dick:2019sy}. Moreover, the $\varphi^\lambda_l$ are essentially tensor products of spherical harmonics, which enables transferring the non-equispaced fast Fourier transform from $\S^2\times\S^2$ to $\G_{2,4}$, see \cite{Dick:2019sy} for details. 
	\end{remark}

\section*{Acknowledgments}
	Part of this research was performed while all authors were visiting the Institute for Pure and Applied Mathematics (IPAM)
	during the long term semester on ``Geometry and Learning from 3D Data and Beyond'' 2019,
	which was supported by the National Science Foundation (Grant No. DMS-1440415).   
	Funding by the German Research Foundation (DFG) with\-in the project STE 571/13-1 and with\-in the RTG 1932, project area P3, and by the Vienna Science and Technology Fund (WWTF) within the project VRG12-009 is gratefully acknowledged.

\bibliographystyle{abbrv}
\bibliography{references_clean}
\end{document}